\documentclass{amsart}
\usepackage{amssymb}
\usepackage{amsfonts}
\usepackage{amsmath}
\usepackage{amsthm, amsmath, amssymb, enumerate}
\usepackage{amssymb, enumitem}
\usepackage{bbm}
\usepackage{amscd}
\usepackage[all]{xy}
\usepackage[hidelinks]{hyperref}
\usepackage{amsmath}
\usepackage{amssymb}
\usepackage{amsthm}
\usepackage[sort&compress,numbers]{natbib}
\usepackage[left=2cm,right=2cm,bottom=3cm,top=3cm]{geometry}
\usepackage{tikz}
\usepackage{tikz-cd}
\usepackage{wrapfig}

\newtheorem{theorem}{Theorem}[section]
\newtheorem{maintheorem}{Theorem}
\newtheorem{proposition}[theorem]{Proposition}
\newtheorem{lemma}[theorem]{Lemma}
\newtheorem{corollary}[theorem]{Corollary}
\theoremstyle{definition}
\newtheorem*{claim*}{Claim}
\newtheorem{definition}[theorem]{Definition}

\AtBeginDocument{   \def\MR#1{}}

\begin{document}
\title{Projective length, phantom extensions, and the structure of flat
modules}
\author{Matteo Casarosa}
\address{Institut de Math\'{e}matiques de Jussieu - Paris Rive Gauche
(IMJ-PRG), Universit\'{e} Paris Cit\'{e}, B\^{a}timent Sophie Germain, 8
Place Aur\'{e}lie Nemours, 75013 Paris, France, and Dipartimento di
Matematica, Universit\`{a} di Bologna, Piazza di Porta S.\ Donato, 5, 40126
Bologna, Italy}
\email{matteo.casarosa@unibo.it}
\email{matteo.casarosa@imj-prg.fr}
\urladdr{https://webusers.imj-prg.fr/~matteo.casarosa/}
\urladdr{https://www.unibo.it/sitoweb/matteo.casarosa/en}
\author{Martino Lupini}
\address{Dipartimento di Matematica, Universit\`{a} di Bologna, Piazza di
Porta S. Donato, 5, 40126 Bologna,\ Italy}
\email{martino.lupini@unibo.it}
\urladdr{http://www.lupini.org/}
\thanks{The authors were partially supported by the Marsden Fund Fast-Start
Grant VUW1816 and the Rutherford Discovery Fellowship VUW2002
\textquotedblleft Computing the Shape of Chaos\textquotedblright\ from the
Royal Society of New Zealand, the Starting Grant 101077154 \textquotedblleft
Definable Algebraic Topology\textquotedblright\ from the European Research
Council, the Universit\'{e} Paris Cit\'{e}, the Gruppo Nazionale per le
Strutture Algebriche, Geometriche e le loro Applicazioni (GNSAGA) of the
Istituto Nazionale di Alta Matematica (INDAM), and the University of Bologna.%
}
\subjclass[2000]{Primary 03E15, 13D07; Secondary 13D09, 13F05}
\keywords{Pr\"{u}fer domain, Dedekind domain, flat module, torsion-free
abelian group, $\mathrm{Ext}$, $\mathrm{lim}^1$, derived functor,
descriptive set theory, Polish group, Borel complexity, phantom morphism,
phantom extension}
\date{\today }

\begin{abstract}
We consider the natural generalization of the notion of the order of a
phantom map from the topological setting to triangulated categories. When
applied to the derived category of the category of countable flat (i.e.,
torsion-free) modules over a countable Dedekind domain $R$, this yields a
notion of\emph{\ phantom extension} of order $\alpha <\omega _{1}$. We
provide a complexity-theoretic characterization of the module $\mathrm{Ph}%
^{\alpha }\mathrm{Ext}\left( C,A\right) $ of phantom extensions of order $%
\alpha $, as the smallest submodule of $\mathrm{Ext}\left( C,A\right) $ of
complexity $\boldsymbol{\Pi }_{1+\alpha +1}^{0}$ with respect to the
structure of \emph{phantom Polish module} on $\mathrm{Ext}\left( C,A\right) $
obtained by considering it as an object of the left heart of the
quasi-abelian category of Polish modules. We use this characterization to
prove the following \emph{Dichotomy Theorem}: either all the extensions of a
countable flat module $A$ are trivial (which happens precisely when $A$ is
divisible) or $A$ has phantom extensions of arbitrarily high order.

By producing canonical phantom projective resolutions of order $\alpha $, we
prove that phantom extensions of order $\alpha $ define on the category of
countable flat modules an exact structure $\mathcal{E}_{\alpha }$ that is
hereditary with enough projectives, and the functor $\mathrm{Ph}^{\alpha }%
\mathrm{Ext}$ is the derived functor of $\mathrm{Hom}$ with respect to $%
\mathcal{E}_{\alpha }$. Furthermore, $\left( \mathcal{E}_{\alpha }\right)
_{\alpha <\omega _{1}}$ is a strictly decreasing chain of nontrivial exact
structures on the category of countable flat modules. We prove a Structure Theorem
characterizing the objects of the class $\mathcal{P}_{\alpha }$ of countable
flat modules that have \emph{projective length at most }$\alpha $ (i.e., are 
$\mathcal{E}_{\alpha }$-projective) as the direct summands of colimits of
presheaves of finite flat modules over well-founded forests of rank $%
1+\alpha $ regarded as ordered sets. In a companion paper, we obtain an
analogous characterization of the \emph{torsion} modules of projective
length at most $\alpha $, which are precisely the reduced torsion modules of
Ulm length at most $\alpha $. Thus, the Structure Theorem for $\mathcal{P}%
_{\alpha }$ can be seen as the first analogue in the flat case of the
classical Ulm Classification Theorem for torsion modules. We also show that
each class $\mathcal{P}_{\alpha }$ is a Borel subset of the Polish space of
countable flat modules, while their union $\mathcal{P}_{\infty }$ is
coanalytic and not Borel.

We also prove a Structure Theorem applicable to an arbitrary countable fat
module $C$, expressing it as a colimit of a chain $\left( \sigma _{\alpha
}C\right) _{\alpha <\omega _{1}}$ of recursively defined characteristic pure
submodules. The submodule $\sigma _{\alpha }C$ is the largest submodule of $%
C $ of $R$-projective length at most $\alpha $, in the sense that $\mathrm{Ph%
}^{\alpha }\mathrm{Ext}\left( \sigma _{\alpha }C,R\right) =0$. The
extensions of the quotient $\partial _{\alpha }C:=C/\sigma _{\alpha }C$ by $%
R $ parametrize phantom extensions of order $\alpha $ of $C$ by $R$ via the
canonical isomorphism%
\begin{equation*}
\mathrm{Ph}^{\alpha }\mathrm{Ext}\left( C,R\right) \cong \mathrm{Ext}\left(
\partial _{\alpha }C,R\right)
\end{equation*}%
induced by the quotient map.
\end{abstract}

\maketitle

%\urladdr{http://www.logibo.org/}

\renewcommand{\themaintheorem}{\Alph{maintheorem}}

\section{Introduction\label{Section:introduction}}

The subject of homological algebra can be traced back to the seminal paper
\textquotedblleft Group Extensions and Homology\textquotedblright\ \cite%
{eilenberg_group_1942} where Eilenberg and MacLane introduced the group $%
\mathrm{Ext}\left( C,A\right) $ parametrizing abelian group extensions of an
abelian group $C$ by $A$. In this work, they were interested in the natural 
\emph{topology }this group is endowed with, coming from the canonical Polish
topology on the cochain complex that has $\mathrm{Ext}$ as cohomology group.
In terms of this topology, they characterized the submodule $\mathrm{PExt}%
\left( C,A\right) $ parametrizing pure extensions of $C$ by $A$ as the
closure of the trivial submodule, a manifestation of the fact that the
topology on $\mathrm{Ext}$ fails in general to be Hausdorff. This failure is
drastic when $C$ is torsion-free, in which case all extensions are pure, and
the topology is trivial.

In this work, we refine the original analysis of Eilenberg and MacLane by
considering additional structure on $\mathrm{Ext}$ obtained by regarding it
as an object of the \emph{left heart }of the category of abelian Polish
groups. The objects of this category can be regarded as abelian \emph{groups
with a Polish cover}, which are groups explicitly presented as a quotient $%
G/N$ where $G$ is an abelian Polish group and $N$ is a \emph{Polishable
subgroup }of $G$, i.e., a (not necessarily closed) subgroup that is also
Polish, such that the inclusion map $N\rightarrow G$ is continuous. This
structure allows one to define the \emph{Borel complexity }of a subgroup $%
H/N\subseteq G/N$ as the Borel complexity of $H$ within the Polish group $G$%
. It is proved by Solecki in \cite{solecki_polish_1999}---see also \cite%
{lupini_complexity_2024,lupini_looking_2024}---that a group with a Polish
cover $G/N$ has a canonical chain of subgroups $s_{\alpha }\left( G/N\right) 
$ indexed by countable ordinals, called \emph{Solecki subgroups}, where $%
s_{\alpha }\left( G/N\right) $ is characterized as the smallest subgroup of $%
G/N$ of Borel complexity $\boldsymbol{\Pi }_{1+\alpha +1}^{0}$. The least $%
\alpha $ for which $s_{\alpha }\left( G/N\right) $ is trivial (the \emph{%
Solecki length }of $G$) corresponds to the Borel rank of $N$ in $G$. As it
happens, for $\alpha =0$ one recovers the closure of the trivial subgroup
with respect to the quotient topology. We say that $G/N$ is a \emph{phantom
Polish group }when $N$ is dense in $G$, as in this case $G/N$ has
\textquotedblleft no classical trace\textquotedblright .

We generalize the Eilenberg--MacLane Theorem, working in the more general
context of modules over a Pr\"{u}fer domain. Pr\"{u}fer domains are a very
well-studied class of rings, which admits several equivalent
characterizations \cite{bazzoni_prufer_2006}. In particular, a domain is Pr%
\"{u}fer if and only if each of its finitely-generated ideals is projective
(as a module). This entails that a module is flat if and only if it is
torsion-free, and that every finitely-generated flat module is projective. Pr%
\"{u}fer domains generalize Dedekind domains, which are precisely the
Noetherian Pr\"{u}fer domains. In particular, every PID is a Pr\"{u}fer
domain.

We consider the modules $\mathrm{Ext}\left( C,A\right) $ for countable flat
modules $C$ and $A$ over a Pr\"{u}fer domain as phantom Polish modules,
characterizing the Solecki submodule $s_{\alpha }\mathrm{Ext}\left(
C,A\right) $ as the submodule $\mathrm{Ph}^{\alpha }\mathrm{Ext}\left(
C,A\right) $ parametrizing the \emph{phantom }(or pure)\emph{\ extensions of
order }$\alpha $. (Pure extensions correspond to the case $\alpha =0$.) This
concept originates from topology, as it was initially defined for phantom
maps between CW complexes \cite{ha_higher_2003}. The notion of phantom map
has been generalized to an arbitrary triangulated category in \cite%
{christensen_ideals_1998}. The \emph{order} of a phantom map also admits a
natural analogue in a triangulated context. In particular, by regarding $%
\mathrm{Ext}\left( C,A\right) $ as the module of morphisms $C\rightarrow
A[1] $ in the derived category of countable flat modules, one obtains the
notion of phantom extension of order $\alpha $. More generally, we define
the $\alpha $-th \emph{phantom subfunctor }of a cohomological functor on a
triangulated category, which yields the submodule of phantom morphisms of
order $\alpha $ when applied to $\mathrm{Hom}$ with a fixed target. The
characterization of pure extensions in terms of existence of lifts of
elements of the same order admits a natural generalization in terms of
\textquotedblleft higher order lifts\textquotedblright , characterizing
phantom\ (or pure) extensions of order $\alpha $.

We identify $\mathrm{Ph}^{\alpha }\mathrm{Ext}$ for countable flat modules
as the derived functor $\mathrm{Ext}_{\mathcal{E}_{\alpha }}$ of $\mathrm{Hom%
}$ with respect to a canonical exact structure $\mathcal{E}_{\alpha }$ on
the category $\mathbf{Flat}\left( R\right) $ of flat modules. This exact
structure $\mathcal{E}_{\alpha }$ can be identified as the exact structure 
\emph{projectively generated} by a class $\mathcal{S}_{\alpha }$ of
countable flat modules. Recall that for a category $\mathcal{C}$, a presheaf
of finite flat modules over $\mathcal{C}$ is a contravariant functor from $%
\mathcal{C}$ to the category $\mathbf{FinFlat}\left( R\right) $ of finite
flat modules. Particularly, this applies when $\mathcal{C}$ is an ordered
set, such as a rooted tree with the order defined for distinct nodes $x,y$, $%
x\prec y$ if and only if $y$ belongs to the unique path from the root to $x$%
. By a \emph{forest}, we mean a disjoint union of rooted trees, with the
induced order.

The classes $\mathcal{S}_{\alpha }$ for $\alpha <\omega _{1}$ are explicitly
described as the colimits of presheaves of finite flat modules over \emph{%
well-founded tree of rank }$\alpha $, when $\alpha $ is a successor, or less
than $\alpha $, when $\alpha $ is a limit. We prove that the exact category $%
\left( \mathbf{Flat}\left( R\right) ,\mathcal{E}_{\alpha }\right) $ is
hereditary with enough projectives. Furthermore, when $R$ is a Dedekind
domain, $\left( \mathcal{E}_{\alpha }\right) _{\alpha <\omega _{1}}$ is a
strictly decreasing chain of nontrivial exact structures on $\mathbf{Flat}%
\left( R\right) $.

We say that a countable flat module $C$ has \emph{projective length }at most 
$\alpha $ if it is $\mathcal{E}_{\alpha }$-projective, i.e., $\mathrm{Ph}%
^{\alpha }\mathrm{Ext}\left( C,A\right) =0$ for every countable flat module $%
A$, and we let $\mathcal{P}_{\alpha }$ be the class of $\mathcal{E}_{\alpha
} $-projective objects. We prove the following \emph{Structure Theorem }for $%
\mathcal{P}_{\alpha }$:

\begin{maintheorem}
\label{Theorem:Structure}Let $R$ be a countable Dedekind domain. Suppose
that $C$ is a countable flat module. Then $C$ has projective length at most $%
\alpha $ if and only if it is a direct summand of a colimit of a presheaf of
finite flat modules over a countable well-founded forest of rank $\alpha $.
\end{maintheorem}

In the companion paper \cite{lupini_phantom_2025-1}, an analogous result is
obtained in the torsion case, showing that a countable \emph{torsion} module
has projective length at most $\alpha $ if and only it is a colimit presheaf
of finite \emph{torsion }modules over a well-founded forest of rank $\alpha $%
. In this case, this is also equivalent to being a reduced countable torsion
module of Ulm length at most $\alpha $. Thus, the notion of projective
length can be seen as an extension to the flat case of the Ulm length of
torsion modules. (The usual definition of the Ulm length in terms of the Ulm
submodules is not applicable in the flat case since the Ulm submodules of a
reduced flat module are all trivial.) Furthermore, Theorem \ref%
{Theorem:Structure} can be seen as the analogue in the flat case of the
classical Ulm Classification Theorem of reduced torsion modules of Ulm
length at most $\alpha $.

We also obtain another\ Structure Theorem applicable to arbitrary countable
flat modules. Let us say that a countable flat module $C$ has $A$-projective
length at most $\alpha $ for some other countable flat module $A$ if and
only if $\mathrm{Ph}^{\alpha }\mathrm{Ext}\left( C,A\right) =0$. For each $%
\alpha <\omega _{1}$, a countable flat module $C$ admits a canonical largest
submodule $\sigma _{\alpha }C$ of $R$-projective length at most $\alpha $.
Such submodules can be explicitly defined by recursion on $\alpha $, and are
characterized by the isomorphism%
\begin{equation*}
\mathrm{Ext}\left( \partial _{\alpha }C,R\right) \cong \mathrm{Ph}^{\alpha }%
\mathrm{Ext}\left( C,R\right)
\end{equation*}%
induced by the quotient map $C\rightarrow \partial _{\alpha }C=C/\sigma
_{\alpha }C$. The $R$-projective length of $C$ is then the least $\alpha $
for which $\partial _{\alpha }C$ is trivial. This provides a description of $%
C$ as colimit of the (eventually constant, continuous at limits) chain of
pure submodules $\left( \sigma _{\alpha }C\right) _{\alpha <\omega _{1}}$
such that $\sigma _{\alpha +1}C/\sigma _{\alpha }C$ has $R$-projective
length at most $1$ for every $\alpha <\omega _{1}$.

We also isolate a natural class of countable flat modules---here named \emph{%
extractable}---which are recursively constructed starting from the trivial
one by taking countable direct sums, direct summands, and extensions of
finite-rank modules. Letting the \emph{extractable length} of such a
countable flat module $C$ to be the index of the stage at which it is
obtained, one has that the extractable length coincides with the projective
length and the $R$-projective length of $C$, and is simply called the \emph{%
length} of $C$.

For a countable flat module $A$, we consider those extractable modules that
are obtained as above using building blocks $F$ that are not only
finite-rank, but also satisfy $\mathrm{Hom}\left( F,A\right) =0$. We call
these extractable modules $A$-\emph{independent}. For such modules, the
length is also equal to the $A$-projective length. Using modules of this
form, we prove the following Dichotomy Theorem concerning extensions of a
given countable flat module.

\begin{maintheorem}
Let $R$ be a countable Dedekind domain. Suppose that $A$ is a countable flat
module. Then either $A$ is injective, and hence $\mathrm{Ext}\left(
-,A\right) =0$, which happens precisely when $A$ is divisible, or for every $%
\alpha <\omega _{1}$ there exists an $A$-independent extractable countable
flat module of length $\alpha $. In particular, in this second case $A$
admits phantom extensions of arbitrarily high order.
\end{maintheorem}

Considering the complexity-theoretic characterization of phantom extensions
provides a reformulation of the Dichotomy Theorem in terms of the complexity
of classifying extensions of countable flat modules over a Dedekind domain
in the sense of\ Borel complexity theory \cite{gao_invariant_2009}. Thus, if 
$A$ is a countable flat module, then either for all countable flat modules $%
C $ the relation of isomorphism of extensions of $C$ by $A$ is trivial
(which happens precisely when $A$ is divisible), or there exist countable
flat modules $C$ such that the relation of isomorphism of extensions of $C$
by $A$ has arbitrarily high (potential) Borel complexity. In fact, in the
latter case, for each possible potential complexity class $\Gamma $ of a
Borel orbit equivalence relation that is classifiable by countable
structures, there exists a countable flat module $C_{\Gamma }$ such that the
relation of isomorphism of extensions of $C_{\Gamma }$ by $A$ has potential
complexity class $\Gamma $.

We also prove that for every $\alpha <\omega _{1}$, the class $\mathcal{P}%
_{\alpha }$ of countable flat modules of projective length at most $\alpha $
is a Borel subset of the Polish space of all countable flat modules; see
Theorem \ref{Therem:complexity-classes}(2). However, the union $\mathcal{P}%
_{\infty }$ of the classes $\mathcal{P}_{\alpha }$ for $\alpha <\omega _{1}$
is a co-analytic set that is not Borel; Theorem \ref%
{Therem:complexity-classes}(3).

The proofs of our main results hinge on the canonical contravariant
equivalence 
\begin{equation*}
\mathrm{colim}_{n\in \omega }C_{n}\mapsto \left( \mathrm{Hom}\left(
C_{n},R\right) \right) _{n\in \omega }
\end{equation*}%
between the category of countable flat modules and the category of \emph{%
towers of finite} (i.e., finitely-generated) \emph{flat modules}. Via this
equivalence, the functor \textrm{Ext}$\left( -,R\right) $ corresponds to the
functor $\mathrm{lim}^{1}$. We analyze the latter regarded as a functor to
the category of phantom Polish modules, characterizing the $\alpha $-th
Solecki submodule of $\mathrm{lim}^{1}\boldsymbol{A}$ as $\mathrm{lim}^{1}%
\boldsymbol{A}_{\alpha }$.\ Here, $\boldsymbol{A}_{\alpha }$ is the $\alpha $%
-th \emph{derived tower} of $\boldsymbol{A}$, obtained via the operation of
derivative applied to a combinatorial tree canonically associated with $%
\boldsymbol{A}$. We also observe that the functor $\mathrm{Ext}\left(
-,R\right) $ provides an equivalence between the category of countable flat
modules that are \emph{coreduced }(have no nonzero homomorphisms to $R$) and
the category of phantom pro-finiteflat Polish modules---which are phantom
Polish modules of the form $G/N$ where both $G$ and $N$ are inverse limit of
towers of finite flat modules. Whence, $\mathrm{lim}^{1}$ provides an
equivalence between the category of towers of countable finite flat modules
that are \emph{reduced }(have trivial inverse limit) and the category of
phantom pro-finiteflat modules.

The main results of this work have a host of implications for homological
invariants from algebraic, complex, and coarse geometry, topology, and
operator algebras, such as \v{C}ech cohomology \cite%
{eilenberg_group_1942,bergfalk_definable_2024-1}, coarse cohomology \cite%
{roe_coarse_1993}, Lie group and Lie algebra cohomology \cite%
{vogan_unitary_2008,januszewski_hausdorffness_2025,kyed_topologizing_2016},
discrete, measurable, and continuous cohomology of topological groups and
groupoids \cite%
{austin_euclidean-valued_2018,austin_continuity_2013,moore_group_1976,moore_group_1976-1,castellano_rational_2016,castellano_finiteness_2020}%
, cohomology of Banach algebras \cite%
{kamowitz_cohomology_1962,johnson_cohomology_1972}, $\mathrm{Ext}$ of Banach
spaces \cite{cabello_sanchez_homological_2023}, cohomology of group actions
on standard probability spaces \cite%
{kechris_global_2010,schmidt_cocycles_1977,schmidt_asymptotically_1980,schmidt_amenability_1981}
von Neumann algebras \cite{popa_some_2006,kawahigashi_cohomology_1991}
C*-algebras \cite{herman_models_1983,izumi_finite_2004,izumi_finite_2004-1}
and Banach spaces \cite%
{shalom_harmonic_2004,martin_free_2010,martin_first_2007}, Dolbeault
cohomology \cite%
{coltoiu_separation_2005,kazama_cohomology_1990,gunning_analytic_2009,chakrabarti_some_2015,silva_rungescher_1978}%
, and KK-theory \cite%
{matui_ext_2001,kishimoto_ext_1998,schochet_fine_2001,schochet_fine_2002,schochet_fine_2005}%
. All these invariants can be described as groups with a Polish cover. The
study of the corresponding quotient topology, as well as the problem of
determining when it is Hausdorff, or more generally to compute the closure
of the trivial subgroup, has attracted in each of these cases a lot of
attention, as the references above show. A group with a Polish cover is
Hausdorff precisely when it is in fact a Polish group or, equivalently, it
has Solecki length $0$. More generally, the closure of the trivial subgroup
is precisely the Solecki subgroup of index $0$. Thus, the problem of
determining the Solecki length, which can be an arbitrary countable ordinal,
can be seen as a wide-reaching generalization of the Hausdorffness problem.
In turn, the problem of describing the Solecki subgroup of index $\alpha $
for an arbitrary countable ordinal $\alpha $ subsumes the problem of
describing the closure of the trivial subgroup as the particular case when $%
\alpha =0$.

In many cases the invariants mentioned above can be expressed in terms of $%
\mathrm{Ext}$ via so-called Universal Coefficient Theorems, thus allowing
one to characterize their \emph{phantom }cohomological subfunctors as those
corresponding to $\mathrm{Ph}^{\alpha }\mathrm{Ext}$. Alternatively, one can
consider Milnor exact sequences expressing the closure of $\left\{ 0\right\} 
$ in these invariants as a $\mathrm{lim}^{1}$ of a tower of countable
modules, whose derived towers correspond via $\mathrm{lim}^{1}$ to their
phantom subfunctors. This perspective allows one to refine the classical
approach by providing for these invariants a canonical family of subfunctors
indexed by countable ordinals.

This article is divided into 13 sections, including this introduction. In
Section \ref{Section:categories} we recall some necessary background from
category theory, including the notion of \emph{(quasi-)abelian category} and
the construction of\emph{\ derived categories} and \emph{derived functors}.
In Section \ref{Section:countable} we present some results from commutative
algebra concerning \emph{Pr\"{u}fer and Dedekind domains} and their \emph{%
modules}. Fundamental properties of the class of inverse limits of towers
countable modules (\emph{pro-countable Polish modules}) and its notable
subcategories are established in Section \ref{Section:pro-countable}.
Building on this analysis, \emph{modules with a Polish cover}, phantom
Polish modules, and their \emph{Solecki submodules}, as previously defined
in \cite{bergfalk_definable_2024,lupini_looking_2024}, are introduced in
Section \ref{Section:modules-polish-cover}.

This work considers various notions of \textquotedblleft
length\textquotedblright\ and \textquotedblleft rank\textquotedblright .
Ordinal-valued ranks are a well-studied topic in descriptive set theory,
particularly the notion of \emph{co-analytic rank} on a co-analytic set. We
recall in Section \ref{Section:ranks} these notions and fundamental results
concerning them. The definition and fundamental properties of the first
derived functor $\mathrm{lim}^{1}$ of the limit functor for towers of
countable modules are recalled in Section \ref{Section:lim1}. The definition
of \emph{derived tower} and the characterization of the Solecki subfunctors
of $\mathrm{lim}^{1}$ in terms of derived towers are also included in this
section.

\emph{Phantom morphisms} as previously defined by Christensen \cite%
{christensen_ideals_1998} are introduced in Section \ref{Section:phantom}.
Here, we define the analogue for morphisms in a triangulated category of the 
\emph{order} of a phantom map in the topological setting. More generally, we
consider the \emph{phantom subfunctors} of a given cohomological functor on
a triangulated category---phantom morphisms being recovered in the
particular case of corepresentable functors.

The derived category of the category of countable (flat) modules, and the
derived functor $\mathrm{Ext}$ of $\mathrm{Hom}$ are introduced in Section %
\ref{Section:Ext}. Here the \emph{order of a phantom extension} is
introduced as a particular instance of the order of a phantom morphism in
the triangulated category. The \emph{projective length} and $A$-projective
length of a countable module are defined as relaxations of projectivity in
terms of the phantom subfunctors $\mathrm{Ph}^{\alpha }\mathrm{Ext}$ of $%
\mathrm{Ext}$. A characterization of modules of projective length at most
one and $R$-projective length at most one is also obtained.

In Section \ref{Section:phantom-resolutions} we produce phantom projective
resolutions of order $\alpha $. We realize $\mathrm{Ph}^{\alpha }\mathrm{Ext}
$ as a derived functor of $\mathrm{Hom}$ with respect to the exact structure 
$\mathcal{E}_{\alpha }$ projectively generated by the class $\mathcal{S}%
_{\alpha }$ of colimits of presheaves of finite flat modules over
well-founded rooted trees of rank $\alpha $, when $\alpha $ is a successor,
or less than $\alpha $, when $\alpha $ is limit. We also completely
characterize the countable flat modules that have projective length at most $%
\alpha $, i.e., are $\mathcal{E}_{\alpha }$-projective, as the direct
summands of colimits of presheaves of finite flat modules over well-founded
forests of rank $\alpha $.

In Section \ref{Section:higher}, higher order phantom extensions are further
analyzed. In particular, we construct for every countable flat module a
canonical chain of submodules characterized by their projective length,
providing the first \emph{Structure Theorem} for arbitrary countable flat
modules over a Pr\"{u}fer domain. \emph{Extractable modules} and the
corresponding length are also introduced here, and their main homological
properties are established.

Section \ref{Section:constructions-Dedekind} provides explicit constructions
of phantom extensions in the case of modules over countable Dedekind
domains. For every countable flat module $A$ that admits nontrivial
extensions (i.e., it is not divisible) we produce extensions of arbitrarily
high order. We conclude in Section \ref{Section:complexity} by isolating the
complexity-theoretic content of phantom extensions, in terms of the \emph{%
Borel complexity} of sets and equivalence relations as studied in
descriptive set theory and invariant complexity theory. In particular, we
notice how the main results of the paper can be seen as offering a \emph{%
Dichotomy Theorem} for the complexity of isomorphism of extensions of a
given countable flat module $A$, a classification problem that is either
trivial or has arbitrarily high complexity.

Notationally, we denote by $\mathbb{N}$ the set of positive integers \emph{%
excluding zero}. The set of nonnegative integers is denoted by $\omega $ and
identified with the first infinite ordinal. A nonnegative integer $n\in
\omega $ is identified with the set $\left\{ 0,1,\ldots ,n-1\right\} $ of
its predecessors. We let $\omega _{1}$ be the set of countable ordinals,
which is the first uncountable ordinal. Thus, we write $\alpha <\omega _{1}$
to mean that $\alpha $ is a countable ordinal.

\subsubsection*{Acknowledgments}

We are grateful to Jeffrey Bergfalk, Luigi Caputi, Nicola Carissimi, Dan
Christensen, Alessandro Codenotti, Ivan Di Liberti, Luisa Fiorot, Pietro
Freni, Eusebio Gardella, Fosco Loregian, Nicholas Meadows, Andr\'{e} Nies,
Aristotelis Panagiotopoulos, Luca Reggio, Claude Schochet, and Joseph
Zielinski for several helpful conversations and useful remarks.

\section{Category theory background\label{Section:categories}}

In this section we recall some notions from category theory that will be
used in the rest of the paper.

\subsection{Quasi-abelian categories}

A (locally small) \emph{preadditive category} or $\mathbf{Ab}$-category is a
category $\mathcal{A}$ enriched over the category $\mathbf{Ab}$ of abelian
groups. Thus, for all objects $x,y$ of $\mathcal{A}$, the corresponding
hom-set $\mathrm{Hom}_{\mathcal{A}}\left( x,y\right) $ is an abelian group,
such that composition of arrows is \emph{bilinear}. An \emph{additive
category }\cite[Section VIII.2]{mac_lane_categories_1998} is a preadditive
category that has a terminal object, which is automatically also an initial
object and called a \emph{zero object}, and all binary products, which are
also binary coproducts and called \emph{biproducts}. A functor between
preadditive categories is additive if it preserves biproducts or,
equivalently, it induces group homomorphisms at the level of hom-sets.

An additive category $\mathcal{A}$ is \emph{quasi-abelian }\cite[Section 1.1]%
{schneiders_quasi-abelian_1999}\emph{\ }if it satisfies the following
axioms, together with their duals obtained by reversing all the arrows:

\begin{itemize}
\item $\mathcal{A}$ has all kernels (and hence all finite limits);

\item the class of kernels is closed under pushout along arbitrary arrows.
\end{itemize}

In a quasi-abelian category, an arrow $f:A\rightarrow B$ is called \emph{%
strict }or \emph{admissible }if the canonical arrow $\mathrm{Coim}\left(
f\right) \rightarrow \mathrm{Im}\left( f\right) $ that it induces is an
isomorphism, where $\mathrm{Coim}\left( f\right) =\mathrm{Coker}\left( 
\mathrm{\mathrm{Ker}}\left( f\right) \right) $ and $\mathrm{Im}\left(
f\right) =\mathrm{K\mathrm{er}}\left( \mathrm{Coker}\left( f\right) \right) $%
. This is equivalent to the assertion that one can write $f=m\circ e$ where $%
e$ is a cokernel and $m$ is a kernel \cite[Section 1.1]%
{schneiders_quasi-abelian_1999}. A monic arrow is strict if and only if it
is a kernel, while an epic arrow is strict if and only if it is a cokernel.

An \emph{abelian category} is a quasi-abelian category where every arrow is
strict \cite[Section 1.1]{schneiders_quasi-abelian_1999}. Equivalently, an
additive category is abelian if and only if it satisfies the following
axioms, together with their duals:

\begin{itemize}
\item $\mathcal{A}$ has all kernels (and hence all finite limits);

\item every monic arrow is a kernel.
\end{itemize}

An additive category is called \emph{idempotent-complete }if every
idempotent arrow $p$ in $\mathcal{A}$ (satisfying $p\circ p=p$) has a kernel 
\cite[Definition 6.1]{buhler_exact_2010}. This implies that every idempotent
arrow in $\mathcal{A}$ also has an image \cite[Remark 6.2]{buhler_exact_2010}%
.

\subsection{Exact categories}

Exact categories provide the natural context to develop \emph{relative
homological algebra }and other generalizations of notions concerning
homological algebra and derived functors.

Let $\mathcal{A}$ be an additive category. A kernel-cokernel pair in $%
\mathcal{A}$ is a pair $\left( f,g\right) $ of arrows of $\mathcal{A}$ such
that $f$ is a kernel of $g$ and $g$ is a cokernel of $f$. An \emph{exact
structure }on $\mathcal{A}$ is a class $\mathcal{E}$ of kernel-cokernel
pairs $\left( f,g\right) $, where $f$ is called an \emph{admissible monic }%
of $\mathcal{E}$ or $\mathcal{E}$-inflation and $g$ is called an \emph{%
admissible epic} of $\mathcal{E}$ or $\mathcal{E}$-deflation, satisfying the
following axioms, as well as their duals:

\begin{enumerate}
\item For every object of $\mathcal{A}$, the corresponding identity morphism
is an admissible monic;

\item The class of admissible monics is closed under composition;

\item The pushout of an admissible monic along an arbitrary morphism exists
and yields an admissible monic.
\end{enumerate}

An \emph{exact category }is a pair $\left( \mathcal{A},\mathcal{E}\right) $
where $\mathcal{A}$ is an additive category and $\mathcal{E}$ is an exact
structure on $\mathcal{A}$ \cite[Section 2]{buhler_exact_2010}. The elements
of $\mathcal{E}$ are called the\emph{\ short exact sequences} for the given
exact category. A morphism $f$ of $\mathcal{A}$ is called \emph{admissible}
if it admits a factorization $f=m\circ e$ where $m$ is an admissible monic
and $e$ is an admissible epic \cite[Section 2]{buhler_exact_2010}. In this
case, such a factorization is essentially unique \cite[Lemma 8.4]%
{buhler_exact_2010}.

A sequence 
\begin{equation*}
A^{\prime }\overset{f}{\rightarrow }A\overset{g}{\rightarrow }A^{\prime
\prime }
\end{equation*}%
of admissible arrows with admissible factorizations $f=i\circ p$ and $%
g=j\circ q$ is \emph{exact }or \emph{acyclic} if the sequence%
\begin{equation*}
\bullet \overset{i}{\rightarrow }A\overset{q}{\rightarrow }\bullet
\end{equation*}%
is short-exact. A sequence%
\begin{equation*}
A^{0}\rightarrow A^{1}\rightarrow \cdots \rightarrow A^{n}
\end{equation*}%
for $n\geq 2$ is exact if each of the sequences%
\begin{equation*}
A^{i}\rightarrow A^{i+1}\rightarrow A^{i+2}
\end{equation*}%
for $0\leq i\leq n-2$ are exact.

If $\mathcal{A}$ is a quasi-abelian category, then the collection $\mathcal{E%
}$ of all kernel-cokernel pairs in $\mathcal{A}$ is an exact structure on $%
\mathcal{A}$ \cite[Proposition 4.4]{buhler_exact_2010}. Unless specified
otherwise, we will regard a quasi-abelian category as an exact category with
respect to this (maximal) exact structure.

The natural notion of morphism between exact categories is provided by exact
functors; see \cite[Definition 5.1]{buhler_exact_2010}.

\begin{definition}
An additive functor $F:\mathcal{A}\rightarrow \mathcal{B}$ between exact
categories is \emph{exact }if it maps short exact sequences in $\mathcal{A}$
to short exact sequences in $\mathcal{B}$.
\end{definition}

An additive functor is exact if and only if it preserves kernels and
cokernels of admissible morphisms. An exact functor between exact categories
preserves pushouts along admissible monics, and pullbacks along admissible
epics \cite[Definition 5.1]{buhler_exact_2010}.

\begin{definition}
An additive functor $F:\mathcal{A}\rightarrow \mathcal{B}$ between exact
categories is \emph{left exact }if for every exact sequence%
\begin{equation*}
0\rightarrow A\rightarrow B\rightarrow C\rightarrow 0
\end{equation*}%
in $\mathcal{A}$, the sequence%
\begin{equation*}
0\rightarrow FA\rightarrow FB\rightarrow FC
\end{equation*}%
in $\mathcal{B}$ is exact.
\end{definition}

An additive functor is left exact if and only if it preserves kernels of
admissible arrows. The notion of \emph{right exact }functor is obtained by
duality.

A\emph{\ fully exact subcategory }$\mathcal{B}$ of an exact category $%
\mathcal{A}$ is a full subcategory whose collection of objects is closed
under extensions, in the sense that if 
\begin{equation*}
0\rightarrow B^{\prime }\rightarrow A\rightarrow B^{\prime \prime
}\rightarrow 0
\end{equation*}%
is an exact sequence with $B^{\prime }$ and $B^{\prime \prime }$ in $%
\mathcal{B}$, then $A$ is isomorphic to an object of $\mathcal{B}$ \cite[%
Definition 10.21]{buhler_exact_2010}. We then have that $\mathcal{B}$ is
itself an exact category, where a short exact sequence in $\mathcal{B}$ is a
short exact sequence in $\mathcal{A}$ whose objects are in $\mathcal{B}$.

A \emph{thick }(\emph{quasi-})\emph{abelian }subcategory $\mathcal{B}$ of a
(quasi-)abelian category $\mathcal{A}$ is a fully exact subcategory of $%
\mathcal{A}$ that is also (quasi-)abelian and such that the inclusion
functor $\mathcal{B}\rightarrow \mathcal{A}$ is exact, finitely continuous,
and finitely cocontinuous.

\subsection{Triangulated categories}

A \emph{triangulated category }graded by the translation functor $T$\emph{\ }%
is an additive category $\mathcal{A}$ endowed with an additive automorphism $%
T$ together with a collection of \emph{distinguished triangles }$\left(
X,Y,Z,u,v,w\right) $ where $u\in \mathrm{Hom}^{0}\left( X,Y\right) $, $v\in 
\mathrm{Hom}^{0}\left( Y,Z\right) $, and $w\in \mathrm{Hom}^{1}\left(
Z,X\right) $, satisfying the axioms of triangulated categories \cite[Section
1.1]{verdier_categories_1977}; see also \cite[Definition 10.1.6]%
{kashiwara_categories_2006}. In this definition, we let $\mathrm{Hom}%
^{k}\left( X,Y\right) $ be the group $\mathrm{Hom}_{\mathcal{A}}\left(
X,T^{k}Y\right) $. The composition of $f\in \mathrm{Hom}^{n}\left(
X,Y\right) $ and $g\in \mathrm{Hom}^{m}\left( Y,Z\right) $ is defined to be $%
T^{n}g\circ f\in \mathrm{Hom}^{n+m}\left( X,Z\right) $. This defines a new
category with graded groups as hom-sets.

A functor $F:\mathcal{A}\rightarrow \mathcal{B}$ between triangulated
categories is \emph{triangulated }if it is additive, graded (which means
that $TF$ and $FT$ are naturally isomorphic), and it maps distinguished
triangles to distinguished triangles. A \emph{morphism} $\mu :F\Rightarrow G$
between triangulated functors $\mathcal{A}\rightarrow \mathcal{B}$ is a
natural transformation $\mu $ such that $\mu T\circ \alpha _{F}\cong \alpha
_{G}\circ Tu$ where $\alpha _{F}$ is the isomorphism $TF\cong FT$ and $%
\alpha _{G}$ is the isomorphism $TG\cong GT$ as in the definition of
triangulated functor.

An additive functor $F:\mathcal{A}\rightarrow \mathcal{E}$ from a
triangulated category $\mathcal{A}$ to an abelian category $\mathcal{E}$ is 
\emph{cohomological} if for every distinguished triangle $\left(
X,Y,Z,u,v,w\right) $ of $\mathcal{A}$, the sequence%
\begin{equation*}
FX\overset{Fu}{\rightarrow }FY\overset{Fv}{\rightarrow }FZ
\end{equation*}%
is exact. This yields long exact sequence%
\begin{equation*}
\cdots \rightarrow F^{d}X\overset{F^{d}u}{\rightarrow }F^{d}Y\overset{F^{d}v}%
{\rightarrow }F^{d}Z\overset{F^{d+1}w}{\rightarrow }F^{d+1}X\rightarrow
\cdots
\end{equation*}%
where we set $F^{d}:=F\circ T^{d}$ for $d\in \mathbb{Z}$. Dually, a functor
is \emph{homological }if it is a cohomological functor on the opposite
category.

For example, for objects $X,Y$ of $\mathcal{A}$, the functors $\mathrm{Hom}%
\left( X,-\right) $ and, respectively, $\mathrm{Hom}\left( -,Y\right) $ from 
$\mathcal{A}$ to the category of abelian groups are homological and
cohomological, respectively.

Suppose now that $\mathcal{A}$, $\mathcal{B}$, $\mathcal{C}$ are
triangulated categories, and $F:\mathcal{A}\times \mathcal{B}\rightarrow 
\mathcal{C}$ is a functor. One says that $F$ is \emph{graded} if there exist
natural isomorphisms%
\begin{equation*}
F\left( TX,Y\right) \cong _{\ell _{X,Y}}TF\left( X,Y\right)
\end{equation*}%
and%
\begin{equation*}
F\left( X,TY\right) \cong _{r_{X,Y}}TF\left( X,Y\right)
\end{equation*}%
for $X\in \mathcal{A}$ and $Y\in \mathcal{B}$ such that%
\begin{equation*}
T\ell _{X,Y}\circ r_{TX,Y}=Tr_{X,Y}\circ \ell _{X,TY}\text{;}
\end{equation*}%
see \cite[Definition 10.1.1]{kashiwara_categories_2006}. The functor $F$ is 
\emph{triangulated }if it is graded, as witnessed by natural transformations 
$\ell $ and $r$ as above, and for every distinguished triangle%
\begin{equation*}
X\overset{u}{\rightarrow }Y\overset{v}{\rightarrow }Z\overset{w}{\rightarrow 
}TX
\end{equation*}%
in $\mathcal{A}$ and object $B$ of $\mathcal{B}$,%
\begin{equation*}
F\left( X,B\right) \rightarrow F\left( Y,B\right) \rightarrow F\left(
Z,B\right) \rightarrow TF\left( X,B\right)
\end{equation*}%
is a distinguished triangle in $\mathcal{C}$, where the rightmost arrow is
the composition%
\begin{equation*}
F\left( Z,B\right) \overset{F\left( w,1_{B}\right) }{\rightarrow }F\left(
TX,B\right) \overset{\ell _{X,B}}{\rightarrow }TF\left( X,B\right) \text{,}
\end{equation*}%
and for every distinguished triangle as above in $\mathcal{B}$ and object $A$
of $\mathcal{A}$, 
\begin{equation*}
F\left( A,X\right) \rightarrow F\left( A,Y\right) \rightarrow F\left(
A,z\right) \rightarrow TF\left( A,X\right)
\end{equation*}%
is a distinguished triangle in $\mathcal{C}$, where the rightmost arrow is
the composition%
\begin{equation*}
F\left( A,Z\right) \overset{F\left( 1_{A},w\right) }{\rightarrow }F\left(
A,TX\right) \overset{r_{A,X}}{\rightarrow }TF\left( A,X\right) \text{;}
\end{equation*}%
see \cite[Definition 10.3.6]{kashiwara_categories_2006}.

\subsection{Derived categories}

Suppose that $\left( \mathcal{A},\mathcal{E}\right) $ is an
idempotent-complete exact category. We let \textrm{Ch}$^{b}\left( \mathcal{A}%
\right) $ be the category of bounded complexes over $\mathcal{A}$, and 
\textrm{K}$^{b}\left( \mathcal{A}\right) $ be the category of bounded
complexes of $\mathcal{A}$ whose morphisms are the\emph{\ homotopy classes}
of morphisms of complexes \cite[Section 10]{buhler_exact_2010}. Then we have
that \textrm{K}$^{b}\left( \mathcal{A}\right) $ is a triangulated category
with translation functor $T$ defined by $\left( TA\right) ^{n}=A^{n+1}$ for $%
n\in \mathbb{Z}$. A \emph{strict triangle }in \textrm{K}$^{b}\left( \mathcal{%
A}\right) $ over a morphism $f:A\rightarrow B$ in \textrm{Ch}$^{b}\left( 
\mathcal{A}\right) $ is $\left( A,B,\mathrm{cone}\left( f\right)
,f,i,j\right) $ where $i:B\rightarrow \mathrm{cone}\left( f\right) $ and $j:%
\mathrm{cone}\left( f\right) \rightarrow TA$ are the canonical morphisms of
complexes \cite[Definition 9.2]{buhler_exact_2010}. A distinguished triangle
in \textrm{K}$^{b}\left( \mathcal{A}\right) $ is one that is isomorphic to a
strict triangle. Similarly one defines the categories $\mathrm{K}^{+}\left( 
\mathcal{A}\right) $ and $\mathrm{K}^{-}\left( \mathcal{A}\right) $ of
left-bounded complexes and right-bounded complexes, respectively.

A complex $A$ over $\mathcal{A}$ is \emph{acyclic }if for every $n\in 
\mathbb{Z}$,%
\begin{equation*}
A^{n}\rightarrow A^{n+1}\rightarrow A^{n+2}
\end{equation*}%
is an exact sequence of admissible morphisms. The subcategory $\mathrm{N}%
^{b}\left( \mathcal{A}\right) $ of \textrm{K}$^{b}\left( \mathcal{A}\right) $
spanned by bounded acyclic complexes is \emph{thick }\cite[Corollary 10.11]%
{buhler_exact_2010}, namely it is strictly full and it is closed under
taking direct summands of its objects \cite[Remark 10.10]{buhler_exact_2010}%
. One can thus define the Verdier quotient 
\begin{equation*}
\mathrm{D}^{b}\left( \mathcal{A}\right) :=\frac{\mathrm{K}^{b}\left( 
\mathcal{A}\right) }{\mathrm{N}^{b}\left( \mathcal{A}\right) }\text{.}
\end{equation*}%
By definition, this is obtained as the \emph{category of fractions }\textrm{K%
}$^{b}\left( \mathcal{A}\right) [\Sigma ^{-1}]$. Here, $\Sigma $ is the 
\emph{multiplicative system} of \emph{quasi-isomorphisms} in \textrm{K}$%
^{b}\left( \mathcal{A}\right) $. By definition, a morphism $%
u=[f]:X\rightarrow Y$ is a quasi-isomorphism if and only if it fits into a
distinguished triangle $\left( X,Y,Z,u,v,w\right) $. This is equivalent to
the assertion that $\mathrm{cone}\left( f\right) $ is acyclic; see \cite[%
Definition 10.16 and Remark 10.17]{buhler_exact_2010}. We let $Q_{\mathcal{A}%
}:\mathrm{K}^{b}\left( \mathcal{A}\right) \rightarrow \mathrm{D}^{b}\left( 
\mathcal{A}\right) $ be the canonical quotient functor.

One similarly define the categories $\mathrm{D}^{+}\left( \mathcal{A}\right) 
$ and $\mathrm{D}^{-}\left( \mathcal{A}\right) $ as quotients of $\mathrm{K}%
^{+}\left( \mathcal{A}\right) $ and $\mathrm{K}^{-}\left( \mathcal{A}\right) 
$, respectively.

\subsection{Left hearts of quasi-abelian categories}

Let $\mathcal{A}$ be a quasi-abelian category. A canonical \textquotedblleft
completion\textquotedblright\ of $\mathcal{A}$ to an abelian category is
described in \cite{beilinson_faisceaux_1982,schneiders_quasi-abelian_1999}.
Let us consider $\mathcal{A}$ as an exact category, endowed with its maximal
exact structure. Let $\mathrm{LH}\left( \mathcal{A}\right) $ be the heart of 
\textrm{D}$^{b}\left( \mathcal{A}\right) $ with respect to its canonical
left truncation structure. Explicitly, $\mathrm{LH}\left( \mathcal{A}\right) 
$ is the full subcategory of $\mathrm{D}^{b}\left( \mathcal{A}\right) $
spanned by complexes $A$ with $A^{n}=0$ for $n\in \mathbb{Z}\setminus
\left\{ -1,0\right\} $ and such that the coboundary morphism $%
A^{-1}\rightarrow A^{0}$ is a monomorphism. For brevity, we call $\mathrm{LH}%
\left( \mathcal{A}\right) $ the \emph{left heart }of $\mathcal{A}$.\ The
definition of $\mathrm{LH}\left( \mathcal{A}\right) $ as the heart of 
\textrm{D}$^{b}\left( \mathcal{A}\right) $ with respect to its canonical
left truncation structure yields a canonical \emph{cohomological} functor $%
\mathrm{H}^{0}:\mathrm{D}^{b}\left( \mathcal{A}\right) \rightarrow \mathrm{LH%
}\left( \mathcal{A}\right) $.

It is proved in \cite{schneiders_quasi-abelian_1999} that the left heart of $%
\mathcal{A}$ is an abelian category satisfying the following universal
property, which characterizes it up to equivalence: the inclusion $\mathcal{A%
}\rightarrow \mathrm{LH}\left( \mathcal{A}\right) $ is exact and finitely
continuous, and for any abelian category $\mathcal{M}$ and exact and
finitely continuous functor $F:\mathcal{A}\rightarrow \mathcal{M}$ there
exists an essentially unique exact (and finitely continuous) functor $%
\mathrm{LH}\left( \mathcal{A}\right) \rightarrow \mathcal{M}$ whose
restriction to $\mathcal{A}$ is isomorphic to $F$.

Recall that a \emph{torsion pair }\cite[Definition 2.4]{tattar_torsion_2021}%
---see also \cite{fiorot_quasi-abelian_2021,happel_tilting_1996}---in a
quasi-abelian category $\mathcal{M}$ with torsion class $\mathcal{T}$ and
torsion-free class $\mathcal{F}$ is a pair $\left( \mathcal{T},\mathcal{F}%
\right) $ of full subcategories of $\mathcal{M}$ such that:

\begin{enumerate}
\item for all objects $T$ of $\mathcal{T}$ and $F$ of $\mathcal{F}$, $%
\mathrm{Hom}\left( T,F\right) =0$;

\item for all objects $M$ of $\mathcal{M}$ there exists a short exact
sequence%
\begin{equation*}
0\rightarrow {}_{\mathcal{T}}M\rightarrow M\rightarrow M_{\mathcal{F}%
}\rightarrow 0
\end{equation*}%
where $_{\mathcal{T}}M$ is in $\mathcal{T}$ and $M_{\mathcal{F}}$ is in $%
\mathcal{F}$.
\end{enumerate}

In this case, we say that $\mathcal{A}$ is an extension of $\mathcal{F}$ by $%
\mathcal{T}$ and write%
\begin{equation*}
0\rightarrow \mathcal{T}\rightarrow \mathcal{A}\rightarrow \mathcal{F}%
\rightarrow 0\text{.}
\end{equation*}%
We have that an object $X$ of $\mathcal{M}$ is in $\mathcal{T}$ if and only
if $\mathrm{Hom}\left( X,F\right) =0$ for all $F$ in $\mathcal{F}$, and it
is in $\mathcal{F}$ if and only if $\mathrm{Hom}\left( T,X\right) =0$ for
all $T$ in $\mathcal{T}$. Furthermore, $\mathcal{T}$ is closed under
quotients and $\mathcal{F}$ under subobjects. The torsion pair $\left( 
\mathcal{T},\mathcal{F}\right) $ is \emph{hereditary} when $\mathcal{T}$ is
also closed under subobjects, in which case $\mathcal{T}$ is a thick
quasi-abelian subcategory of $\mathcal{A}$.

If $\left( \mathcal{T},\mathcal{F}\right) $ is a torsion pair for $\mathcal{M%
}$, then $\mathcal{T}$ and $\mathcal{F}$ are fully exact quasi-abelian
subcategories of $\mathcal{M}$ \cite[Theorem 2]{rump_almost_2001}.
Conversely, if $\mathcal{A}$ is a quasi-abelian category, let $\kappa :%
\mathrm{LH}\left( \mathcal{A}\right) \rightarrow \mathcal{A}$ be the left
adjoint of the inclusion $\mathcal{A}\rightarrow \mathrm{LH}\left( \mathcal{A%
}\right) $; see \cite[Defintion 1.2.26 and Proposition 1.2.27]%
{schneiders_quasi-abelian_1999}. Thus, we have that%
\begin{equation*}
\mathrm{Hom}_{\mathcal{A}}\left( \kappa \left( X\right) ,B\right) \cong 
\mathrm{Hom}_{\mathrm{LH}\left( \mathcal{A}\right) }\left( X,B\right)
\end{equation*}%
for all objects $B$ of $\mathcal{A}$. Define $\mathrm{Ph}\left( \mathcal{A}%
\right) $ to be the full subcategory of $\mathrm{LH}\left( \mathcal{A}%
\right) $ spanned by the objects $X$ such that $\kappa \left( X\right) =0$.
Then we have that $\left( \mathrm{Ph}\left( \mathcal{A}\right) ,\mathcal{A}%
\right) $ is a torsion pair in $\mathrm{LH}\left( \mathcal{A}\right) $ \cite[%
Theorem 2]{rump_almost_2001}.

\begin{definition}
\label{Definition:phantom-category}Let $\mathcal{A}$ be a quasi-abelian
category. The full subcategory $\mathrm{Ph}\left( \mathcal{A}\right) $ of $%
\mathrm{LH}\left( \mathcal{A}\right) $ defined as above is the \emph{phantom
category }of $\mathcal{A}$, and its object are called phantom objects of $%
\mathcal{A}$.
\end{definition}

\begin{lemma}
\label{Lemma:fully-qa}If $\mathcal{B}$ is a fully exact quasi-abelian
subcategory of $\mathcal{A}$, then the inclusion $\mathcal{B}\rightarrow 
\mathcal{A}$ induces a fully faithful exact functor $\mathrm{LH}\left( 
\mathcal{B}\right) \rightarrow \mathrm{LH}\left( \mathcal{A}\right) $.
\end{lemma}

\begin{proof}
Since the inclusion $\mathcal{B}\rightarrow \mathcal{A}$ is exact and
finitely continuous, the functor $\mathrm{LH}\left( \mathcal{B}\right)
\rightarrow \mathrm{LH}\left( \mathcal{A}\right) $ that it induces is exact.
It remains to prove that the functor $\mathrm{LH}\left( \mathcal{B}\right)
\rightarrow \mathrm{LH}\left( \mathcal{A}\right) $ is fully faithful. We
will use the description of $\mathrm{LH}\left( \mathcal{A}\right) $ as a
localization as in \cite[Defintion 1.2.26 and Proposition 1.2.27]%
{schneiders_quasi-abelian_1999}. Let $\mathrm{LK}\left( \mathcal{A}\right) $
be the full subcategory of $\mathrm{K}\left( \mathcal{A}\right) $ spanned by
complexes $A$ such that $A^{n}=0$ for $n\in \mathbb{Z}\setminus \left\{
-1,0\right\} $ and such that the coboundary morphism $A^{-1}\rightarrow
A^{0} $ is a monomorphism. Then quasi-isomorphisms between such complexes
are precisely the chain maps that induce a square commutative diagram that
is both a pullback and a pushout. The corresponding localization of \textrm{%
LK}$\left( \mathcal{A}\right) $ is equivalent to $\mathrm{LH}\left( \mathcal{%
A}\right) $.

Suppose that $X,Y$ are objects of $\mathrm{LH}\left( \mathcal{B}\right) $.
Let $f:X\rightarrow Y$ be morphism in $\mathrm{LH}\left( \mathcal{A}\right) $%
. Let $\mathcal{R}\left( f\right) $ be the collection of triples $\left(
Z,\sigma ,g\right) $ that \emph{represent }$f$, where $Z$ is an object of $%
\mathrm{LK}\left( \mathcal{A}\right) $, $\sigma :Z\rightarrow X$ is a
morphism in $\mathrm{N}\left( \mathcal{A}\right) $, and $g:Z\rightarrow Y$
is a morphism in $\mathrm{K}\left( \mathcal{A}\right) $. We endow $\mathcal{R%
}\left( f\right) $ with the preorder relation defined by setting $\left(
Z,\sigma ,g\right) \leq \left( Z^{\prime },\sigma ^{\prime },g^{\prime
}\right) $ if and only if there exists a morphism $\eta :Z\rightarrow
Z^{\prime }$ in $\mathrm{K}\left( \mathcal{A}\right) $ such that $\sigma
^{\prime }\eta =\sigma $ and $g^{\prime }\eta =g$. It suffices to prove that
the collection $\mathcal{R}_{\mathcal{B}}\left( f\right) $ triples $\left(
Z,\sigma ,g\right) $ in $\mathcal{R}\left( f\right) $ with $Z$ in $\mathrm{LK%
}\left( \mathcal{B}\right) $ is cofinal.

Suppose that $\left( Z,\sigma ,g\right) $ is in $\mathcal{R}\left( f\right) $%
. After replacing $\mathcal{A}$ with a small quasi-abelian subcategory
containing $\mathcal{A}$, we can assume without loss of generality that $%
\mathcal{A}$ is small. In this case, we have that $\mathrm{LH}\left( 
\mathcal{A}\right) $ is also small. By the Freyd--Mitchell Embedding Theorem 
\cite[Definition 8.3.21]{kashiwara_sheaves_1994}, we can assume that $%
\mathrm{LH}\left( \mathcal{A}\right) $ is a fully exact abelian subcategory
of the category $\mathrm{Mod}_{R}$ of $R$-modules for some ring $R$. Whence, 
$\mathcal{A}$ is a fully exact quasi-abelian subcategory of $\mathrm{Mod}%
_{R} $.

Consider the morphism%
\begin{equation*}
\tau =\sigma \oplus g:Z\rightarrow X\oplus Y\text{.}
\end{equation*}%
After replacing $Z_{1}$ with $Z_{1}\oplus X_{1}\oplus Y_{1}\oplus X_{0}$ and 
$Z_{0}$ with $Z_{0}\oplus X_{1}\oplus Y_{1}\oplus X_{0}$, we can assume that 
$\tau _{1}$ and $\sigma _{0}$ are epimorphisms in $\mathrm{LH}\left( 
\mathcal{A}\right) $ (and hence admissible epimorphisms in $\mathcal{A}$).

After replacing $Z_{1}$ with $\mathrm{Coker}\left( \mathrm{\mathrm{Ker}}%
\left( \tau _{1}\right) \right) $ and $Z_{0}$ with $\mathrm{Coker}\left( 
\mathrm{ker}\left( \tau _{0}\right) \right) $, we can assume that $\tau
_{1}:Z_{1}\rightarrow X_{1}\oplus Y_{1}$ is an isomorphism. Identifying $%
Z_{1}$ with $X_{1}\oplus Y_{1}$ via this isomorphism, we have that 
\begin{equation*}
\mathrm{\mathrm{\mathrm{Ker}}}\left( \sigma _{1}\right) =0\oplus Y_{1}\text{.%
}
\end{equation*}%
Thus, we have a short exact sequence%
\begin{equation*}
0\rightarrow 0\oplus Y_{1}\rightarrow Z_{0}\overset{\sigma _{0}}{\rightarrow 
}X_{0}\rightarrow 0
\end{equation*}%
in $\mathcal{A}$. Since $\mathcal{B}$ is a fully exact subcategory of $%
\mathcal{A}$, this implies that $Z_{0}$ is in $\mathcal{B}$. Thus, we have
that $Z$ is in \textrm{K}$\left( \mathcal{B}\right) $, concluding the proof.
\end{proof}

\subsection{Relative homological algebra}

Suppose that $\mathcal{A}$ is a category. Recall that an arrow $%
r:A\rightarrow B$ is a \emph{retraction} if it has a right inverse, which is
an arrow $g:B\rightarrow A$ satisfying $r\circ g=1_{B}$. A class $\mathbb{E}$
of cokernels in $\mathcal{A}$ is called a proper class of cokernels \cite[%
Section 1]{generalov_relative_1996} if it satisfies the following:

\begin{enumerate}
\item Every retraction is in $\mathbb{E}$;

\item The class $\mathbb{E}$ is closed under composition;

\item The class $\mathbb{E}$ is closed under pullback along arbitrary arrows;

\item If $\sigma ,\tau $ are arrows such that the composition $\sigma \tau $
exists, and $\sigma \tau $ and $\tau $ belong to $\mathbb{E}$, then $\sigma $
belongs to $\mathbb{E}$.
\end{enumerate}

A class $\mathbb{M}$ of kernels is called a proper class of kernels if it is
a proper class of cokernels in the opposite category $\mathcal{A}^{\mathrm{op%
}}$.

\begin{lemma}
Suppose that $\mathcal{E}$ is a class of short exact sequences in a finitely
complete and finitely cocomplete additive category $\mathcal{A}$. The
following assertions are equivalent:

\begin{enumerate}
\item $\left( \mathcal{A},\mathcal{E}\right) $ is an exact category;

\item The class $\mathbb{E}$ of cokernels that appear in $\mathcal{E}$ is a
proper class of cokernels, and the class $\mathbb{M}$ of kernels that appear
in $\mathcal{E}$ is a proper class of kernels.
\end{enumerate}
\end{lemma}

\begin{proof}
(1)$\Rightarrow $(2) It suffices to prove that$\mathbb{\ M}$ is a proper
class of kernels, as the other assertion follows by duality. Since $\mathcal{%
A}$ is in particular idempotent complete, every coretraction is an
admissible monic by \cite[Corollary 7.5]{buhler_exact_2010}. By definition,
we also have that the class $\mathbb{M}$ is closed under composition and
pushout by arbitrary arrows. Suppose now that $\sigma ,\tau $ are arrows
such that the composition $\sigma \tau $ exists, and $\sigma \tau $ and $%
\sigma $ are admissible monics. We need to prove that $\tau $ is an
admissible monic. This follows from the \textquotedblleft Obscure
Axiom\textquotedblright\ of exact categories \cite[Proposition 7.6]%
{buhler_exact_2010}, asserting that if $f,g$ are morphisms in an
idempotent-complete exact category such that $gf$ is defined and is an
admissible monic, then $f$ is an admissible monic.

(2)$\Rightarrow $(1) This is obvious.
\end{proof}

Suppose that $\mathfrak{I}$ is a class of objects in a quasi-abelian
category $\mathcal{A}$ . Let $\mathcal{E}_{\mathfrak{I}}$ be the collection
of kernel-cokernel pairs $A\rightarrow B\rightarrow C$ in $\mathcal{A}$ such
that for every object $X$ of $\mathfrak{I}$ the induced morphism 
\begin{equation*}
\mathrm{Hom}\left( X,B\right) \rightarrow \mathrm{Hom}\left( X,C\right)
\end{equation*}%
is surjective. We say that $\mathcal{E}_{\mathfrak{I}}$ is the exact
structure \emph{projectively generated} by $\mathfrak{I}$.

\begin{proposition}
The exact structure $\mathcal{E}_{\mathfrak{I}}$ projectively generated by a
class of objects $\mathfrak{I}$ in a quasi-abelian category $\mathcal{A}$ is
indeed an exact structure on $\mathcal{A}$.
\end{proposition}

\begin{proof}
We verify the axioms of exact structures. It is clear that the identity
morphisms are $\mathcal{E}_{\mathfrak{I}}$-inflations and $\mathcal{E}_{%
\mathfrak{I}}$-deflations, and that composition of $\mathcal{E}_{\mathfrak{I}%
}$-deflations is an $\mathcal{E}_{\mathfrak{I}}$-deflation.\ Let us prove
that composition of $\mathcal{E}_{\mathfrak{I}}$-inflations is an $\mathcal{E%
}_{\mathfrak{I}}$-inflation. Suppose that $A\rightarrow B\rightarrow B/A$
and $B\rightarrow C\rightarrow C/B$ are kernel-cokernel pairs that belong to 
$\mathcal{E}_{\mathfrak{I}}$. We let $\pi _{B/A}^{B}:B\rightarrow B/A$ and $%
\pi _{C/B}^{C}:C\rightarrow C/B$ be the corresponding $\mathcal{E}_{%
\mathfrak{I}}$-deflations. We claim that $A\rightarrow C\rightarrow C/A$
belongs to $\mathcal{E}_{\mathfrak{I}}$, where $C\rightarrow C/A$ is the
cokernel of $A\rightarrow C$. Let $X$ be an object of $\mathfrak{I}$ and let 
$f:X\rightarrow C/A$ be a morphism. Let $B/A\rightarrow C/A\rightarrow C/B$
be the kernel-cokernel pair induced by $B\rightarrow C\rightarrow C/B$ and
the strict monic $A\rightarrow B$. Let $\pi _{C/B}^{C/A}:C/A\rightarrow C/B$
be the corresponding strict epic and $\iota _{C/A}^{B/A}:B/A\rightarrow C/A$
be the corresponding strict monic. Consider $\pi _{C/B}^{C/A}f:X\rightarrow
C/B$. Since $B\rightarrow C\rightarrow C/B$ is $\mathcal{E}_{\mathfrak{I}}$%
-exact, there exists a morphism $\varphi :X\rightarrow C$ such that $\pi
_{C/B}^{C}\varphi =\pi _{C/B}^{C/A}f$. Consider $\eta :=f-\pi
_{C/A}^{C}\varphi :X\rightarrow C/A$ and observe that $\pi _{C/B}^{C/A}\eta
=\pi _{C/B}^{C/A}f-\pi _{C/B}^{C}\varphi =0$. Therefore we have that $\eta
=\iota _{C/A}^{B/A}\eta _{0}$ for some $\eta _{0}:X\rightarrow B/A$. Since $%
A\rightarrow B\rightarrow B/A$ is $\mathcal{E}_{\mathfrak{J}}$-exact, there
exists a morphism $\tau :X\rightarrow B$ such that $\eta _{0}=\pi
_{B/A}^{B}\tau $.

Define now $g:=\iota _{C}^{B}\tau +\varphi :X\rightarrow C$. Then we have
that 
\begin{eqnarray*}
\pi _{C/A}^{C}g &=&\pi _{C/A}^{C}\iota _{C}^{B}\tau +\pi _{C/A}^{C}\varphi
=\iota _{C/A}^{B/A}\pi _{B/A}^{B}\tau +\pi _{C/A}^{C}\varphi \\
&=&\iota _{C/A}^{B/A}\eta _{0}+\pi _{C/A}^{C}\varphi =\eta +\pi
_{C/A}^{C}\varphi =f\text{.}
\end{eqnarray*}%
This concludes the proof that $A\rightarrow C\rightarrow C/A$ is $\mathcal{E}%
_{\mathfrak{I}}$-exact.

We now show that the pullback of an $\mathcal{E}_{\mathfrak{I}}$-deflation
along an arbitrary morphism is an $\mathcal{E}_{\mathfrak{I}}$-deflation.
Suppose that $A\rightarrow B\rightarrow B/A$ is $\mathcal{E}_{\mathfrak{I}}$%
-exact and that $t:L\rightarrow B/A$ is a morphism. Let $P$ be the pullback
of $\pi _{B/A}^{B}:B\rightarrow B/A$ and $t$ with canonical arrows $%
a:P\rightarrow B$ and $b:P\rightarrow L$. We claim that $b$ is an $\mathcal{E%
}_{\mathfrak{I}}$-deflation. We have that $b$ is a cokernel by the axioms of
quasi-abelian categories. Suppose that $f:X\rightarrow L$ is a morphism, and
let $f_{0}:=tf:X\rightarrow B/A$. Since $A\rightarrow B\rightarrow B/A$ is $%
\mathcal{E}_{\mathfrak{I}}$-exact, we have that there exists an arrow $%
g_{0}:X\rightarrow B$ such that $\pi _{B/A}^{B}g_{0}=f_{0}=tf$. By the
universal property of the pullback, there exists a unique arrow $\xi
:X\rightarrow P$ such that $b\xi =f$ and $a\xi =g_{0}$. This shows that $b$
is an $\mathcal{E}_{\mathfrak{I}}$-deflation.

We now prove that the pushout of an $\mathcal{E}_{\mathfrak{I}}$-inflation
along an arbitrary morphism is an $\mathcal{E}_{\mathfrak{I}}$-inflation.
Suppose that $A\rightarrow B\rightarrow B/A$ is $\mathcal{E}_{\mathfrak{I}}$%
-exact. Let $t:A\rightarrow L$ be a morphism, and $P$ be the pushout of $t$
and $\iota _{B}^{A}:A\rightarrow B$, with canonical arrows $a:L\rightarrow P$
and $b:B\rightarrow P$. By the axioms of quasi-abelian categories, $a$ is a
kernel. Let $c=\mathrm{Coker}\left( a\right) :P\rightarrow Q$. By the
universal property of the pushout, there exists a unique arrow $\sigma
_{0}:P\rightarrow B/A$ such that $\sigma _{0}a=0$ and $\sigma _{0}b=\pi
_{B/A}^{B}$. Since $\sigma _{0}a=0$, by the universal property of the
cokernel there exists a unique arrow $\sigma :Q\rightarrow B/A$ such that $%
\sigma c=\sigma _{0}$. Suppose that $X$ is an object of $\mathfrak{I}$ and $%
f:X\rightarrow Q$ is a morphism. We need to show that there exists a
morphism $g:X\rightarrow P$ such that $cg=f$. Consider the morphism $\sigma
f:X\rightarrow B/A$. Since $A\rightarrow B\rightarrow B/A$ is $\mathcal{E}_{%
\mathfrak{I}}$-exact, there exists a morphism $g:X\rightarrow B$ such that $%
\pi _{B/A}^{B}g=\sigma f$. Thus, we have that $bg:X\rightarrow P$ is a
morphism such that $\sigma cbg=\sigma _{0}bg=\pi _{B/A}^{B}g=\sigma f$. It
remains to prove that $\sigma :Q\rightarrow B/A$ is an isomorphism. To this
purpose, it suffices to prove that the arrow $cb:B\rightarrow Q$ is a
cokernel of $\iota _{B}^{A}:A\rightarrow B$. Suppose that $Z$ is an object
and $\xi :Q\rightarrow Z$ is a morphism such that $\xi cb=0$. Since $c$ is
the cokernel of $a$, we have $\xi ca=0$. Hence, by the universal property of
the pushout $P$, $\xi c=0$. Hence we have that $\xi =0$ since $c$ is an
epimorphism. Suppose now that $\eta :B\rightarrow Z$ is an arrow such that $%
\eta \iota _{B}^{A}=0$. Then considering the arrow $0:L\rightarrow Z$ and
the universal property of the pushout $P$, there exists an arrow $\tau
:P\rightarrow Z$ such that $\tau b=\eta $ and $\tau a=0$. By the universal
property of $c=\mathrm{Coker}\left( a\right) $ there exists an arrow $\xi
:Q\rightarrow Z$ such that $\xi c=\tau $ and hence $\xi cb=\tau b=\eta $.
This concludes the proof.
\end{proof}

\subsection{Injective and projective objects}

Let $\mathcal{A}$ be an exact category. An object $I$ of $\mathcal{A}$ is 
\emph{injective} if the functor $\mathrm{Hom}_{\mathcal{A}}\left( -,I\right) 
$ is exact. The category $\mathcal{A}$ has\emph{\ enough injectives }if for
every object $A$ of $\mathcal{A}$ there exists an admissible monic $%
A\rightarrow I$, where $I$ is injective. Dually, one defines the notion of 
\emph{projective }object, and the notion of category with\emph{\ enough
projectives}. The following lemma is the content of \cite[Proposition 1.3.24]%
{schneiders_quasi-abelian_1999}.

\begin{lemma}[Schneiders]
\label{Lemma:projectives-left-heart}Suppose that $\mathcal{A}$ is a
quasi-abelian category. Then $\mathcal{A}$ has enough projectives if and
only if $\mathrm{LH}\left( \mathcal{A}\right) $ has enough projectives.
Furthermore, in this case every projective object of $\mathcal{A}$ is
projective in $\mathrm{LH}\left( \mathcal{A}\right) $, and every projective
object of $\mathrm{LH}\left( \mathcal{A}\right) $ is isomorphic to one of
this form.
\end{lemma}

Suppose that $\mathcal{A}$ is an exact category. Then $\mathrm{Hom}_{%
\mathcal{A}}$ is a functor $\mathcal{A}^{\mathrm{op}}\times \mathcal{A}%
\rightarrow \mathbf{Ab}$. The corresponding functor 
\begin{equation*}
\mathrm{Hom}^{\bullet }:\mathrm{K}^{b}\left( \mathcal{A}\right) ^{\mathrm{op}%
}\times \mathrm{K}^{b}\left( \mathcal{A}\right) \rightarrow \mathrm{K}%
^{b}\left( \mathbf{Ab}\right)
\end{equation*}%
is such that $\left( \mathrm{H}^{0}\circ \mathrm{Hom}^{\bullet }\right)
\left( A,B\right) $ is naturally isomorphic to $\mathrm{Hom}_{\mathrm{K}%
^{b}\left( \mathcal{A}\right) }\left( A,B\right) $; see \cite[Proposition
11.7.3]{kashiwara_categories_2006}. For objects $X,Y$ of $\mathcal{A}$ and $%
k\in \mathbb{Z}$, define $\mathrm{Ext}^{k}\left( X,Y\right) :=\mathrm{Hom}_{%
\mathrm{D}^{b}\left( \mathcal{A}\right) }\left( X,T^{k}Y\right) $; see \cite[%
Notation 13.1.9]{kashiwara_categories_2006}. In particular, $\mathrm{Ext}%
^{0}\left( X,Y\right) \cong \mathrm{Hom}_{\mathcal{A}}\left( X,Y\right) $
and $\mathrm{Ext}^{k}\left( X,Y\right) =0$ for $k<0$ \cite[Proposition
13.1.10]{kashiwara_categories_2006}. One says that $\mathcal{A}$ is \emph{%
hereditary }or has homological dimension at most $1$ if $\mathrm{Ext}%
^{n}\left( X,Y\right) =0$ for $n>1$ \cite[Definition 13.1.18]%
{kashiwara_categories_2006}. In this case, we set $\mathrm{Ext}\left(
X,Y\right) :=\mathrm{Ext}^{1}\left( X,Y\right) $ for objects $X,Y$ of $%
\mathcal{A}$. If $\mathcal{A}$ has enough injectives, this is equivalent to
the assertion that for some (equivalently, every) short exact sequence $%
A\rightarrow I\rightarrow C$ with $I$ injective, $C$ is also injective \cite[%
Exercise 13.8]{kashiwara_categories_2006}. When $\mathcal{A}$ has enough
projectives, this is equivalent to the assertion that for some
(equivalently, every) short exact sequence $A\rightarrow P\rightarrow C$
with $P$ projective, $A$ is also projective. The group $\mathrm{Ext}\left(
X,Y\right) $ parameterizes the \emph{extensions }of $X$ by $Y$ \cite[Section
II.2]{gelfand_methods_2003}.

\begin{lemma}
\label{Lemma:left-heart-projectives}Suppose that $\mathcal{A}$ is a
quasi-abelian category with enough projectives. The left heart $\mathrm{LH}%
\left( \mathcal{A}\right) $ is equivalent to its full subcategory consisting
of objects of the form $P/S:=\mathrm{Coker}_{\mathrm{LH}\left( \mathcal{A}%
\right) }\left( S\rightarrow P\right) $ for some monic arrow $S\rightarrow P$
in $\mathcal{A}$ with $P$ projective. Furthermore, the morphisms $%
P/S\rightarrow P^{\prime }/S^{\prime }$ are precisely those induced by pairs
of morphisms $S\rightarrow S^{\prime }$ and $P\rightarrow P^{\prime }$ that
make the diagram%
\begin{equation*}
\begin{array}{ccc}
P & \rightarrow & P^{\prime } \\ 
\uparrow &  & \uparrow \\ 
S & \rightarrow & S^{\prime }%
\end{array}%
\end{equation*}%
commute.
\end{lemma}

\begin{proof}
The first assertion follows from the fact that $\mathcal{A}$ has enough
projectives, and the pullback of a monic arrow is monic. Suppose now that $%
\varphi :P/S\rightarrow P^{\prime }/S^{\prime }$ is a morphism in $\mathrm{LH%
}\left( \mathcal{A}\right) $. Then by \cite[Corollary 1.2.20]%
{schneiders_quasi-abelian_1999} there exist an object $Q/R$ and morphisms $%
\sigma :Q/R\rightarrow P/S$ and $\psi :Q/R\rightarrow P^{\prime }/S^{\prime
} $ that are induced by commuting diagrams%
\begin{equation*}
\begin{array}{ccc}
Q & \overset{\psi _{Q}}{\rightarrow } & P^{\prime } \\ 
\uparrow &  & \uparrow \\ 
R & \overset{\psi _{R}}{\rightarrow } & S^{\prime }%
\end{array}%
\end{equation*}%
and%
\begin{equation*}
\begin{array}{ccc}
Q & \overset{\sigma _{Q}}{\rightarrow } & P \\ 
i_{R}\uparrow &  & \uparrow i_{S} \\ 
R & \overset{\sigma _{R}}{\rightarrow } & S%
\end{array}%
\end{equation*}%
where the latter is both a pullback and a pushout, $\sigma $ is invertible,
and $\varphi =\psi \circ \sigma ^{-1}$. After replacing $Q$ with $Q\oplus P$
and $R$ with $R\oplus S$ we can assume that the morphism $\sigma
_{Q}:Q\rightarrow P$ is a strict epimorphism. Thus, since $P$ is projective, 
$\sigma _{Q}$ admits a right inverse $\eta _{P}:P\rightarrow Q$. We also
have by the universal property of the pullback a morphism $\eta
_{S}:S\rightarrow R$ such that $i_{R}\eta _{S}=\eta _{P}i_{S}$ and $\sigma
_{R}\eta _{S}=1_{S}$. This shows that the morphism $\sigma
^{-1}:P/S\rightarrow Q/R$ is induced by the commuting diagram%
\begin{equation*}
\begin{array}{ccc}
P & \overset{\eta _{P}}{\rightarrow } & Q \\ 
i_{S}\uparrow &  & \uparrow i_{R} \\ 
S & \overset{\eta _{S}}{\rightarrow } & R%
\end{array}%
\end{equation*}%
and hence $\varphi $ is induced by the commuting diagram%
\begin{equation*}
\begin{array}{ccc}
P & \overset{\psi _{Q}\circ \eta _{P}}{\rightarrow } & P^{\prime } \\ 
i_{S}\uparrow &  & \uparrow i_{S^{\prime }} \\ 
S & \overset{\psi _{R}\circ \eta _{S}}{\rightarrow } & S^{\prime }\text{.}%
\end{array}%
\end{equation*}%
This concludes the proof.
\end{proof}

The same proof as \cite[Theorem 10.22 and Remark 10.23]{buhler_exact_2010}
gives the following:

\begin{lemma}
\label{Lemma:fully-exact}Suppose that $\mathcal{A}$ is an
idempotent-complete exact category and $\mathcal{B}$ is a fully exact
subcategory of $\mathcal{A}$ closed under direct summands. Suppose that:

\begin{enumerate}
\item for each object $A$ of $\mathcal{A}$ there exists an admissible monic $%
A\rightarrow B$ with $B$ in $\mathcal{B}$;

\item for all short exact sequences $B^{\prime }\rightarrow B\rightarrow
A^{\prime \prime }$ in $\mathcal{A}$ with $B^{\prime }$ and $B$ in $\mathcal{%
B}$, one has that $A^{\prime \prime }$ is isomorphic to an object of $%
\mathcal{B}$.
\end{enumerate}

Then the inclusion $\mathcal{B}\rightarrow \mathcal{A}$ induces an
equivalence of categories $\mathrm{D}^{b}\left( \mathcal{B}\right)
\rightarrow \mathrm{D}^{b}\left( \mathcal{A}\right) $.
\end{lemma}

\begin{corollary}
\label{Corollary:derived-projectives}Suppose that $\mathcal{A}$ is a
hereditary idempotent-complete exact category with enough projectives, and $%
\mathcal{P}$ is the fully exact subcategory of $\mathcal{A}$ consisting of
projective objects. The canonical quotient map $\mathrm{D}^{b}\left( 
\mathcal{P}\right) \rightarrow \mathrm{K}^{b}\left( \mathcal{A}\right) $ is
an equivalence of categories, while the inclusion $\mathcal{P}\rightarrow 
\mathcal{A}$ induces an equivalence of categories $\mathrm{D}^{b}\left( 
\mathcal{P}\right) \rightarrow \mathrm{D}^{b}\left( \mathcal{A}\right) $.
\end{corollary}

The following more general version of Lemma \ref{Lemma:fully-exact} is (the
dual of) \cite[Theorem 10.22 and Remark 10.23]{buhler_exact_2010}:

\begin{lemma}
\label{Lemma:fully-exact2}Suppose that $\mathcal{A}$ is an
idempotent-complete exact category and $\mathcal{B}$ is a fully exact
subcategory of $\mathcal{A}$ closed under direct summands. Suppose that:

\begin{enumerate}
\item for each object $A$ of $\mathcal{A}$ there exists an admissible epic $%
B\rightarrow A$ with $B$ in $\mathcal{B}$;

\item for all short exact sequences $A^{\prime }\rightarrow A\rightarrow
B^{\prime \prime }$ in $\mathcal{A}$ with $B^{\prime \prime }$ in $\mathcal{B%
}$, there exists a commutative diagram with exact rows%
\begin{equation*}
\begin{array}{ccccc}
B^{\prime } & \rightarrow & B & \rightarrow & B^{\prime \prime } \\ 
\downarrow &  & \downarrow &  & \downarrow \\ 
A^{\prime } & \rightarrow & A & \rightarrow & B^{\prime \prime }%
\end{array}%
\end{equation*}%
with $B^{\prime }$ and $B$ in $\mathcal{B}$.
\end{enumerate}

Then the inclusion $\mathcal{B}\rightarrow \mathcal{A}$ induces an
equivalence $\mathrm{D}^{+}\left( \mathcal{B}\right) \rightarrow \mathrm{D}%
^{+}\left( \mathcal{A}\right) $.
\end{lemma}

\begin{corollary}
Suppose that $\mathcal{A}$ is an idempotent-complete exact category with
enough projectives, and $\mathcal{P}$ is the fully exact subcategory of $%
\mathcal{A}$ consisting of projective objects. The canonical quotient map $%
\mathrm{K}^{-}\left( \mathcal{P}\right) \rightarrow \mathrm{D}^{-}\left( 
\mathcal{P}\right) $ is an equivalence of categories, while the inclusion $%
\mathcal{P}\rightarrow \mathcal{A}$ induces an equivalence of categories $%
\mathrm{D}^{-}\left( \mathcal{P}\right) \rightarrow \mathrm{D}^{-}\left( 
\mathcal{A}\right) $.
\end{corollary}

\subsection{Derived functors\label{Section:derived}}

Suppose that $F:\mathcal{A}\rightarrow \mathcal{B}$ is a functor between
exact categories. Then $F$ induces a triangulated functor \textrm{K}$%
^{b}\left( \mathcal{A}\right) \rightarrow \mathrm{K}^{b}\left( \mathcal{B}%
\right) $, which we still denote by $F$. If $F$ is exact, then its extension 
\textrm{K}$^{b}\left( \mathcal{A}\right) \rightarrow \mathrm{K}^{b}\left( 
\mathcal{B}\right) $ maps acyclic complexes to acyclic complexes, and hence
it induces a canonical triangulated functor $G:\mathrm{D}^{b}\left( \mathcal{%
A}\right) \rightarrow \mathrm{D}^{b}\left( \mathcal{B}\right) $ such that $%
GQ_{\mathcal{A}}$ is isomorphic to $Q_{\mathcal{B}}F$.

In general, we say that $F$ has a \emph{total right derived functor} if
there exists an (essentially unique) triangulated functor $\mathrm{R}F:%
\mathrm{D}^{b}\left( \mathcal{A}\right) \rightarrow \mathrm{D}^{b}\left( 
\mathcal{B}\right) $ with an isomorphism $\mu :Q_{\mathcal{B}}F\Rightarrow
\left( \mathrm{R}F\right) Q_{\mathcal{A}}$ of triangulated functors, and
such that for any other triangulated functor $G:\mathrm{D}^{b}\left( 
\mathcal{A}\right) \rightarrow \mathrm{D}^{b}\left( \mathcal{B}\right) $ and
morphism $\nu :Q_{\mathcal{B}}F\Rightarrow GQ_{\mathcal{A}}$ of triangulated
functors there exists a unique morphism $\sigma :\mathrm{R}F\Rightarrow G$
of triangulated functors such that $\mu \circ \sigma Q_{\mathcal{A}}=\nu $.

Suppose that $F:\mathcal{A}\times \mathcal{B}\rightarrow \mathcal{C}$ is an
additive functor, where $\mathcal{A}$, $\mathcal{B}$, and $\mathcal{C}$ are
exact categories. Let $\mathrm{Ch}^{b}\left( \mathcal{A}\right) $ be the
category of bounded complexes over $\mathcal{A}$ and $\mathrm{Ch}^{b}\left( 
\mathcal{A}\right) ^{2}$ be the category of bounded \emph{double} complexes
over $\mathcal{A}$. Then we have that $F$ induces a functor $F:\mathrm{Ch}%
^{b}\left( \mathcal{A}\right) \times \mathrm{Ch}^{b}\left( \mathcal{B}%
\right) \rightarrow \mathrm{Ch}^{b}\left( \mathcal{C}\right) ^{2}$ defined by%
\begin{equation*}
F\left( A,B\right) ^{i,j}:=F\left( A^{i},B^{j}\right)
\end{equation*}%
see \cite[Section 11.6]{kashiwara_categories_2006}. The vertical and
horizontal maps in $F\left( A,B\right) $ are given by%
\begin{equation*}
\delta _{\mathrm{v}}^{ij}:=F\left( \delta _{A}^{i},\mathrm{id}%
_{B^{j}}\right) :F\left( A^{i},B^{j}\right) \rightarrow F\left(
A^{i+1},B^{j}\right)
\end{equation*}%
\begin{equation*}
\delta _{\text{\textrm{h}}}^{ij}:=F(\mathrm{id}_{A_{i}},\delta
_{B}^{j}):F\left( A^{i},B^{j}\right) \rightarrow F\left(
A^{i},B^{j+1}\right) \text{.}
\end{equation*}%
The functor $F^{\bullet }:\mathrm{Ch}^{b}\left( \mathcal{A}\right) ^{\mathrm{%
op}}\times \mathrm{Ch}^{b}\left( \mathcal{B}\right) \rightarrow \mathrm{Ch}%
^{b}\left( \mathcal{C}\right) $ is defined by%
\begin{equation*}
F^{\bullet }\left( A,B\right) =\mathrm{Tot}\left( F\left( A,B\right) \right)
\end{equation*}%
for bounded complexes $A$ and $B$, where $\mathrm{Tot}\left( F\left(
A,B\right) \right) $ denotes the \emph{total complex} of $F\left( A,B\right) 
$ \cite[Section 11.5]{kashiwara_categories_2006}. This induces a \emph{%
triangulated} bifunctor $F^{\bullet }:\mathrm{K}^{b}\left( \mathcal{A}%
\right) \times \mathrm{K}^{b}\left( \mathcal{B}\right) \rightarrow \mathrm{K}%
^{b}\left( \mathcal{C}\right) $ in the sense of \cite[Definition 10.3.6]%
{kashiwara_categories_2006}; see \cite[Proposition 11.6.4]%
{kashiwara_categories_2006}.

If $\Phi :\mathrm{K}^{b}\left( \mathcal{A}\right) \times \mathrm{K}%
^{b}\left( \mathcal{B}\right) \rightarrow \mathrm{K}^{b}\left( \mathcal{C}%
\right) $ is a triangulated bifunctor, then a \emph{total right derived
functor }$\mathrm{R}\Phi $ of $\Phi $ is a triangulated bifunctor $\mathrm{D}%
^{b}\left( \mathcal{A}\right) \times \mathrm{D}^{b}\left( \mathcal{B}\right)
\rightarrow \mathrm{D}^{b}\left( \mathcal{C}\right) $ with a natural
isomorphism $g:Q_{\mathcal{C}}\Phi \Rightarrow \mathrm{R}\Phi \left( Q_{%
\mathcal{A}}\times Q_{\mathcal{B}}\right) $ such that for any other
triangulated bifunctor $\Psi :\mathrm{D}^{b}\left( \mathcal{A}\right) \times 
\mathrm{D}^{b}\left( \mathcal{B}\right) \rightarrow \mathrm{D}^{b}\left( 
\mathcal{C}\right) $ with natural transformation $h:Q_{\mathcal{C}}\Phi
\Rightarrow \Psi \left( Q_{\mathcal{A}}\times Q_{\mathcal{B}}\right) $ there
exists a unique morphism $\gamma :\mathrm{R}\Phi \Rightarrow \Psi $ of
triangulated functors such that $\gamma \left( Q_{\mathcal{A}}\times Q_{%
\mathcal{B}}\right) \circ g=h$.

Similar consideration apply if one replaces the bounded categories of
complexes with the left-bounded or right-bounded categories of complexes;
see \cite[Proposition 11.6.4]{kashiwara_categories_2006}. The notion of 
\emph{left derived functor }is obtained by duality.

\subsection{Construction of derived functors}

The most common way to produce derived functors is by using \emph{projective
resolutions}. The following can be sen as a particular instance of the
version for bounded complexes of \cite[Proposition 13.4.3]%
{kashiwara_categories_2006}.

\begin{proposition}
\label{Proposition:derived-functor}Let $\mathcal{A}$, $\mathcal{B}$, $%
\mathcal{C}$ be idempotent-complete exact categories. Suppose that $\mathcal{%
A}$ is hereditary with enough projectives. Let also $F:\mathcal{A}\times 
\mathcal{B}\rightarrow \mathcal{C}$ be an additive functor such that $%
F\left( P,-\right) :\mathcal{B}\rightarrow \mathcal{C}$ is exact for every
projective object $P$ of $\mathcal{A}$. Then the triangulated bifunctor $%
F^{\bullet }:\mathrm{K}^{b}(\mathcal{A)}\times \mathrm{K}^{b}(\mathcal{B)}%
\rightarrow \mathrm{K}^{b}\left( \mathcal{C}\right) $ has a total left
derived functor $\mathrm{L}F^{\bullet }:\mathrm{D}^{b}(\mathcal{A)}\times 
\mathrm{D}^{b}(\mathcal{B)}\rightarrow \mathrm{D}^{b}(\mathcal{C)}$.
Furthermore, this can be defined as follows: For a bounded complex $A$ over $%
\mathcal{A}$, pick a bounded complex $P_{A}$ over $\mathcal{P}$ together
with quasi-isomorphisms $\eta _{A}:P_{A}\rightarrow A$.\ Then one can define 
\begin{equation*}
\mathrm{L}F^{\bullet }\left( A,B\right) :=F^{\bullet }\left( P_{A},B\right)
\end{equation*}
for $A\in \mathrm{D}^{b}\left( \mathcal{A}\right) $ and $B\in \mathrm{D}%
^{b}\left( \mathcal{A}\right) $.

When $\mathcal{C}$ is quasi-abelian, one sets 
\begin{equation*}
\mathrm{L}^{k}F:=\left( \mathrm{H}^{k}\circ \mathrm{L}F\right) |_{\mathcal{A}%
\times \mathcal{B}}:\mathcal{A}\times \mathcal{B}\rightarrow \mathrm{LH}%
\left( \mathcal{C}\right)
\end{equation*}%
If $F$ is right exact in each variable, then $\mathrm{L}^{0}F\cong F$ and $%
\mathrm{L}^{k}F=0$ for $k>0$.
\end{proposition}

\begin{proof}
Let $\mathcal{P}$ be the class of projective objects of $\mathcal{A}$.
Consider the full subcategory $\mathcal{P}\times \mathcal{B}$ of $\mathcal{A}%
\times \mathcal{B}$. By\ Lemma \ref{Lemma:fully-exact}, the inclusion $%
\mathcal{P}\times \mathcal{B}\rightarrow \mathcal{A}\times \mathcal{B}$
induces an equivalence between their bounded derived categories. Suppose
that $P$ is a bounded complex in $\mathcal{P}$ and $A$ is a bounded complex
in $\mathcal{A}$. Considering that $\mathcal{P}$ is the class of projective
objects of $\mathcal{A}$, and that $F\left( P,-\right) :\mathcal{B}%
\rightarrow \mathcal{C}$ is exact for every $P$ in $\mathcal{P}$, one can
see that $F^{\bullet }\left( P,B\right) $ is acyclic whenever $P$ is an
acyclic complex in $\mathcal{P}$ or $B$ is an acyclic in $\mathcal{B}$.

The conclusion that $\mathrm{L}F^{\bullet }$ as defined in the statement is
a total left derived functor for $F^{\bullet }$ is a consequence of the
above remarks; see also \cite[Section 13.4]{kashiwara_categories_2006}. When 
$F$ is right-exact in each variable, it follows as in (the dual of) \cite[%
Proposition 1.3.11]{schneiders_quasi-abelian_1999} that $\mathrm{L}^{0}F|_{%
\mathcal{A}\times \mathcal{B}}\cong F$ and $\mathrm{L}^{k}F=0$ for $k>0$.
(Notice that in the terminology of \cite{schneiders_quasi-abelian_1999},
\textquotedblleft RL left exact\textquotedblright\ is equivalent to left
exact; see \cite[Proposition 1.1.15]{schneiders_quasi-abelian_1999}.)
\end{proof}

A similar proof as Proposition \ref{Proposition:derived-functor} gives the
following result (a particular instance of \cite[Proposition 13.4.3]%
{kashiwara_categories_2006}), which applies to exact categories that are not
necessarily hereditary.

\begin{proposition}
\label{Proposition:derived-functor2}Let $\mathcal{A}$, $\mathcal{B}$, $%
\mathcal{C}$ be idempotent-complete exact categories. Suppose that $\mathcal{%
A}$ has enough projectives. Let also $F:\mathcal{A}\times \mathcal{B}%
\rightarrow \mathcal{C}$ be an additive functor such that $F\left(
P,-\right) :\mathcal{B}\rightarrow \mathcal{C}$ is exact for every
projective object $P$ of $\mathcal{A}$. Then the triangulated bifunctor $%
F^{\bullet }:\mathrm{K}^{+}(\mathcal{A)}\times \mathrm{K}^{+}(\mathcal{B)}%
\rightarrow \mathrm{K}^{+}\left( \mathcal{C}\right) $ has a total left
derived functor $\mathrm{L}F^{\bullet }:\mathrm{D}^{+}(\mathcal{A)}\times 
\mathrm{D}^{+}(\mathcal{B)}\rightarrow \mathrm{D}^{+}(\mathcal{C)}$.
Furthermore, this can be defined as follows: For a left-bounded complex $A$
over $\mathcal{A}$, pick a left-bounded complex $P_{A}$ over $\mathcal{P}$
together with a quasi-isomorphism $\eta _{A}:P_{A}\rightarrow A$.\ Then one
can define 
\begin{equation*}
\mathrm{L}F^{\bullet }\left( A,B\right) :=F^{\bullet }\left( P_{A},B\right)
\end{equation*}
for $A\in \mathrm{D}^{+}\left( \mathcal{A}\right) $ and $B\in \mathrm{D}%
^{+}\left( \mathcal{A}\right) $.

When $\mathcal{C}$ is quasi-abelian, one sets 
\begin{equation*}
\mathrm{L}^{k}F:=\left( \mathrm{H}^{k}\circ \mathrm{L}F\right) |_{\mathcal{A}%
\times \mathcal{B}}:\mathcal{A}\times \mathcal{B}\rightarrow \mathrm{LH}%
\left( \mathcal{C}\right)
\end{equation*}%
If $F$ is right exact in each variable, then $\mathrm{L}^{0}F\cong F$.
\end{proposition}

\subsection{Monoidal categories}

Let $\mathcal{C}$ be a category. A \emph{monoidal structure }on $\mathcal{C}$
consists of a quintuple $\left( \otimes ,I,\alpha ,\lambda ,\rho \right) $
where $\otimes $ is a functor $\mathcal{C}\times \mathcal{C}\rightarrow 
\mathcal{C}$ called \emph{tensor product}, $I$ is an object of $\mathcal{C}$
called \emph{tensor identity}, $\alpha $ is a natural isomorphism 
\begin{equation*}
\otimes \circ \left( \otimes \times 1_{\mathcal{C}}\right) \cong \otimes
\circ \left( 1_{\mathcal{C}}\times \otimes \right)
\end{equation*}%
called \emph{associator}, $\rho $ is a natural isomorphism 
\begin{equation*}
\otimes \circ \left( 1_{\mathcal{C}}\times I\right) \cong 1_{\mathcal{C}}
\end{equation*}%
called \emph{right unitor}, and $\lambda $ is a natural isomorphism 
\begin{equation*}
\otimes \circ \left( I\times 1_{\mathcal{C}}\right) \cong 1_{\mathcal{C}}
\end{equation*}%
called \emph{left unitor}, such that $\alpha ,\lambda ,\rho $ satisfy the
associativity and unit coherence axioms; see \cite[Section VII.1]%
{mac_lane_categories_1998} or \cite[Definition 6.1.1]%
{borceux_handbook_1994-2}. Thus, for objects $x,y,z$ of $\mathcal{C}$, $%
\alpha _{x,y,z}$ is an isomorphism%
\begin{equation*}
\left( x\otimes y\right) \otimes z\cong x\otimes \left( y\otimes z\right) 
\text{,}
\end{equation*}%
$\rho _{x}$ is an isomorphism%
\begin{equation*}
x\otimes I\cong x\text{,}
\end{equation*}%
and $\lambda _{x}$ is an isomorphism%
\begin{equation*}
I\otimes x\cong x\text{.}
\end{equation*}%
A \emph{monoidal category} is a category endowed with a monoidal structure.

A \emph{tensor category }or \emph{symmetric monoidal category} is a monoidal
category together with a natural isomorphism $\beta $, called \emph{braiding}%
, between $\otimes $ and $\otimes \circ \sigma $, where $\sigma :\mathcal{C}%
\times \mathcal{C}\rightarrow \mathcal{C}\times \mathcal{C}$ is the flip,
satisfying $\beta \circ \left( \beta \sigma \right) =\mathrm{1}_{\otimes }$ (%
\emph{symmetry}), and the associativity and unit cohere axioms \cite[%
Definition 6.1.2]{borceux_handbook_1994-2}. Thus, for objects $x,y$ of $%
\mathcal{C}$, $\beta _{x,y}$ is an isomorphism%
\begin{equation*}
x\otimes y\cong y\otimes x\text{.}
\end{equation*}%
The symmetry requirement asserts that $\beta _{y,x}\circ \beta _{x,y}$ is
the identity of $x\otimes y$.

A functor $F:\mathcal{A}\rightarrow \mathcal{B}$ between monoidal categories
is strict monoidal if $F\left( I\right) =I$ and $F\left( A\otimes B\right)
=FA\otimes FB$ for objects $A,B$ of $\mathcal{A}$. The notion of (lax)
monoidal functor is obtained in a similar fashion by replacing equalities
with natural transformations that satisfy coherence conditions that involve
the associator and unitors of $\mathcal{A}$ and $\mathcal{B}$.

Let $\mathcal{V}$ be a monoidal category. A $\mathcal{V}$-category is given
by a collection of objects and, for objects $x,y,z$, an object $\mathrm{Hom}%
\left( x,y\right) $ of $\mathcal{V}$, a morphism $1_{x}:I\rightarrow \mathrm{%
Hom}\left( x,x\right) $ (the \emph{identity}), and a morphism $\mathrm{Hom}%
\left( y,z\right) \otimes \mathrm{Hom}\left( x,y\right) \rightarrow \mathrm{%
Hom}\left( x,z\right) $ (\emph{composition}), that satisfy the natural
version of the associativity and identity axioms of a category, phrased in
terms of commuting diagrams involving the associator and the left and right
unitor of $\mathcal{V}$. A $\mathcal{V}$-functor $F:\mathcal{C}\rightarrow 
\mathcal{D}$ between $\mathcal{V}$-categories is given by a function $%
x\mapsto F\left( x\right) $ mapping objects of $\mathcal{C}$ to objects of $%
\mathcal{D}$ together with morphisms $\mathrm{Hom}\left( x,y\right)
\rightarrow \mathrm{Hom}\left( Fx,Fy\right) $ in $\mathcal{V}$ for objects $%
x,y$ of $\mathcal{C}$ that preserve identities and composition in the
obvious sense.

\subsection{Tensor quasi-abelian categories}

A tensor additive category is a tensor category that is also additive, and
such that the tensor product functor is additive in each variable. The
notion of tensor exact category is defined in the same fashion.

A tensor quasi-abelian category is a tensor additive category that is also
quasi-abelian, where the quasi-abelian and monoidal structures are
compatible in a suitable sense. For our purposes, it will be convenient to
consider the following:

\begin{definition}
\label{Defnition:monoidal-quasiabelian}A\emph{\ tensor quasi-abelian} \emph{%
category} is a tensor category which is also quasi-abelian, such that the
tensor identity $I$ is projective and the tensor product functor is right
exact in each variable.

An object $X$ of a tensor quasi-abelian category is \emph{flat }if the
functor $X\otimes -$ is exact.
\end{definition}

The reason for this definition is that it includes the categories of modules
over a commutative ring, and other categories that we will consider. If $%
\mathcal{A}$ is a monoidal quasi-abelian category, then by the universal
property of $\mathrm{LH}\left( \mathcal{A}\right) $, the tensor product
functor $\otimes _{\mathcal{A}}$ extend to a functor $\otimes _{\mathrm{LH}%
\left( \mathcal{A}\right) }:\mathrm{LH}\left( \mathcal{A}\right) \times 
\mathrm{LH}\left( \mathcal{A}\right) \rightarrow \mathrm{LH}\left( \mathcal{A%
}\right) $ which is still right exact in each variable. The tensor identity
for $\mathcal{A}$ is also a tensor identity for $\mathrm{LH}\left( \mathcal{A%
}\right) $. Furthermore, the associator, unitors, and braiding for $\otimes
_{\mathcal{A}}$ induce associator, unitors, and braiding for $\otimes _{%
\mathrm{LH}\left( \mathcal{A}\right) }$.\ This endows $\mathrm{LH}\left( 
\mathcal{A}\right) $ with a canonical structure of tensor abelian category.

\subsection{Tensor triangulated categories}

A variety of notions of tensor triangulated category have been considered in
the literature. Informally, a tensor triangulated category is a triangulated
category that is also a symmetric monoidal category, with compatibility
requirements between the triangulated and monoidal structures. Depending on
the source and the desired applications, the compatibility requirements
between the triangulated and monoidal structures may vary. In this paper, we
consider the ones from \cite{dellambrogio_triangulated_2016}; see also \cite%
{xu_hopf_2025}, \cite[Definition 3.1]{dubey_compactly_2023}, and \cite[%
Appendix A]{hovey_axiomatic_1997}. Unlike these sources, we consider
monoidal categories that are not necessarily closed, i.e., we do not require
the tensor product functor to have a right adjoint.

Let thus $\mathcal{T}$ be a triangulated category with translation functor $%
T $ (called \emph{suspension }in \cite[Appendix A]{hovey_axiomatic_1997})
endowed with a tensor structure which we denote as above (the tensor functor
is denoted by $\wedge $ and called smash product in \cite[Appendix A]%
{hovey_axiomatic_1997}). We require that the tensor product functor be
triangulated in the sense of \cite[Definition 10.3.6]%
{kashiwara_categories_2006}, as witnessed by natural isomorphisms%
\begin{equation*}
\left( Tx\right) \otimes y\cong _{\ell _{x,y}}T\left( x\otimes y\right)
\end{equation*}%
and%
\begin{equation*}
x\otimes \left( Ty\right) \cong _{r_{x,y}}T\left( x\otimes y\right)
\end{equation*}%
for objects $x,y$ of $\mathcal{T}$. The triangulated and monoidal structures
together with the natural transformations $\ell $ and $r$, called\emph{\
left and right translator }respectively, define a \emph{tensor triangulated
category} provided they satisfy the following compatibility requirements as
in \cite{dellambrogio_triangulated_2016} for all objects $x,y,z$:

\begin{enumerate}
\item the right translator and the left unitor satisfy%
\begin{equation*}
T\lambda _{x}\circ r_{x}=\lambda _{Tx}\text{;}
\end{equation*}

\item the left translator and the right unitor satisfy%
\begin{equation*}
T\rho _{x}\circ \ell _{x}=\rho _{Tx}
\end{equation*}

\item the left and right translator satisfy%
\begin{equation*}
T\ell _{x,y}\circ r_{Tx,y}=Tr_{x,y}\circ \ell _{x,Ty}
\end{equation*}

\item the associator and the left translator satisfy%
\begin{equation*}
\ell _{x,y\otimes z}\circ \alpha _{Tx,y,z}=T\alpha _{x,y,z}\circ \ell
_{x\otimes y,z}\circ \left( \ell \otimes 1_{z}\right)
\end{equation*}

\item the left and right translator determine each other via the braiding%
\emph{\ }by the identity%
\begin{equation*}
r_{x,y}=\left( T\beta _{y,x}\right) \circ \ell _{y,x}\circ \beta _{x,Ty}%
\text{.}
\end{equation*}
\end{enumerate}

Suppose that $\mathcal{A}$ is a tensor additive category. Then the tensor
product functor extends as in Section \ref{Section:derived} to a
triangulated bifunctor $\otimes ^{\bullet }$ on $\mathrm{K}^{b}\left( 
\mathcal{A}\right) $ which turns it into a tensor triangulated category with
the same tensor identity as $\mathcal{A}$, and associators, unitors, and
braiding induced by those of $\mathcal{A}$. The same applies to $\mathrm{K}%
^{+}\left( \mathcal{A}\right) $.

\begin{proposition}
Let $\mathcal{A}$ be a tensor exact category. Suppose that:

\begin{itemize}
\item the tensor product functor on $\mathcal{A}$ is right exact in each
variable;

\item $\mathcal{A}$ has enough \emph{flat }projectives.
\end{itemize}

Then the triangulated bifunctor $\otimes :\mathrm{K}^{+}(\mathcal{A)}\times 
\mathrm{K}^{+}(\mathcal{A)}\rightarrow \mathrm{K}^{+}\left( \mathcal{A}%
\right) $ has a total left derived functor $\otimes :\mathrm{D}^{+}(\mathcal{%
A)}\times \mathrm{D}^{+}(\mathcal{A)}\rightarrow \mathrm{D}^{+}(\mathcal{M)}$%
. This turns $\mathrm{D}^{+}\left( \mathcal{A}\right) $ into a tensor
triangulated category, such that the inclusion $\mathcal{A}\rightarrow 
\mathrm{D}^{+}\left( \mathcal{A}\right) $ is monoidal and triangulated.

If $\mathcal{A}$ is hereditary, the same conclusions hold replacing the
category of left-bounded complexes with the category of bounded complexes.
\end{proposition}

\begin{proof}
By Proposition \ref{Proposition:derived-functor2}, the tensor product
functor $\otimes ^{\bullet }$ on $\mathrm{K}^{+}\left( \mathcal{A}\right) $
admits a total left derived functor $\otimes $ on $\mathrm{D}^{+}\left( 
\mathcal{A}\right) $. By definition, for each complex $A$ one picks a
quasi-isomorphism $A\rightarrow P_{A}$, where $P_{A}$ is a complex of \emph{%
flat} projectives. On then has%
\begin{equation*}
A\otimes B:=P_{A}\otimes ^{\bullet }P_{B}\text{.}
\end{equation*}%
The associators, unitors, and braiding for $\otimes ^{\bullet }$ on $\mathrm{%
K}^{+}\left( \mathcal{A}\right) $ induce those for $\otimes $ on $\mathrm{D}%
^{+}\left( \mathcal{A}\right) $.
\end{proof}

\subsection{Tilting}

Suppose that $\left( \mathcal{E},\mathcal{F}\right) $ is a torsion pair in a
quasi-abelian category $\mathcal{A}$. The pair is called \emph{tilting}, and 
$\mathcal{E}$ a tilting torsion class, if $\mathcal{E}$ cogenerates $%
\mathcal{A}$, i.e., for every object $A$ of $\mathcal{A}$ there exists a
strict monic $A\rightarrow E$ with $E$ in $\mathcal{A}$ \cite%
{happel_tilting_1996}; see also \cite[Appendix B]{bondal_generators_2003}.
Dually, the pair $\left( \mathcal{E},\mathcal{F}\right) $ is called
cotilting, and $\mathcal{F}$ a cotilting torsion-free class, if $\mathcal{F}$
generates $\mathcal{A}$, i.e., for every object $A$ of $\mathcal{A}$ there
exists a strict epic $F\rightarrow \mathcal{A}$ with $F$ in $\mathcal{F}$.
In this case, $\mathcal{A}$ is equivalent to $\mathrm{LH}\left( \mathcal{E}%
\right) $ via an equivalence that is the identity of $\mathcal{E}$.

Suppose that $\left( \mathcal{E},\mathcal{F}\right) $ is a torsion pair in $%
\mathcal{A}$. Then the inclusion $\mathcal{E}\rightarrow \mathcal{A}$ is
exact and cocontinuous. For an arrow $f$ in $\mathcal{E}$ one has $\mathrm{%
\mathrm{Ker}}_{\mathcal{E}}\left( f\right) =\pi _{\mathcal{E}}\mathrm{%
\mathrm{Ker}}_{\mathcal{A}}\left( f\right) $. For an object $A$ of $\mathcal{%
A}$, by definition we have a short exact sequence%
\begin{equation*}
0\rightarrow E_{A}\rightarrow A\rightarrow F_{A}\rightarrow 0
\end{equation*}%
with $E_{A}$ in $\mathcal{E}$ and $F_{A}$ in $\mathcal{F}$. The assignment $%
A\rightarrow F_{A}$ defines a functor $\pi _{\mathcal{E}}:\mathcal{A}%
\rightarrow \mathcal{E}$ that is right adjoint of the inclusion functor $%
\mathcal{E}\rightarrow \mathcal{A}$ and satisfies $\pi _{\mathcal{E}}|_{%
\mathcal{E}}=\mathrm{id}_{\mathcal{E}}$ and $\pi _{\mathcal{E}}|_{\mathcal{F}%
}=0$. It is transparent that $\mathcal{E}$ is a $\pi _{\mathcal{E}}$%
-injective category in the sense of \cite[Proposition 1.3.5]%
{schneiders_quasi-abelian_1999}. Therefore by \cite[Proposition 1.3.5]%
{schneiders_quasi-abelian_1999}, $\pi _{\mathcal{E}}$ is explicitly right
derivable in the sense of \cite[Definition 1.3.6]%
{schneiders_quasi-abelian_1999}. The derived functor $\mathrm{R}^{1}\pi _{%
\mathcal{E}}:\mathcal{A}\rightarrow \mathrm{LH}\left( \mathcal{E}\right) $
is defined as follows. For an object $A$ of $\mathcal{A}$, since $\mathcal{E}
$ is a tilting torsion class, there exists a strict monic $A\rightarrow I$.
Then the complex%
\begin{equation*}
0\rightarrow I\rightarrow \mathrm{Coker}_{\mathcal{A}}\left( I\rightarrow
A\right) \rightarrow 0
\end{equation*}%
with $I$ in degree $0$, is quasi-isomorphic to $A$. Its cohomology in degree 
$1$ is $\mathrm{R}^{1}\pi _{\mathcal{E}}\left( A\right) $. Thus,%
\begin{equation*}
\mathrm{R}^{1}\pi _{\mathcal{E}}\left( A\right) =\mathrm{Coker}_{\mathrm{LH}%
\left( \mathcal{E}\right) }\left( I\rightarrow \mathrm{Coker}_{\mathcal{A}%
}\left( I\rightarrow A\right) \right) \text{.}
\end{equation*}%
Recall the definition of the phantom category of a quasi-abelian category
from Definition \ref{Definition:phantom-category}.

\begin{proposition}
\label{Proposition:phantom-category}Let $\mathcal{A}$ be a quasi-abelian
category, and $\left( \mathcal{E},\mathcal{F}\right) $ be a tilting torsion
pair in $\mathcal{A}$:

\begin{enumerate}
\item the restriction of $\mathrm{R}^{1}\pi _{\mathcal{E}}$ to $\mathcal{F}$
is faithful, and its essential image is the phantom category $\mathrm{Ph}%
\left( \mathcal{E}\right) $ of $\mathcal{E}$;

\item if for every object $F$ of $\mathcal{F}$ there exists a strict monic $%
F\rightarrow P$ where $P$ is projective in $\mathcal{E}$, $\mathrm{R}^{1}\pi
_{\mathcal{E}}$ establishes an equivalence of categories between $\mathcal{F}
$ and $\mathrm{Ph}\left( \mathcal{E}\right) $.
\end{enumerate}
\end{proposition}

\begin{proof}
(1) Suppose that $F\in \mathcal{F}$. We need to prove that $\mathrm{R}%
^{1}\pi _{\mathcal{E}}\left( F\right) \in \mathrm{Ph}\left( \mathcal{E}%
\right) $. Let $\kappa :\mathrm{LH}\left( \mathcal{E}\right) \rightarrow 
\mathcal{E}$ be the left adjoint of the inclusion $\mathcal{E}\rightarrow 
\mathrm{LH}\left( \mathcal{E}\right) $. Then we have that%
\begin{eqnarray*}
\kappa \left( \mathrm{R}^{1}\pi _{\mathcal{E}}\left( A\right) \right)
&=&\kappa \left( \mathrm{Coker}_{\mathrm{LH}\left( \mathcal{E}\right)
}\left( I\rightarrow \mathrm{Coker}_{\mathcal{A}}\left( I\rightarrow
A\right) \right) \right) \\
&=&\mathrm{Coker}_{\mathcal{E}}\left( I\rightarrow \mathrm{Coker}_{\mathcal{A%
}}\left( I\rightarrow A\right) \right) \\
&=&\mathrm{Coker}_{\mathcal{A}}\left( I\rightarrow \mathrm{Coker}_{\mathcal{A%
}}\left( I\rightarrow A\right) \right) \\
&=&0\text{.}
\end{eqnarray*}%
Conversely, suppose that $X$ is a phantom object of $\mathcal{E}$, i.e., an
object of $\mathrm{Ph}\left( \mathcal{E}\right) $. Then 
\begin{equation*}
X\cong \mathrm{Coker}_{\mathrm{LH}\left( \mathcal{E}\right) }\left(
X^{0}\rightarrow X^{1}\right)
\end{equation*}%
for some monic arrow $X^{0}\rightarrow X^{1}$ in $\mathcal{E}$ satisfying%
\begin{equation*}
\mathrm{Coker}_{\mathcal{E}}\left( X^{0}\rightarrow X^{1}\right) =0\text{.}
\end{equation*}%
Define now%
\begin{equation*}
F:=\mathrm{\mathrm{Ker}}_{\mathcal{A}}\left( X^{0}\rightarrow X^{1}\right) 
\text{.}
\end{equation*}%
By definition $F$ is an object of $\mathcal{A}$, but since $X^{0}\rightarrow
X^{1}$ is a monic arrow of $\mathcal{E}$ and $\left( \mathcal{E},\mathcal{F}%
\right) $ is a torsion pair in $\mathcal{A}$, we must have that $F$ is an
object of $\mathcal{A}$. Since%
\begin{equation*}
\mathrm{Coker}_{\mathcal{E}}\left( X^{0}\rightarrow X^{1}\right) =\mathrm{%
Coker}_{\mathcal{A}}\left( X^{0}\rightarrow X^{1}\right) =0
\end{equation*}%
we have that the arrow $X^{0}\rightarrow X^{1}$ is the cokernel in $\mathcal{%
A}$ of the canonical arrow $F\rightarrow X^{0}$. Thus, we have that by
definition%
\begin{equation*}
\mathrm{R}^{1}\pi _{\mathcal{E}}\left( F\right) =\mathrm{Coker}_{\mathcal{A}%
}\left( F\rightarrow X^{0}\right) =\mathrm{Coker}_{\mathrm{LH}\left( 
\mathcal{E}\right) }\left( X^{0}\rightarrow X^{1}\right) \cong X\text{.}
\end{equation*}%
This shows that the essential image of $\mathcal{F}$ under $\mathrm{R}%
^{1}\pi _{\mathcal{E}}$ is $\mathrm{Ph}\left( \mathcal{E}\right) $.

We now prove that $\mathrm{R}^{1}\pi _{\mathcal{E}}$ is faitfhul on $%
\mathcal{F}$. Suppose that $f:F\rightarrow F^{\prime }$ is an arrow in $%
\mathcal{F}$ with $\mathrm{R}^{1}\pi _{\mathcal{E}}\left( f\right) =0$. By
definition, we have that $\mathrm{R}^{1}\pi _{\mathcal{E}}\left( f\right) $
is obtained by choosing strict monic arrows%
\begin{equation*}
i:F\rightarrow E
\end{equation*}%
and%
\begin{equation*}
i^{\prime }:F^{\prime }\rightarrow E^{\prime }
\end{equation*}%
and an arrow%
\begin{equation*}
\varphi :E\rightarrow E^{\prime }
\end{equation*}%
such that 
\begin{equation*}
\varphi i=fi^{\prime }\text{.}
\end{equation*}%
Then the pair $\left( f,\varphi \right) $ induces an arrow%
\begin{equation*}
\rho :\mathrm{Coker}_{\mathcal{A}}\left( i\right) \rightarrow \mathrm{Coker}%
\left( i^{\prime }\right)
\end{equation*}%
which in turn induces together with $f$ an arrow%
\begin{equation*}
\mathrm{R}^{1}\pi _{\mathcal{E}}\left( f\right) :\mathrm{Coker}_{\mathrm{LH}%
\left( \mathcal{E}\right) }\left( E\rightarrow \mathrm{Coker}_{\mathcal{A}%
}\left( i\right) \right) \rightarrow \mathrm{Coker}_{\mathrm{LH}\left( 
\mathcal{E}\right) }\left( E^{\prime }\rightarrow \mathrm{Coker}_{\mathcal{A}%
}\left( i^{\prime }\right) \right) \text{.}
\end{equation*}%
This implies that $\rho $ factors through $i^{\prime }$. In turn, this
implies that the arrow $\varphi $ factors through $F^{\prime }$. Since $%
\mathrm{Hom}\left( \mathcal{E},\mathcal{F}\right) =0$ this implies $\varphi
=0$ and hence $f=0$.

(2) We now assume that for every object $F$ of $\mathcal{F}$ there exists a
strict monic $i:F\rightarrow P$ where $P$ is projective in $\mathcal{E}$. In
this case, in the definition of $\mathrm{R}^{1}\pi _{\mathcal{E}}\left(
F\right) $ we can use such a strict monic $i$.\ Suppose that $F,F^{\prime }$
are objects of $\mathcal{F}$ and $i:F\rightarrow P$ and $i^{\prime
}:F^{\prime }\rightarrow P^{\prime }$ are the corresponding strict monics in 
$\mathcal{A}$ with $P$ and $P^{\prime }$ projective in $\mathcal{E}$.
Consider a morphism $g:\mathrm{R}^{1}\pi _{\mathcal{E}}\left( F\right)
\rightarrow \mathrm{R}^{1}\pi _{\mathcal{E}}\left( F^{\prime }\right) $.
Since $P$ is projective in $\mathcal{E}$, by Lemma \ref%
{Lemma:left-heart-projectives} $g$ is induced by a pair of morphisms $\gamma
:P\rightarrow P^{\prime }$, $\psi :\mathrm{Coker}_{\mathcal{A}}\left(
i\right) \rightarrow \mathrm{Coker}_{\mathcal{A}}\left( i^{\prime }\right) $%
. By the universal property of $\mathrm{Coker}_{\mathcal{A}}\left( i^{\prime
}\right) $, we must have that $\psi \circ i$ factors through $i^{\prime }$,
whence there exists a morphism $f:F\rightarrow F^{\prime }$ such that $%
i^{\prime }f=\gamma i$. It follows from the construction that $\mathrm{R}%
^{1}\pi _{\mathcal{E}}\left( f\right) =\gamma $.
\end{proof}

\subsection{Categories of towers\label{Subsection:towers}}

Let $\mathcal{C}$ be a category. We let $\mathrm{pro}\left( \mathcal{C}%
\right) $ be the full subcategory of the pro-category $\mathrm{Pro}\left( 
\mathcal{C}\right) $ spanned by the pro-objects with countable index set. Up
to equivalence, one can identify $\mathrm{pro}\left( \mathcal{C}\right) $
with the category whose objects are towers (i.e., $\omega ^{\mathrm{op}}$%
-diagrams) $\left( A^{\left( n\right) }\right) $ of objects of $\mathcal{C}$%
, and the hom-set $\left( A^{\left( n\right) }\right) \rightarrow \left(
B^{\left( n\right) }\right) $ is by definition%
\begin{equation*}
\mathrm{\mathrm{lim}}_{i}\mathrm{co\mathrm{lim}}_{j}\mathrm{Hom}(A^{\left(
j\right) },B^{\left( i\right) })\text{;}
\end{equation*}%
see \cite[Definition 6.1.1]{prosmans_derived_1999} and \cite[Chapter 6]%
{kashiwara_categories_2006}. When $\mathcal{A}$ is a quasi-abelian category, 
$\mathrm{pro}\left( \mathcal{A}\right) $ is a countably complete
quasi-abelian category, and the canonical inclusion $\mathcal{A}\rightarrow 
\mathrm{pro}\left( \mathcal{A}\right) $ is exact and finitely continuous 
\cite[Proposition 7.1.5, Proposition 7.1.7]{prosmans_derived_1999}; see also 
\cite[Section 8.6]{kashiwara_categories_2006} and \cite[Section 2.1]%
{edwards_cech_1976}.

Let us say that a tower $\left( A^{\left( n\right) }\right) $ is \emph{%
reduced }if $\mathrm{lim}_{n}{}A^{\left( n\right) }=0$, the limit being
computed in $\mathrm{pro}\left( \mathcal{A}\right) $. A tower $\left(
A^{\left( n\right) }\right) $ is \emph{epimorphic }if the bonding maps $%
A^{\left( n+1\right) }\rightarrow A^{\left( n\right) }$ are admissible
epimorphisms, and \emph{essentially epimorphic} if it is isomorphic to an
epimorphic tower. We define $\boldsymbol{\Pi }\left( \mathcal{A}\right) $ to
be the full subcategory of $\mathrm{pro}\left( \mathcal{A}\right) $ spanned
by the essentially epimorphic towers, and $\mathrm{P}\left( \mathcal{A}%
\right) $ to be the full subcategory of $\mathrm{pro}\left( \mathcal{A}%
\right) $ spanned by the reduced towers. Then the pair $\left( \boldsymbol{%
\Pi }\left( \mathcal{A}\right) ,\mathrm{P}\left( \mathcal{A}\right) \right) $
defines a torsion pair on $\mathrm{pro}\left( \mathcal{A}\right) $, i.e., we
have an extension%
\begin{equation*}
0\rightarrow \boldsymbol{\Pi }\left( \mathcal{A}\right) \rightarrow \mathrm{%
pro}\left( \mathcal{A}\right) \rightarrow \mathrm{P}\left( \mathcal{A}%
\right) \rightarrow 0\text{.}
\end{equation*}%
Furthermore, $\boldsymbol{\Pi }\left( \mathcal{A}\right) $ is a tilting
torsion class, and the right adjoint of the inclusion $\boldsymbol{\Pi }%
\left( \mathcal{A}\right) \rightarrow \mathrm{pro}\left( \mathcal{A}\right) $
is the functor $\mathrm{\mathrm{lim}}:\mathrm{pro}\left( \mathcal{A}\right)
\rightarrow \boldsymbol{\Pi }\left( \mathcal{A}\right) $ that maps a tower
to its limit. In this case, the derived functor $\mathrm{R}^{1}\mathrm{%
\mathrm{lim}}$ is denoted by $\mathrm{lim}^{1}$. Explicitly, given a tower $%
\left( A^{\left( n\right) }\right) $ of objects of $\mathcal{A}$ one can
realize $\mathrm{lim}_{n}^{1}{}A^{\left( n\right) }$ as the cokernel of the
morphism%
\begin{equation*}
\prod_{n}A^{\left( n\right) }\rightarrow \prod_{n}A^{\left( n\right) }
\end{equation*}%
defined in terms of generalized elements by%
\begin{equation*}
\left( x_{n}\right) \mapsto \left( x_{n}-p^{\left( n,n+1\right) }\left(
x_{n+1}\right) \right)
\end{equation*}%
where $p^{\left( n,n+1\right) }:A^{\left( n+1\right) }\rightarrow A^{\left(
n\right) }$ is the bonding map. (Here, the product is computed in $%
\boldsymbol{\Pi }\left( \mathcal{A}\right) $.) As a particular instance of
Proposition \ref{Proposition:phantom-category} we have the following:

\begin{proposition}
\label{Proposition:phantom-tower}Let $\mathcal{A}$ be a quasi-abelian
category:

\begin{enumerate}
\item the restriction of the functor $\mathrm{lim}^{1}:\mathrm{pro}\left( 
\mathcal{A}\right) \rightarrow \boldsymbol{\Pi }\left( \mathcal{A}\right) $
to the full subcategory $\mathrm{P}\left( \mathcal{A}\right) $ comprising
the reduced towers is faithful, and its essential image is the phantom
category $\mathrm{Ph}\left( \boldsymbol{\Pi }\left( \mathcal{A}\right)
\right) $ of $\boldsymbol{\Pi }\left( \mathcal{A}\right) $;

\item if every object of $\mathcal{A}$ is projective, then $\mathrm{lim}^{1}$
establishes an equivalence of categories between the category of reduced
towers over $\mathcal{A}$ and\textrm{\ }the phantom category of $\boldsymbol{%
\Pi }\left( \mathcal{A}\right) $.
\end{enumerate}
\end{proposition}

It is easy to see that if $\mathcal{A}$ is hereditary with enough
projectives, then $\boldsymbol{\Pi }(\mathcal{A)}$ is also hereditary with
enough projectives. Furthermore, products of projectives of $\mathcal{A}$
are projective in $\boldsymbol{\Pi }\left( \mathcal{A}\right) $.

Suppose that $\mathcal{A}$ is a tensor quasi-abelian category. As by
assumption the tensor product functor is right exact in each variable, it
preserves strict epimorphisms. Therefore, it induces a tensor product
functor on $\boldsymbol{\Pi }\left( \mathcal{A}\right) $. The associator,
unitors, and braiding for the tensor product of $\mathcal{A}$ induce
associator, unitors, and braiding for the tensor product on $\boldsymbol{\Pi 
}\left( \mathcal{A}\right) $, which turn it into a tensor quasi-abelian
category.

\subsection{Categories of modules\label{Subsection:M-categories}}

We would like to consider generalizations of additive, quasi-abelian, and
triangulated categories, where the hom-sets are not just abelian groups, but
objects in a more general category, such as the category of pro-countable
Polish modules, or its left heart. To this purpose, we consider the abstract
notion of \emph{category of modules}. We think of a category of modules as a
category whose objects are modules \textquotedblleft with additional
structure\textquotedblright\ and the morphisms are the module homomorphisms
that \textquotedblleft preserve\textquotedblright\ this additional
structure. For example, the additional structure can be a topology, in which
case the module homomorphisms that preserve the additional structure are the
continuous ones. Consistently with this interpretation, we will notice that
a category of modules in the sense we are about to define is always a (not
necessarily full) subcategory of the category of $R$-modules for some ring $%
R $.

Recall that one can define the notion of \emph{(abelian) group object} in an
abstract category with finite products \cite[Definition 4.1]%
{awodey_category_2006} (here and in the following, we only consider group
objects with respect to a cartesian monoidal structure). The group objects
in the category of groups are precisely the abelian groups \cite[Corollary
4.6]{awodey_category_2006}. Furthermore, such a group object structure is
unique, and it is induced by the given operation on the abelian group. When $%
R$ is a ring, $R$ itself is not only a tensor identity, but also a
projective generator for the category of $R$-modules. (This means that $R$
is projective, and two parallel arrows $f,g$ are equal if and only if $fx=gx$
for every arrow $x$ with source $R$.)

Suppose now that $\mathcal{M}$ is a category with finite products such that
every object has a unique group object structure. By the \emph{%
Eckmann--Hilton argument} \cite[Example 4.4]{awodey_category_2006}, the
unique group object structure on a given object is necessarily abelian.
Furthermore, every morphism in $\mathcal{M}$ is necessarily a homomorphism 
\cite[Definition 4.2]{awodey_category_2006} with respect to the unique group
object structures on the source and target. The group structure on the
objects induce a group operation on the hom-sets, that turns $\mathcal{M}$
into an $\mathbf{Ab}$-category. We say that $\mathcal{M}$ is (quasi-)abelian
if it is a (quasi-)abelian category with respect to this canonical $\mathbf{%
Ab}$-category structure.

\begin{definition}
A \emph{category of modules }is a \emph{locally small tensor category} $%
\mathcal{M}$ such that:

\begin{enumerate}
\item $\mathcal{M}$ has finite products, and every object has a unique group
object structure;

\item with respect to the induced $\mathbf{Ab}$-category structure and the
given tensor structure, $\mathcal{M}$ is a quasi-abelian tensor category
with enough projectives;

\item the tensor identity $R$ is projective and a generator;

\item for projective objects $P$ and $Q$, $P\otimes Q$ is projective;

\item every projective object $P$ is flat, i.e., the functor $P\otimes -$ is
exact.
\end{enumerate}
\end{definition}

If $\mathcal{M}$ is a category of modules, then it is easily seen
considering the remarks above that $\boldsymbol{\Pi }\left( \mathcal{M}%
\right) $ is also a category of modules. The fact that each object of $%
\boldsymbol{\Pi }\left( \mathcal{M}\right) $ admits a unique group object
structure follows from instance from \cite[Proposition 2.1.5]%
{edwards_cech_1976}. Likewise, the left heart $\mathrm{LH}\left( \mathcal{M}%
\right) $ is also a category of modules, where the unique group structure on
an object $P/R$ is the one induced by $P$. The uniqueness follows for
example from the Eckmann--Hilton argument \cite{eckmann_group-like_1962}
together with Lemma \ref{Lemma:left-heart-projectives}.

If $R$ is the tensor identity of $\mathcal{M}$, one can identify $\mathcal{M}
$ with the concrete category whose objects are the abelian groups $\mathrm{%
Hom}\left( R,A\right) $ for $A\in \mathcal{M}$, and whose morphisms $\mathrm{%
Hom}\left( R,A\right) \rightarrow \mathrm{Hom}\left( R,B\right) $ are the
group homomorphisms induced by arrows $A\rightarrow B$ in $\mathcal{M}$.
Notice that the group $\mathrm{Hom}\left( R,R\right) $ is in fact a ring,
and $\mathrm{Hom}\left( R,A\right) $ a right $\mathrm{Hom}\left( R,R\right) $%
-module for every $A\in \mathcal{M}$.

As $\mathcal{M}$ is a monoidal category, one can consider the notion of $%
\mathcal{M}$-category. For example, if $R$ is a commutative ring, then the
category $\mathrm{Mod}\left( R\right) $ of $R$-modules is an abelian
category of modules in the sense above. A $\mathrm{Mod}\left( R\right) $%
-category is then precisely an $R$-linear category in the usual sense.

Given an $\mathcal{M}$-category $\mathcal{A}$, for every finite category $J$%
, one can endow the category of diagrams $\mathcal{A}^{J}$ of shape $J$ with
a natural structure of $\mathcal{M}$-category. This allows one to define a 
\emph{(quasi-)abelian }$\mathcal{M}$\emph{-category} to be an $\mathcal{M}$%
-category that is also (quasi-)abelian, and such that for every finite
category $J$, finite limits and finite colimits of shape $J$ are given by $%
\mathcal{M}$-functors $\mathcal{A}^{J}\rightarrow \mathcal{A}$. Likewise,
one defines the notion of \emph{exact }$\mathcal{M}$-\emph{category}.

One can similarly generalize the notion of triangulated category to the
context of $\mathcal{M}$-categories. A\emph{\ triangulated }$\mathcal{M}$%
\emph{-category} will be an $\mathcal{M}$-category $\mathcal{T}$ that is
also triangulated, such that the shift functor is an $\mathcal{M}$-functor,
and the $\mathrm{Hom}$-functor is a cohomological functor in the first
variable and a homological functor in the second variable \cite[Definition
1.1.7]{neeman_triangulated_2001}, when regarded as a functor with values in
the left heart of $\mathcal{M}$.

As a particular instances of Corollary \ref{Corollary:derived-projectives}
and Proposition \ref{Proposition:derived-functor} we obtain the following;
see also \cite[Section 13.4]{kashiwara_categories_2006}. Notice that we can
identify $\mathrm{Ch}^{b}\left( \mathcal{A}\right) ^{\mathrm{op}}$ with $%
\mathrm{Ch}^{b}\left( \mathcal{A}^{\mathrm{op}}\right) $ using the
convention on signs from \cite[Proposition 11.3.12]%
{kashiwara_categories_2006}.

\begin{proposition}
Suppose that $\mathcal{M}$ is a category of modules. Let $\mathcal{P}$ be
its full subcategory spanned by its projective objects. Then the inclusion $%
\mathcal{P}\rightarrow \mathcal{M}$ induces an equivalence $\mathrm{K}%
^{b}\left( \mathcal{P}\right) \rightarrow \mathrm{D}^{b}\left( \mathcal{M}%
\right) $. In particular, $\mathrm{D}^{b}\left( \mathcal{M}\right) $ is a
triangulated $\mathrm{LH}\left( \mathcal{P}\right) $-category.
\end{proposition}

\begin{proposition}
Suppose that $\mathcal{M}$ is a category of modules.\ Let $\mathcal{A}$ be a
hereditary idempotent-complete exact $\mathcal{M}$-category with enough
projectives. Then the triangulated bifunctor $\mathrm{Hom}_{\mathcal{A}%
}^{\bullet }:\mathrm{K}^{b}(\mathcal{A)}^{\mathrm{op}}\times \mathrm{K}^{b}(%
\mathcal{A)}\rightarrow \mathrm{K}^{b}\left( \mathcal{M}\right) $ has a
total right derived functor $\mathrm{RHom}_{\mathcal{A}}^{\bullet }:\mathrm{D%
}^{b}(\mathcal{A)}^{\mathrm{op}}\times \mathrm{D}^{b}(\mathcal{A)}%
\rightarrow \mathrm{D}^{b}(\mathcal{M)}$. Furthermore, this can be defined
as follows: For a bounded complex $A$ over $\mathcal{A}$, pick a bounded
complex $P_{A}$ over $\mathcal{P}$ together with quasi-isomorphisms $\eta
_{A}:P_{A}\rightarrow A$.\ Then one can define $\left( \mathrm{RHom}_{%
\mathcal{A}}^{\bullet }\right) \left( A,B\right) :=\mathrm{Hom}_{\mathcal{A}%
}^{\bullet }\left( P_{A},B\right) $ for $A\in \mathrm{D}^{b}\left( \mathcal{A%
}\right) ^{\mathrm{op}}$ and $B\in \mathrm{D}^{b}\left( \mathcal{A}\right) $%
. We have that 
\begin{equation*}
\mathrm{H}^{0}\circ \mathrm{RHom}_{\mathcal{A}}^{\bullet }:\mathrm{D}%
^{b}\left( \mathcal{A}\right) ^{\mathrm{op}}\times \mathrm{D}^{b}\left( 
\mathcal{A}\right) \rightarrow \mathrm{LH}\left( \mathcal{M}\right)
\end{equation*}%
is a cohomological right derived functor of $\mathrm{Hom}_{\mathcal{A}%
}^{\bullet }$ such that 
\begin{equation*}
\mathrm{H}^{k}\circ \mathrm{RHom}^{\bullet }\left( A,B\right) \cong \mathrm{%
Ext}^{k}\left( A,B\right) \in \mathrm{LH}\left( \mathcal{M}\right)
\end{equation*}%
for objects $A,B$ of $\mathcal{A}$.
\end{proposition}

\section{Countable modules\label{Section:countable}}

By a \emph{domain} we will mean a \emph{unital commutative }ring $R$ with no
zero divisors. We will denote by $K$ its field of fractions. A module will
always mean an $R$-module. We say that a module $M$ is \emph{finite }or
finitely-generated if there exist $n\in \omega $ and $a_{0},\ldots ,a_{n}\in
M$ such that $M$ is equal to the submodule of $M$ generated by $a_{0},\ldots
,a_{n}$. An ideal is finite or finitely-generated if it is so as a module.

\subsection{Some ring terminology}

A ring is called:

\begin{itemize}
\item \emph{semihereditary} if every finite ideal is projective \cite[%
Section 4.4]{rotman_introduction_2009};

\item \emph{hereditary }if every ideal is projective \cite[Section 4.3]%
{rotman_introduction_2009}.
\end{itemize}

A \emph{Pr\"{u}fer domain} is a semihereditary domain, while a \emph{%
Dedekind domain} is a hereditary domain. Pr\"{u}fer and Dedekind domains are
well-studied classes of rings, which admit several equivalent
characterizations and extend many other important classes of rings.

A \emph{Bezout domain }is a domain in which every finite ideal is principal,
while a \emph{principal ideal domain} (PID) is a domain where every ideal is
principal. A \emph{valuation ring} is a domain $R$ that satisfies $a|b$ or $%
b|a$ for all $a,b\in R$. A \emph{local ring }is a ring with a unique maximal
ideal. A domain is \emph{Noetherian }if every ideal is finitely generated.
Every PID is a Dedekind domain, every valuation ring is a Bezout domain, and
every\ Bezout domain is a Pr\"{u}fer domain. Conversely, a domain $R$ is Pr%
\"{u}fer if and only if its localization $R_{P}$ with respect to any prime
(or, equivalently, to any maximal) ideal $P$ is a valuation ring \cite[%
Example 4.28]{rotman_introduction_2009}. A ring is a Dedekind domain if and
only if it is a Noetherian Pr\"{u}fer domain, and a valuation domain if and
only if it is a local Pr\"{u}fer domain.

A module $M$ is \emph{finitely-presented }if there exists an exact sequence $%
R^{d}\rightarrow R^{n}\rightarrow M\rightarrow 0$. When $R$ is a Pr\"{u}fer
domain, this is equivalent to the assertion that there exists a short exact
sequence $0\rightarrow Q\rightarrow P\rightarrow M\rightarrow 0$ where $P$
and $Q$ are finite projective modules. When $R$ is a Dedekind domain, every
finite module is finitely-presented. A module $M$ is \emph{flat} if the
functor $M\otimes -$ is exact \cite[Section 3.3]{rotman_introduction_2009}.
We say that $M$ is \emph{finiteflat} if it is both finite and flat, and 
\emph{countableflat }if it is both countable and flat. A short exact sequence%
\begin{equation*}
0\rightarrow A\rightarrow B\rightarrow C\rightarrow 0
\end{equation*}%
of modules is \emph{pure-exact }if, for every module $N$,%
\begin{equation*}
0\rightarrow N\otimes A\rightarrow N\otimes B\rightarrow N\otimes
C\rightarrow 0
\end{equation*}%
is exact \cite[Section 3.3.1]{rotman_introduction_2009}. In this case, one
says that $A$ is a pure submodule of $B$. The pure-exact sequences are
precisely those that belong to the exact structure projectively generated by
the finitely-presented modules \cite[Theorem 2.1.2]{benson_phantom_1999}.

\begin{lemma}
\label{Lemma:pure-exact-extension}Suppose that $M^{\prime }\rightarrow
M\rightarrow M^{\prime \prime }$ is a short exact sequence of modules.

\begin{itemize}
\item If $M^{\prime \prime }$ is flat, then $M$ is flat if and only if $%
M^{\prime }$ is flat;

\item If $M$ is flat, then $M^{\prime }\rightarrow M\rightarrow M^{\prime
\prime }$ is pure if and only if $M^{\prime \prime }$ is flat;

\item If $M^{\prime \prime }$ is flat, then $M^{\prime }\rightarrow
M\rightarrow M^{\prime \prime }$ is pure.
\end{itemize}
\end{lemma}

\begin{proof}
This combines \cite[Corollary 7.4]{rotman_introduction_2009}, \cite[%
Corollary 4.86]{lam_lectures_1999}, and \cite[Section 2.7]%
{benson_phantom_1999}.
\end{proof}

A nonzero element $a$ of a module $M$ is \emph{torsion }if $ra=0$ for some $%
r\in R$. A module is called \emph{torsion }if every nonzero element is
torsion, and \emph{torsion-free }if\emph{\ }no nonzero element is torsion.
If $M$ is a module, then the set $M_{\mathrm{t}}$ of torsion elements is a
submodule such that $M\left/ M_{\mathrm{t}}\right. $ is torsion-free. Every
flat module is torsion-free, while the converse does not holds in general.

A submodule $A$ of a torsion-free module $B$ is \emph{relatively divisible }%
if $rA=rB\cap A$ for every $r\in R$. Every pure submodule is relatively
divisible, but the converse does not necessarily hold.

\subsection{Characterization of Pr\"{u}fer and Dedekind domains}

One can characterize the classes of Pr\"{u}fer and Dedekind domains in terms
of properties of modules.

\begin{proposition}
\label{Proposition:characterize-prufer}Let $R$ be a domain. The following
assertions are equivalent:

\begin{enumerate}
\item $R$\emph{\ }is a Pr\"{u}fer domain;

\item every torsion-free $R$-module is flat;

\item every finite torsion-free $R$-module is projective;

\item Every finite $R$-submodule of a projective $R$-module is projective;

\item If $M$ is a flat $R$-module and $N$ is a relatively divisible $R$%
-submodule of $M$, then $N$ is pure in $M$;

\item an $R$-module is projective if and only if it is a sum of finite
ideals of $R$;

\item if $M$ is a finite $R$-module, then $M_{\mathrm{t}}$ is a direct
summand of $M$.
\end{enumerate}
\end{proposition}

\begin{proof}
See \cite[Section 4.4]{rotman_introduction_2009} for the equivalence of
(1)--(4).

(2)$\Rightarrow $(5) Suppose that $N$ is relatively divisible in $M$. Then $%
M/N$ is torsion-free, whence flat. Thus, $N$ is pure in $M$ by Lemma \ref%
{Lemma:pure-exact-extension}.

(5)$\Rightarrow $(2) Suppose that $M$ is a torsion-free module. Then there
exists a short-exact sequence $A\rightarrow P\rightarrow M$ where $P$ is
projective. Since $M$ is torsion-free, $A$ is relatively divisible in $P$,
whence pure in $P$. Thus, by Lemma \ref{Lemma:pure-exact-extension}, $M$ is
flat.

(1)$\Leftrightarrow $(6) This is \cite[Theorem 3]{kaplansky_projective_1958}.

(1)$\Leftrightarrow $(7) This is proved in \cite[Theorem 1]%
{kaplansky_modules_1952} and \cite{kaplansky_characterization_1960}.
\end{proof}

\begin{proposition}
\label{Proposition:characterize-Dedekind}Let $R$ be a domain. Then the
following assertions are equivalent:

\begin{enumerate}
\item $R$ is a Dedekind domain;

\item every $R$-submodule of a projective $R$-module is projective;

\item an $R$-module is projective if and only if it is a sum of ideals of $R$%
.
\end{enumerate}

Likewise, we have that $R$ is a PID if and only if every $R$-submodule of a
free $R$-module is free.
\end{proposition}

\begin{proof}
See \cite[Section 4.3]{rotman_introduction_2009}.
\end{proof}

Let $M$ be a flat module over a Pr\"{u}fer domain. We define its\emph{\ rank}
to be the dimension of $K\otimes M$ as a $K$-vector space. Clearly, if $M$
is finite, then it has finite rank. The converse holds if $M$ is projective,
as in this case it is a sum of finite ideals of $R$.

By Proposition \ref{Proposition:characterize-prufer}, if $X$ is a subset of $%
M$, then there exists a smallest pure submodule $\langle X\rangle _{\ast }$
of $M$ containing $X$, called the \emph{pure hull} of $X$ in $M$. If $X$ has 
$n$ elements, then $\langle X\rangle _{\ast }$ has rank at most $n$.

\subsection{Finite flat modules\label{Subsection:finiteflat}}

In this section, we assume that $R$ is a Pr\"{u}fer domain. Let $M$ be a
finite flat (whence, projective) module. We set $M^{\ast }:=\mathrm{Hom}%
\left( M,R\right) $, which is also a finite flat module. We have a canonical
homomorphism $M\rightarrow M^{\ast \ast }$.

\begin{lemma}
\label{Lemma:finiteflat-duality}Suppose that $R$ is a Pr\"{u}fer domain. Let 
$M$ be a finite flat module. Then the canonical homomorphism $\eta
_{M}:M\rightarrow M^{\ast \ast }$ is an isomorphism.
\end{lemma}

\begin{proof}
As $M$ is a finite direct sum of finite ideals of $R$, it suffices to
consider the case when $M$ is a finite ideal $I$ of $R$. We claim that the
image of the canonical homomorphism $\eta _{I}:I\rightarrow I^{\ast \ast }$
is a pure submodule of $I^{\ast \ast }$. Suppose that $\Phi \in I^{\ast \ast
}$ and $r\in R$ are such that $r\Phi \in \mathrm{Ran}\left( \eta _{I}\right) 
$. Thus, there exists $a$ such that $r\Phi =\eta _{I}\left( a\right) $. For
every $\psi \in I^{\ast }$ we have that%
\begin{equation*}
r\Phi \left( \psi \right) =\left( r\Phi \right) \left( \psi \right) =\left(
\eta _{I}\left( a\right) \right) \left( \psi \right) =\psi \left( a\right)
\end{equation*}%
In particular when $\psi $ is the inclusion $\iota :I\rightarrow R$,%
\begin{equation*}
r\Phi \left( \iota \right) =\iota \left( a\right) =a\text{.}
\end{equation*}%
Thus, setting%
\begin{equation*}
b:=\Phi \left( \iota \right) =\frac{1}{r}a
\end{equation*}%
we have that, for every $\psi \in I^{\ast }$,%
\begin{equation*}
\Phi \left( \psi \right) =\frac{1}{r}\psi \left( a\right) =\psi (\frac{1}{r}%
a)=\psi \left( b\right) \text{.}
\end{equation*}%
This shows that $\mathrm{Ran}\left( \iota \right) $ is pure in $I^{\ast \ast
}$. Thus, we can write $I^{\ast \ast }=I\oplus I^{\bot }$ for some pure
submodule $I^{\bot }$ of $I^{\ast \ast }$.

Since $I$ is projective and finite, it is a direct summand of a free finite
module. Thus, there exist $n\in \omega $ such that $R^{n}=I\oplus M$ for
some pure submodule $M$ of $R^{n}$. Thus,%
\begin{equation*}
R^{n}\cong \left( R^{n}\right) ^{\ast \ast }\cong I^{\ast \ast }\oplus
M^{\ast \ast }\cong I\oplus I^{\bot }\oplus M^{\ast \ast }\text{.}
\end{equation*}%
This implies%
\begin{eqnarray*}
n &=&\mathrm{\mathrm{\mathrm{rk}}}\left( R^{n}\right) =\mathrm{\mathrm{%
\mathrm{rk}}}\left( I\right) +\mathrm{\mathrm{\mathrm{rk}}}\left( I^{\bot
}\right) +\mathrm{\mathrm{\mathrm{rk}}}\left( M^{\ast \ast }\right) \\
&\geq &\mathrm{\mathrm{\mathrm{rk}}}\left( I^{\bot }\right) +\mathrm{\mathrm{%
\mathrm{rk}}}\left( I\right) +\mathrm{\mathrm{\mathrm{rk}}}\left( M\right) \\
&=&n+\mathrm{\mathrm{\mathrm{rk}}}\left( I^{\bot }\right) \text{.}
\end{eqnarray*}%
This forces $I^{\bot }=0$, concluding the proof.
\end{proof}

\subsection{Finitely-presented modules}

In this section, we continue to assume that $R$ is a Pr\"{u}fer domain with
field of fractions $K$. A \emph{fractional ideal }of $R$ is a submodule $%
I\subseteq K$ such that there exists $a\in R$ with $aI\subseteq R$. Given
two fractional ideals $I$ and $J$, one defines $IJ$ to be the submodule of $%
K $ generated by $xy$ for $x\in I$ and $y\in J$. A fractional ideal $I$ is 
\emph{invertible }if there exists a fractional ideal $J$ such that $IJ=R$,
in which case $J$ is the ideal $I^{-1}=\left\{ a\in R:aI\subseteq R\right\} $%
.

\begin{lemma}
\label{Lemma:projective-invertible}Suppose that $I$ is a fractional ideal of 
$R$. Then $I$ is invertible if and only if $I$ is projective, and in this
case $I$ is finite.
\end{lemma}

\begin{proof}
Notice that a projective ideal is necessarily finite. Indeed, if $F$ is a
free module, $F\rightarrow I$ is a surjective homomorphism, and $%
I\rightarrow F$ is a right inverse, then $I\rightarrow F$ factors through a
finite direct summand of $F$, since $I$ is an ideal of $R$ and $R$ is a
domain. The conclusion thus follows from \cite[Theorem 2.5]%
{bazzoni_prufer_2006}.
\end{proof}

Notice that a module is finite flat of rank $1$ if and only if it is
isomorphic to a finite (fractional) ideal of $R$. Let us say that a
finitely-presented torsion module is \emph{cyclic} if it is of the form $L/I$
for some rank $1$ finite flat module $L$ and finite submodule $I\subseteq L$.

\begin{lemma}
\label{Lemma:structure-torsion}Every finitely-presented torsion module is a
finite direct sum of finitely-presented cyclic torsion modules.
\end{lemma}

\begin{proof}
Let $M$ be a finitely-presented torsion module. Since $M$ is
finitely-presented, $M\cong F/S$ for some finite free module $F$ and finite
submodule $S\subseteq F$. Since $M$ is torsion, $\langle S\rangle _{\ast }=F$%
. As $S$ is finite and flat, there exist $n\in \omega $ and finite ideals $%
I_{0},\ldots ,I_{n}$ in $R$ such that $S=I_{0}\oplus \cdots \oplus I_{n}$.
If $F_{k}:=\langle I_{k}\rangle _{\ast }\subseteq F $ for $0\leq k\leq n$,
then we have that $F=F_{0}\oplus \cdots \oplus F_{n} $, and%
\begin{equation*}
F/S\cong F_{0}/I_{0}\oplus \cdots \oplus F_{n}/I_{n}\text{.}
\end{equation*}%
As $F_{k}$ is a pure submodule of $F$, it is a finite projective module,
which has rank $1$ since so does $I_{k}$.
\end{proof}

\begin{lemma}
\label{Lemma:structure-finitely-presented}Every finitely-presented module is
a finite direct sum of finitely-presented cyclic torsion modules and finite
ideals.
\end{lemma}

\begin{proof}
Let $M$ be a finitely-presented module. We have a short exact sequence%
\begin{equation*}
0\rightarrow M_{\mathrm{t}}\rightarrow M\rightarrow A\rightarrow 0
\end{equation*}%
where $A$ is finite and torsion-free, whence projective.\ Therefore $M\cong
M_{\mathrm{t}}\oplus A$. Being direct summands of $M$, $M_{\mathrm{t}}$ and $%
A$ are finitely-presented. The conclusion thus follows from Lemma \ref%
{Lemma:structure-torsion} and Proposition \ref%
{Proposition:characterize-prufer}.
\end{proof}

Every Pr\"{u}fer domain is, in particular, a coherent ring \cite[Section II.3%
]{lombardi_commutative_2015}. It follows that finitely-presented modules
form a thick subcategory of the category of modules \cite[Section IV.3]%
{lombardi_commutative_2015}.

\subsection{Countably presented modules}

A module $M$ is \emph{countably presented }if there exists an exact sequence%
\begin{equation*}
R^{\left( \omega \right) }\rightarrow R^{\left( \omega \right) }\rightarrow
M\rightarrow 0\text{.}
\end{equation*}%
This is equivalent to the assertion that $M\cong \mathrm{co\mathrm{lim}}%
_{n}M_{n}$ for some inductive sequence $\left( M_{n}\right) $ of \emph{%
finitely presented }modules, as well as to the assertion that $M\cong P/Q$
where $Q\subseteq P$ are countable projective modules. Similar arguments as
for finitely-presented modules show that countably presented modules form an
abelian category $\mathbf{Mod}_{\aleph _{0}}\left( R\right) $. When $R$ is
countable, every countably-presented module is countable, and $\mathbf{Mod}%
_{\aleph _{0}}\left( R\right) $ is a thick abelian subcategory of the
abelian category of countable modules.

\section{Pro-countable Polish modules\label{Section:pro-countable}}

Throughout this section, we assume that $R$ is a \emph{countable} domain,
and assume that all the modules are $R$-modules.

\subsection{Polish modules}

A \emph{Polish }module is a module that is also a Polish abelian group, such
that for every $r\in R$ the function $M\rightarrow M$, $x\mapsto rx$ is
continuous. We have that the category $\mathbf{PolMod}\left( R\right) $ of 
\emph{Polish} $R$-modules and \emph{continuous} homomorphisms is a
quasi-abelian category \cite[Section 6]{lupini_looking_2024}. The admissible
arrows in this category are the continuous homomorphisms with closed image.
We let $\mathbf{Mod}\left( R\right) $ be the category of countable modules, $%
\mathbf{Flat}\left( R\right) $ be its full subcategory of countable flat
modules, and\textrm{\ }$\mathbf{FinFlat}\left( R\right) $ be its full
subcategory of finite flat modules.

We say that a Polish module $M$ is \emph{pro-countable} if it is isomorphic
to the inverse limit of a tower of countably presented modules, and \emph{%
pro-finite }if it is isomorphic to the inverse limit of a tower of finite
modules. More generally, if $\mathcal{A}$ is a fully exact subcategory of
the category $\mathbf{Mod}_{\aleph _{0}}\left( R\right) $ of countably
presented modules, we say that a Polish module is pro-$\mathcal{A}$ if it is
isomorphic to the inverse limit of a tower of modules in $\mathcal{A}$. We
let $\mathbf{\Pi }\left( \mathcal{A}\right) $ be the full subcategory of $%
\mathbf{PolMod}\left( R\right) $ spanned by the pro-$\mathcal{A}$ modules.

The same proof as \cite[Proposition 2.1]{ding_non-archimedean_2017} shows
that, when $\mathcal{A}$ is closed under submodules, $M$ is pro-$\mathcal{A}$
if and only if $M$ is isomorphic to a closed submodule of a countable
product of modules in $\mathcal{A}$. In particular, in this case $%
\boldsymbol{\Pi }\left( \mathcal{A}\right) $ is closed under taking closed
submodules. Note that, for pro-countable Polish modules $N\cong \mathrm{%
\mathrm{lim}}_{k}N_{k}$ and $M\cong \mathrm{\mathrm{lim}}_{k}M_{k}$ with $%
N_{k},M_{k}$ in $\mathcal{A}$, we have that the group of continuous
homomorphisms $\mathrm{Hom}\left( N,M\right) $ is isomorphic to 
\begin{equation*}
\mathrm{\mathrm{lim}}_{k}\mathrm{co\mathrm{lim}}_{j}\mathrm{Hom}\left(
N_{j},M_{k}\right) \text{.}
\end{equation*}%
This establishes an equivalence between $\boldsymbol{\Pi }(\mathcal{A})$ and
the full subcategory of the category $\mathrm{pro}\left( \mathcal{A}\right) $
of towers of modules in $\mathcal{A}$ spanned by essentially epimorphic
towers. Thus, the notation adopted here is consistent with the notation from
Section \ref{Subsection:towers}.

\begin{lemma}
\label{Lemma:quotient-prodiscrete}Let $\mathcal{A}$ be a thick abelian
subcategory of $\mathbf{Mod}\left( R\right) $. Then $\mathbf{\Pi }\left( 
\mathcal{A}\right) $ is closed under quotients by closed submodules.
\end{lemma}

\begin{proof}
Suppose that $\left( A_{i}\right) _{i\in \omega }$ is a sequence of modules
in $\mathcal{A}$ and $N$ is a closed submodule of $A:=\prod_{i\in \omega
}A_{i}$. Define $A_{\leq k}:=\prod_{i\leq k}A_{i}$, which we can identify
with a closed submodule of $A$, and $N_{\leq k}:=A_{\leq k}\cap N$. Then we
have that $A\left/ N\right. $ is isomorphic to the inverse limit%
\begin{equation*}
\mathrm{\mathrm{lim}}_{k}(A_{\leq k}/N_{\leq k})\text{.}
\end{equation*}%
Since $A_{\leq k}/N_{\leq k}$ is isomorphic to an object of $\mathcal{A}$,
this shows that $A/N$ is pro-countable.
\end{proof}

\begin{lemma}
\label{Lemma:extension-prodiscrete1}Suppose $\mathcal{A}$ is a fully exact
subcategory of $\mathbf{Mod}\left( R\right) $ closed under submodules.
Suppose that $A\rightarrow B\rightarrow C$ is a short exact sequence in $%
\mathbf{PolMod}\left( R\right) $. If $A$ is pro-$\mathcal{A}$ and $C$ is in $%
\mathcal{A}$, then $B$ is pro-$\mathcal{A}$.
\end{lemma}

\begin{proof}
Suppose initially that $A=\prod_{i\in \omega }A_{i}$ for a sequence $\left(
A_{i}\right) $ of modules in $\mathcal{A}$. For $i\in \omega $, consider the
commutative diagram%
\begin{equation*}
\begin{array}{ccccc}
A & \rightarrow & B & \rightarrow & C \\ 
\pi _{i}\downarrow &  & \downarrow \xi _{i} &  & \downarrow \\ 
A_{i} & \overset{\tau _{i}}{\rightarrow } & B_{i} & \overset{p_{i}}{%
\rightarrow } & C%
\end{array}%
\end{equation*}%
obtained via pushout, where $\pi _{i}:A\rightarrow A_{i}$ is the canonical
projection. Thus, we have that $B_{i}$ is the quotient of $A_{i}\oplus B$ by
the submodule%
\begin{equation*}
\Xi _{i}:=\left\{ \left( -\pi _{i}\left( x\right) ,x\right) :x\in A\right\} 
\text{.}
\end{equation*}%
Since $B_{i}$ is an extension of $C$ by $A_{i}$, it is isomorphic to an
object of $\mathcal{A}$. Define 
\begin{equation*}
B_{\omega }:=\prod_{i\in \omega }B_{i}
\end{equation*}%
and the continuous homomorphism%
\begin{equation*}
\xi :B\rightarrow B_{\omega }\text{, }b\mapsto \left( \xi _{i}\left(
b\right) \right) _{i\in \omega }\text{.}
\end{equation*}%
Then $\xi $ is injective, and its image is equal to the preimage of the
diagonal $\Delta _{C}\subseteq C^{\omega }$ under the continuous
homomorphism $B_{\omega }\rightarrow C^{\omega }$, $\left( b_{i}\right)
\mapsto \left( p_{i}\left( b_{i}\right) \right) $. Thus, $B$ is isomorphic
to a closed submodule of $B_{\omega }$, and therefore pro-$\mathcal{A}$.

If $A$ is an arbitrary pro-$\mathcal{A}$ module, then we have an admissible
monic $A\rightarrow \prod_{k\in \omega }A_{k}$ for some sequence $\left(
A_{k}\right) $ of modules in $\mathcal{A}$. By taking pushouts, we obtain a
commutative diagram%
\begin{equation*}
\begin{array}{ccccc}
A & \rightarrow & X & \rightarrow & C \\ 
\downarrow &  & \downarrow &  & \downarrow \\ 
\prod_{k\in \omega }A_{k} & \rightarrow & Y & \rightarrow & C%
\end{array}%
\end{equation*}%
where the arrow $X\rightarrow Y$ is an admissible monic. By the above
discussion, $Y$ is pro-$\mathcal{A}$, and hence also $X$ is pro-$\mathcal{A}$%
.
\end{proof}

\begin{lemma}
\label{Lemma:enough-projectives}Suppose that $\mathcal{A}$ is a fully exact
subcategory of $\mathbf{Mod}\left( R\right) $. Suppose that $\mathcal{P}%
\subseteq \mathcal{A}$ is a class of projective objects of $\mathcal{A}$
closed under direct summands such that for every $A\in \mathcal{A}$ there
exists $Q\in \mathcal{P}$ and a strict epimorphism $Q\rightarrow A$. For
every pro-$\mathcal{A}$ module $C$ there exists a surjective continuous
module homomorphism $P\rightarrow C$ where $P$ is a countable product of
objects of $\mathcal{P}$.
\end{lemma}

\begin{proof}
Let $\left( A^{\left( n\right) }\right) $ be a tower of modules in $\mathcal{%
A}$ with epimorphic bonding maps $A^{\left( n+1\right) }\rightarrow
A^{\left( n\right) }$ such that $A=\mathrm{lim}_{n}A^{\left( n\right) }$. We
define by recursion on $k\in \omega $ modules $P^{\left( k\right) }$ in $%
\mathcal{P}$, surjective homomorphisms $P^{\left( k\right) }\rightarrow
A^{\left( k\right) }$ and $P^{\left( k+1\right) }\rightarrow P^{\left(
k\right) }$ for $k\geq 1$ such that the diagram%
\begin{equation*}
\begin{array}{ccc}
P^{\left( k+1\right) } & \rightarrow & A^{\left( k+1\right) } \\ 
\downarrow &  & \downarrow \\ 
P^{\left( k\right) } & \rightarrow & A^{\left( k\right) }%
\end{array}%
\end{equation*}%
commutes. Define $P^{\left( 0\right) }$ to be a module in $\mathcal{P}$
together with a surjective homomorphism $P^{\left( 0\right) }\rightarrow
A^{\left( 0\right) }$. Supposing that $P^{\left( i\right) }$ has been
defined for $i<k$, let $B^{\left( k\right) }$ be the pullback%
\begin{equation*}
\begin{array}{ccc}
B^{\left( k\right) } & \rightarrow & A^{\left( k\right) } \\ 
\downarrow &  & \downarrow \\ 
P^{\left( k-1\right) } & \rightarrow & A^{\left( k-1\right) }%
\end{array}%
\end{equation*}%
Define $P^{\left( k\right) }$ to be a module in $\mathcal{P}$ together with
a surjective homomorphism $P^{\left( k\right) }\rightarrow B^{\left(
k\right) }$. Consider then the homomorphisms $P^{\left( k\right)
}\rightarrow B^{\left( k\right) }\rightarrow P^{\left( k-1\right) }$ and $%
P^{\left( k\right) }\rightarrow B^{\left( k\right) }\rightarrow A^{\left(
k\right) }$. This concludes the recursive definition. One then sets $P:=%
\mathrm{lim}_{n}P^{\left( k\right) }$, and $P\rightarrow A$ be the
surjective continuous homomorphism induced by the morphisms $P^{\left(
k\right) }\rightarrow A^{\left( k\right) }$ for $k\in \omega $. Since $%
\left( P^{\left( k\right) }\right) $ is an epimorphic tower of projective
modules of $\mathcal{A}$ and $\mathcal{P}$ is closed under direct summands, $%
P$ is isomorphic to a product of objects of $\mathcal{P}$.
\end{proof}

\subsection{Extensions of modules and cocycles\label{Section:cocycles}}

Suppose that $C,A$ are Polish modules. We can define the group $\mathrm{Ext}%
_{\mathrm{Yon}}\left( C,A\right) $ of isomorphism classes of extensions $%
A\rightarrow X\rightarrow C$, where the group operation is induced by Baer
sum of extensions \cite[Section 7.2]{rotman_introduction_2009}. The trivial
element of $\mathrm{Ext}_{\mathrm{Yon}}\left( C,A\right) $ corresponds to 
\emph{split }extensions of $C$ by $A$. One can obtain an equivalent
description of such a group in terms of cocycles.

A ($2$-)cocycle on $C$ with values in $A$ is a pair $\left( c,\rho \right) $
where $c:C\times C\rightarrow A$ and $\rho :R\times C\rightarrow A$ satisfy
the following identities, for all $r,s\in R$ and $x,y,z\in C$:

\begin{enumerate}
\item $c\left( x,0\right) =0$;

\item $c\left( x,y\right) =c\left( y,x\right) $;

\item $c\left( x+y,z\right) +c\left( x,y\right) =c\left( x,y+z\right)
+c\left( y,z\right) $;

\item $\rho \left( r+s,x\right) =c\left( rx,ry\right) +\rho \left(
r,x\right) +\rho \left( s,x\right) $;

\item $c\left( rx,ry\right) +\rho \left( r,x\right) +\rho \left( r,y\right)
=\rho \left( r,x+y\right) +rc\left( x,y\right) $.
\end{enumerate}

We say that a cocycle is Borel (respectively, continuous) if each of the
functions $c$ and $\rho $ is Borel (respectively, continuous). A Borel
cocycle is a coboundary if there exists a Borel function $\xi :C\rightarrow
A $ such that%
\begin{equation*}
\left( c,\rho \right) :=\delta \xi
\end{equation*}%
where $\delta \xi $ is the pair of functions $\left( c_{\xi },\rho _{\xi
}\right) $ defined by%
\begin{equation*}
c_{\xi }\left( x,y\right) :=\xi \left( x+y\right) -\xi \left( x\right) -\xi
\left( y\right)
\end{equation*}%
and%
\begin{equation*}
\rho _{\xi }\left( r,x\right) :=\xi \left( rx\right) -r\xi \left( x\right) 
\text{.}
\end{equation*}%
Given a continuous cocycle $\left( c,\rho \right) $ on $C$ with values in $A$%
, one can define an extension%
\begin{equation*}
A\rightarrow X_{\left( c,\rho \right) }\rightarrow C
\end{equation*}%
where the Polish module $X_{\left( c,\rho \right) }$ is defined as follows.
As a topological space, $X_{\left( c,\rho \right) }$ is equal to $C\times A$%
, with module operation defined by setting%
\begin{equation*}
\left( x,a\right) +\left( y,b\right) :=\left( x+y,c\left( x,y\right)
+a+b\right)
\end{equation*}%
\begin{equation*}
r\left( x,a\right) :=\left( rx,\rho \left( r,x\right) +a\right)
\end{equation*}%
for $r\in R$, $x,y\in C$, and $a,b\in A$.

\begin{lemma}
Suppose that $C,A$ are Polish modules, and $\left( c,\rho \right) $ is a
cocycle on $C$ with values in $A$. Let $\xi $ be a Borel function such that $%
c\left( x,y\right) =\xi \left( x+y\right) -\xi \left( x\right) -\xi \left(
y\right) $ for all $x,y\in C$. Then $\xi $ is continuous.
\end{lemma}

\begin{proof}
Define the Polish module $X_{\left( c,\rho \right) }$ as above. Consider
also the direct sum $C\oplus A$. We can define a Borel group isomorphism $%
\varphi :X_{\left( c,\rho \right) }\rightarrow C\oplus A$, $\left(
x,a\right) \mapsto \left( x,\xi \left( x\right) +a\right) $. Since $%
X_{\left( c,\rho \right) }$ and $C\oplus A$ are Polish groups, we must have
that $\varphi $ is continuous \cite[Theorem 9.9]{kechris_classical_1995},
and $\xi $ is continuous.
\end{proof}

Borel cocycles on $C$ with values in $A$ form a module $\mathrm{Z}\left(
C,A\right) $ that has coboundaries as a submodule $\mathrm{B}\left(
C,A\right) $. We let $\mathrm{Ext}_{\mathrm{c}}\left( C,A\right) $ be the
module obtained as the quotient $\mathrm{Z}\left( C,A\right) \left/ \mathrm{B%
}\left( C,A\right) \right. $.

When $C,A$ are pro-countable, we define a homomorphism $\mathrm{Ext}_{%
\mathrm{Yon}}\left( C,A\right) \rightarrow \mathrm{Ext}_{\mathrm{c}}\left(
C,A\right) $ as follows. Given an extension $A\rightarrow X\rightarrow C$,
then by \cite[Theorem 4.5]{bergfalk_definable_2024}, the map $X\rightarrow C$
has a continuous right inverse $f:C\rightarrow X$. Identifying $A$ with a
closed submodule of $X$, we can define $c\left( x,y\right) =f\left(
x+y\right) -f\left( x\right) -f\left( y\right) $ and $\rho \left( r,x\right)
=\rho \left( rx\right) -r\rho \left( x\right) $. This defines a continuous
cocycle $\left( c,\xi \right) $ on $C$ with values in $A$, whose
corresponding element of $\mathrm{Ext}_{\mathrm{c}}\left( C,A\right) $ only
depends on the isomorphism class of the given extension.

\begin{lemma}
Suppose that $C$ and $A$ are pro-countable Polish modules. The assignment
described above defines an injective homomorphism $\mathrm{Ext}_{\mathrm{Yon}%
}\left( C,A\right) \rightarrow \mathrm{Ext}_{\mathrm{c}}\left( C,A\right) $,
whose image is the submodule of $\mathrm{Ext}_{\mathrm{c}}\left( C,A\right) $
consisting of elements that are represented by a continuous cocycle.
\end{lemma}

\begin{proof}
Suppose that $A\rightarrow B\rightarrow C$ is a short exact sequence of
pro-countable Polish modules. Then by \cite[Proposition 4.6]%
{bergfalk_definable_2024} the homomorphism $B\rightarrow C$ has a continuous
right inverse $t:C\rightarrow B$. This yields a continuous cocycle
representing an element of $\mathrm{Ext}_{\mathrm{c}}\left( C,A\right) $
that is the image of the element of $\mathrm{Ext}_{\mathrm{Yon}}\left(
C,A\right) $ represented by the given extension. Conversely, if $\left(
c,\rho \right) $ is a continuous cocycle on $C$ with values in $A$, then one
can define an extension $A\rightarrow B\rightarrow C$ associated with $%
\left( c,\rho \right) $ as above. This represents an element of $\mathrm{Ext}%
_{\mathrm{Yon}}\left( C,A\right) $ whose image in $\mathrm{Ext}_{\mathrm{c}%
}\left( C,A\right) $ is represented by the cocycle $\left( c,\rho \right) $.
\end{proof}

One can also easily prove directly the following:

\begin{lemma}
Suppose that $X,Y$ are Polish modules, and $A\rightarrow B\rightarrow C$ is
a short exact sequence of Polish modules.

\begin{enumerate}
\item If $X$ and $C$ are pro-countable, then we have an exact sequence of
abelian groups%
\begin{equation*}
0\rightarrow \mathrm{Hom}\left( X,A\right) \rightarrow \mathrm{Hom}\left(
X,B\right) \rightarrow \mathrm{Hom}\left( X,C\right) \rightarrow \mathrm{Ext}%
_{\mathrm{Yon}}\left( X,A\right) \rightarrow \mathrm{Ext}_{\mathrm{Yon}%
}\left( X,B\right) \rightarrow \mathrm{Ext}_{\mathrm{Yon}}\left( X,C\right)
\end{equation*}

\item If $A,B,C$ are pro-countable, then we have an exact sequence of
abelian groups%
\begin{equation*}
0\rightarrow \mathrm{Hom}\left( C,Y\right) \rightarrow \mathrm{Hom}\left(
B,Y\right) \rightarrow \mathrm{Hom}\left( A,Y\right) \rightarrow \mathrm{Ext}%
_{\mathrm{Yon}}\left( C,Y\right) \rightarrow \mathrm{Ext}_{\mathrm{Yon}%
}\left( B,Y\right) \rightarrow \mathrm{Ext}_{\mathrm{Yon}}\left( A,Y\right)
\end{equation*}
\end{enumerate}
\end{lemma}

\subsection{The thick subcategory of pro-countable Polish modules}

Using the representation of extensions in terms of cocycles we can show that
pro-countable modules form a thick subcategory of the category of Polish
modules. The same proof as \cite[Lemma 3.37]{casarosa_homological_2024}
gives the following result.

\begin{lemma}
\label{Lemma:pro-projective}Suppose that $\mathcal{A}$ is a fully exact
subcategory of $\mathbf{Mod}\left( R\right) $ closed under submodules, and $%
\left( C_{i}\right) _{i\in \omega }$ is a sequence of projective objects in $%
\mathcal{A}$. Then, setting $C:=\prod_{i\in \omega }C_{i}$, we have that $%
\mathrm{Ext}_{\mathrm{Yon}}\left( C,A\right) =0$ for every pro-$\mathcal{A}$
module $A$.
\end{lemma}

\begin{proof}
Assume initially that $A$ is isomorphic to an object of $\mathcal{A}$. Let $%
\left( c,\rho \right) $ be a continuous cocycle representing an element of $%
\mathrm{Ext}_{\mathrm{Yon}}\left( C,A\right) $.\ For $n_{0}<n_{1}\leq \omega 
$ define 
\begin{equation*}
C_{[n_{0},n_{1})}=\prod_{n_{0}\leq i<n_{1}}C_{i}\text{,}
\end{equation*}%
which we identify with a closed submodule of $C$. Since $\left( c,\rho
\right) $ is continuous, there exists $n_{0}\in \omega $ such that $\left(
c,\rho \right) $ is trivial on $C_{[n_{0},\omega )}$. Since $C_{[0,n_{0})}$
is isomorphic to a projective object of $\mathcal{A}$, and $A$ is isomorphic
to an object of $\mathcal{A}$, we also have that $\left( c,\rho \right) $ is
trivial on $C_{[0,n_{0})}$. This easily implies that $\left( c,\rho \right) $
is trivial.

Consider now the general case. Let $\left( V_{n}\right) $ be a basis of open
submodules of $A$ such that $V_{0}=A$ and $A/V_{n}$ is isomorphic to an
object of $\mathcal{A}$ for every $n\in \omega $. Thus, there exists a
continuous function $\xi _{0}:C\rightarrow V_{0}$ such that 
\begin{equation*}
\left( c,\rho \right) \equiv \delta \xi _{0}\ \mathrm{mod}V_{1}\text{,}
\end{equation*}%
which means that the cocycle%
\begin{equation*}
\left( c_{1},\rho _{1}\right) =\left( c,\rho \right) -\delta \xi _{0}
\end{equation*}%
takes values in $V_{1}$. Again, there exists a continuous function $\xi
_{1}:C\rightarrow V_{1}$ such that 
\begin{equation*}
\left( c_{1},\rho _{1}\right) \equiv \delta \xi _{1}\ \mathrm{\mathrm{mod}}%
V_{2}\text{.}
\end{equation*}
Proceeding in this fashion one defines a sequence of continuous cocycles $%
\left( c_{n},\rho _{n}\right) $ on $C$ with values in $V_{n}$ with $\left(
c_{0},\rho _{0}\right) =\left( c,\rho \right) $ and continuous functions $%
\xi _{n}:C\rightarrow V_{n}$ such that 
\begin{equation*}
\left( c_{n+1},\rho _{n+1}\right) =\left( c_{n},\rho _{n}\right) -\delta \xi
_{n}\text{.}
\end{equation*}
Setting%
\begin{equation*}
\xi :=\sum_{n\in \omega }\xi _{n}:C\rightarrow A
\end{equation*}%
we obtain a continuous function such that $\left( c,\rho \right) =\delta \xi 
$.
\end{proof}

\begin{proposition}
\label{Proposition:pro-countable}Let $R$ be a countable domain.

\begin{enumerate}
\item If $\mathcal{A}$ is a fully exact subcategory of $\mathbf{Mod}\left(
R\right) $ with enough projectives and closed under submodules, then $%
\boldsymbol{\Pi }(\mathcal{A})$ is a fully exact subcategory of $\mathbf{%
PolMod}\left( R\right) $ with enough projectives and closed under closed
submodules;

\item If $\mathcal{A}$ is a thick subcategory of $\mathbf{Mod}\left(
R\right) $ with enough projectives, then $\boldsymbol{\Pi }\left( \mathcal{A}%
\right) $ is a thick subcategory of $\mathbf{PolMod}\left( R\right) $ with
enough projectives.
\end{enumerate}
\end{proposition}

\begin{proof}
In view of Lemma \ref{Lemma:quotient-prodiscrete}, it remains to prove that
if $A\rightarrow X\rightarrow C$ is a short exact sequence of Polish modules
where $A$ and $C$ are pro-$\mathcal{A}$, then $X$ is pro-$\mathcal{A}$. As $%
\mathcal{A}$ has enough projectives, it is easy to see that there exists a
continuous surjective homomorphism $P\rightarrow C$ where $P$ is a product
of projective objects of $\mathcal{A}$. By Lemma \ref%
{Lemma:enough-projectives}, $P$ is projective in $\boldsymbol{\Pi }\left( 
\mathcal{A}\right) $. Via pullback we obtain a commuting diagram%
\begin{equation*}
\begin{array}{ccccc}
A & \rightarrow & Y & \rightarrow & P \\ 
\downarrow &  & \downarrow &  & \downarrow \\ 
A & \rightarrow & X & \rightarrow & C%
\end{array}%
\end{equation*}%
By Lemma \ref{Lemma:pro-projective}, $Y$ is isomorphic to $A\oplus P$. Thus, 
$Y$ is pro-$\mathcal{A}$, and hence $X$ is pro-$\mathcal{A}$ by Lemma \ref%
{Lemma:quotient-prodiscrete}, being isomorphic to a quotient of $Y$ by a
closed submodule.
\end{proof}

\subsection{Projective pro-countable Polish modules}

We now characterize the projective objects in $\mathbf{\Pi }\left( \mathcal{A%
}\right) $.

\begin{lemma}
\label{Lemma:submodule-projective}Let $\mathcal{A}$ be a \emph{hereditary }%
fully exact subcategory of $\mathbf{Mod}\left( R\right) $ with enough
projectives. Then $\mathbf{\Pi }\left( \mathcal{A}\right) $ is hereditary,
and its projective objects are precisely the products of projective objects
of $\mathcal{A}$.
\end{lemma}

\begin{proof}
Let $P$ be a projective pro-$\mathcal{A}$ module. Then we have that there
exists a surjective continuous homomorphism $Q\rightarrow P$ where 
\begin{equation*}
Q=\prod_{k\in \omega }Q_{k}
\end{equation*}%
where $Q_{k}$ is a projective object of $\mathcal{A}$. Thus, $P$ is a direct
summand of $Q$. In order to prove that $\mathcal{A}$ is hereditary, it
suffices to prove that if $L$ is a closed submodule of $Q$, then $L$ is a
product of projective objects of $\mathcal{A}$. For $k\in \omega $ define%
\begin{equation*}
N_{k}=\left\{ \left( x_{i}\right) \in Q:\forall i\leq k,x_{i}=0\right\} 
\text{.}
\end{equation*}%
Then we have that $N_{k}$ is an open submodule of $Q$, and $\left(
N_{k}\right) $ is a basis of zero neighborhoods of $Q$. Thus, we have that $%
\left( N_{k}\cap L\right) $ is a basis of zero neighborhoods of $L$. Hence,
we have $L\cong \mathrm{\mathrm{lim}}_{k}(L\left/ \left( N_{k}\cap L\right)
\right. )$. For $k\in \omega $, $L\left/ \left( N_{k}\cap L\right) \right. $
is isomorphic to a closed submodule of $Q\left/ N_{k}\right. \cong
Q_{0}\oplus \cdots \oplus Q_{k}$, and hence a projective object of $\mathcal{%
A}$. Thus, $L$ is isomorphic to a product of projective objects of $\mathcal{%
A}$.
\end{proof}

\subsection{Pro-flat modules}

We let $\mathbf{Flat}\left( R\right) $ be the category of
countably-presented flat modules, and $\mathbf{FinFlat}\left( R\right) $ the
category of finite flat modules.

\begin{lemma}
\label{Lemma:characterize-proflat}Suppose that $R$ is a Pr\"{u}fer domain.
Let $M $ be a pro-countable module. The following assertions are equivalent:

\begin{enumerate}
\item $M$ is torsion-free and has a basis of zero neighborhoods consisting
of pure open submodules;

\item $M$ is pro-countableflat;
\end{enumerate}
\end{lemma}

\begin{proof}
(1)$\Rightarrow $(2) Since $R$ is a Pr\"{u}fer domain, $M$ is flat. Suppose
that $U\subseteq M$ is a pure open submodule. Then $M/U$ is a countable flat
module by Lemma \ref{Lemma:pure-exact-extension}.{}

(2)$\Rightarrow $(1) If $M=\mathrm{lim}_{k}M_{k}$ where for every $k\in
\omega $, $M_{k}$ is flat, then $M$ is torsion-free, whence flat.
Furthermore, for every $k\in \omega $, $\mathrm{\mathrm{Ker}}\left(
M\rightarrow M_{k}\right) $ is a pure open submodule by Lemma \ref%
{Lemma:pure-exact-extension}, since $M_{k}$ and $M$ are flat.
\end{proof}

\begin{lemma}
\label{Lemma:pure-proflat}Suppose that $R$ is a Pr\"{u}fer domain. Let $M$
be a pro-countableflat module and $N\subseteq M$ be a closed submodule. The
following assertions are equivalent:

\begin{enumerate}
\item $N$ is pure in $M$;

\item $U\cap N$ is pure in $N$ for every pure open submodule $U$ of $M$.
\end{enumerate}

Furthermore, in this case $N$ is pro-countableflat.
\end{lemma}

\begin{proof}
We have that $M$ is flat, whence $N$ is also flat. If (2) holds, then $N$
has a basis of zero neighborhoods consisting of pure open submodules.

(1)$\Rightarrow $(2) Let $U$ be a pure open submodule of $M$. Then $U\cap N$
is pure in $N$. Indeed, if $x\in rU\cap N$, then $x=rz$ for some $z\in U$.
Since $x\in N$ and $N$ is pure in $M$, necessarily $z\in U\cap N$.

(2)$\Rightarrow $(1) Obvious.
\end{proof}

\begin{theorem}
\label{Theorem:hd-2}Suppose that $R$ is a countable domain.

\begin{enumerate}
\item The category $\boldsymbol{\Pi }\left( \mathbf{Mod}\left( R\right)
\right) $ of pro-countable modules is a thick subcategory of $\mathbf{PolMod}%
\left( R\right) $ with enough projectives.

\item If $R$ is Pr\"{u}fer, then the category $\boldsymbol{\Pi }\left( 
\mathbf{Mod}_{\aleph _{0}}\left( R\right) \right) $ is a thick subcategory
of $\boldsymbol{\Pi }\left( \mathbf{Mod}\left( R\right) \right) $ with
enough projectives;

\item If $R$ is Dedekind, then $\mathbf{\Pi }\left( \mathbf{Mod}\left(
R\right) \right) $ is hereditary, and its projective objects are precisely
the countable products of countable direct sums of ideals of $R$;

\item If $R$ is Pr\"{u}fer, then the category $\mathbf{\Pi }\left( \mathbf{%
Flat}\left( R\right) \right) $ is a hereditary quasi-abelian category with
enough projectives, which is a fully exact subcategory of $\mathbf{\Pi }%
\left( \mathbf{Mod}\left( R\right) \right) $ closed under taking closed
submodules, and its projective objects are precisely the countable products
of countable direct sums of finite ideals of $R$;

\item If $R$ is Pr\"{u}fer, then the category $\mathbf{\Pi }\left( \mathbf{%
FinFlat}\left( R\right) \right) $ is a thick subcategory of $\mathbf{\Pi }%
\left( \mathbf{Flat}\left( R\right) \right) $ that is hereditary with enough
projectives, whose projective objects are precisely the countable products
of finite sums of finite ideals of $R$.
\end{enumerate}
\end{theorem}

\begin{proof}
(1) This follows from Proposition \ref{Proposition:pro-countable} applied
when $\mathcal{A}=\mathbf{Mod}\left( R\right) $.

(2) This follows from Proposition \ref{Proposition:pro-countable} applied
when $\mathcal{A}=\mathbf{Mod}_{\aleph _{0}}\left( R\right) $.

(3) This follows from Proposition \ref{Proposition:characterize-Dedekind}
and Lemma \ref{Lemma:submodule-projective} applied when $\mathcal{A}=\mathbf{%
Mod}\left( R\right) $.

(4) This follows from \ref{Proposition:characterize-prufer}, Lemma \ref%
{Lemma:submodule-projective} applied when $\mathcal{A}=\mathbf{Flat}\left(
R\right) $, Lemma \ref{Lemma:characterize-proflat}, and Lemma \ref%
{Lemma:pure-proflat}, together with the following observations.

Recall that Polish modules form a quasi-abelian category, whose kernels are
the injective homomorphisms with closed image and whose cokernels are the
surjective continuous homomorphisms. If $\varphi :M\rightarrow N$ is a
continuous homomorphism between Polish modules in $\boldsymbol{\Pi }\left( 
\mathbf{Flat}\left( R\right) \right) $, then it is easily seen that its
kernel in $\mathbf{\Pi }\left( \mathbf{Flat}\left( R\right) \right) $ is $%
\left\{ x\in M:\varphi \left( x\right) =0\right\} $ while its cokernel is
the quotient of $N$ by the closed pure hull $\langle \varphi \left( X\right)
\rangle _{\ast }$ of $\varphi \left( X\right) $. Suppose that%
\begin{equation*}
\begin{array}{ccc}
A^{\prime } & \overset{f^{\prime }}{\rightarrow } & B^{\prime } \\ 
g\uparrow &  & \uparrow h \\ 
A & \overset{f}{\rightarrow } & B%
\end{array}%
\end{equation*}%
is a pushout diagram, where $f$ is a kernel in $\boldsymbol{\Pi }\left( 
\mathbf{Flat}\right) $. We have that%
\begin{equation*}
N:=\left\{ \left( g\left( a\right) ,-f\left( a\right) \right) :a\in A\right\}
\end{equation*}%
is a pure submodule of $A^{\prime }\oplus B$. Indeed, if $\left( a^{\prime
},b\right) \in A^{\prime }\oplus B$ and $r\in R\setminus \left\{ 0\right\} $
are such that $r\left( a^{\prime },b\right) =\left( g\left( a\right)
,-f\left( a\right) \right) $ for some $a\in A$, then, since $f$ is a kernel,
there exists $a_{0}\in A$ such that $-f\left( a_{0}\right) =b$ and hence $%
g\left( a_{0}\right) =a^{\prime }$. Since $N$ is also closed, we have that $%
B^{\prime }$ is the quotient of $A^{\prime }\oplus B$ by $N$.\ We claim that 
$f^{\prime }$ is kernel. If $b^{\prime }\in B^{\prime }$ and $a^{\prime }\in
A^{\prime }$ and $r\in R\setminus \left\{ 0\right\} $ are such that $%
rb^{\prime }=f^{\prime }\left( a^{\prime }\right) $, then we have that $%
b^{\prime }=\left( x,y\right) +N$ for some $x\in A^{\prime }$ and $y\in B$
and $\left( rx,ry\right) \equiv \left( a^{\prime },0\right) \ \mathrm{mod}N$%
. Thus, there exists $a\in A$ such that $rx+g\left( a\right) =a^{\prime }$
and $ry=f\left( a\right) $. Since $f$ is a kernel, there exists $a_{0}\in A$
such that $y=f\left( a_{0}\right) $ and hence $r\left( x+g\left(
a_{0}\right) \right) =a^{\prime }$. Thus, we have that $\left( x,y\right)
\equiv \left( x+g\left( a_{0}\right) ,0\right) \mathrm{mod}N$. Thus, if $%
a^{\prime \prime }:=x+g\left( a_{0}\right) \in A^{\prime }$ then we have $%
f^{\prime }\left( a^{\prime \prime }\right) =b^{\prime }$.

Suppose now that%
\begin{equation*}
\begin{array}{ccc}
A^{\prime } & \overset{f^{\prime }}{\rightarrow } & B^{\prime } \\ 
g\uparrow &  & \uparrow h \\ 
A & \overset{f}{\rightarrow } & B%
\end{array}%
\end{equation*}%
is a pullback diagram. Observe that $\left\{ \left( a^{\prime },b\right) \in
A^{\prime }\oplus B:f^{\prime }\left( a^{\prime }\right) =h\left( b\right)
\right\} $ is pure in $A^{\prime }\oplus B$. Since it is also closed, is
equal to $A$. It follows now that $f$ is a cokernel.

(5) This follows from Lemma \ref{Lemma:submodule-projective} applied when $%
\mathcal{A}=\mathbf{FinFlat}\left( R\right) $.
\end{proof}

It follows from Proposition \ref{Proposition:derived-functor2} that the
functor $\mathrm{Hom}$ on $\boldsymbol{\Pi }(\mathbf{Mod}\left( R\right) )$
admits a total right derived functor $\mathrm{RHom}$. We let $\mathrm{Ext}%
^{n}=\mathrm{H}^{n}\circ \mathrm{RHom}$ for $n\in \mathbb{Z}$. For
pro-countable Polish modules $C,A$ we denote $\mathrm{Ext}^{1}\left(
C,A\right) $ by $\mathrm{Ext}\left( C,A\right) $. By \cite[Section\ III.2]%
{gelfand_methods_2003}, $\mathrm{Ext}\left( C,A\right) $ is isomorphic to $%
\mathrm{Ext}_{\mathrm{Yon}}\left( C,A\right) $. If $R$ is Pr\"{u}fer and $%
C,A $ are pro-flat, then $\mathrm{Ext}^{n}\left( C,A\right) =0$ for $n\geq 2$%
. If $R$ is Dedekind, then $\mathrm{Ext}^{n}=0$ for $n\geq 2$.

\section{Modules with a Polish cover\label{Section:modules-with-polish-cover}%
}

In this section we recall the definition and fundamental properties of \emph{%
modules with a Polish cover}, and how they provide an explicit description
of the left heart of the quasi-abelian category of Polish modules. We denote
by $R$ a countable Pr\"{u}fer domain and assume that all the modules are $R$%
-modules.

\subsection{Modules with a Polish cover}

The following notion was introduced in \cite{bergfalk_definable_2024-1}. A 
\emph{module with a Polish cover} is a module $G$ explicitly presented as a
quotient $\hat{G}/N$ where $\hat{G}$ is a Polish module and $N$ is a
Polishable submodule of $\hat{G}$. This means that $N$ is a Borel submodule
of $\hat{G}$ such that there exists a (necessarily unique) Polish module
topology on $N$ whose Borel sets are precisely the Borel subsets of $\hat{G}$
contained in $N$ or, equivalently, such that the inclusion $N\rightarrow 
\hat{G}$ is continuous. Equivalently, there exist a Polish module $L$ and a
continuous (injective) homomorphism $L\rightarrow \hat{G}$ with image equal
to $N$. In what follows, unless otherwise specified, all the topological
notions pertaining to subsets of $N$ will be considered in reference its
unique Polish module topology that makes the inclusion $N\rightarrow \hat{G}$
is continuous.

Suppose that $G=\hat{G}/N$ and $H=\hat{H}/M$ are modules with a Polish
cover. A homomorphism $\varphi :G\rightarrow H$ is \emph{Borel} if there
exists a \emph{Borel} \emph{lift }for $\varphi $, namely a Borel function $%
\hat{\varphi}:\hat{G}\rightarrow \hat{H}$ such that $\varphi \left(
x+N\right) =\hat{\varphi}\left( x\right) +M$ for $x\in G$. We regard modules
with a Polish cover as objects of a category with Borel group homomorphisms
as morphisms. It is proved in \cite[Section 6]{lupini_looking_2024} that
this is an abelian category, where the sum of two Borel homomorphisms $%
\varphi ,\psi :G\rightarrow H$ is defined by setting $\left( \varphi +\psi
\right) \left( x\right) =\varphi \left( x\right) +\psi \left( x\right) $ for 
$x\in G$. By \cite[Remark 3.10]{lupini_looking_2024}, a Borel homomorphism
between modules with a Polish cover is an isomorphism in the category of
modules with a Polish cover if and only if it is a bijection.

We identify a Polish module $G$ with the module with a Polish cover $\hat{G}%
/N$ where $\hat{G}=G$ and $N=\left\{ 0\right\} $. By \cite[Theorem 9.10 and
Theorem 12.17]{kechris_classical_1995}, this realizes the category of Polish
modules and a continuous module homomorphisms as a full subcategory of the
category of modules with a Polish cover. A module with a Polish cover $\hat{G%
}/N$ is Borel isomorphic to a Polish module if and only if $N$ is a closed
submodule of $\hat{G}$.

The category $\mathbf{PolMod}\left( R\right) $ of Polish modules is a
quasi-abelian category, whose left heart is equivalent to the category of
modules with a Polish cover \cite[Theorem 6.3]{lupini_looking_2024}. More
generally, if $\mathcal{B}$ is a thick subcategory of $\mathbf{PolMod}\left(
R\right) $, then $\mathrm{LH}\left( \mathcal{B}\right) $ can be identified
with the full subcategory of the category of groups with a Polish cover
spanned by modules with a Polish cover of the form $\hat{G}/N$ where both $%
\hat{G}$ and $N$ are in $\mathcal{B}$; see \cite[Theorem 6.14 and
Proposition 6.16]{lupini_looking_2024}. For example, this applies to the
category $\boldsymbol{\Pi }(\mathbf{Mod}\left( R\right) )$ of pro-countable
Polish modules. We say that $\hat{G}/N$ is a pro-countable module with a
Polish cover if both $\hat{G}$ and $N$ are pro-countable. A submodule with a
pro-countable Polish cover of $\hat{G}/N$ is a submodule of the form $\hat{H}%
/N$ where $\hat{H}$ is a pro-countable Polishable submodule of $\hat{G}$.

\begin{lemma}
\label{Lemma:thick-left-heart}Suppose that $\mathcal{A}$ is a strictly full
thick subcategory of $\mathbf{PolMod}\left( R\right) $ and $\mathcal{B}$ is
a strictly full thick subcategory of $\mathcal{A}$. Suppose that for every
object $B$ of $\mathcal{B}$ there exists an object $P$ of $\mathcal{B}$ that
is projective in $\mathcal{A}$ and a strict epimorphism $P\rightarrow B$.
Then\textrm{\ }$\mathrm{LH}\left( \mathcal{B}\right) $ is a thick
subcategory of $\mathrm{LH}\left( \mathcal{A}\right) $.
\end{lemma}

\begin{proof}
Suppose that%
\begin{equation*}
0\rightarrow L/K\rightarrow G/N\rightarrow H/M\rightarrow 0
\end{equation*}%
is a short exact sequence in \textrm{LH}$\left( \mathcal{A}\right) $ such
that $L,K,H,M$ are objects of $\mathcal{B}$. We need to prove that $G/N$ is
isomorphic in $\mathrm{LH}\left( \mathcal{A}\right) $ to an object of $%
\mathrm{LH}\left( \mathcal{B}\right) $. By hypothesis, we can assume that $H$
and $L$ are projective in $\mathcal{B}$. Without loss of generality, after
replacing $G/N$ with an isomorphic object in $\mathrm{LH}\left( \mathcal{A}%
\right) $, we can assume that the morphism $G/N\rightarrow H/M$ is induced
by a continuous group homomorphism $G\rightarrow H$. Furthermore, after
replacing $G$ with $G\oplus M$ and $N$ with $N\oplus M$, we can assume
without loss of generality that the continuous group homomorphism $%
G\rightarrow H$ is surjective. Thus, we have that $G=H\oplus H^{\bot }$ and
the homomorphism $G\rightarrow H$ is the first-coordinate projection. Thus,
we have that $H^{\bot }\subseteq N$. As $H^{\bot }$ is a closed submodule of 
$G$, and hence of $N$, after replacing $G$ with $G/H^{\bot }$ and $N$ with $%
N/H^{\bot }$, we can assume without loss of generality that $H^{\bot }=0$
and hence $G=H$.

Arguing as before, we can assume that the morphism $L/K\rightarrow G/N$ is
induced by an injective continuous homomorphism $L\rightarrow G$. By
identifying $L$ with its (not necessarily closed) image inside of $G$ we
have that $K=L\cap N$ and $L+N=M$. Thus, the inclusion $N\rightarrow L+N$
induces a Borel isomorphism%
\begin{equation*}
N/\left( L\cap N\right) \cong \left( L+N\right) /L\text{.}
\end{equation*}%
Consider%
\begin{equation*}
\Gamma =\left\{ \left( x,y\right) \in N\oplus \left( L+N\right) :x+\left(
L\cap N\right) =y+L\right\} \text{.}
\end{equation*}%
Then we have a short exact sequence%
\begin{equation*}
0\rightarrow L\rightarrow \Gamma \rightarrow N\rightarrow 0\text{.}
\end{equation*}%
Since $\mathcal{B}$ is a strictly thick subcategory of $\mathcal{A}$, this
implies that $\Gamma $ is in $\mathcal{B}$. We also have an exact sequence%
\begin{equation*}
0\rightarrow L\cap N\rightarrow \Gamma \rightarrow L+N\rightarrow 0\text{.}
\end{equation*}%
This implies that $K=L\cap N$ is in $\mathcal{B}$, concluding the proof.
\end{proof}

If $G$ is a module with a Polish cover, and $N$ is the closure of $\left\{
0\right\} $ in $G$, then $G/N$ is a Polish module. Furthermore, we have an
extension%
\begin{equation*}
0\rightarrow N\rightarrow G\rightarrow G/N\rightarrow 0\text{.}
\end{equation*}%
As a particular instance of Definition \ref{Definition:phantom-category} we
have the following:

\begin{definition}
A \emph{phantom Polish module} is a module with a Polish cover whose trivial
submodule is dense. Likewise, a phantom pro-countable Polish module is a
module with a pro-countable Polish cover whose trivial submodule is dense.
\end{definition}

We let $\mathrm{Ph}\left( \mathbf{\Pi (Mod}\left( R\right) )\right) $ be the
category of phantom pro-countable Polish modules. Then $\mathrm{Ph}\left( 
\boldsymbol{\Pi }(\mathbf{Mod}\left( R\right) )\right) $ is a thick abelian
subcategory of $\boldsymbol{\Pi }(\mathbf{Mod}\left( R\right) )$, and we
have an extension%
\begin{equation*}
0\rightarrow \mathrm{Ph}(\boldsymbol{\Pi }(\mathbf{Mod}\left( R\right)
))\rightarrow \mathrm{LH}(\boldsymbol{\Pi }(\mathbf{Mod}\left( R\right)
))\rightarrow \boldsymbol{\Pi }(\mathbf{Mod}\left( R\right) )\rightarrow 0
\end{equation*}

\begin{lemma}
Suppose that $\boldsymbol{G}=\left( G^{\left( n\right) }\right) $ is an
inverse sequence of phantom Polish modules. Then its limit $\mathrm{lim}%
\boldsymbol{G}$ in $\boldsymbol{\Pi }(\mathbf{Mod}\left( R\right) )$ is a
phantom Polish module.
\end{lemma}

\begin{proof}
Write $G^{\left( n\right) }=\hat{G}^{\left( n\right) }/N^{\left( n\right) }$
for $n\in \omega $, where $\hat{G}^{\left( n\right) }$ and $N^{\left(
n\right) }$ are Polish modules. Thus, we have that $\mathrm{lim}\boldsymbol{G%
}=\hat{G}/N$ where%
\begin{equation*}
\hat{G}=\left\{ \left( x_{n},z_{n}\right) _{n\in \omega }\in \prod_{n\in
\omega }\hat{G}^{\left( n\right) }\oplus N^{\left( n\right) }:\forall n\in
\omega ,x_{n}-p^{\left( n,n+1\right) }\left( x_{n+1}\right) =z_{n}\right\}
\end{equation*}%
and%
\begin{equation*}
N=\left\{ \left( x_{n},z_{n}\right) _{n\in \omega }\in \hat{G}:\forall n\in
\omega ,x_{n}\in N^{\left( n\right) }\right\} \text{.}
\end{equation*}%
Suppose that $\left( a_{n},c_{n}\right) _{n\in \omega }\in \hat{G}$ and $W$
is an open neighborhood of $\left( a_{n},c_{n}\right) _{n\in \omega }$ in $%
\hat{G}$. Without loss of generality, we can assume that there exist $%
n_{0}\geq 1$ and, for $n<n_{0}$, open neighborhoods $U_{n}$ of $0$ in $\hat{G%
}^{\left( n\right) }$ and $V_{n}$ of $0$ in $N^{\left( n\right) }$ such that 
$V_{n}\subseteq U_{n}$ and 
\begin{equation*}
W=\left\{ \left( x_{n},z_{n}\right) _{n\in \omega }:\forall n<n_{0},\left(
x_{n}-a_{n},z_{n}-c_{n}\right) \in U_{n}\times V_{n}\right\} \text{.}
\end{equation*}%
For $n\leq n_{0}$, we have that%
\begin{equation*}
p^{\left( n,n_{0}\right) }a_{n_{0}}-a_{n}=p^{\left( n,n_{0}-1\right)
}c_{n_{0}-1}+p^{\left( n,n_{0}-2\right) }c_{n_{0}-2}+\cdots +p^{\left(
n,n+1\right) }c_{n+1}+c_{n}\text{.}
\end{equation*}%
By assumption, $N^{\left( n_{0}\right) }$ is dense in $\hat{G}^{\left(
n_{0}\right) }$. Thus, there exists a sequence $(a_{n_{0}}^{\left( k\right)
})$ in $N^{\left( n_{0}\right) }$ that converges to $a_{n_{0}}$ in $\hat{G}%
^{\left( n_{0}\right) }$. Thus, 
\begin{equation*}
(p^{\left( n,n_{0}\right) }a_{n_{0}}^{\left( k\right) }-(p^{\left(
n,n_{0}-1\right) }c_{n_{0}-1}+p^{\left( n,n_{0}-2\right) }c_{n_{0}-2}+\cdots
+p^{\left( n,n+1\right) }c_{n+1}+c_{n}))
\end{equation*}%
converges in $\hat{G}^{\left( n\right) }$ to $a_{n}$ for $n<n_{0}$. Thus,
there exists $k\in \omega $ such that, for $n<n_{0}$,%
\begin{equation*}
b_{n}:=p^{\left( n,n_{0}\right) }a_{n_{0}}^{\left( k\right) }-(p^{\left(
n,n_{0}-1\right) }c_{n_{0}-1}+p^{\left( n,n_{0}-2\right) }c_{n_{0}-2}+\cdots
+p^{\left( n,n+1\right) }c_{n+1}+c_{n})\in U_{n}+a_{n}\text{.}
\end{equation*}%
Set $b_{n_{0}}:=a_{n_{0}}^{\left( k\right) }$, $d_{n_{0}}:=a_{n_{0}}^{\left(
k\right) }$, and $d_{n}:=c_{n}$ for $n<n_{0}$. Define for $n>n_{0}$, $%
b_{n}=d_{n}=0$. Then we have that $\left( b_{n},d_{n}\right) _{n\in \omega }$
defines an element of $\hat{G}$, since for $n<n_{0}$ we have that%
\begin{eqnarray*}
&&p^{\left( n,n+1\right) }b_{n+1}-b_{n} \\
&=&p^{\left( n,n+1\right) }\left( p^{\left( n+1,n_{0}\right)
}a_{n_{0}}^{\left( k\right) }-(p^{\left( n+1,n_{0}-1\right)
}c_{n_{0}-1}+p^{\left( n+1,n_{0}-2\right) }c_{n_{0}-2}+\cdots
+c_{n+1})\right) \\
&&-\left( p^{\left( n,n_{0}\right) }a_{n_{0}}^{\left( k\right) }-(p^{\left(
n,n_{0}-1\right) }c_{n_{0}-1}+p^{\left( n,n_{0}-2\right) }c_{n_{0}-2}+\cdots
+p^{\left( n,n+1\right) }c_{n+1}+c_{n})\right) \\
&=&c_{n}=d_{n}\text{.}
\end{eqnarray*}%
Furthermore, we have that $\left( b_{n},d_{n}\right) _{n\in \omega }$
belongs to $W$, since for $n<n_{0}$ we have that 
\begin{equation*}
d_{n}-c_{n}=0
\end{equation*}%
and%
\begin{equation*}
b_{n}-a_{n}\in U_{n}\text{.}
\end{equation*}%
This concludes the proof.
\end{proof}

\subsection{Submodules with a Polish cover\label%
{Section:modules-polish-cover}}

Suppose that $G=\hat{G}/N$ is a module with a Polish cover, and let $H=\hat{H%
}/N$ be a submodule of $G$. Then we say that $H\ $is a\emph{\ submodule with
a Polish cover} of $G$ if $\hat{H}$ is a Polishable submodule of $\hat{G}$.
We denote by $\overline{H}^{G}$ the closure of $H$ in $G$ with respect to
the quotient topology. It is proved in \cite[Section 4]{lupini_looking_2024}
that images and preimages of submodules with a Polish cover of modules with
a Polish cover under Borel homomorphisms are submodules with a Polish cover.

Recall that a\emph{\ complexity class} $\Gamma $ is an assignment $X\mapsto
\Gamma \left( X\right) $ where $X$ is a Polish space and $\Gamma \left(
X\right) $ is a collection of subsets of $X$, such that $f^{-1}\left(
A\right) \in \Gamma \left( X\right) $ for every continuous function $%
f:X\rightarrow Y$ between Polish spaces and for every $A\in \Gamma \left(
Y\right) $. We say that a submodule with a Polish cover $H=\hat{H}/N$ of a
module with a Polish cover $G=\hat{G}/N$ is $\Gamma $ in $G$ or belongs to $%
\Gamma \left( G\right) $ if $\hat{H}\in \Gamma (\hat{G})$. Furthermore, $H$
has complexity class $\Gamma $ in $G$ if $\hat{H}\in \Gamma (\hat{G})$ and $%
\hat{G}\setminus \hat{H}\notin \Gamma (\hat{G})$. Let $\Gamma $ be the
complexity class $\boldsymbol{\Sigma }_{\alpha }^{0}$, $\boldsymbol{\Pi }%
_{\alpha }^{0}$, or $D(\boldsymbol{\Pi }_{\alpha }^{0})$ for $\alpha <\omega
_{1}$; see \cite[Section 1.3]{gao_invariant_2009}. (Here, $D(\boldsymbol{\Pi 
}_{\alpha }^{0})$ is the class of sets that can be written as intersection
of a $\boldsymbol{\Sigma }_{\alpha }^{0}$ set and a $\boldsymbol{\Pi }%
_{\alpha }^{0}$ set.) It is proved in \cite[Section 3]%
{lupini_complexity_2024} that if $\varphi :G\rightarrow H=\hat{H}/M$ is a
Borel homomorphism between modules with a pro-countable Polish cover and $%
H_{0}=\hat{H}_{0}/M$ is a submodule with a pro-countable Polish cover of $H$%
, then:

\begin{itemize}
\item $H_{0}\in \Gamma \left( H\right) $ implies $\varphi ^{-1}\left(
H_{0}\right) \in \Gamma \left( G\right) $;

\item $H_{0}\in \Gamma \left( H\right) $ if and only if the coset
equivalence relation of $\hat{H}_{0}$ in $\hat{H}$ is \emph{potentially} $%
\Gamma $, namely Borel reducible to an equivalence relation $E$ on a Polish
space $X$ such that $E\in \Gamma \left( X\times X\right) $.
\end{itemize}

\subsection{Solecki submodules\label{Section:solecki-submodules}}

Let $G=\hat{G}/N$ be a module with a Polish cover. The first Solecki
submodule $s_{1}\left( G\right) $ of $G$ is the smallest $\boldsymbol{\Pi }%
_{3}^{0}$ submodule with a Polish cover of $G$. It was shown by Solecki in 
\cite[Theorem 2.1]{solecki_polish_1999} that such a submodule exists, and $%
s_{1}\left( G\right) \subseteq \overline{\left\{ 0\right\} }^{G}$. In fact,
one can explicitly define $s_{1}(G)$ as $\hat{G}_{1}/N$ where $\hat{G}_{1}$
is the submodule of $\hat{G}$ comprising the $x\in \hat{G}$ such that for
every zero neighborhood $V$ of $N$ there exists $z\in N$ such that $x+z$
belongs to the closure $\overline{V}^{\hat{G}}$ of $V$ in $\hat{G}$ \cite[%
Lemma 2.3]{solecki_polish_1999}.

One then defines the sequence $\left( s_{\alpha }\left( G\right) \right)
_{\alpha <\omega _{1}}$ of Solecki submodules of $G$ by recursion on $\alpha 
$ by setting:

\begin{itemize}
\item $s_{0}\left( G\right) =\overline{\left\{ 0\right\} }^{G}$;

\item $s_{\alpha +1}\left( G\right) =s_{1}\left( s_{\alpha }\left( G\right)
\right) $;

\item $s_{\lambda }\left( G\right) =\bigcap_{\alpha <\lambda }s_{\alpha
}\left( G\right) $ for $\lambda $ limit.
\end{itemize}

It is shown in \cite[Theorem 2.1]{solecki_polish_1999} that there exists $%
\sigma <\omega _{1}$ such that $s_{\sigma }\left( G\right) =\left\{
0\right\} $. The least such a countable ordinal is called the \emph{Solecki
length} of $G$ in \cite{lupini_looking_2024}. When $\alpha $ is a successor
ordinal, we say that $G$ has plain Solecki length $\alpha $ if it has
Solecki length $\alpha $ and $\left\{ 0\right\} \in \boldsymbol{\Sigma }%
_{2}^{0}\left( s_{\alpha -1}\left( G\right) \right) $. In this case, we also
say that $G$ is plain.

For every $\alpha <\omega _{1}$, $s_{\alpha }\left( G\right) $ is the
smallest $\boldsymbol{\Pi }_{1+\alpha +1}^{0}$ submodule of $G$ \cite[%
Theorem 5.4]{lupini_looking_2024}. If $G$ is a pro-countable module with a
Polish cover, then $s_{\alpha }\left( G\right) $ is a pro-countable
submodule of $G$ for every $\alpha <\omega _{1}$. The following
characterization of the first Solecki submodule can be easily obtained with
the methods of \cite{solecki_polish_1999}; see \cite[Lemma 4.2]%
{lupini_complexity_2024}.

\begin{lemma}
\label{Lemma:1st-Solecki}Suppose that $G=\hat{G}/N$ is a module with a
Polish cover, and let $H=\hat{H}/N$ be a $\boldsymbol{\Pi }_{3}^{0}$
submodule with a Polish cover of $G$. Suppose that:

\begin{enumerate}
\item $N$ is dense in $\hat{H}$;

\item for every open neighborhood $V$ of the identity in $N$, $\overline{V}^{%
\hat{G}}$ contains an open neighborhood of the identity in $\hat{H}$.
\end{enumerate}

Then $H=s_{1}\left( G\right) $.
\end{lemma}

To simplify the statement of the following result, we introduce some
notation. For $\lambda <\omega _{1}$ either zero or limit and $n<\omega $ we
define the complexity class:%
\begin{equation*}
\Gamma _{\lambda +n}:=\left\{ 
\begin{array}{ll}
\boldsymbol{\Pi }_{1}^{0} & \text{for }\lambda =n=0\text{;} \\ 
\boldsymbol{\Pi }_{\lambda } & \text{for }n=0\text{ and }\lambda >0\text{;}
\\ 
\boldsymbol{\Pi }_{1+\lambda +n+1}^{0} & \text{for }\lambda >0\text{ and }n>0%
\text{;}%
\end{array}%
\right.
\end{equation*}%
and%
\begin{equation*}
\Gamma _{\lambda +n}^{\mathrm{semiplain}}:=\left\{ 
\begin{array}{ll}
\Gamma _{\lambda +n} & \text{for }n=0\text{;} \\ 
D(\boldsymbol{\Pi }_{1+\lambda +n}^{0}) & \text{for }n\geq 1\text{;}%
\end{array}%
\right.
\end{equation*}%
and%
\begin{equation*}
\Gamma _{\lambda +n}^{\mathrm{plain}}:=\left\{ 
\begin{array}{ll}
\Gamma _{\lambda +n}^{\mathrm{semiplain}} & \text{for }n\neq 1\text{;} \\ 
\boldsymbol{\Sigma }_{1+\lambda +1}^{0} & \text{for }n=1\text{.}%
\end{array}%
\right.
\end{equation*}%
By \cite[Theorem 6.1]{lupini_complexity_2024} we have the following
description of the complexity of $\left\{ 0\right\} $ in a module with a
pro-countable Polish cover.

\begin{theorem}
\label{Theorem:complexity}Suppose that $G$ is a module with a pro-countable
Polish cover.\ Let $\alpha <\omega _{1}$ be the Solecki length of $G$.

\begin{enumerate}
\item If $G$ is plain, then the complexity class of $\left\{ 0\right\} $ in $%
G$ is $\Gamma _{\alpha }^{\mathrm{plain}}$;

\item If $G$ is not plain, then the complexity class of $\left\{ 0\right\} $
in $G$ is $\Gamma _{\alpha }$.
\end{enumerate}
\end{theorem}

\begin{corollary}
\label{Corollary:complexity}Let $G$ be a module with a pro-countable Polish
cover. The the complexity class of $\left\{ 0\right\} $ in $G$ is either $%
\Gamma _{\alpha }$ or $\Gamma _{\alpha }^{\mathrm{plain}}$ for some $\alpha
<\omega _{1}$.
\end{corollary}

In fact, it is also proved in \cite[Theorem 1.2]{lupini_complexity_2024}
that the ones in Corollary \ref{Corollary:complexity} are the only possible
complexity classes of a non-Archimedean Polishable subgroup of a Polish
group. Recall that $\Gamma _{1}^{\mathrm{plain}}$ is by definition the
complexity class $\boldsymbol{\Sigma }_{2}^{0}$.

\begin{lemma}
\label{Lemma:compact-phantom}Suppose that $\left( G_{i}\right) $ is an
inverse sequence of phantom pro-countable modules of plain Solecki length at
most $1$, and $H$ is a phantom pro-countable module. Suppose that $G=\mathrm{%
lim}_{i}G_{i}$ and $\Phi :G\rightarrow H$ is a Borel homomorphism. If $H$
has plain Solecki length at most $1$, then there exists $n\in \omega $ such
that $\Phi $ factors through $G_{n}$.
\end{lemma}

\begin{proof}
As $\left\{ 0\right\} \in \Sigma _{2}^{0}\left( G_{i}\right) $, we can write 
$G_{i}=\hat{G}_{i}/L_{i}$ for some pro-countable Polish module $\hat{G}_{i}$
and some countable submodule $L_{i}$. Set 
\begin{equation*}
\hat{G}=\left\{ \left( x_{i},y_{i}\right) _{i\in \omega }\in \prod_{i\in
\omega }\hat{G}_{i}\oplus L_{i}:\forall i,p^{\left( i,i+1\right) }\left(
x_{i+1}\right) -x_{i}=y_{i}\right\}
\end{equation*}%
and%
\begin{equation*}
L:=\left\{ \left( x_{i},y_{i}\right) _{i\in \omega }\in \hat{G}:\forall
i,x_{i}\in L_{i}\right\} \text{.}
\end{equation*}%
Then we have that $G=\hat{G}/L$ and $\mathrm{\mathrm{\mathrm{Ker}}}\left(
\Phi \right) =K/L$ for some $\boldsymbol{\Sigma }_{2}^{0}$ pro-countable
Polishable submodule $K$ of $\hat{G}$ containing $L$.

Since $K$ is a $\boldsymbol{\Sigma }_{2}^{0}$ pro-countable Polishable
submodule of $\hat{G}$, by the Baire Category Theorem it has an open
submodule $U$ that is closed in $\hat{G}$. As $U$ is open in $K$, $U\cap L$
is open in $L$. Thus, there exists $n\in \omega $ such that $U\cap L$
contains 
\begin{equation*}
\left\{ \left( x_{i},y_{i}\right) _{i\in \omega }\in L:\forall i\leq
n,x_{i}=0\right\} \text{.}
\end{equation*}%
As $U$ is closed in $\hat{G}$, this implies that $U$ contains 
\begin{equation*}
\left\{ \left( x_{i},y_{i}\right) _{i\in \omega }\in \hat{G}:\forall i\leq
n,x_{i}=0\right\} .
\end{equation*}%
Thus, the same holds for $K$, and the conclusion follows.
\end{proof}

\begin{corollary}
\label{Corollary:compact-phantom}Let $\left( G^{\left( n\right) },p^{\left(
n,n+1\right) }:G^{\left( n+1\right) }\rightarrow G^{\left( n\right) }\right) 
$ be an inverse sequence of phantom pro-countable modules such that $%
p^{\left( n,n+1\right) }$ is surjective for every $n\in \omega $. Assume
that $G^{\left( n\right) }$ has plain Solecki length at most $1$ for every $%
n\in \omega $. Set $G:=\mathrm{lim}_{n}G^{\left( n\right) }$. The following
assertions are equivalent:

\begin{enumerate}
\item $G$ has plain Solecki length at most $1$;

\item there exists $k\in \omega $ such that the canonical map $G\rightarrow
G^{\left( k\right) }$ is an isomorphism;

\item there exists $k\in \omega $ such that for all $n\geq k$, $p^{\left(
n,n+1\right) }\ $is an isomorphism.
\end{enumerate}
\end{corollary}

\begin{lemma}
\label{Lemma:compact-phantom2}Suppose that $G,H$ are modules with a
pro-countable cover. Let $\varphi :G\rightarrow H$ be a Borel group
homomorphism. If $G$ has plain Solecki length at most $1$, then so does $%
\varphi \left( G\right) $.
\end{lemma}

\begin{proof}
Since $G$ is a module with a pro-countable cover, we can write $G=\hat{G}/N$
where $\hat{G}$ and $N$ are pro-countable modules. Since $\left\{ 0\right\}
\in \boldsymbol{\Sigma }_{2}^{0}\left( G\right) $, we have that $N$ has a
zero neighborhood $V$ that is a closed submodule of $\hat{G}$. After
replacing $\hat{G}$ with $\hat{G}/V$ and $N$ with $N/V$, we can assume that $%
N$ is countable. We can also assume without loss of generality that $\varphi 
$ is surjective. If $\mathrm{\mathrm{Ker}}\left( \varphi \right) $ is the
submodule $K/N$, then we have that $H\cong \hat{G}/(N+K)$. Since $K$ is
closed in $\hat{G}$, we have that $N+K\in \boldsymbol{\Sigma }_{2}^{0}(\hat{G%
})$, and hence $\left\{ 0\right\} \in \boldsymbol{\Sigma }_{2}^{0}\left(
H\right) $.
\end{proof}

\section{Ranks and games\label{Section:ranks}}

In this section, we recall some notions from descriptive set theory
concerning \emph{ranks }and their description in terms of games. More
information can be found in \cite{kechris_classical_1995}.

\subsection{Rank of well-founded binary relation}

Suppose that $Y$ is a set, and $\prec $ is an binary relation on $Y$. For a
subset $A$ of $Y$, define its $\prec $-\emph{derivative} to be the set%
\begin{equation*}
D_{\prec }\left( A\right) :=\left\{ y\in Y:\exists a\in A\text{, }a\prec
y\right\} \text{.}
\end{equation*}%
Define the recursively, for every ordinal $\alpha $,%
\begin{equation*}
D_{\prec }^{0}\left( A\right) :=A\text{,}
\end{equation*}%
\begin{equation*}
D_{\prec }^{\alpha +1}\left( A\right) :=D_{\prec }\left( D_{\prec }^{\alpha
}\left( A\right) \right) \text{,}
\end{equation*}%
and for $\lambda $ limit%
\begin{equation*}
D_{\prec }^{\lambda }\left( A\right) :=\bigcap_{\alpha <\lambda }D_{\prec
}^{\alpha }\left( A\right) \text{.}
\end{equation*}%
The binary relation $\prec $ is \emph{ill-founded }if there exists a
sequence $\left( x_{n}\right) $ in $Y$ with $x_{n+1}\prec x_{n}$ for every $%
n\in \omega $, and \emph{well-founded }otherwise. If $\prec $ is
well-founded, one defines the rank $\rho _{\prec }:Y\rightarrow \mathbf{ORD}$%
, where $\mathbf{ORD}$ is the class of ordinals, by well-founded recursion
on $\prec $ by setting%
\begin{equation*}
\rho _{\prec }\left( x\right) =\mathrm{sup}\left\{ \rho _{\prec }\left(
y\right) +1:y\in Y\text{ and }y\prec x\right\} \text{.}
\end{equation*}%
One defines then the rank $\rho \left( \prec \right) $ of $\prec $ to be the
least ordinal not in the range of $\rho _{\prec }$, thus%
\begin{equation*}
\rho \left( \prec \right) =\mathrm{\mathrm{sup}}\left\{ \rho _{\prec }\left(
y\right) +1:y\in Y\right\} \text{;}
\end{equation*}%
see \cite[Appendix B]{kechris_classical_1995}. When $\prec $ is ill-founded,
one sets $\rho \left( \prec \right) =\infty $. It is clear that we have the
equality%
\begin{equation*}
D_{\prec }^{\alpha }\left( Y\right) =\left\{ y\in Y:\alpha \leq \rho _{\prec
}\left( y\right) \right\} \text{.}
\end{equation*}%
Therefore, the following result is immediate.

\begin{lemma}
Suppose that $Y$ is a set, and $\prec $ is a binary relation on $Y$. Then $%
\prec $ is well-founded if and only if there exists an ordinal $\alpha $
such that $D_{\prec }^{\alpha }\left( Y\right) =\varnothing $. In this case,
we have that $\rho \left( \prec \right) $ is the least such an ordinal.
\end{lemma}

Suppose now that $Y$ is a \emph{standard Borel space}, and $\prec $ is \emph{%
analytic }binary relation on $Y$. (Recall that a subset $A$ of a standard
Borel space $Y$ is analytic if it is the image of a continuous function $%
f:Z\rightarrow Y$ for some standard Borel space $Z$, and co-analytic if its
complement in $Y$ is analytic.) In this case, if $\prec $ is well-founded,
then its rank is a \emph{countable }ordinal \cite[Theorem 31.1]%
{kechris_classical_1995}. Furthermore, for an analytic subset $A\subseteq Y$%
, its derivative $D_{\prec }^{\alpha }\left( A\right) $ is still analytic.

\subsection{Coanalytic ranks}

Recall that if $X$ is a standard Borel space, then a co-analytic rank on $X$
is a function $\varphi :X\rightarrow \omega _{1}\cup \left\{ \infty \right\} 
$ such that there exist analytic binary relations $\leq _{\varphi }$ and $%
<_{\varphi }$ on $X$ satisfying, for $x,y\in X$ with $\varphi \left(
y\right) <\infty $:%
\begin{equation*}
\left( \varphi \left( x\right) <\infty \wedge \varphi \left( x\right) \leq
\varphi \left( y\right) \right) \Leftrightarrow x\leq _{\varphi }y
\end{equation*}%
and%
\begin{equation*}
\left( \varphi \left( x\right) <\infty \wedge \varphi \left( x\right)
<\varphi \left( y\right) \right) \Leftrightarrow x<_{\varphi }y\text{;}
\end{equation*}%
see \cite[Section 34.B]{kechris_classical_1995}.

Suppose now that $X,Y$ are standard Borel spaces, and $\prec \subseteq
X\times Y\times Y$ is an analytic set. For $x\in X$ we define the analytic
binary relation $\prec _{x}$ on $Y$ by setting, for $y,y^{\prime }\in Y$, 
\begin{equation*}
y\prec _{x}y^{\prime }\Leftrightarrow \left( x,y,y^{\prime }\right) \in
{}\prec \text{.}
\end{equation*}%
We think of $\prec $ as an analytic assignment of binary relations on $Y$ to
elements of $X$. Define now 
\begin{equation*}
\varphi _{\prec }:X\rightarrow \omega _{1}\text{, }x\mapsto \rho (\prec _{x})%
\text{.}
\end{equation*}%
Define $\mathrm{LO}$ to be the standard Borel space of $t\in 2^{\omega
\times \omega }$ such that $D\left( t\right) :=\left\{ i\in \omega :t\left(
i,i\right) =1\right\} $ is a nonempty subset of $\omega $, and setting%
\begin{equation*}
i<_{t}j\Leftrightarrow t\left( i,j\right) =1
\end{equation*}%
defines a linear order on $D\left( t\right) $ with least element $0$. Define
also $\mathrm{WO}\subseteq \mathrm{LO}$ to be the co-analytic set of $t\in 
\mathrm{LO}$ such that $<_{t}$ is a well-order, in which case we define $%
\alpha _{t}<\omega _{1}$ to be its order type.

\begin{proposition}
\label{Proposition:rank}Suppose that $X,Y$ are standard Borel spaces, and $%
x\mapsto \prec _{x}$ is an analytic assignment of binary relations on $Y$ to
elements of $X$. Then the relation%
\begin{equation*}
\left\{ \left( t,x\right) \in \mathrm{LO}\times X:\alpha _{t}\leq \varphi
_{\prec }\left( x\right) \right\}
\end{equation*}%
is analytic, and $\varphi _{\prec }$ is a co-analytic rank of $X$.
\end{proposition}

\begin{proof}
For $x\in X$ with $\varphi _{\prec }\left( x\right) <\infty $ and $t\in 
\mathrm{WO}$ we have that $\alpha _{t}\leq \varphi _{\prec }\left( x\right) $
if and only if for every $n\in \omega $ we have%
\begin{equation*}
\forall a_{0}\exists b_{0}\forall a_{1}\exists b_{1}\cdots \forall
a_{n}\exists b_{n}\text{, }\left( \text{if }a_{n}<_{t}a_{n-1}<_{t}\cdots
<_{t}a_{0}\text{ then }b_{n}\prec _{x}\cdots \prec _{x}b_{0}\right)
\end{equation*}%
where $a_{0},\ldots ,a_{n}\in D\left( t\right) $ and $b_{0},\ldots ,b_{n}\in
X$. More generally, for $a\in D\left( t\right) $ and $b\in X$ we have that $%
\rho _{<_{t}}\left( a\right) \leq \rho _{\prec _{x}}\left( b\right) $ if and
only if%
\begin{equation*}
\forall a_{0}\exists b_{0}\forall a_{1}\exists b_{1}\cdots \forall
a_{n}\exists b_{n}\text{, }\left( \text{if }a_{n}<_{t}a_{n-1}<_{t}\cdots
<_{t}a_{0}<_{t}a\text{ then }b_{n}\prec _{y}\cdots \prec _{y}b_{0}\prec
_{y}b\right) \text{.}
\end{equation*}%
The second assertion follows from the first one as in the proof of \cite[%
Lemma 34.11]{kechris_classical_1995}.
\end{proof}

A fundamental property of co-analytic ranks is that their restriction to
Borel subsets is necessarily bounded; see \cite[Theorem 35.23]%
{kechris_classical_1995}.

\begin{proposition}
\label{Proposition:bounded-rank}Suppose that $X$ is a standard Borel space
and $\varphi :X\rightarrow \omega _{1}\cup \left\{ \infty \right\} $ is a
co-analytic rank on $X$. If $A$ is a Borel subset of $X$ such that $\varphi
|_{A}\ $takes values in $\omega _{1}$, then there exists $\alpha <\omega
_{1} $ such that $\varphi |_{A}:A\rightarrow \alpha $.
\end{proposition}

\subsection{Trees}

A canonical example of co-analytic rank is the rank of combinatorial trees.
Let $X$ be a countable set. A \emph{tree }on $X$ is a set $T$ of finite
tuples of elements of $X$ (called the \emph{nodes} of $T$) that is closed
under taking initial segments. The \emph{root }of the tree is the empty
tuple $\varnothing $. The \emph{body }$[T]$ of $T$ is the closed subsets of $%
X^{\omega }$ comprising the infinite sequences whose initial segments all
belong to $T$. The binary relation $\prec _{T}$ on the set $X^{<\omega }$ of
finite tuples of elements of $X$ is defined by setting $a\prec _{T}b$ if and
only if $a,b\in T$ and $b$ is an initial segment of $a$.\ Then we have that $%
\prec _{T}$ is \emph{well-founded} if and only if the body of $T$ is empty.
Furthermore, the $\alpha $-th derivative $D_{\prec _{T}}^{\alpha }\left(
T\right) $ is a subtree $T^{\alpha }$ of $T$. The rank $\rho \left( T\right) 
$ of $T$ is defined to be the rank of $\prec _{T}$, which is thus the least $%
\alpha $ such that $T^{\alpha }$ is empty. As the root is easily seen to
have the largest rank among the nodes of $T$, we have by definition 
\begin{equation*}
\rho \left( T\right) =\rho _{\prec _{T}}\left( \varnothing \right) +1\text{.}
\end{equation*}%
In particular, the rank of $T$ is always a successor ordinal.

Trees form a closed subspace $\mathrm{Tree}$ of $2^{X^{<\omega }}$, of which
the well-founded trees are a co-analytic subspace. If $Z$ is a standard
Borel space, and $z\mapsto T_{z}$ is an analytic assignment of trees to
elements of $Z$, then $z\mapsto {}\prec _{T_{z}}$ is an analytic assignments
of binary relations on $X^{<\omega }$ to elements of $Z$. It follows from
this and Proposition \ref{Proposition:rank} that the map $z\mapsto \rho
\left( T_{z}\right) $ is a co-analytic rank on $Z$.

Graph-theoretically, a rooted tree is a connected acyclic graph with a
distinguished vertex (the root). When $X$ is the set $\omega $ of natural
numbers, one can think of a tree $T$ on $\omega $ as a graph-theoretic
countable rooted tree such that each level-set is endowed with a well-order
of order type at most $\omega $. We regard a rooted tree $T$ as an ordered
set with respect to the relation between nodes of $T$%
\begin{equation*}
v\prec w\Leftrightarrow w\text{ belongs to the unique path from the root to }%
v\text{.}
\end{equation*}%
We define a forest to be a disjoint union of rooted trees, with the induced
order.

\subsection{Games}

Let $X$ be a countable discrete set, and denote by $X^{\omega }$ the Polish
space of infinite sequences of elements of $X$. A subset $A$ of $X^{\omega }$
defines a game $G_{A}$ between two players, Alice and Bob. The game starts
with Alice picking an element $x_{0}$ of $X$, to which Bob must reply with
an element $x_{1}$ of $X$. The games continues in this fashion, with Alice
and Bob alternating to pick elements $x_{2n}$ and, respectively, $x_{2n+1}$
at Alice's and, respectively, Bob's turn of index $n$. A run of the game
produces an infinite sequence $\left( x_{n}\right) _{n\in \omega }$. Alice
wins the game if such a sequence belongs to $A$, otherwise Bob wins the
game. One can similarly define a run of the game with Bob as starting
player. The payoff set for Bob is clearly $B:=X^{\omega }\setminus A$.

One says that Alice has a winning strategy for such a game if there exists a
way for Alice to choose her $n$-th move (as given by a function from finite
tuples of elements of $X$ to $X$) in a way that lets her win any run of the
game. A winning strategy for Bob is defined in the same fashion. The game is
called \emph{determined }if either Alice or Bob has a winning strategy. If $%
A $ is a \emph{Borel }subset of $X^{\omega }$, then the game is determined.

Following \cite{allison_class_2024}, we define the notion of \emph{game rank}
for a game $G_{A}$ with \emph{open }payoff set $A$ for Alice (and, hence,
closed payoff set for Bob), and more generally the notion of \emph{game rank}
for a finite tuple $x\in X^{<\omega }$ representing a position of the game
after finitely many steps. In fact, we define two ranks, the $\sigma $-game
rank $\sigma _{A}$ and $\pi $-game rank $\pi _{A}$, depending on whether one
considers Alice and Bob as starting player. We say that a finite tuple $x$
is a\emph{\ winning position} for Alice if the open set $N_{x}$ of sequences
that have $x$ as initial segment is a subset of $A$. Define:

\begin{itemize}
\item $\pi _{A}\left( x\right) =0$ if and only if $\sigma _{A}\left(
x\right) =0$ if and only if $x$ is a winning position for Alice;

\item $\sigma _{A}\left( x\right) \leq \alpha $ if and only if $\exists z\in
X$, $\pi _{A}\left( x,z\right) <\alpha $;

\item $\pi _{A}\left( x\right) \leq \alpha $ if and only if $\forall z\in X$%
, $\sigma _{A}\left( x,z\right) <\alpha $.
\end{itemize}

We then define the rank $\rho \left( G_{A}\right) $ of the game $G_{A}$ to
be the least $\alpha $ such that $\pi _{A}\left( \varnothing \right) \leq
\alpha $.

This description of ranks is a reformulation of the one in terms of trees,
considering the following: If the payoff set $A$ for Alice is an open subset
of $X^{\omega }$, then the payoff set $B:=X^{\omega }\setminus A$ for Bob is
a \emph{closed }subset of $X^{\omega }$. Thus, there exists a tree $T$ such
that $B$ is the body of $T$. Then we have that the rank of $G_{A}$ is at
most $\alpha $ if and only if the rank of $T$ is at most $\alpha $.

\subsection{Indices\label{Subsection:indices}}

We define by recursion on countable ordinals sets of indices which can be
seen as a version of \emph{walks on ordinals }\cite[Chapter 2]%
{todorcevic_walks_2007}. Recall that for each limit ordinal $\lambda $ we
have fixed an increasing sequence of successor ordinals $\left( \lambda
_{i}\right) $ cofinal in $\lambda $. Likewise, if $\alpha $ is a successor
or zero, we set $\alpha _{i}=\alpha $ for every $i<\omega $. We now define
by recursion on $\alpha <\omega _{1}$ the sets of indices $I_{\alpha }$ and,
when $\alpha $ is a successor, $I_{\alpha }^{\mathrm{plain}}$. Define $%
I_{1}^{\mathrm{plain}}=\left\{ 0\right\} $.\ If $I_{\beta }$ and $I_{\beta
}^{\mathrm{plain}}$ has been defined for $1\leq \beta <\alpha $ define $%
I_{\alpha }$ to be the set of pairs $\left( i;n\right) $ with $i\in
I_{\alpha _{n}}^{\mathrm{plain}}$. Define then $I_{\alpha }^{\mathrm{plain}}$
to be $I_{\alpha -1}\cup \left\{ \alpha -1\right\} $.

We also define, for $\beta <\gamma $:

\begin{itemize}
\item $\sigma _{\gamma }I_{\beta }=I_{\beta }$;

\item $\sigma _{\gamma }I_{\beta }^{\mathrm{plain}}=I_{\beta }^{\mathrm{plain%
}}$;

\item $\partial _{\gamma }I_{\beta }=\varnothing $;

\item $\partial _{\gamma }I_{\beta }^{\mathrm{plain}}=\varnothing $;

\item $\sigma _{\gamma }I_{\gamma +1}^{\mathrm{plain}}=I_{\gamma }$;

\item $\partial _{\gamma }I_{\gamma +1}^{\mathrm{plain}}=\left\{ \gamma
\right\} $.
\end{itemize}

If $\sigma _{\gamma }I_{\beta }^{\mathrm{plain}}$ and $\partial _{\gamma
}I_{\beta }^{\mathrm{plain}}$ have been defined for $\gamma +1\leq \beta
\leq \alpha $ define $\sigma _{\gamma }I_{\alpha }$ to be the set of pairs $%
\left( i;n\right) $ with $i\in \sigma _{\gamma }I_{\alpha _{n}}^{\mathrm{%
plain}}$, and $\partial _{\gamma }I_{\alpha }$ to be the set of pairs $%
\left( i;n\right) $ with $i\in \partial _{\beta }I_{\alpha _{n}}^{\mathrm{%
plain}}$. We also set, for $\gamma <\alpha $, 
\begin{equation*}
\sigma _{\gamma }I_{\alpha }^{\mathrm{plain}}=\sigma _{\gamma }I_{\alpha -1}%
\text{ and }\partial _{\gamma }I_{\alpha }^{\mathrm{plain}}=\partial
_{\gamma }I_{\alpha -1}\cup \left\{ \alpha -1\right\} \text{.}
\end{equation*}%
Thus, by definition, $\sigma _{\gamma }$ and $\partial _{\gamma }$ describe
complementary subsets of $I_{\alpha }$ and $I_{\alpha }^{\mathrm{plain}}$.

We consider $I_{\alpha }$ and $I_{\alpha }^{\mathrm{plain}}$ as ordered set,
with respect to the relation $\prec $ defined recursively as follows. Say
that $\alpha -1$ is the largest element of $I_{\alpha }^{\mathrm{plain}}$,
and the map $I_{\left( \alpha -1\right) _{n}}^{\mathrm{plain}}\rightarrow
I_{\alpha }^{\mathrm{plain}}$, $i\mapsto \left( i;n\right) $ is
order-preserving. This can be seen as the relation of belonging to the same
branch on a canonical tree structure on $I_{\alpha }^{\mathrm{plain}}$ with
root $\alpha $. In fact, this can be regarded as a tree on $\omega $, where
the order on each level set in $I_{\alpha }^{\mathrm{plain}}$ is recursively
defined by setting $\left( \left( \alpha -1\right) _{n};n\right) \prec
\left( \left( \alpha -1\right) _{k};k\right) \Leftrightarrow n<k$. It is
easily seen by induction that $I_{\alpha }^{\mathrm{plain}}$ is a countable
well-founded tree of rank $\alpha $. Furthermore, every countable
well-founded rooted tree of rank at most $\alpha $ is isomorphic to a
subtree of $I_{\alpha }^{\mathrm{plain}}$, i.e., a nonempty subset $%
S\subseteq I_{\alpha }^{\mathrm{plain}}$ that contains, for each of its
nodes, the path from it to the root.

Clearly, if $\left( x_{i}\right) _{i\in S}$ is a diagram of shape $S$ in a
category $\mathcal{C}$ that has a limit $x$, then one can extend it to a
diagram of shape over $I_{\alpha }^{\mathrm{plain}}$ in $\mathcal{C}$ with
the same limit. This shows that every limit of a diagram whose shape is a
countable well-founded rooted tree of rank at most $\alpha $ can be realized
as a limit of a diagram of shape $I_{\alpha }^{\mathrm{plain}}$. Likewise,
we have that $I_{\alpha }$ is a countable forest of rank $\alpha $, and
every limit of a diagram whose shape is a countable forest of rank at most $%
\alpha $ is also a limit of a diagram of shape $I_{\alpha }$. By duality,
the same applies to colimits of diagrams of shape $S^{\mathrm{op}}$, i.e.,
presheaves on $S$.

There is a canonical linear order $<$ on $I_{\alpha }^{\mathrm{plain}}$
defined recursively by letting $\alpha -1$ be the largest element, every
descendent of $\left( \left( \alpha -1\right) _{n};n\right) $ less than
every descendent of $\left( \left( \alpha -1\right) _{k};k\right) $ for $n<k$%
, and the map $I_{\left( \alpha -1\right) _{n}}^{\mathrm{plain}}\rightarrow
I_{\alpha }^{\mathrm{plain}}$, $i\mapsto \left( i;n\right) $ is
order-preserving. In terms of the structure as tree on $\omega $ on $%
I_{\alpha }^{\mathrm{plain}}$, this is the order on nodes defined by setting 
$x<y$ if and only if $x$ is less than $y$ in the lexicographic order, or $%
x\prec y$. Such a linear order on $I_{\alpha }^{\mathrm{plain}}$ is easily
seen by induction to be isomorphic to $\omega ^{\alpha -1}+1$. The subset of
elements of $I_{\alpha }^{\mathrm{plain}}$ whose rank with respect to the
tree structure is at most $\beta <\alpha -1$ corresponds via this
isomorphism to the set of ordinals whose Cantor normal form is $\omega
^{\beta _{0}}k_{0}+\cdots +\omega ^{\beta _{\ell }}k_{\ell }$ for $%
k_{0},\ldots ,k_{\ell }\in \omega $ and $\alpha -1>\beta _{0}\geq \beta
_{1}\geq \cdots \geq \beta _{\ell }\geq \beta $. In particular, for $\beta
=1 $ we obtain that the\emph{\ terminal nodes} of $I_{\alpha }^{\mathrm{plain%
}}$ correspond to the \emph{successor ordinals} in $\omega ^{\alpha -1}$.

\section{The functor $\mathrm{lim}^{1}$\label{Section:lim1}}

In this section we regard the $\mathrm{lim}^{1}$ of a tower of countable
modules as a phantom pro-countable module. We explicitly describe its
Solecki submodules in terms of an ordinal sequence of towers obtained from
the given one by iterating the operation which we call \emph{derivation}. As
an application, we give a purely algebraic characterization of the towers $%
\boldsymbol{A}$ of countable modules such that $\left\{ 0\right\} $ has a
given complexity within $\mathrm{lim}^{1}\boldsymbol{A}$. We continue to let 
$R$ be a countable domain and to consider all modules to be $R$-modules.

\subsection{$\mathrm{lim}$\textrm{\ }and $\mathrm{lim}^{1}$ for a tower of
modules}

We adopt the notations from Section \ref{Subsection:towers} concerning the
category of towers of Polish modules. Thus, a \emph{tower }of Polish modules
is a sequence $\boldsymbol{G}=\left( G^{\left( n\right) },p^{\left(
n,n+1\right) }\right) _{n\in \omega }$ of Polish modules such that, for
every $n\in \omega $, $p^{\left( n,n+1\right) }:G^{\left( n+1\right)
}\rightarrow G^{\left( n\right) }$ is a continuous homomorphism. For $n\leq
m $, we define recursively $p^{\left( n,n\right) }:=\mathrm{id}_{G^{\left(
n\right) }}$ and $p^{\left( n,m+1\right) }:=p^{\left( n,m\right) }\circ
p^{\left( m,m+1\right) }$.

Suppose that $\boldsymbol{G}=(G^{\left( n\right) })$ and $\boldsymbol{H}%
=(H^{\left( n\right) })$ are towers of\emph{\ }Polish modules. A morphism
from $\boldsymbol{G}$ to $\boldsymbol{H}$ is represented by a sequence $%
\left( m_{k},f^{\left( k\right) }\right) $ such that $\left( m_{k}\right) $
is an increasing sequence in $\omega $, and $f^{\left( k\right) }:G^{\left(
m_{k}\right) }\rightarrow H^{\left( k\right) }$ is a continuous homomorphism
such that $p^{\left( k,k+1\right) }f^{\left( m_{k+1}\right) }=f^{\left(
k\right) }p^{\left( m_{k},m_{k+1}\right) }$ for every $k\in \omega $. Two
such sequences $\left( m_{k},f^{\left( k\right) }\right) $ and $\left(
m_{k}^{\prime },f^{\prime \left( k\right) }\right) $ define the same
morphism if and only if there exists an increasing sequence $\left( \tilde{m}%
_{k}\right) _{k\in \omega }$ in $\omega $ such that $\tilde{m}_{k}\geq \max
\left\{ m_{k},m_{k}^{\prime }\right\} $ for every $k\in \omega $, and%
\begin{equation*}
f^{\left( k\right) }p^{\left( m_{k},\tilde{m}_{k}\right) }=f^{\prime \left(
k\right) }p^{(m_{k}^{\prime },\tilde{m}_{k})}
\end{equation*}%
for every $k\in \omega $. This defines an additive category $\mathrm{pro}%
\left( \mathbf{PolMod}\left( R\right) \right) $ of towers of Polish modules.

We let:

\begin{itemize}
\item $\mathrm{pro}\left( \mathbf{Mod}\left( R\right) \right) $ be the thick
subcategory of $\mathrm{pro}\left( \mathbf{PolMod}\left( R\right) \right) $
spanned by towers of countable modules;

\item $\mathrm{pro}\left( \mathbf{Fin}\left( R\right) \right) $ be the thick
subcategory of $\mathrm{pro}\left( \mathbf{Mod}\left( R\right) \right) $
spanned by towers of finite modules.
\end{itemize}

Suppose that $\boldsymbol{G}=(G^{\left( n\right) })_{n\in \omega }$ is a
tower of Polish modules. For $n\in \omega $, set%
\begin{equation*}
\Lambda _{n}:=\left\{ \left( \lambda _{0},\ldots ,\lambda _{n}\right) \in
\omega \times \cdots \times \omega :\lambda _{0}\leq \lambda _{1}\leq \cdots
\leq \lambda _{n}\right\} \text{.}
\end{equation*}%
Define $K\left( \boldsymbol{G}\right) $ to be the cochain complex of Polish
modules obtained by setting%
\begin{equation*}
K^{n}\left( \boldsymbol{G}\right) =\left\{ 0\right\}
\end{equation*}%
for $n<0$ and%
\begin{equation*}
K^{n}\left( \boldsymbol{G}\right) =\prod_{\left( \lambda _{0},\ldots
,\lambda _{n}\right) \in \Lambda _{n}}G^{\left( \lambda _{0}\right) }
\end{equation*}%
for $n\geq 0$. The continuous homomorphism $\delta ^{n}:K^{n-1}\left(
G\right) \rightarrow K^{n}\left( G\right) $ is given by%
\begin{equation*}
\delta ^{n}\left( c\right) _{\left( \lambda _{0},\ldots ,\lambda _{n}\right)
}=p^{\left( \lambda _{0},\lambda _{1}\right) }\left( c_{\left( \lambda
_{1},\ldots ,\lambda _{n}\right) }\right) +\sum_{i=1}^{n}\left( -1\right)
^{i}c_{(\lambda _{0},\ldots ,\widehat{\lambda }_{i},\ldots ,\lambda _{n})}%
\text{.}
\end{equation*}%
One then defines, for $n\in \omega $,%
\begin{equation*}
\mathrm{Z}^{n}\left( K(\boldsymbol{G})\right) =\mathrm{\mathrm{Ker}}\left(
\delta ^{n}:K^{n}\left( \boldsymbol{G}\right) \rightarrow K^{n+1}\left( 
\boldsymbol{G}\right) \right)
\end{equation*}%
\begin{equation*}
\mathrm{B}^{n}\left( K(\boldsymbol{G})\right) =\mathrm{\mathrm{Ran}}\left(
\delta ^{n-1}:K^{n-1}\left( \boldsymbol{G}\right) \rightarrow K^{n}\left( 
\boldsymbol{G}\right) \right) \subseteq \mathrm{Z}^{n}\left( K(\boldsymbol{G}%
)\right)
\end{equation*}%
and%
\begin{equation*}
H^{n}\left( K(\boldsymbol{G})\right) =\frac{\mathrm{Z}^{n}\left( K(%
\boldsymbol{G})\right) }{\mathrm{B}^{n}\left( K(\boldsymbol{G})\right) }%
\text{.}
\end{equation*}%
By definition one has that 
\begin{equation*}
\mathrm{lim}\boldsymbol{G}=\mathrm{H}^{0}\left( K\left( \boldsymbol{G}%
\right) \right)
\end{equation*}%
and 
\begin{equation*}
\mathrm{lim}^{1}\boldsymbol{G}=\mathrm{H}^{1}\left( K\left( \boldsymbol{G}%
\right) \right) \text{;}
\end{equation*}%
see \cite[Section 11.5]{mardesic_strong_2000}. Furthermore, $\mathrm{H}%
^{n}\left( K\left( \boldsymbol{G}\right) \right) =0$ for $n\geq 2$.

We have that $\mathrm{lim}$ defines a countably continuous functor $\mathrm{%
pro}\left( \boldsymbol{\Pi }(\mathbf{Mod}\left( R\right) )\right)
\rightarrow \boldsymbol{\Pi }(\mathbf{Mod}\left( R\right) )$, while $\mathrm{%
lim}^{1}$ defines a countably continuous functor $\mathrm{pro}\left( 
\boldsymbol{\Pi }(\mathbf{Mod}\left( R\right) )\right) \rightarrow \mathrm{Ph%
}\left( \boldsymbol{\Pi }(\mathbf{Mod}\left( R\right) )\right) $. A short
exact sequence%
\begin{equation*}
0\rightarrow \boldsymbol{A}\rightarrow \boldsymbol{B}\rightarrow \boldsymbol{%
C}\rightarrow 0
\end{equation*}%
in $\mathrm{pro}\left( \boldsymbol{\Pi }(\mathbf{Mod}\left( R\right)
)\right) $ gives rise to a six-term exact sequence%
\begin{equation*}
0\rightarrow \mathrm{lim}\boldsymbol{A}\rightarrow \mathrm{lim}\boldsymbol{B}%
\rightarrow \mathrm{lim}\boldsymbol{C}\rightarrow \mathrm{lim}^{1}%
\boldsymbol{A}\rightarrow \mathrm{lim}^{1}\boldsymbol{B}\rightarrow \mathrm{%
lim}^{1}\boldsymbol{C}\rightarrow 0
\end{equation*}%
in $\mathrm{LH}\left( \mathbf{\Pi }\left( \mathbf{Mod}\left( R\right)
\right) \right) $.

As in Section \ref{Subsection:towers}, we say that $\boldsymbol{G}$ is an
epimorphic (respectively, monomorphic) tower if $p^{\left( n,n+1\right)
}:G^{\left( n+1\right) }\rightarrow G^{\left( n\right) }$ is a strict
epimorphism (respectively, a monomorphism) for every $n\in \omega $. We say
that $\boldsymbol{G}$ is essentially epimorphic (respectively, essentially
monomorphic) if it is isomorphic to an epimorphic (respectively,
monomorphic) tower.

A tower $\boldsymbol{G}$ is essentially epimorphic if and only if it
satisfies the so-called \emph{Mittag-Leffler condition}: for every $n\in
\omega $ there exists $k_{0}\in \omega $ such that, for every $k\geq k_{0}$, 
$p^{\left( n,n+k\right) }\left( G^{\left( n+k\right) }\right) =p^{\left(
n,n+k_{0}\right) }\left( G^{\left( n+k_{0}\right) }\right) $. This implies
that $\mathrm{lim}^{1}\boldsymbol{G}$ is trivial, where the converse holds
if $\boldsymbol{G}$ is a tower of countable modules; see \cite[Ch. II,
Section 6.2]{mardesic_shape_1982}. We let $\mathbf{\Pi }\left( \mathbf{Mod}%
\left( R\right) \right) $ be the full subcategory of $\mathrm{pro}\left( 
\mathbf{Mod}\left( R\right) \right) $ spanned by the (essentially)
epimorphic towers.

\begin{definition}
Let $\boldsymbol{A}=\left( A^{\left( n\right) }\right) $ be a tower of
pro-countable modules. Define $A_{\infty }^{\left( n\right) }$ to be the
submodule of $A^{\left( n\right) }$ consisting of $a\in A^{\left( n\right) }$
such that there exist $b_{i}\in A^{\left( i\right) }$ for $i\geq n$ such
that $b_{n}=a$ and $p^{\left( i,i+1\right) }\left( b_{i+1}\right) =b_{i}$
for $i\geq n$. Then $\boldsymbol{A}_{\infty }=(A_{\infty }^{\left( n\right)
})$ is a subtower of $\boldsymbol{A}$.
\end{definition}

Since $\boldsymbol{A}_{\infty }$ satisfies the Mittag-Leffler condition, $%
\mathrm{lim}^{1}\boldsymbol{A}_{\infty }=0$, and the quotient $\boldsymbol{A}%
\rightarrow \boldsymbol{A}/\boldsymbol{A}_{\infty }$ induces an isomorphism
at the level of $\mathrm{lim}^{1}$ modules. We say that a tower $\boldsymbol{%
A}$ is $\emph{reduced}$ if $\boldsymbol{A}_{\infty }=0$ or, equivalently, $%
\mathrm{lim}\boldsymbol{A}=0$.

Notice that the full subcategory $\mathrm{P}\left( \mathbf{Mod}\left(
R\right) \right) $ of \textrm{pro}$\left( \mathbf{Mod}\left( R\right)
\right) $ spanned by reduced towers is a fully exact quasi-abelian category
that is closed under subobjects but not under quotients. Let $\boldsymbol{B}$
be a tower of countable modules, and $\boldsymbol{A}$ be a subtower of $%
\boldsymbol{B}$. Consider the quotient $\boldsymbol{B}/\boldsymbol{A}$ of $%
\boldsymbol{B}$ by $\boldsymbol{A}$ in $\boldsymbol{\Pi }\left( \mathbf{Mod}%
\left( R\right) \right) $. Then we have an exact sequence%
\begin{equation*}
\mathrm{\mathrm{lim}}\boldsymbol{B}\rightarrow \mathrm{\mathrm{lim}}%
\boldsymbol{B/A}\rightarrow \mathrm{\mathrm{lim}}^{1}\boldsymbol{A}%
\rightarrow \mathrm{lim}^{1}\boldsymbol{B}\text{.}
\end{equation*}%
Thus, if $\boldsymbol{B}$ is reduced, then $\boldsymbol{B/A}$ is reduced if
and only if the inclusion $\mathrm{lim}^{1}\boldsymbol{A}\rightarrow \mathrm{%
lim}^{1}\boldsymbol{B}$ is injective, in which case $\boldsymbol{B}/%
\boldsymbol{A}$ is the quotient of $\boldsymbol{B}$ by $\boldsymbol{A}$ in
the category of reduced towers of countable modules.

\subsection{Derived towers}

Let $\boldsymbol{A}$ be a reduced tower of countable modules. We define a
sequence $\left( \boldsymbol{A}_{\alpha }\right) $ of reduced subtowers of $%
\boldsymbol{A}$, indexed by countable ordinals, which we called the \emph{%
derived towers} of $\boldsymbol{A}$. For every limit ordinal $\lambda $ we
fix a strictly increasing sequence $\left( \lambda _{n}\right) $ of
successor ordinals with $\sup_{n}\lambda _{n}=\lambda $. We also set $\alpha
_{k}=\alpha $ if $\alpha $ is either zero or a successor ordinal.

Given a tower $\boldsymbol{A}=\left( A^{\left( n\right) }\right) $ we define 
$A_{0}^{\left( n\right) }=A^{\left( n\right) }$ and for each successor
ordinal $\alpha $,%
\begin{equation*}
A_{\alpha }^{\left( n\right) }:=\bigcap_{k\in \omega }p^{\left( n,n+k\right)
}(A_{\left( \alpha -1\right) _{k}}^{\left( n+k\right) })\text{.}
\end{equation*}%
We then define for each ordinal $\alpha $, $\boldsymbol{A}_{\alpha }$ to be
the tower $(A_{\alpha _{n}}^{\left( n\right) })_{n\in \omega }$, which we
call the $\alpha $-th derived tower of $\boldsymbol{A}$.

Notice that for limit $\lambda $ we have%
\begin{equation*}
A_{\lambda +1}^{\left( n\right) }=\bigcap_{\alpha <\lambda }A_{\alpha
}^{\left( n\right) }\text{.}
\end{equation*}

\begin{definition}
Let $\boldsymbol{A}$ be a reduced tower of countable modules, and $\alpha
<\omega _{1}$. We say that $\boldsymbol{A}$ has (Mittag-Leffler) \emph{length%
} at most $\alpha $ if $\boldsymbol{A}_{\alpha }\cong 0$. When $\alpha $ is
a successor ordinal, we say that $\boldsymbol{A}$ has plain length at most $%
\alpha $ if $\boldsymbol{A}_{\alpha -1}$ is essentially monomorphic.
\end{definition}

A similar definition is also considered in \cite[Definition 3.2]%
{boardman_conditionally_1999}; see also \cite{ha_higher_2003}.

Let $X$ be the disjoint union of $A^{\left( n\right) }$ for $n\in \omega $.\
Then we have that the set $T_{\boldsymbol{A}}$ of $x=\left( a_{0},\ldots
,a_{n}\right) $ such that, for all $i\in \left\{ 0,1,\ldots ,n\right\} $ and 
$j\in \left\{ 1,2,\ldots ,n\right\} $, $a_{i}\in A^{\left( i\right) }$ and $%
p^{\left( j-1,j\right) }\left( a_{j}\right) =a_{j-1}$, is a tree on $X$.\
Clearly, we have that $\boldsymbol{A}$ is reduced if and only if $T_{%
\boldsymbol{A}}$ is well-founded, and $\boldsymbol{A}\cong 0$ if and only if 
$T_{\boldsymbol{A}}$ is finite. By recursion on $\alpha $ one can see that 
\begin{equation*}
a_{n}\in A_{\alpha }^{\left( n\right) }\Leftrightarrow \left( p^{\left(
0,n\right) }\left( a_{n}\right) ,\ldots ,p^{\left( n-1,n\right) }\left(
a_{n}\right) ,a_{n}\right) \in T_{\boldsymbol{A}}^{\left( \omega \alpha
\right) }\text{.}
\end{equation*}%
Thus, the length of $\boldsymbol{A}$ is at most $\alpha $ if and only if $T_{%
\boldsymbol{A}}^{\left( \omega \alpha \right) }$ is finite. This also shows
the length defines a co-analytic rank on towers of countable modules.

\begin{definition}
Let $\boldsymbol{A}$ be a reduced tower of countable modules, and $\alpha
<\omega _{1}$. For $n\in \omega $, define the chain of submodules 
\begin{equation*}
S^{\left( n\right) }\left( \boldsymbol{A}\right) =(p^{\left( n,k\right)
}(A^{\left( k\right) }))_{k\geq n}
\end{equation*}%
of $A^{\left( n\right) }$. We regard $S^{\left( n\right) }\left( \boldsymbol{%
A}\right) $ as a monomorphic tower, where the bonding maps are the inclusion
maps.
\end{definition}

We let $L^{\left( n\right) }\left( \boldsymbol{A}\right) $ be the completion
of $A^{\left( n\right) }$ with respect to the chain of submodules $%
(S^{\left( n\right) }\left( \boldsymbol{A}\right) )$ and $\kappa _{%
\boldsymbol{A}}^{\left( n\right) }:A^{\left( n\right) }\rightarrow L^{\left(
n\right) }\left( \boldsymbol{A}\right) $ be the canonical homomorphism. The
cokernel $E^{\left( n\right) }\left( \boldsymbol{A}\right) =\mathrm{Coker}%
(\kappa _{\boldsymbol{A}}^{\left( n\right) })$ is then isomorphic to $%
\mathrm{lim}^{1}S^{\left( n\right) }\left( \boldsymbol{A}\right) $; see \cite%
[Theorem 5.13]{bergfalk_definable_2024}. This gives an epimorphic tower $%
(E^{\left( n\right) }\left( \boldsymbol{A}\right) )$ of phantom
pro-countable modules. By definition, $A_{1}^{\left( n\right) }$ is equal to 
$\mathrm{\mathrm{Ker}}(\kappa _{\boldsymbol{A}}^{\left( n\right) })$. We
have an exact sequence%
\begin{equation*}
0\rightarrow \boldsymbol{A}_{1}\rightarrow \boldsymbol{A}\rightarrow \mathrm{%
lim}_{n}{}S^{\left( n\right) }\left( \boldsymbol{A}\right) \rightarrow 0%
\text{,}
\end{equation*}%
and hence an exact sequence%
\begin{equation*}
0\rightarrow \mathrm{lim}^{1}\boldsymbol{A}_{1}\rightarrow \mathrm{lim}^{1}%
\boldsymbol{A}\rightarrow \mathrm{lim}_{n}{}E^{\left( n\right) }\left( 
\boldsymbol{A}\right) \rightarrow 0\text{.}
\end{equation*}%
If $\lambda $ is a limit ordinal, then%
\begin{equation*}
\mathrm{lim}^{1}\boldsymbol{A}_{\lambda }=\bigcap_{\beta <\lambda }\mathrm{%
lim}^{1}\boldsymbol{A}_{\beta }\subseteq \mathrm{lim}^{1}\boldsymbol{A}
\end{equation*}%
and we have an exact sequence%
\begin{equation*}
0\rightarrow \boldsymbol{A}_{\lambda }\rightarrow \boldsymbol{A}\rightarrow 
\mathrm{lim}_{k}\boldsymbol{A}/\boldsymbol{A}_{\lambda _{k}}\rightarrow 0%
\text{.}
\end{equation*}%
The sequence of submodules $(\mathrm{\mathrm{Ker}}(\mathrm{lim}^{1}%
\boldsymbol{A}\rightarrow E^{\left( n\right) }\left( \boldsymbol{A}\right)
))_{n\in \omega }$ is sometimes called the \emph{Gray filtration}; see \cite%
{ha_gray_2001}.

\subsection{The Solecki submodules of $\mathrm{\mathrm{\mathrm{lim}}}^{1}$}

In this section we describe the Solecki submodules of the $\mathrm{lim}^{1}$
of a tower of countable modules in terms of its derived towers.

\begin{proposition}
\label{Proposition:lim1-Solecki1}Suppose that $\boldsymbol{A}=\left(
A^{\left( n\right) }\right) $ is a tower of countable modules.\ Then%
\begin{equation*}
\mathrm{lim}^{1}\boldsymbol{A}_{1}=s_{1}\left( \mathrm{lim}^{1}\boldsymbol{A}%
\right) \text{.}
\end{equation*}
\end{proposition}

\begin{proof}
By the remarks in the previous section, we have that%
\begin{equation*}
\mathrm{lim}_{n}^{1}A_{1}^{\left( n\right) }=\mathrm{\mathrm{Ker}}\left( 
\mathrm{lim}_{n}^{1}A^{\left( n\right) }\rightarrow \mathrm{lim}%
_{n}{}E_{1}(A^{\left( n\right) })\right) \text{.}
\end{equation*}%
For $n\in \omega $, since $A^{\left( n\right) }$ is countable, $\left\{
0\right\} $ is $\boldsymbol{\Sigma }_{2}^{0}$ in $E_{1}(A^{\left( n\right)
}) $. As this holds for every $n\in \omega $, we have that $\left\{
0\right\} $ is $\boldsymbol{\Pi }_{3}^{0}$ in $\mathrm{lim}_{n}E_{1}\left(
A^{\left( n\right) }\right) $.\ Hence, $\mathrm{lim}_{n}^{1}{}A_{1}^{\left(
n\right) }$ is $\boldsymbol{\Pi }_{3}^{0}$ in $\mathrm{lim}_{n}^{1}A^{\left(
n\right) }$. We have that $\left\{ 0\right\} $ is dense in $\mathrm{lim}%
_{n}^{1}A_{1}^{\left( n\right) }$ by Proposition \ref%
{Proposition:phantom-category}(1). By Lemma \ref{Lemma:1st-Solecki}, it
remains to prove that if $V$ is an open neighborhood of $0$ in $\mathrm{B}%
^{1}\left( K\left( \boldsymbol{A}\right) \right) $, then $\overline{V}^{%
\mathrm{Z}^{1}\left( K\left( \boldsymbol{A}\right) \right) }$ contains an
open neighborhood of $0$ in $\mathrm{Z}^{1}\left( K\left( \boldsymbol{A}%
_{1}\right) \right) \subseteq \mathrm{Z}^{1}\left( K\left( \boldsymbol{A}%
\right) \right) $. Without loss of generality, we can assume that there
exists $N\in \omega $ such that%
\begin{equation*}
V=\left\{ \left( p^{\left( n_{0},n_{1}\right) }\left( b_{n_{1}}\right)
-b_{n_{0}}\right) \in \mathrm{B}^{1}\left( K\left( \boldsymbol{A}\right)
\right) :\forall n\leq N,b_{n}=0\right\} \text{.}
\end{equation*}%
Consider now%
\begin{equation*}
U=\left\{ \left( a_{n_{0},n_{1}}\right) \in \mathrm{Z}^{1}\left( K\left( 
\boldsymbol{A}_{1}\right) \right) :a_{n_{0},n_{1}}\in A_{1}^{\left(
n_{0}\right) },\forall n_{0},n_{1}\leq N,a_{n_{0},n_{1}+1}=0\right\} \text{.}
\end{equation*}%
Then we have that $U$ is an open neighborhood of $0$ in $\mathrm{Z}%
^{1}\left( K\left( \boldsymbol{A}_{1}\right) \right) $.

We claim that $U\subseteq \overline{V}^{\mathrm{Z}^{1}\left( K\left( 
\boldsymbol{A}\right) \right) }$. Suppose that%
\begin{equation*}
\tilde{x}=\left( \tilde{a}_{n_{0},n_{1}}\right) \in U\text{.}
\end{equation*}%
Thus, we have that $\tilde{a}_{n_{0},n_{1}}=0$ for $n_{0},n_{1}\leq N$. Fix
an open neighborhood $W$ of $\tilde{x}$ in $\mathrm{Z}^{1}\left( K\left( 
\boldsymbol{A}\right) \right) $. Then there exists $M>N$ such that $W$
contains%
\begin{equation*}
\left\{ \left( a_{n_{0},n_{1}}\right) \in \mathrm{Z}^{1}\left( K\left( 
\boldsymbol{A}\right) \right) :a_{n_{0},n_{1}}\in A^{\left( n_{0}\right)
},\forall n_{0},n_{1}\leq M,a_{n_{0},n_{1}}=\tilde{a}_{n_{0},n_{1}}\right\} 
\text{.}
\end{equation*}%
We need to find $b_{n}\in A^{\left( n\right) }$ for $n\in \omega $ such that:

\begin{enumerate}
\item $b_{n}=0$ for $n\leq N$;

\item $p^{\left( n,n+1\right) }\left( b_{n+1}\right) -b_{n}=\tilde{a}%
_{n,n+1} $ for $n<M$.
\end{enumerate}

This will give an element $\left( p^{\left( n,n+1\right) }\left(
b_{n+1}\right) -b_{n}\right) _{n\in \omega }$ in $V\cap W$. Set $b_{n}=0$
for $n\leq N+1$ and for $n>M$. We now define recursively $%
b_{N+2},b_{N+2},\ldots ,b_{M}$ such that%
\begin{equation*}
b_{n}\in p^{\left( n,M\right) }(A^{\left( M\right) })\subseteq A^{\left(
n\right) }
\end{equation*}%
and%
\begin{equation*}
p^{\left( n,n+1\right) }\left( b_{n+1}\right) =b_{n}+\tilde{a}_{n,n+1}
\end{equation*}%
for $n<M$. Since $\tilde{a}_{N+1,N+2}\in A_{1}^{\left( N+1\right) }$ there
exists $b_{N+2}\in p^{\left( N+1,M\right) }(A^{\left( M\right) })$ such that 
\begin{equation*}
p^{\left( N+1,N+2\right) }\left( b_{N+2}\right) =\tilde{a}_{N+1,N+2}=b_{N+1}+%
\tilde{a}_{N+1,N+2}\text{.}
\end{equation*}%
Suppose that $b_{N+2},\ldots ,b_{k}$ have been defined for some $k<M$. Since 
$\tilde{a}_{k,k+1}\in A_{1}^{\left( N+1\right) }$, we have that%
\begin{equation*}
b_{k}+\tilde{a}_{k,k+1}\in p^{\left( k,M\right) }(A^{\left( M\right)
})+A_{1}^{\left( k\right) }\subseteq p^{\left( k,M\right) }(A^{\left(
M\right) })\text{.}
\end{equation*}%
Thus, there exists $b_{k+1}\in p^{\left( k+1,M\right) }\left( A^{\left(
M\right) }\right) $ such that%
\begin{equation*}
p^{\left( k,k+1\right) }\left( b_{k+1}\right) =b_{k}+\tilde{a}_{k,k+1}\text{.%
}
\end{equation*}%
This concludes the recursive definition, and the proof.
\end{proof}

\begin{lemma}
\label{Lemma:monomorphic-tower}Let $\boldsymbol{A}=\left( A^{\left( n\right)
}\right) $ be a reduced tower of countable modules.\ The following
assertions are equivalent:

\begin{enumerate}
\item $\boldsymbol{A}$ has plain length at most $1$;

\item $\mathrm{lim}^{1}\boldsymbol{A}$ has plain Solecki length at most $1$;

\item $\boldsymbol{A}\cong S_{1}^{\left( n\right) }\left( \boldsymbol{A}%
\right) $ for some $n\in \omega $.
\end{enumerate}
\end{lemma}

\begin{proof}
(1)$\Rightarrow $(2) If $\boldsymbol{A}$ is a monomorphic tower, then the $%
E^{\left( n\right) }\left( \boldsymbol{A}\right) \rightarrow E^{\left(
0\right) }\left( \boldsymbol{A}\right) $ is an isomorphism for every $n\in
\omega $. Furthermore, by Proposition \ref{Proposition:lim1-Solecki1} the
map $\mathrm{lim}^{1}\boldsymbol{A}\rightarrow E^{\left( 0\right) }\left( 
\boldsymbol{A}\right) $ is an isomorphism. As $\left\{ 0\right\} $ is $%
\boldsymbol{\Sigma }_{2}^{0}$ in $E^{\left( 0\right) }\left( \boldsymbol{A}%
\right) $, the same holds for $\mathrm{lim}^{1}\boldsymbol{A}$.

(2)$\Rightarrow $(1) If $\left\{ 0\right\} $ is $\boldsymbol{\Sigma }%
_{2}^{0} $ in $\mathrm{lim}^{1}\boldsymbol{A}$, then $\boldsymbol{A}$ has
length $1$. As $\boldsymbol{A}$ is reduced, after passing to an isomorphic
tower we can assume that for every $n\in \omega $, $\kappa _{\boldsymbol{A}%
}^{\left( n\right) }:A^{\left( n\right) }\rightarrow L^{\left( n\right)
}\left( \boldsymbol{A}\right) $ is injective. We can thus identify $%
A^{\left( n\right) }$ with its image inside $L^{\left( n\right) }\left( 
\boldsymbol{A}\right) $. By Proposition \ref{Proposition:lim1-Solecki1}, we
have that $\left\{ 0\right\} $ is $\boldsymbol{\Sigma }_{2}^{0}$ in $\mathrm{%
lim}_{n}E^{\left( n\right) }\left( \boldsymbol{A}\right) $. Hence, by
Corollary \ref{Corollary:compact-phantom}, without loss of generality we can
assume that for every $n\in \omega $, $E^{\left( n+1\right) }\left( 
\boldsymbol{A}\right) \rightarrow E^{\left( n\right) }\left( \boldsymbol{A}%
\right) $ is an isomorphism. The fact that $E^{\left( 1\right) }\left( 
\boldsymbol{A}\right) \rightarrow E^{\left( 0\right) }\left( \boldsymbol{A}%
\right) $ is injective implies that the closure of $N^{\left( 1\right) }:=%
\mathrm{Ker}\left( p^{\left( 0,1\right) }:A^{\left( 1\right) }\rightarrow
A^{\left( 0\right) }\right) $ within $L^{\left( 1\right) }\left( \boldsymbol{%
A}\right) $ is contained in $A^{\left( 1\right) }$. Therefore, we have that
the chain of submodules $\left( N^{\left( 1\right) }\cap p^{\left(
1,n\right) }\left( A^{\left( n\right) }\right) \right) _{n\in \omega }$ of $%
N^{\left( 1\right) } $ is eventually constant. Hence, there exists $%
n_{2}\geq 1$ such that $p^{\left( 1,n_{2}\right) }\left( A^{\left(
n_{2}\right) }\right) \cap N^{\left( 1\right) }=\left\{ 0\right\} $.
Proceeding in this fashion, we can define an increasing sequence $\left(
n_{k}\right) $ in $\omega $ with $n_{0}=0$ and $n_{1}=1$ such that $%
p^{\left( n_{i},n_{i+1}\right) }\left( A^{\left( n_{i+1}\right) }\right)
\cap \mathrm{\mathrm{Ker}}\left( p^{\left( n_{i-1},n_{i}\right) }\right)
=\left\{ 0\right\} $ for every $i\geq 1$. The tower $\left( p^{\left(
n_{i},n_{i+1}\right) }\left( A^{\left( n_{i+1}\right) }\right) \right)
_{i\in \omega }$ is the desired monomorphic tower isomorphic to $\boldsymbol{%
A}$.

(1)$\wedge $(2)$\Rightarrow $(3) If $\boldsymbol{A}$ is a monomorphic tower,
then $E^{\left( n\right) }\left( \boldsymbol{A}\right) \rightarrow E^{\left(
0\right) }\left( \boldsymbol{A}\right) $ is an isomorphism. Furthermore, if $%
\left\{ 0\right\} $ is $\boldsymbol{\Sigma }_{2}^{0}$ in $\mathrm{lim}^{1}$ $%
\boldsymbol{A}$ then by Proposition \ref{Proposition:lim1-Solecki1}, 
\begin{equation*}
\mathrm{\mathrm{\mathrm{Ke}}r}\left( \mathrm{\mathrm{lim}}^{1}\boldsymbol{A}%
\rightarrow E^{\left( 0\right) }\left( \boldsymbol{A}\right) \right) =%
\mathrm{\mathrm{Ker}}\left( \mathrm{lim}^{1}\boldsymbol{A}\rightarrow 
\mathrm{lim}_{n}E^{\left( n\right) }\left( \boldsymbol{A}\right) \right) =0
\end{equation*}%
for every $n\in \omega $.

(3)$\Rightarrow $(2) This follows from the fact that $\left\{ 0\right\} $ is 
$\boldsymbol{\Sigma }_{2}^{0}$ in $E^{\left( n\right) }\left( \boldsymbol{A}%
\right) $ for every $n\in \omega $.
\end{proof}

\begin{theorem}
\label{Theorem:Solecki-lim1}Suppose that $\left( A^{\left( n\right) }\right) 
$ is a tower of countable modules and $\alpha <\omega _{1}$. Then we have
that%
\begin{equation*}
s_{\alpha }\left( \mathrm{lim}^{1}\boldsymbol{A}\right) =\mathrm{lim}^{1}%
\boldsymbol{A}_{\alpha }\text{.}
\end{equation*}
\end{theorem}

\begin{proof}
We prove that the conclusion holds by induction on $\alpha $. The limit case
follows from the observation that, when $\lambda $ is limit,%
\begin{equation*}
\mathrm{lim}^{1}\boldsymbol{A}_{\lambda }=\bigcap_{\beta <\lambda }\mathrm{%
lim}^{1}\boldsymbol{A}_{\beta }\text{.}
\end{equation*}%
The successor stage follows from Proposition \ref{Proposition:lim1-Solecki1}
applied to $\boldsymbol{A}_{\alpha }$.
\end{proof}

\begin{corollary}
\label{Corollary:Solecki-lim1}A reduced tower $\boldsymbol{A}$ has (plain)
length at most $\alpha $ if and only if $\mathrm{lim}^{1}\boldsymbol{A}$ has
(plain)\ Solecki length at most $\alpha $.
\end{corollary}

For a tower $\boldsymbol{A}=\left( A^{\left( n\right) }\right) $, we let $%
\boldsymbol{A}^{\oplus \omega }$ be the tower whose $n$-th term is the
countable direct sum of infinitely many copies of $A^{\left( n\right) }$,
with bonding maps induced by those of $\boldsymbol{A}$. It is easy to prove
by induction on $\alpha <\omega $ that $(\boldsymbol{A}^{\oplus \omega
})_{\alpha }=(\boldsymbol{A}_{\alpha })^{\oplus \omega }$. Furthermore, $%
\boldsymbol{A}$ is plain if and only if $\boldsymbol{A}^{\oplus \omega }$ is
plain. As a consequence of these observations and Corollary \ref%
{Corollary:Solecki-lim1} we obtain the following.

\begin{corollary}
\label{Corollary:sum-tower}If $\boldsymbol{A}$ is a reduced tower of
countable modules, then $\left\{ 0\right\} $ has the same complexity class
in $\mathrm{lim}^{1}\boldsymbol{A}$ and in $\mathrm{lim}^{1}\boldsymbol{A}%
^{\oplus \omega }$.
\end{corollary}

\begin{lemma}
\label{Lemma:quotient-tower}Suppose that $\left( A^{\left( k\right) }\right) 
$ is a tower of countable modules, and $0\rightarrow B\rightarrow A^{\left(
k\right) }\rightarrow A^{\left( k\right) }/B\rightarrow 0$ is an exact
sequence for every $k\in \omega $. Then the map $\mathrm{\mathrm{lim}}%
_{k}^{1}A^{\left( k\right) }\rightarrow \mathrm{\mathrm{lim}}%
_{k}^{1}(A^{\left( k\right) }/B)$ is an isomorphism, which restricts to an
isomorphism $\mathrm{\mathrm{\mathrm{li}}m}_{k}^{1}{}A_{\alpha }^{\left(
k\right) }\rightarrow \mathrm{\mathrm{lim}}_{k}^{1}(A^{\left( k\right)
}/B)_{\alpha }$ for every $\alpha <\omega _{1}$.
\end{lemma}

\begin{proof}
It follows from the six-term exact sequence relating $\mathrm{lim}$ and $%
\mathrm{lim}^{1}$, considering that $\mathrm{\mathrm{lim}}%
_{k}^{1}{}B^{\left( k\right) }=0$ if $B^{\left( k\right) }=B$ for every $%
k\in \omega $.
\end{proof}

\begin{lemma}
\label{Lemma:lim1-derived-tower}Suppose that $\boldsymbol{A}$ is a reduced
tower of countable modules, and let $\boldsymbol{A}_{\alpha }$ be its
derived tower for $\alpha <\omega _{1}$. For $\ell \leq i<\omega $ and $%
\alpha <\omega _{1}$, define%
\begin{equation*}
A_{\alpha }^{\left( i\right) }[\ell ]=\mathrm{\mathrm{Ker}}(A_{\alpha
}^{\left( i\right) }\rightarrow A_{\alpha }^{\left( \ell \right) })\text{,}
\end{equation*}%
For fixed $\ell <\omega $ and $\alpha <\omega _{1}$ we let $\boldsymbol{A}%
_{\alpha }[\ell ]$ be the tower $(A_{\alpha }^{\left( i\right) }[\ell
])_{i\geq \ell }$. The canonical homomorphism $\mathrm{\mathrm{lim}}^{1}%
\boldsymbol{A}_{\alpha }[\ell ]\rightarrow \mathrm{lim}^{1}\boldsymbol{A}%
[\ell ]$ is injective. For $\ell \leq i<\omega $ and $\alpha <\omega _{1}$,
we have a short exact sequence%
\begin{equation*}
0\rightarrow A_{\alpha }^{\left( i\right) }[\ell ]\rightarrow A_{\alpha
}^{\left( i\right) }\rightarrow \mathrm{Ran}(A_{\alpha }^{\left( i\right)
}\rightarrow A_{\alpha }^{\left( \ell \right) })\rightarrow 0\text{.}
\end{equation*}%
Considering the six-term exact sequence relating $\mathrm{lim}$ and $\mathrm{%
lim}^{1}$, this gives rise to an exact sequence%
\begin{equation*}
\mathrm{lim}^{1}A_{\alpha }^{\left( i\right) }\rightarrow A_{\alpha
+1}^{\left( \ell \right) }\rightarrow \mathrm{\mathrm{lim}}^{1}\boldsymbol{A}%
_{\alpha }[\ell ]\rightarrow \mathrm{\mathrm{lim}}^{1}\boldsymbol{A}_{\alpha
}\rightarrow \mathrm{lim}_{i}^{1}\mathrm{Ran}(A_{\alpha }^{\left( i\right)
}\rightarrow A_{\alpha }^{\left( \ell \right) })\rightarrow 0\text{.}
\end{equation*}%
We have that%
\begin{equation*}
\mathrm{\mathrm{lim}}^{1}\boldsymbol{A}_{\alpha +1}=\bigcap_{\ell \in \omega
}\mathrm{\mathrm{Ker}}(\mathrm{\mathrm{lim}}^{1}\boldsymbol{A}_{\alpha
}\rightarrow \mathrm{\mathrm{lim}}_{i}^{1}\mathrm{Ran}(A_{\alpha }^{\left(
i\right) }\rightarrow A_{\alpha }^{\left( \ell \right) }))=\bigcap_{\ell \in
\omega }\mathrm{Ran}(\mathrm{\mathrm{lim}}^{1}\boldsymbol{A}_{\alpha }[\ell
]\rightarrow \mathrm{\mathrm{lim}}^{1}\boldsymbol{A}_{\alpha })\text{.}
\end{equation*}
\end{lemma}

\begin{proof}
After replacing $\boldsymbol{A}$ with $\boldsymbol{A}_{\alpha }$, it
suffices to consider the case $\alpha =1$.

Suppose that $a$ is an element of $\mathrm{\mathrm{lim}}^{1}\boldsymbol{A}%
_{1}$, represented by $\left( a_{ij}\right) \in \prod_{i\leq j}A_{1}^{\left(
i\right) }$. We fix $\ell <\omega $ and show that $a$ belongs to 
\begin{equation*}
\mathrm{\mathrm{Ker}}(\mathrm{\mathrm{lim}}^{1}\boldsymbol{A}\rightarrow 
\mathrm{\mathrm{lim}}_{i}^{1}\mathrm{Ran}(A^{\left( i\right) }\rightarrow
A^{\left( \ell \right) }))\text{.}
\end{equation*}%
Consider the element $\left( p^{\left( \ell ,i\right) }\left( a_{ij}\right)
\right) _{i\leq j}\in \prod_{\ell \leq i\leq j}p^{\left( \ell ,i\right)
}\left( A^{\left( i\right) }\right) $. Using the fact that 
\begin{equation*}
a_{ij}\in A_{1}^{\left( i\right) }=\bigcap_{t>i}p^{\left( i,t\right)
}(A^{\left( t\right) })
\end{equation*}%
we can define recursively $b_{i}\in A^{\left( i\right) }$ for $i\geq \ell $
such that $b_{\ell }=0$ and 
\begin{equation*}
p^{\left( \ell ,i\right) }\left( a_{i,i+1}\right) =p^{\left( \ell ,i\right)
}\left( b_{i}\right) -p^{\left( \ell ,i+1\right) }\left( b_{i+1}\right)
\end{equation*}%
for $i\geq \ell $. This shows that $\left( p^{\left( \ell ,i\right) }\left(
a_{ij}\right) \right) _{i\leq j}$ represents the trivial element of $\mathrm{%
\mathrm{lim}}_{i}^{1}\mathrm{Ran}\left( A^{\left( i\right) }\rightarrow
A^{\left( \ell \right) }\right) $. Therefore, $a$ belongs to 
\begin{equation*}
\mathrm{\mathrm{Ker}}(\mathrm{\mathrm{lim}}^{1}\boldsymbol{A}\rightarrow 
\mathrm{\mathrm{lim}}_{i}^{1}\mathrm{Ran}(A^{\left( i\right) }\rightarrow
A^{\left( \ell \right) }))\text{.}
\end{equation*}

Conversely, suppose that $a$ is an element of 
\begin{equation*}
\bigcap_{\ell \in \omega }\mathrm{\mathrm{Ker}}(\mathrm{\mathrm{lim}}^{1}%
\boldsymbol{A}\rightarrow \mathrm{\mathrm{lim}}_{i}^{1}p^{\left( \ell
,i\right) }(A^{\left( i\right) }))\text{,}
\end{equation*}%
represented by $\left( a_{ij}\right) \in \prod_{i\leq j}A^{\left( i\right) }$%
. Fix $\ell \in \omega $ and define $\hat{A}^{\left( \ell \right) }=\mathrm{%
lim}_{i}(A^{\left( \ell \right) }/p^{\left( \ell ,i\right) }\left( A^{\left(
i\right) }\right) )$. We also let $\kappa ^{\left( \ell \right) }:A^{\left(
\ell \right) }\rightarrow \hat{A}^{\left( \ell \right) }$ be the canonical
homomorphism obtained by mapping $a$ to $\left( a+p^{\left( \ell ,i\right)
}\left( A^{\left( i\right) }\right) \right) _{i\geq \ell }$. Consider the
isomorphism%
\begin{equation*}
\mathrm{\mathrm{lim}}_{i}^{1}p^{\left( \ell ,i\right) }(A^{\left( i\right)
})\cong \hat{A}^{\left( \ell \right) }/\kappa (A^{\left( \ell \right) })
\end{equation*}%
obtained by mapping an element $b$ of $\mathrm{\mathrm{lim}}%
_{i}^{1}p^{\left( \ell ,i\right) }\left( A^{\left( i\right) }\right) $
represented by $\left( b_{ij}\right) $ to $\left( \mathrm{lim}_{j}\kappa
\left( b_{\ell j}\right) \right) +\kappa \left( A^{\left( \ell \right)
}\right) $; see \cite[Theorem 5.13]{bergfalk_definable_2024}. Since $\left(
a_{ij}\right) $ represents an element of $\mathrm{\mathrm{Ker}}(\mathrm{%
\mathrm{lim}}^{1}\boldsymbol{A}\rightarrow \mathrm{\mathrm{lim}}%
_{i}^{1}p^{\left( \ell ,i\right) }\left( A^{\left( i\right) }\right) )$, we
have that $\left( p^{\left( \ell ,i\right) }\left( a_{ij}\right) \right) $
represents the trivial element in $\mathrm{\mathrm{lim}}_{i}^{1}p^{\left(
\ell ,i\right) }\left( A^{\left( i\right) }\right) $. Thus, there exists $%
c_{\ell }\in A^{\left( \ell \right) }$ such that $\mathrm{\mathrm{lim}}%
_{j}\kappa ^{\left( \ell \right) }(a_{\ell j})=\kappa ^{\left( \ell \right)
}(c_{\ell })$. This means that 
\begin{equation*}
\kappa ^{\left( \ell \right) }\left( a_{\ell ,\ell +1}\right) =\mathrm{%
\mathrm{lim}}_{i}\kappa ^{\left( \ell \right) }(p^{\left( \ell ,i\right)
}\left( a_{\ell i}\right) )-\mathrm{\mathrm{lim}}_{i}\kappa ^{\left( \ell
\right) }(p^{\left( \ell ,\ell +1\right) }\left( a_{\ell +1,i}\right)
)=\kappa ^{\left( \ell \right) }\left( c_{\ell }-p^{\left( \ell ,\ell
+1\right) }\left( c_{\ell +1}\right) \right) .
\end{equation*}%
Thus, 
\begin{equation*}
b_{\ell ,\ell +1}:=a_{\ell ,\ell +1}-(c_{\ell }-p^{\left( \ell ,\ell
+1\right) }\left( c_{\ell +1}\right) )\in \mathrm{\mathrm{Ker}}(\kappa
^{\left( \ell \right) })=A_{1}^{\left( \ell \right) }.
\end{equation*}%
As $\left( b_{\ell ,\ell +1}\right) $ and $\left( a_{ij}\right) $ represent
the same element $\mathrm{\mathrm{lim}}^{1}\boldsymbol{A}$, this shows that $%
a$ belongs to $\mathrm{\mathrm{lim}}^{1}\boldsymbol{A}_{1}$.
\end{proof}

\subsection{Fishbone towers}

We now present a way to construct new towers from old. In order to present
this construction more succinctly, we introduce the notion of \emph{fishbone}
of towers.

\begin{definition}
\label{Definition:fishbone-tower}Consider a nontrivial reduced essentially
monomorphic tower $\boldsymbol{A}=\left( A^{\left( n\right) }\right) _{n\in
\omega }$ and reduced towers $\boldsymbol{B}[k]=\left( B^{\left( n\right)
}[k]\right) _{n\in \omega }$ for $k\in \omega $.

\begin{itemize}
\item We say that $(\boldsymbol{A},\boldsymbol{B}[k])_{k\in \omega }$ is a 
\emph{fishbone }with spine $\boldsymbol{A}$ and ribs $\boldsymbol{B}[k]$ for 
$k\in \omega $ if $A^{\left( n\right) }=B^{\left( 0\right) }[n]$ for every $%
n\in \omega $.

\item If $\alpha $ is a successor ordinal, then we say that such a fishbone
is \emph{straight} of length $\alpha $ if for every $k\in \omega $, $%
\boldsymbol{B}[k]$ has plain length $\beta \lbrack k]<\alpha $, $%
\sup_{k}\beta \lbrack k]=\alpha -1$, and 
\begin{equation*}
p^{\left( \ell ,0\right) }(B_{\beta \lbrack k]-1}^{\left( \ell \right)
}[k])+p^{(k,\ell )}(A^{\left( \ell \right) })=A^{\left( k\right) }
\end{equation*}%
for every $\ell \geq k$.

\item Suppose that $(\boldsymbol{A},\boldsymbol{B}[k])_{k\in \omega }$ is a 
\emph{fishbone }with spine $\boldsymbol{A}$ and ribs $\boldsymbol{B}[k]$ for 
$k\in \omega $. We define the corresponding \emph{fishbone tower }to be the
tower $\left( C^{\left( n\right) }\right) _{n\in \omega }$ where 
\begin{equation*}
C^{\left( n\right) }:=B^{\left( n\right) }[0]\oplus B^{\left( n-1\right)
}[1]\oplus \cdots \oplus B^{\left( 1\right) }[n-1]\oplus A^{\left( n\right) }
\end{equation*}%
and bonding maps $C^{\left( n+1\right) }\rightarrow C^{\left( n\right) }$,%
\begin{equation*}
\left( x_{0},x_{1},\ldots ,x_{n+1}\right) \mapsto \left( p^{B[0]}\left(
x_{0}\right) ,p^{B[1]}\left( x_{1}\right) ,\ldots ,p^{B[n-1]}\left(
x_{n-1}\right) ,p^{B[n]}\left( x_{n}\right) +p^{A}\left( x_{n+1}\right)
\right) \text{.}
\end{equation*}
\end{itemize}
\end{definition}

The notion of straight fishbone allows one to easily compute the derived
towers.

\begin{proposition}
\label{Proposition:fishbone-tower}Suppose that $(\boldsymbol{A},\boldsymbol{B%
}[k])_{k\in \omega }$ is a fishbone\emph{\ }with spine $\boldsymbol{A}$ and
ribs $\boldsymbol{B}[k]$ for $k\in \omega $. Let $\boldsymbol{C}$ be the
corresponding fishbone tower. Suppose that $(\boldsymbol{A},\boldsymbol{B}%
[k])_{k\in \omega }$ is \emph{straight} of length $\alpha $. Let $\beta
\lbrack k]$ be the plain length of $\boldsymbol{B}[k]$.

\begin{enumerate}
\item If $k\in \omega $, $\beta \lbrack k]\leq \beta <\beta \lbrack k+1]$,
then for $n\leq k$,%
\begin{equation*}
C_{\beta }^{\left( n\right) }=0\oplus \cdots \oplus 0\oplus p^{\left(
n,k\right) }(A^{\left( k\right) })
\end{equation*}%
and for $n>k$,%
\begin{equation*}
C_{\beta }^{\left( n\right) }=B_{\beta }^{\left( n\right) }[0]\oplus \cdots
\oplus B_{\beta }^{\left( 1\right) }[n-1]\oplus A^{\left( n\right) }\text{;}
\end{equation*}

\item $\boldsymbol{C}$ has plain length $\alpha $.
\end{enumerate}
\end{proposition}

\begin{proof}
Recall that given a tower $\boldsymbol{D}=\left( D^{\left( n\right) }\right) 
$ we have defined $D_{0}^{\left( n\right) }=D^{\left( n\right) }$ and for
each successor ordinal $\beta $%
\begin{equation*}
D_{\beta }^{\left( n\right) }:=\bigcap_{i\in \omega }p^{\left( n,n+i\right)
}(A_{\left( \beta -1\right) _{i}}^{(n+i)})\text{.}
\end{equation*}

(1) Suppose that the conclusion holds for ordinals less then $\beta $.
Suppose that $\beta \lbrack k]\leq \beta <\beta \lbrack k+1]$. By definition
we have%
\begin{equation*}
C_{\beta }^{\left( n\right) }=\bigcap_{i\in \omega }p^{\left( n,n+i\right)
}(C_{\left( \beta -1\right) _{n+i}}^{\left( n+i\right) })\text{.}
\end{equation*}%
Suppose initially that $\beta -1>\beta \lbrack k]$. In this case for $n>k$
we have by the inductive hypothesis%
\begin{equation*}
C_{\beta }^{\left( n\right) }=B_{\beta }^{\left( n\right) }[0]\oplus
B_{\beta }^{\left( n-1\right) }[1]\oplus \cdots \oplus B_{\beta }^{\left(
1\right) }[n-1]\oplus A^{\left( n\right) }
\end{equation*}%
as above, and for $n\leq k$ by the inductive hypothesis%
\begin{equation*}
C_{\beta }^{\left( n\right) }=\bigcap_{i\in \omega }p^{\left( n,n+i\right)
}(C_{\left( \beta -1\right) _{i}}^{\left( n+i\right) })=0\oplus \cdots
\oplus 0\oplus p^{\left( n,k\right) }(A^{\left( k\right) })\text{.}
\end{equation*}%
Suppose now that $\beta -1=\beta \lbrack k]$. This means that $\beta -1$ is
a successor ordinal. Therefore we have $\left( \beta -1\right) _{i}=\beta
\lbrack k]$ for every $i\in \omega $. Therefore as before we conclude from
the inductive hypothesis that, for $n>k$,%
\begin{equation*}
C_{\beta }^{\left( n\right) }=B_{\beta }^{\left( n\right) }[0]\oplus
B_{\beta }^{\left( n-1\right) }[1]\oplus \cdots \oplus B_{\beta }^{\left(
1\right) }[n-1]\oplus A^{\left( n\right) }
\end{equation*}%
and for $n\leq k$,%
\begin{equation*}
C_{\beta }^{\left( n\right) }=0\oplus \cdots \oplus 0\oplus p^{\left(
n,k\right) }(A^{\left( k\right) })\text{.}
\end{equation*}

(2) It follows from (1) that $\boldsymbol{C}_{\alpha -1}\cong \boldsymbol{A}$%
, whence $\boldsymbol{C}$ has plain length $\alpha $.
\end{proof}

\section{Phantom morphisms\label{Section:phantom}}

In this section, we recall the notion of \emph{phantom morphism }in a
triangulated category. We also consider the natural analogue in the context
of triangulated category of the notion of \emph{higher order phantom map}
from topology.

\subsection{Phantom morphisms}

Suppose that $\mathcal{T}$ is a triangulated category with translation
functor $\Sigma $. An object $C$ of $\mathcal{T}$ is \emph{compact} \cite[%
Section 3]{bird_duality_2024} if, for every sequence $\left( X_{i}\right)
_{i\in \omega }$ of objects of $\mathcal{T}$, if $X$ is its coproduct, then 
\begin{equation*}
\mathrm{Hom}(C,X)\cong \bigoplus_{i\in \omega }\mathrm{Hom}\left(
C,X_{i}\right) \text{.}
\end{equation*}%
Let $\mathcal{C}$ be a class of compact objects in $\mathcal{T}$. As the
functor $\mathrm{Hom}\left( -,Y\right) $ is triangulated, a distinguished
triangle 
\begin{equation*}
C\rightarrow X\rightarrow X/C
\end{equation*}%
in $\mathcal{T}$ induces an exact sequence%
\begin{equation*}
\mathrm{Hom}\left( X/C,Y\right) \rightarrow \mathrm{Hom}\left( X,Y\right)
\rightarrow \mathrm{Hom}\left( C,Y\right) \text{.}
\end{equation*}%
Given objects $X,Y$ in $\mathcal{T}$, the collection $\mathrm{PhHom}\left(
X,Y\right) $ of \emph{phantom morphisms} is the intersection%
\begin{eqnarray*}
\mathrm{PhHom}\left( X,Y\right) &=&\bigcap \mathrm{\mathrm{Ker}}(\mathrm{Hom}%
\left( X,Y\right) \rightarrow \mathrm{Hom}\left( C,Y\right) ) \\
&=&\bigcap \mathrm{Ran}\left( \mathrm{Hom}\left( X/C,Y\right) \rightarrow 
\mathrm{Hom}\left( X,Y\right) \right)
\end{eqnarray*}%
ranging over the distinguished triangles $C\rightarrow X\rightarrow X/C$
with $C$ in $\mathcal{C}$; see \cite[Section 5]{christensen_ideals_1998}.
This is a notion inspired by the notion of \emph{phantom map }from topology;
see \cite{mcgibbon_phantom_1995}. One can likewise generalize the notion of 
\emph{higher order }phantom map introduced in the topological context in 
\cite{ha_higher_2003}.

\begin{definition}
We define by recursion on an ordinal $\alpha $ the set $\mathrm{Ph}^{\alpha }%
\mathrm{Hom}\left( X,Y\right) $ of phantom morphisms of order $\alpha $ from 
$X$ to $Y$ by setting%
\begin{equation*}
\mathrm{Ph}^{0}\mathrm{Hom}\left( X,Y\right) =\mathrm{PhHom}\left(
X,Y\right) \text{;}
\end{equation*}%
\begin{equation*}
\mathrm{Ph}^{\alpha +1}\mathrm{Hom}\left( X,Y\right) =\bigcap \mathrm{Ran}%
\left( \mathrm{Ph}^{\alpha }\mathrm{Hom}\left( X/C,Y\right) \rightarrow 
\mathrm{Hom}\left( X,Y\right) \right)
\end{equation*}%
where the intersection is ranging over the distinguished triangles $%
C\rightarrow X\rightarrow X/C$ with $C$ in $\mathcal{C}$;%
\begin{equation*}
\mathrm{Ph}^{\lambda }\mathrm{Hom}\left( X,Y\right) =\bigcap_{\beta <\lambda
}\mathrm{Ph}^{\beta }\mathrm{Hom}\left( X,Y\right)
\end{equation*}%
for $\lambda $ limit.
\end{definition}

Suppose that $\boldsymbol{X}=\left( X_{i},\phi _{i}:X_{i}\rightarrow
X_{i+1}\right) _{i\in \omega }$ is an inductive sequence of compact objects
in $\mathcal{T}$. An object $\mathrm{hoco\mathrm{\mathrm{lim}}}\boldsymbol{X}
$ of $\mathcal{T}$ is called a \emph{homotopy colimit} of $\boldsymbol{X}$
if there exists a distinguished triangle%
\begin{equation*}
\bigoplus \boldsymbol{X}\overset{\mathrm{id}-\phi }{\rightarrow }\bigoplus 
\boldsymbol{X}\rightarrow \mathrm{hoco\mathrm{lim}}\boldsymbol{X}
\end{equation*}%
where $\bigoplus \boldsymbol{X}$ is the coproduct of $\left( X_{i}\right)
_{i\in \omega }$ with canonical maps $\eta _{i}:X_{i}\rightarrow \bigoplus 
\boldsymbol{X}$ for $i\in \omega $, and $\phi :\bigoplus \boldsymbol{X}%
\rightarrow \bigoplus \boldsymbol{X}$ is such that 
\begin{equation*}
\phi \circ \eta _{i}=\eta _{i+1}\circ \phi _{i}
\end{equation*}
for $i\in \omega $.

If $C$ is a compact object in $\mathcal{T}$, and $X$ is the homotopy colimit
of $\boldsymbol{X}$, then by \cite[Lemma 3.4.3]{krause_homological_2022} we
have that $\mathrm{Hom}\left( C,X\right) $ is the colimit of $\left( \mathrm{%
Hom}\left( C,X_{i}\right) \right) _{i\in \omega }$. Thus, we have that%
\begin{eqnarray*}
\mathrm{PhHom}\left( X,Y\right) &=&\bigcap_{\ell \in \omega }\mathrm{\mathrm{%
Ker}}(\mathrm{Hom}\left( X,Y\right) \rightarrow \mathrm{Hom}\left( X_{\ell
},Y\right) ) \\
&=&\bigcap_{\ell \in \omega }\mathrm{Ran}\left( \mathrm{Hom}\left( X/X_{\ell
},Y\right) \rightarrow \mathrm{Hom}\left( X,Y\right) \right) \text{.}
\end{eqnarray*}%
Furthermore, by \cite[Lemma 5.2.6]{krause_homological_2022}, the exact
triangle defining the homotopy colimit induces a natural exact sequence 
\begin{equation*}
0\rightarrow \mathrm{PhHom}\left( X,Y\right) \rightarrow \mathrm{Hom}\left(
X,Y\right) \rightarrow \mathrm{lim}_{n}\mathrm{Hom}\left( X_{n},Y\right)
\rightarrow 0
\end{equation*}%
and a natural isomorphism%
\begin{equation*}
\mathrm{PhHom}\left( X,Y\right) \cong \mathrm{\mathrm{lim}}_{\ell }^{1}%
\mathrm{Hom}\left( X_{\ell },\Sigma ^{-1}Y\right) \cong \mathrm{\mathrm{lim}}%
_{\ell }^{1}\mathrm{Hom}\left( \Sigma X_{\ell },Y\right) \text{.}
\end{equation*}

\subsection{Phantom cohomological subfunctors}

Suppose that $\mathcal{T}$ is a triangulated category, $\mathcal{C}$ is the
class of compact objects of $\mathcal{T}$, $\mathcal{M}$ is a \emph{%
countably complete} abelian category, and $F:\mathcal{T}\rightarrow \mathcal{%
M}$ is a contravariant cohomological functor. We generalize the notion of
(higher order) phantom morphisms to obtain a chain of \emph{phantom
subfunctors }of $F$.

Define $\mathrm{Pro}\left( \mathcal{M}\right) $ to be the pro-category of $%
\mathcal{M}$, which is the opposite category of $\mathrm{Ind}\left( \mathcal{%
M}^{\mathrm{op}}\right) $ \cite[Section 6.1]{kashiwara_categories_2006}.
Then we have that $\mathrm{Pro}\left( \mathcal{M}\right) $ is still abelian
and contains $\mathcal{M}$ as a thick subcategory \cite[Proposition 8.6.11,
Theorem 8.6.5]{kashiwara_categories_2006}. Given an object $F$ of $X$, the
phantom subobject \textrm{Ph}$F\left( X\right) $ of $F\left( X\right) $
(which is computed in $\mathrm{Pro}\left( \mathcal{M}\right) $) is defined
to be%
\begin{equation*}
\bigcap \mathrm{\mathrm{Ker}}\left( F\left( X\right) \rightarrow F\left(
C\right) \right) =\bigcap \mathrm{Ran}\left( F\left( X/C\right) \rightarrow
F\left( X\right) \right)
\end{equation*}%
ranging over the distinguished triangles $C\rightarrow X\rightarrow X/C$
with $C$ in $\mathcal{C}$. Recursively, one defines the phantom subobject of
order $\alpha $ for every ordinal $\alpha $ by setting%
\begin{equation*}
\mathrm{Ph}^{\alpha +1}F\left( X\right) :=\bigcap \mathrm{Ran}\left( \mathrm{%
Ph}^{\alpha }F\left( X/C\right) \rightarrow F\left( X\right) \right)
\end{equation*}%
ranging over the distinguished triangles $C\rightarrow X\rightarrow X/C$,
and for a limit ordinal $\lambda $%
\begin{equation*}
\mathrm{Ph}^{\lambda }F\left( X\right) :=\bigcap_{\beta <\lambda }\mathrm{Ph}%
^{\beta }F\left( X\right) \text{.}
\end{equation*}%
This defines for each ordinal $\alpha $ a subfunctor $\mathrm{Ph}^{\alpha }F$
of $F$.

Suppose that $\boldsymbol{X}=\left( X_{i},\phi _{i}:X_{i}\rightarrow
X_{i+1}\right) $ is an inductive sequence of compact objects in $\mathcal{T}$%
, with homotopy colimit $X:=\mathrm{hoco\mathrm{\mathrm{lim}}}\boldsymbol{X}$%
. The same argument as in the case of the functor $\mathrm{Hom}\left(
-,Y\right) $ shows that%
\begin{equation*}
\mathrm{Ph}F\left( X\right) =\bigcap_{\ell \in \omega }\mathrm{\mathrm{Ker}}%
(F\left( X\right) \rightarrow F\left( X_{i}\right) )=\bigcap_{\ell \in
\omega }\mathrm{Ran}\left( F\left( X/X_{\ell }\right) \rightarrow F\left(
X\right) \right)
\end{equation*}%
Likewise, we have that for every ordinal $\alpha $:%
\begin{equation*}
\mathrm{Ph}^{\alpha +1}F\left( X\right) :=\bigcap_{\ell \in \omega }\mathrm{%
Ran}\left( \mathrm{Ph}^{\alpha }F\left( X/X_{\ell }\right) \rightarrow
F\left( X\right) \right)
\end{equation*}%
As $\mathcal{M}$ is countably complete, this guarantees that $\mathrm{Ph}%
^{\alpha }F\left( X\right) $ is (naturally isomorphic to) an object of $%
\mathcal{M}$ for every countable ordinal $\alpha $.

Furthermore, the exact triangle defining $X$ as the homotopy colimit of%
\textrm{\ }$\boldsymbol{X}$ induces a natural exact sequence 
\begin{equation*}
0\rightarrow \mathrm{Ph}F\left( X\right) \rightarrow F\left( X\right)
\rightarrow \mathrm{lim}_{n}F\left( X_{n}\right) \rightarrow 0
\end{equation*}%
and a natural isomorphism%
\begin{equation*}
\mathrm{Ph}F\left( X\right) \cong \mathrm{\mathrm{lim}}_{\ell }^{1}F\left(
\Sigma X_{\ell }\right) \text{.}
\end{equation*}

\begin{theorem}
\label{Theorem:phantom-contravariant-subfunctor}Let $R$ be a countable
domain. Suppose that $\mathcal{T}$ is a triangulated category, and 
\begin{equation*}
F:\mathcal{T}\rightarrow \mathrm{LH}\left( \mathbf{\Pi }(\mathbf{Mod}\left(
R\right) )\right)
\end{equation*}%
is a contravariant cohomological functor. Suppose that $\boldsymbol{X}%
=\left( X_{i},\phi _{i}:X_{i}\rightarrow X_{i+1}\right) $ is an inductive
sequence of compact objects in $\mathcal{T}$, with homotopy colimit $X:=%
\mathrm{hoco\mathrm{\mathrm{lim}}}\boldsymbol{X}$. If $F\left( X_{i}\right) $
is countable for every $i\in \omega $, then the isomorphism%
\begin{equation*}
\mathrm{Ph}F\left( X\right) \cong \mathrm{\mathrm{lim}}_{\ell
}^{1}{}F(\Sigma X_{\ell })
\end{equation*}%
maps $\mathrm{Ph}^{\alpha }F\left( X\right) $ to $s_{\alpha }\left( \mathrm{%
\mathrm{lim}}_{\ell }^{1}F\left( \Sigma X_{\ell }\right) \right) $ for $%
\alpha <\omega _{1}$. As a consequence, $\mathrm{Ph}^{\alpha }F\left(
X\right) =s_{\alpha }\left( F\left( X\right) \right) $ for every $\alpha
<\omega _{1}$.
\end{theorem}

\begin{proof}
Define $F^{n}:=F\circ \Sigma ^{n}$ for $n\in \mathbb{Z}$. Since $\left\{
0\right\} $ is dense in $\mathrm{lim}_{\ell }^{1}F^{1}\left( X_{\ell
}\right) $ and $\mathrm{\mathrm{Ker}}\left( F\left( X\right) \rightarrow 
\mathrm{lim}_{n}F\left( X_{n}\right) \right) $ is closed, we have that the
latter is equal to $s_{0}\left( F\left( X\right) \right) $. Without loss of
generality, we can assume $X_{0}=0$. Fix $\ell \in \omega $ and for $i\geq
\ell $ consider distinguished triangles%
\begin{equation*}
X_{\ell }\rightarrow X_{i}\rightarrow X_{i}/X_{\ell }\text{.}
\end{equation*}%
Then we have an exact sequence%
\begin{equation*}
F^{2}\left( X_{\ell }\right) \rightarrow F^{1}\left( X_{i}/X_{\ell }\right)
\rightarrow F^{1}\left( X_{i}\right) \rightarrow F^{1}(X_{\ell })\text{.}
\end{equation*}%
It follows from the Octahedral Axiom of triangulated categories \cite[%
Definition 10.1.6 (TR5)]{kashiwara_categories_2006} that $X/X_{\ell }$ is
the homotopy colimit of the inductive sequence $\left( X_{i}/X_{\ell
}\right) _{i\geq \ell }$. Thus, we have a natural isomorphism%
\begin{equation*}
F\left( X/X_{\ell }\right) \cong \mathrm{\mathrm{lim}}_{i\geq \ell
}^{1}F^{1}\left( X_{i}/X_{\ell }\right) \text{.}
\end{equation*}%
We prove by induction on $\alpha <\omega _{1}$ that under such an
isomorphism, 
\begin{equation*}
\mathrm{Ph}^{\alpha }F\left( X/X_{\ell }\right)
\end{equation*}%
corresponds to 
\begin{equation*}
s_{\alpha }\left( \mathrm{\mathrm{lim}}_{i}^{1}F^{1}\left( X_{i}/X_{\ell
}\right) \right) \text{.}
\end{equation*}%
Suppose that the conclusion holds for $\alpha $. Define $\boldsymbol{A}$ to
be the tower $\left( F^{1}\left( X_{i}\right) \right) _{i\in \omega }$.
Adopting the notation from Lemma \ref{Lemma:lim1-derived-tower}, we set 
\begin{eqnarray*}
\boldsymbol{A}[\ell ] &=&\left( \mathrm{\mathrm{Ker}}\left( F^{1}\left(
X_{i}\right) \rightarrow F^{1}\left( X_{\ell }\right) \right) \right)
_{i\geq \ell } \\
&=&\left( \mathrm{Ran}\left( F^{1}\left( X_{i}/X_{\ell }\right) \rightarrow
F^{1}\left( X_{i}\right) \right) \right) _{i\geq \ell }\text{.}
\end{eqnarray*}%
Then we have by Lemma \ref{Lemma:lim1-derived-tower},%
\begin{eqnarray*}
s_{\alpha +1}\left( \mathrm{\mathrm{lim}}_{i}^{1}\boldsymbol{A}\right)
&=&\bigcap_{\ell \in \omega }\mathrm{Ran}\left( s_{\alpha }(\mathrm{lim}^{1}%
\boldsymbol{A}[\ell ])\rightarrow \mathrm{lim}^{1}\boldsymbol{A}\right) \\
&\cong &\bigcap_{\ell \in \omega }\mathrm{Ran}\left( s_{\alpha }(\mathrm{lim}%
^{1}\frac{F^{1}(X_{i}/X_{\ell })}{F^{2}\left( X_{\ell }\right) })\rightarrow 
\mathrm{\mathrm{lim}}^{1}\boldsymbol{A}\right) \\
&\cong &\bigcap_{\ell \in \omega }\mathrm{Ran}\left( s_{\alpha }(\mathrm{lim}%
^{1}F^{1}\left( X_{i}/X_{\ell }\right) )\rightarrow \mathrm{\mathrm{lim}}^{1}%
\boldsymbol{A}\right) \\
&\cong &\bigcap_{\ell \in \omega }\mathrm{Ran}\left( \mathrm{Ph}^{\alpha
}F\left( X/X_{\ell }\right) \rightarrow F\left( X\right) \right) \\
&=&\mathrm{Ph}^{\alpha +1}F\left( X\right) \text{.}
\end{eqnarray*}%
The case of limit stages is trivial.
\end{proof}

\subsection{Phantom extensions in derived categories}

Suppose that $\mathcal{M}$ is a category of modules, and $\mathcal{A}$ is an
idempotent-complete hereditary exact $\mathcal{M}$-category with enough
projectives; see Section \ref{Subsection:M-categories}.\ Then we have that $%
\mathrm{D}^{b}\left( \mathcal{A}\right) $ is a triangulated $\mathrm{LH}%
\left( \mathcal{M}\right) $-category.\ Fix a bounded complex $X$ over $%
\mathcal{A}$. As $\mathrm{Hom}_{\mathrm{D}^{b}\left( \mathcal{A}\right)
}\left( -,X\right) :\mathrm{D}^{b}\left( \mathcal{A}\right) \rightarrow 
\mathrm{LH}\left( \mathcal{M}\right) $ is a cohomological functor, we can
consider for every ordinal $\alpha $ its $\alpha $-th phantom subfunctor 
\begin{equation*}
\mathrm{Ph}^{\alpha }\mathrm{Hom}_{\mathrm{D}^{b}\left( \mathcal{A}\right)
}\left( -,X\right) :\mathrm{D}^{b}\left( \mathcal{A}\right) \rightarrow 
\mathrm{LH}\left( \mathcal{M}\right) \text{.}
\end{equation*}%
For objects $A,B$ of $\mathcal{A}$, recall that one sets 
\begin{equation*}
\mathrm{Ext}\left( A,B\right) :=\mathrm{Hom}_{\mathrm{D}^{b}\left( \mathcal{A%
}\right) }\left( A,B[1]\right) .
\end{equation*}%
We define its $\alpha $-th phantom subobject%
\begin{equation*}
\mathrm{Ph}^{\alpha }\mathrm{Ext}\left( A,B\right) :=\mathrm{Ph}^{\alpha }%
\mathrm{Hom}_{\mathrm{D}^{b}\left( \mathcal{A}\right) }\left( A,B[1]\right) 
\text{.}
\end{equation*}%
As a particular instance of \cite[Lemma 5.2.6]{krause_homological_2022} and
Theorem \ref{Theorem:phantom-contravariant-subfunctor} we have the following:

\begin{theorem}
\label{Theorem:phantom-Ext}Suppose that $\mathcal{A}$ is an
idempotent-complete hereditary exact $\boldsymbol{\Pi }(\mathbf{Mod}\left(
R\right) )$-category with enough projectives. Suppose that $\boldsymbol{X}%
=\left( X_{i},\phi _{i}:X_{i}\rightarrow X_{i+1}\right) $ is an inductive
sequence of compact objects in $\mathcal{A}$ with colimit $X:=\mathrm{co%
\mathrm{\mathrm{lim}}}\boldsymbol{X}$, and $Y$ is an object of $A$. If $%
\mathrm{Hom}\left( X_{i},Y\right) $ is countable for every $i\in \omega $,
then we have an exact sequence%
\begin{equation*}
0\rightarrow \mathrm{PhExt}\left( X,Y\right) \rightarrow \mathrm{Ext}\left(
X,Y\right) \rightarrow \mathrm{lim}_{n}\mathrm{Ext}\left( X_{n},Y\right)
\rightarrow 0
\end{equation*}%
and an isomorphism%
\begin{equation*}
\mathrm{PhExt}\left( X,Y\right) \cong \mathrm{\mathrm{lim}}_{\ell }^{1}%
\mathrm{Hom}\left( X_{\ell },Y\right)
\end{equation*}%
in $\mathrm{LH}\left( \boldsymbol{\Pi }\left( \mathbf{Mod}\left( R\right)
\right) \right) $. Under such an isomorphism, $\mathrm{Ph}^{\alpha }\mathrm{%
Ext}\left( X,Y\right) $ corresponds to $s_{\alpha }\left( \mathrm{Ext}\left(
X,Y\right) \right) $ for every $\alpha <\omega _{1}$.
\end{theorem}

\section{The functors $\mathrm{Ext}$ and $\mathrm{PExt}$\label{Section:Ext}}

In this section we assume that $R$ is a countable Pr\"{u}fer domain, and
consider all modules to be $R$-modules.

\subsection{Extensions of countable modules}

We let $\mathbf{Mod}_{\aleph _{0}}\left( R\right) $ be the category of
countably-presented modules. For countably-presented modules $C,A$, we
regard $\mathrm{Hom}\left( C,A\right) $ as a pro-countably-presented Polish
module endowed with the topology of pointwise convergence, where $A$ and $C$
have the discrete topology. This turns $\mathbf{Mod}_{\aleph _{0}}\left(
R\right) $ into an abelian $\boldsymbol{\Pi }(\mathbf{Mod}_{\aleph
_{0}}\left( R\right) )$-category with enough projectives. (See Section \ref%
{Subsection:M-categories} for the notion of abelian $\mathcal{M}$-category,
where $\mathcal{M}$ is a category of modules.) By Proposition \ref%
{Proposition:derived-functor2} we have that the triangulated functor 
\begin{equation*}
\mathrm{Hom}^{\bullet }:\mathrm{K}^{-}(\mathbf{Mod}_{\aleph _{0}}\left(
R\right) \mathcal{)}^{\mathrm{op}}\times \mathrm{K}^{+}(\mathbf{Mod}_{\aleph
_{0}}\left( R\right) \mathcal{)}\rightarrow \boldsymbol{\Pi }(\mathbf{Mod}%
_{\aleph _{0}}\left( R\right) )
\end{equation*}%
has a total right derived functor 
\begin{equation*}
\mathrm{RHom}^{\bullet }:\mathrm{D}^{-}(\mathbf{Mod}_{\aleph _{0}}\left(
R\right) \mathcal{)}^{\mathrm{op}}\times \mathrm{D}^{+}(\mathbf{Mod}_{\aleph
_{0}}\left( R\right) \mathcal{)}\rightarrow \mathrm{D}^{+}(\boldsymbol{\Pi }(%
\mathbf{Mod}_{\aleph _{0}}\left( R\right) )\mathcal{)}\text{.}
\end{equation*}%
We let $\mathrm{Ext}^{n}=\mathrm{H}^{n}\circ \mathrm{RHom}^{\bullet }$ for $%
n\geq 0$.

Suppose that $\mathcal{C}$ is the class of \emph{finitely-presented}
modules. We let $\mathcal{E}_{0}$ be the exact structure \emph{projectively
generated }by $\mathcal{C}$ on the abelian category $\mathbf{Mod}_{\aleph
_{0}}\left( R\right) $. The sequences in $\mathcal{E}_{0}$ are precisely the 
\emph{pure short-exact} sequences of countable modules; see \cite[Theorem
3.69]{rotman_introduction_2009}. An admissible arrow in this exact category
is called \emph{pure-admissible}. A countable module is projective\emph{\ }%
in $\left( \mathbf{Mod}_{\aleph _{0}}\left( R\right) ,\mathcal{E}_{0}\right) 
$, in which case it is called \emph{pure-projective}, if and only if it is a
direct summand of a countable direct sum of finitely-presented modules \cite[%
Theorem 2.2.3]{benson_phantom_1999}. The exact category $\left( \mathbf{Mod}%
_{\aleph _{0}}\left( R\right) ,\mathcal{E}_{0}\right) $ is hereditary with
enough projectives \cite[Section 2.6]{benson_phantom_1999}. We let $\mathrm{%
PExt}^{n}$ for $n\geq 0$ be the corresponding derived functor of $\mathrm{Hom%
}$. We have that $\mathrm{PExt}^{n}=0$ for $n\geq 2$. We write $\mathrm{PExt}
$ for $\mathrm{PExt}^{1}$. Furthermore, if $C$ is a countable flat module
and $A$ is any countable module, then $\mathrm{Ext}^{n}\left( C,A\right) =%
\mathrm{PExt}^{n}\left( C,A\right) $ for every $n\geq 0$ \cite[Proposition
2.7.1]{benson_phantom_1999}. In this case, we write $\mathrm{Ext}\left(
C,-\right) $ for $\mathrm{Ext}^{1}\left( C,-\right) $.

\begin{lemma}
\label{Lemma:countable-Ext}Let $C$ be a finitely-presented module and $A$ be
a countable flat module. Then $\mathrm{Ext}^{1}\left( C,A\right) $ is
countably presented.
\end{lemma}

\begin{proof}
By Lemma \ref{Lemma:structure-finitely-presented}, without loss of
generality we can assume that $C=L/I$ is a cyclic torsion module, for some
finite ideals $I\subseteq L\subseteq R$. Considering the short exact sequence%
\begin{equation*}
0\rightarrow I\rightarrow L\rightarrow C\rightarrow 0
\end{equation*}%
we have an induced exact sequence%
\begin{equation*}
\mathrm{Hom}\left( L,A\right) \rightarrow \mathrm{Hom}\left( I,A\right)
\rightarrow \mathrm{Ext}\left( C,A\right) \rightarrow \mathrm{Ext}\left(
L,A\right) =0\text{.}
\end{equation*}%
Thus, $\mathrm{Ext}\left( C,A\right) $ is countably presented.
\end{proof}

We let $\mathbf{Flat}\left( R\right) $ be the full subcategory of $\mathbf{%
Mod}\left( R\right) $ spanned by the\emph{\ }countable \emph{flat} modules.
A short exact sequence in $\mathbf{Flat}\left( R\right) $ is necessarily
pure. By Lemma \ref{Lemma:projective-flat}, the projective objects in $%
\mathbf{Flat}\left( R\right) $ are precisely the projective countable
modules. As any countable flat module is countably presented, $\mathbf{Flat}%
\left( R\right) $ is a \emph{hereditary} quasi-abelian category with enough
projectives; see \cite[Corollary 2.7.3]{benson_phantom_1999}.

\begin{lemma}
\label{Lemma:projective-flat}Let $M$ be a countable torsion-free module.
Then $M$ is pure-projective if and only if it is projective.
\end{lemma}

\begin{proof}
Since $R$ is a Pr\"{u}fer domain, $M$ is the colimit of a sequence of finite
projective modules. Whence there is a pure short-exact sequence%
\begin{equation*}
0\rightarrow P\rightarrow Q\rightarrow M\rightarrow 0
\end{equation*}%
where $P$ and $Q$ are projective. If $M$ is pure-projective, then $M$ is
isomorphic to a direct summand of $Q$, whence it is projective.
\end{proof}

By definition, for countably-presented modules $C,A$, $\mathrm{Ph}^{0}%
\mathrm{Ext}^{1}\left( C,A\right) $ is the submodule $\mathrm{PExt}\left(
C,A\right) $ of \textrm{Ext}$^{1}\left( C,A\right) $ parametrizing $\emph{%
pure}$ extensions, which can be identified with the derived functor $\mathrm{%
Ext}_{\mathcal{E}_{0}}$ of $\mathrm{Hom}$ on $\left( \mathbf{Mod}\left(
R\right) ,\mathcal{E}_{0}\right) $. When both $C$ and $A$ are \emph{flat},
one has that $\mathrm{PExt}^{1}\left( C,A\right) =\mathrm{Ext}^{1}\left(
C,A\right) $; see \cite[Proposition 2.7.1]{benson_phantom_1999}. If $A$ is
flat, then it follows from \cite[Lemma 5.2.6]{krause_homological_2022} and
Lemma \ref{Lemma:countable-Ext} that $\mathrm{PExt}\left( C,A\right) $ is
the closure of the trivial submodule in $\mathrm{Ext}^{1}\left( C,A\right) $%
. Furthermore, if $C\cong \mathrm{co\mathrm{lim}}_{n}C_{n}$, where $\left(
C_{n}\right) $ is an inductive sequence of finitely-presented modules, then%
\begin{equation*}
\mathrm{PExt}\left( C,A\right) \cong \mathrm{lim}_{n}^{1}\mathrm{Hom}\left(
C_{n},A\right) \text{;}
\end{equation*}%
see also \cite[Theorem 2.6.1]{benson_phantom_1999}. If $C$ is a
countably-presented torsion module, and $A$ is a countable flat module, then 
$\mathrm{Hom}\left( C_{n},A\right) =0$ for every $n\in \omega $, and hence $%
\mathrm{Ext}^{1}\left( C,A\right) \cong \mathrm{lim}_{n}\mathrm{Ext}%
^{1}\left( C_{n},A\right) $ is a pro-countably-presented Polish module.

\begin{lemma}
\label{Lemma:zero-closure}Suppose that $C,A$ are countably-presented
modules, with $A$ flat. Then the quotient map $C\rightarrow C/C_{\mathrm{t}}$
induces an injective homomorphism $\mathrm{Ext}(C/C_{\mathrm{t}%
},A)\rightarrow \mathrm{Ext}\left( C,A\right) $ whose image is equal to $%
\mathrm{PExt}\left( C,A\right) $.
\end{lemma}

\begin{proof}
Notice that $\mathrm{Hom}\left( C_{\mathrm{t}},A\right) =0$. Furthermore, $%
\mathrm{Ext}\left( C_{\mathrm{t}},A\right) $ is a pro-countably-presented
Polish module. Thus, the homomorphism $\mathrm{Ext}(C/C_{\mathrm{t}%
},A)\rightarrow \mathrm{Ext}\left( C,A\right) $ is injective with closed
image. Since $\left\{ 0\right\} $ is dense in $\mathrm{Ext}\left( C/C_{%
\mathrm{t}},A\right) $, being $C/C_{\mathrm{t}}$ flat, the conclusion
follows.
\end{proof}

\begin{corollary}
\label{Corollary:Ext-sum} Suppose that $C,A$ are countably-presented
modules, with $A$ flat. Then the quotient map $C\rightarrow C/C_{\mathrm{t}}$
induces an isomorphism $\mathrm{Ph}^{\alpha }\mathrm{Ext}\left( C/C_{\mathrm{%
t}},A\right) \rightarrow \mathrm{Ph}^{\alpha }\mathrm{Ext}\left( C,A\right) $
for every $\alpha \in \omega _{1}$. Furthermore, $\left\{ 0\right\} $ has
the same complexity class in $\mathrm{Ext}^{1}\left( C,A\right) $ and $%
\mathrm{Ext}^{1}\left( C,A^{\left( \omega \right) }\right) $.
\end{corollary}

\begin{proof}
The first assertion is an immediate consequence of Lemma \ref%
{Lemma:zero-closure}.

Let $\left( C_{n}\right) $ be an increasing sequence of finite submodules of 
$C$ with union equal to $C$. Then we have that $\mathrm{PExt}\left(
C,A\right) \cong \mathrm{lim}_{n}^{1}\mathrm{Hom}\left( C_{n},A\right) $ is
the closure of $\left\{ 0\right\} $ in $\mathrm{Ext}\left( C,A\right) $.
Thus, the complexity class of $\left\{ 0\right\} $ in $\mathrm{Ext}%
^{1}\left( C,A\right) $ is the same as the complexity class of $\left\{
0\right\} $ in $\mathrm{lim}_{n}^{1}\mathrm{Hom}\left( C_{n},A\right) $. The
same holds replacing $A$ with $A^{\left( \omega \right) }$. Thus, the
conclusion follows from Corollary \ref{Corollary:sum-tower}.
\end{proof}

\begin{corollary}
\label{Corollary:Ext-projective}Suppose that $C$ is a countably-presented
module. If $\mathrm{Ext}\left( C,R\right) $ has (plain) Solecki length at
most $\alpha $, then the same holds for $\mathrm{Ext}\left( C,P\right) $ for
any countable projective nontrivial module.
\end{corollary}

\begin{proof}
By Corollary \ref{Corollary:Ext-sum}, if $\mathrm{Ext}\left( C,R\right) $
has (plain) Solecki length at most $\alpha $, then the same holds for $%
\mathrm{Ext}\left( C,R^{\left( \omega \right) }\right) $. As $P$ is a direct
summand of $R^{\left( \omega \right) }$, the same holds for $\mathrm{Ext}%
\left( C,P\right) $.
\end{proof}

The module with a Polish cover $\mathrm{Ext}^{1}\left( C,A\right) $ is
naturally isomorphic to the module with a Polish cover $\mathrm{Ext}_{%
\mathrm{Yon}}\left( C,A\right) $ of continuous cocycles modulo coboundaries
as defined in Section \ref{Section:cocycles}. The isomorphism is defined by
mapping a given extension to the corresponding cocycle; see \cite[Section
II.2]{gelfand_methods_2003}.

Suppose that $C,A$ are countable flat modules. Let $\left( C_{n}\right)
_{n\in \omega }$ be an increasing sequence of submodules of $C$. For every $%
n\in \omega $, the inclusion map $\iota _{n}:C_{n}\rightarrow C$ induces a
surjective homomorphism $\mathrm{Ext}\left( C,A\right) \rightarrow \mathrm{%
Ext}\left( C_{n},A\right) $. Together, these induce a homomorphism $\zeta :%
\mathrm{Ext}\left( C,A\right) \rightarrow \mathrm{lim}_{n\in \omega }\mathrm{%
Ext}\left( C_{n},A\right) $. Roos' Theorem expresses the kernel of such a
homomorphism as the $\mathrm{lim}^{1}$ of the tower of pro-countable modules 
$\left( \mathrm{Hom}\left( C_{n},A\right) \right) _{n\in \omega }$; see \cite%
[Theorem 6.2]{schochet_pext_2003}.

\begin{theorem}[Roos]
\label{Lemma:Ext-and-lim1}Adopt the notation above. Then we have an exact
sequence 
\begin{equation*}
0\rightarrow \mathrm{lim}_{n}^{1}\mathrm{\mathrm{Hom}}\left( C_{n},A\right)
\rightarrow \mathrm{Ext}\left( C,A\right) \rightarrow \mathrm{lim}_{n}%
\mathrm{Ext}\left( C_{n},A\right) \rightarrow 0
\end{equation*}%
in $\mathrm{LH}\left( \mathbf{\Pi }\left( \mathbf{Mod}\left( R\right)
\right) \right) $.
\end{theorem}

\subsection{Pure extensions}

In this section we assume that $R$ is a PID. For an element $x$ of a module $%
M$ we let its order $\mathfrak{o}\left( x\right) $ to be the kernel of $%
R\rightarrow M$, $\lambda \mapsto \lambda x$. If $N$ is a submodule of $M$
then we define $\mathfrak{o}\left( x\right) $ $\mathrm{\mathrm{mod}}N$ to be
the order of $x+N$ in $M/N$. The notion of pure extension can be phrased in
terms of liftings of elements of given order. Thus, an extension%
\begin{equation*}
A\rightarrow X\overset{\pi }{\rightarrow }C
\end{equation*}%
is \emph{pure} if for every $c\in C$ there exist a \emph{lift} $\hat{c}\in X$
of $c$ (so that $\pi (\hat{c})=c$) such that $\mathfrak{o}(\hat{c})=%
\mathfrak{o}(c)$. The natural \textquotedblleft order one\textquotedblright\
version of this consists in requiring that for $c_{0}\in C$, there exists a
lift $\hat{c}_{0}\in X$ of $c_{0}$ such that $\mathfrak{o}(\hat{c}_{0})=%
\mathfrak{o}(c_{0})$ and for every $c_{1}\in C$ there exists a lift $\hat{c}%
_{1}\in X$ of $c_{1}$ such that $\mathfrak{o}\left( \hat{c}_{1}\right) \ 
\mathrm{mod}\langle \hat{c}_{0}\rangle =\mathfrak{o}\left( c_{1}\right) \ 
\mathrm{mod}\langle c_{0}\rangle $. This motivates the following:

\begin{definition}
Let $R$ be a countable PID. Consider an extension 
\begin{equation*}
\mathfrak{S}:A\rightarrow X\rightarrow C\text{.}
\end{equation*}%
Say that a triple 
\begin{equation*}
(c,C_{0},\hat{C}_{0})
\end{equation*}%
with $c\in C$, $C_{0}$ submodule of $C$, and $\hat{C}_{0}$ submodule of $X$
with $\pi \hat{C}_{0}=C_{0}$ has:

\begin{itemize}
\item a sharp lift of order $0$ if there exists $\hat{c}\in X$ with $\pi 
\hat{c}=c$ and $\mathfrak{o}\left( \hat{c}\right) \ \mathrm{mod}\hat{C}_{0}=%
\mathfrak{o}\left( c\right) \ \mathrm{mod}C_{0}$, and

\item a sharp lift of order $\alpha $ if for every $\beta <\alpha $ there
exists $\hat{c}\in X$ with $\pi \hat{c}=c$ and $\mathfrak{o}\left( \hat{c}%
\right) \ \mathrm{mod}\hat{C}_{0}=\mathfrak{o}\left( x\right) \ \mathrm{mod}%
C_{0}$ and for every $d\in C$, 
\begin{equation*}
(d,\langle C_{0},c\rangle ,\langle \hat{C}_{0},\hat{c}\rangle )
\end{equation*}
has a sharp lift of order $\beta $.
\end{itemize}

The extension $\mathfrak{S}$ is \emph{pure} of order $\alpha $ if and only
if for every $c\in C$ and $r\in R$, the triple $\left( c,0,0\right) $ has a
sharp lift of order $\alpha $.
\end{definition}

\begin{proposition}
Let $R$ be a countable PID. Consider an extension%
\begin{equation*}
\mathfrak{S}:A\rightarrow X\rightarrow C
\end{equation*}%
of $C$ by $A$ as above. Then the following assertions are equivalent:

\begin{enumerate}
\item $\mathfrak{S}$ is phantom of order $\alpha $;

\item $\mathfrak{S}$ is pure of order $\alpha $.
\end{enumerate}
\end{proposition}

\begin{proof}
By recursion of $\alpha $. For $\alpha =0$ this is remarked above. Suppose
that the conclusion holds for $\beta <\alpha $.

(1)$\Rightarrow $(2) Assume that $\mathfrak{S}$ is pure of order $\alpha $.
We claim that it is phantom of order $\alpha $. Fix $\beta <\alpha $. We
need to prove that $\mathfrak{S}$ represents an element of%
\begin{equation*}
\mathrm{Ran}\left( \mathrm{Ph}^{\beta }\mathrm{Ext}\left( C/\langle c\rangle
,A\right) \rightarrow \mathrm{Ext}\left( C,A\right) \right) \text{.}
\end{equation*}%
By hypothesis, the triple $\left( c,0,0\right) $ has a sharp lift of order $%
\alpha $. Thus, there exists $\hat{c}\in X$ with $\pi \hat{c}=c$ and $%
\mathfrak{o}\left( \hat{c}\right) =\mathfrak{o}\left( c\right) $ such that
for every $d\in C$, $\left( d,\langle c\rangle ,\langle \hat{c}\rangle
\right) $ has a sharp lift of order $\beta $. Consider the homomorphism $\pi
_{1}:X/\langle \hat{c}\rangle \rightarrow C/\langle c\rangle $ induced by $%
\pi :X\rightarrow C$. As $A\cap \langle \hat{c}\rangle =0$ we have a short
exact sequence%
\begin{equation*}
A\rightarrow X/\langle \hat{c}\rangle \overset{\pi _{1}}{\rightarrow }%
C/\langle c\rangle \text{.}
\end{equation*}%
Such an extension is pure of order $\beta $. Therefore, by the inductive
hypothesis it is phantom of order at $\beta $. As the extension of $C$ by $A$
that it induces via the quotient map $C\rightarrow C/\langle c\rangle $ is
isomorphic to $\mathfrak{S}$, this concludes the proof.

(2)$\Rightarrow $(1) Assume that $\mathfrak{S}$ is phantom of order $\alpha $%
. We claim that it is pure of order $\alpha $. Fix $\beta <\alpha $. Then $%
\mathfrak{S}$ represents an element of%
\begin{equation*}
\mathrm{Ran}\left( \mathrm{Ph}^{\beta }\mathrm{Ext}\left( C/\langle c\rangle
,A\right) \rightarrow \mathrm{Ext}\left( C,A\right) \right) \text{.}
\end{equation*}%
Thus, there exists an extension%
\begin{equation*}
A\rightarrow Y\rightarrow C/\langle c\rangle
\end{equation*}%
that is phantom of order $\beta $ that yields via pullback along the
quotient map $C\rightarrow C/\langle c\rangle $ an extension isomorphic to $%
\mathfrak{S}$. Consider the pullback diagram%
\begin{equation*}
\begin{array}{ccc}
X & \overset{\pi }{\rightarrow } & C \\ 
\rho \downarrow &  & \downarrow \\ 
Y & \rightarrow & C/\langle c\rangle%
\end{array}%
\end{equation*}%
Let $\hat{c}$ be the unique element of $X$ such that $\rho \hat{c}=0$ and $%
\pi \hat{c}=c$. Then we have that $\mathfrak{o}\left( \hat{c}\right) =%
\mathfrak{o}\left( c\right) $. By the inductive hypothesis, the extension%
\begin{equation*}
A\rightarrow Y\rightarrow C/\langle c\rangle
\end{equation*}%
is pure of order $\beta $. Thus, for every $d\in C$, the triple $\left(
d+\langle c\rangle ,0,0\right) $ has a sharp lift of order $\beta $. This
implies that the triple $\left( d,\langle c\rangle ,\langle \hat{c}\rangle
\right) $ has a sharp lift of order $\beta $. This shows that $\left(
c,0,0\right) $ has a sharp lift of order $\alpha $, concluding the proof.
\end{proof}

\subsection{Thin modules}

Let us say that a countable module $M$ is \emph{thin }if its torsion
submodule $M_{\mathrm{t}}$ is finitely-presented. We denote by $\mathbf{Thin}%
\left( R\right) $ the full subcategory of $\mathbf{Mod}\left( R\right) $
consisting of thin modules. Every countable flat module is, in particular,
thin.

\begin{proposition}
\label{Proposition:resolution-thin}Suppose that $X$ is a countable thin
module. Then there exists a short exact sequence%
\begin{equation*}
0\rightarrow A\rightarrow B\rightarrow X\rightarrow 0
\end{equation*}%
where $A$ and $B$ are countable projective modules.
\end{proposition}

\begin{proof}
Consider an extension%
\begin{equation*}
0\rightarrow T\rightarrow X\rightarrow M\rightarrow 0
\end{equation*}%
where $M$ is flat and $T$ is finitely-presented and torsion. Then since $M$
is flat it has a projective resolution%
\begin{equation*}
0\rightarrow Q\rightarrow P\rightarrow M\rightarrow 0\text{.}
\end{equation*}%
Considering a pullback we obtain a commuting diagram%
\begin{equation*}
\begin{array}{ccccc}
T & \rightarrow & Y & \rightarrow & P \\ 
\downarrow &  & \downarrow &  & \downarrow \\ 
T & \rightarrow & X & \rightarrow & M%
\end{array}%
\end{equation*}%
Since $P$ is projective, $Y\cong T\oplus P$. Since $T$ is
finitely-presented, it has a projective resolution%
\begin{equation*}
0\rightarrow S\rightarrow F\rightarrow T\rightarrow 0\text{.}
\end{equation*}%
Thus,%
\begin{equation*}
0\rightarrow S\oplus Q\rightarrow F\oplus P\rightarrow X\rightarrow 0
\end{equation*}%
is as desired.
\end{proof}

\begin{lemma}
\label{Lemma:split-torsion}Let $M$ be a countable flat module and $T$ be a
finitely-presented torsion module. Then $\mathrm{Ext}\left( M,T\right) =0$.
\end{lemma}

\begin{proof}
As in Section \ref{Subsection:finiteflat}, for a finite flat module $M$ we
let $M^{\ast }$ be the finite flat module $\mathrm{Hom}\left( M,R\right) $.
For a finite flat module $N$ and a countable flat module $X$, we have
natural isomorphisms%
\begin{equation*}
\mathrm{Hom}\left( X,N\right) \cong \mathrm{Hom}\left( X,\mathrm{Hom}\left(
N^{\ast },R\right) \right) \cong \mathrm{Hom}\left( X\otimes N^{\ast
},R\right)
\end{equation*}%
and hence%
\begin{equation*}
\mathrm{Ext}\left( X,N\right) \cong \mathrm{Ext}\left( X\otimes N^{\ast
},R\right) \text{.}
\end{equation*}%
Since $T$ is a finite direct sum of cyclic finitely-presented torsion
modules, we can assume without loss of generality that $T$ is cyclic. Thus,
we have that $T=L/I$ for some finite ideals $I\subseteq L\subseteq R$.
Notice that the inclusion $I\rightarrow L$ induces a morphism $L^{\ast
}\rightarrow I^{\ast }$ that is injective. Indeed, if $\varphi :L\rightarrow
R$ is a homomorphism such that $\varphi |_{I}=0$, if $a\in I$ then for every 
$x\in L$ we have $a\varphi \left( x\right) =\varphi \left( ax\right) =0$ and
hence $\varphi \left( x\right) =0$. Notice also that $\mathrm{Hom}\left(
I^{\ast }/L^{\ast },R\right) =0$. Indeed, if $\varphi :I^{\ast }\rightarrow
R $ is a homomorphism such that $\varphi |_{L^{\ast }}=0$ then considering $%
a\in I$ such that $\varphi \left( \xi \right) =\xi \left( a\right) $ for $%
\xi \in I^{\ast }$, and considering $\xi \in L^{\ast }$ to be the inclusion $%
L\rightarrow R$, we have that $\xi |_{I}$ is the inclusion $I\rightarrow R$,
and $0=\varphi \left( \xi |_{I}\right) =\xi \left( a\right) =a$. This shows
that $\varphi =0$.

We have an exact sequence%
\begin{equation*}
\mathrm{Ext}\left( M,I\right) \rightarrow \mathrm{Ext}\left( M,L\right)
\rightarrow \mathrm{Ext}\left( M,T\right) \rightarrow 0
\end{equation*}%
which is naturally isomorphic to%
\begin{equation*}
\mathrm{Ext}\left( M\otimes I^{\ast },R\right) \rightarrow \mathrm{Ext}%
\left( M\otimes L^{\ast },R\right) \rightarrow \mathrm{Ext}\left( M,T\right)
\rightarrow 0\text{.}
\end{equation*}%
Since the morphism $L^{\ast }\rightarrow I^{\ast }$ is injective and $M$ is
flat, the morphism $M\otimes L^{\ast }\rightarrow M\otimes I^{\ast }$ is
injective. Hence, it induces a surjective homomorphism%
\begin{equation*}
\mathrm{Ext}\left( M\otimes I^{\ast },R\right) \rightarrow \mathrm{Ext}%
\left( M\otimes L^{\ast },R\right) \text{.}
\end{equation*}%
This shows that $\mathrm{Ext}\left( M,T\right) =0$.
\end{proof}

\begin{corollary}
\label{Corollary:thin}If $X$ is a countable thin module, then $X=T\oplus M$
where $T$ is a finitely-presented torsion module and $M$ is a countable flat
module.
\end{corollary}

\begin{lemma}
\label{Lemma:thin-fully-exact}Suppose that%
\begin{equation*}
0\rightarrow A\rightarrow B\rightarrow C\rightarrow 0
\end{equation*}%
is a short exact sequence, where $A$ and $C$ are countable thin modules.
Then $B$ is a countable thin module.
\end{lemma}

\begin{proof}
We need to prove that $B_{\mathrm{t}}$ is finitely-presented. Since $B_{%
\mathrm{t}}$ maps to $C_{\mathrm{t}}$, after replacing $C$ with $C_{\mathrm{t%
}}$ and $B$ with the preimage of $C_{\mathrm{t}}$, we can assume that $C$ is
a finitely-presented torsion module. By the previous corollary, we can write 
$A=N\oplus S$. Considering the canonical isomorphism%
\begin{equation*}
\mathrm{Ext}^{1}\left( C,A\right) \cong \mathrm{Ext}^{1}\left( C,N\right)
\oplus \mathrm{Ext}^{1}\left( C,S\right) \cong \mathrm{Ext}^{1}\left(
C,S\right)
\end{equation*}%
we have that%
\begin{equation*}
B\cong N\oplus T
\end{equation*}%
where $T$ is an extension $S\rightarrow T\rightarrow C$. Since $C$ and $S$
are finitely-presented torsion modules, the same holds for $T$.
\end{proof}

\begin{corollary}
The category $\mathbf{Thin}\left( R\right) $ is a fully exact subcategory of 
$\mathbf{Mod}\left( R\right) $.
\end{corollary}

It follows from the above that $\mathbf{Thin}\left( R\right) $ is a \emph{%
hereditary} quasi-abelian $\mathbf{\Pi }\left( \mathbf{Mod}\left( R\right)
\right) $-category with enough projectives. Thus, the categories $\mathrm{K}%
^{b}\left( \mathbf{Thin}\left( R\right) \right) $ and $\mathrm{D}^{b}\left( 
\mathbf{Thin}\left( R\right) \right) $ are triangulated $\mathrm{LH}\left( 
\mathbf{\Pi }\left( \mathbf{Mod}\left( R\right) \right) \right) $%
-categories. The \emph{compact objects }in $\mathrm{D}^{b}\left( \mathbf{Thin%
}\left( R\right) \right) $ are the bounded complexes of finitely-presented%
\emph{\ }modules. We let $\mathrm{Ext}^{n}$ for $n\geq 0$ be the derived
functor of $\mathrm{Hom}$ on $\mathbf{Thin}\left( R\right) $ as defined in
Section \ref{Section:derived}. Since $\mathbf{Thin}\left( R\right) $ is
hereditary, $\mathrm{Ext}^{n}=0$ for $n\geq 2$. We let $\mathrm{Ext}$ be $%
\mathrm{Ext}^{1}$. Since $\mathbf{Thin}\left( R\right) $ is a fully exact
subcategory of $\mathbf{Mod}\left( R\right) $, this definition of $\mathrm{%
Ext}$ is consistent with the one previously considered in reference to the
derived functor of $\mathrm{Hom}$ on $\mathbf{Mod}\left( R\right) $.

Let now $\mathcal{E}_{0}$ be the exact structure on $\mathbf{Thin}\left(
R\right) $ \emph{projectively generated }by the class of finitely-presented
(torsion) modules. Then the exact sequences in $\mathcal{E}_{0}$ are
precisely the pure short-exact sequences. A countable module is projective%
\emph{\ }in $\left( \mathbf{Thin}\left( R\right) ,\mathcal{E}_{0}\right) $
if and only if it is of the form $T\oplus P$ where $T$ is a
finitely-presented torsion module and $P$ is projective in $\mathbf{Mod}%
\left( R\right) $.

\subsection{Largest free summand}

For a module $C$, define $\Sigma C$ to be the intersection of the kernels of
the idempotent homomorphisms $C\rightarrow C$ whose image is a (finite)
projective module, and $\Phi C$ to be the union of the images of the
idempotent homomorphisms $C\rightarrow C$ whose image is a direct summand of 
$C$ that is a (finite) projective module. We say that $C$ is \emph{coreduced 
}if $\Phi C=0$.

\begin{lemma}
\label{Lemma:coreduced-radical}Suppose that $C$ is a countable flat module.
Then:

\begin{enumerate}
\item $\Phi C$ is a countable projective module;

\item $C=\Sigma C\oplus \Phi C$;

\item $\Sigma \left( \Sigma C\right) =\Sigma C$, and $\Sigma C$ has no
projective direct summands.
\end{enumerate}
\end{lemma}

\begin{proof}
Suppose that $f,g:C\rightarrow C$ are idempotent homomorphisms with images $%
P,Q$, respectively. Suppose that $P,Q$ are finite projective modules, and
direct summands of $C$. Then we have that $P\cap Q$ is a direct summand of $%
Q $. Write $Q=Q^{\prime }\oplus \left( P\cap Q\right) $. Thus, there exists
an idempotent homomorphism $g^{\prime }:C\rightarrow C$ whose image is $%
Q^{\prime }$. Thus, $f+g^{\prime }:C\rightarrow C$ is an idempotent
homomorphism whose image is $P\oplus Q^{\prime }=P\oplus Q$.

Let $\left( \phi _{n}\right) $ be a sequence of idempotent homomorphisms $%
C\rightarrow C$ whose image is a finite projective module, such that%
\begin{equation*}
\Phi C=\bigcup_{n}\mathrm{Ran}\left( \phi _{n}\right) \text{.}
\end{equation*}%
By the above remarks, one can obtain another sequence $\left( \psi
_{n}\right) $ such that, for every $n$,%
\begin{equation*}
\mathrm{Ran}\left( \psi _{0}+\cdots +\psi _{n}\right) =\bigoplus_{i}\mathrm{%
Ran}\left( \psi _{i}\right) =\sum_{i}\mathrm{Ran}\left( \varphi _{i}\right) 
\text{.}
\end{equation*}%
Thus, 
\begin{equation*}
\Phi _{\mathcal{C}}C=\bigoplus_{i\in \omega }\mathrm{Ran}\left( \psi
_{i}\right) \text{.}
\end{equation*}%
The rest of the assertions easily follow from this.
\end{proof}

It follows from Lemma \ref{Lemma:coreduced-radical} that $C\mapsto \Sigma C$
is a radical in the sense of \cite%
{charles_sous-groupes_1968,fay_preradicals_1982}.

\subsection{Injective and projective objects}

Suppose that $M$ is a module. Let $\mathcal{F}$ be a countable
downward-directed collection of submodules of $M$, and let $\mathcal{G}$ be
the collection of cosets of elements of $\mathcal{F}$. A family $\mathcal{G}%
_{0}\subseteq \mathcal{G}$ satisfies the \emph{finite intersection property}
(FIP) if every finite subset of $\mathcal{G}_{0}$ has nonempty intersection.
By definition, the module $M$ is $\mathcal{G}$-compact if every family $%
\mathcal{G}_{0}\subseteq \mathcal{G}$ with the FIP has nonempty
intersection. The family $\mathcal{F}$ satisfies the descending chain
condition (DCC) if every decreasing sequence of elements of $\mathcal{F}$ is
eventually constant.

\begin{lemma}
Adopt the notations above. Then $M$ is $\mathcal{G}$-compact if and only if
for every decreasing chain $\left( S_{n}\right) $ of elements of $\mathcal{F}
$, the canonical homomorphism%
\begin{equation*}
\varphi _{\left( S_{n}\right) }:M\rightarrow \mathrm{lim}_{n}M/S_{n}
\end{equation*}%
is surjective.
\end{lemma}

\begin{proof}
We have that $M$ is $\mathcal{G}$-compact if and only if $\varphi _{\left(
S_{n}\right) }$ has closed image for each decreasing chain $\left(
S_{n}\right) $. Since $\varphi _{\left( S_{n}\right) }$ has dense image, the
conclusion follows.
\end{proof}

\begin{corollary}
\label{Corollary:DCC}Adopt the notations above. If $M$ is countable and $%
\mathcal{G}$-compact, then $\mathcal{F}$ satisfies the \emph{DCC}.
\end{corollary}

\begin{proof}
If $\left( S_{n}\right) $ is a strictly decreasing sequence of elements of $%
\mathcal{F}$, then $\mathrm{lim}_{n}M/S_{n}$ is an uncountable Polish module.
\end{proof}

A submodule $N$ of $M$ is pp-definable if it is of the form%
\begin{equation*}
\left\{ x\in M:\exists y_{0},\ldots ,y_{\ell }\in M\text{, }%
rx=r_{0}y_{0}+\cdots +r_{\ell }y_{\ell }\right\}
\end{equation*}%
for some $\ell \in \omega $ and $r,r_{0},\ldots ,r_{\ell }\in R$. The module 
$M$ is \emph{algebraically compact }if it is $\mathcal{G}$-compact with
respect to the family $\mathcal{G}$ of cosets of pp-definable submodules 
\cite[Section 4]{prest_pure-injective_2009}.

The module $M$ is $r$-divisible for some $r\in R$ if $rM=M$, and it is \emph{%
divisible }if it is $r$-divisible for every nonzero $r\in R$. For $r\in R$
define $R[r^{-1}]\subseteq K$ to be the submodule generated by $r^{-n}$ for $%
n\in \mathbb{N}$. More generally, for a finite ideal $I$, we say that $M$ is 
$I$-divisible if $IM=M$, where $IM$ is the submodule of $M$ generated by $rx$
for $r\in I$ and $x\in M$. We also let $R[I^{-1}]\subseteq K$ be the union
of $I^{-n}$ for $n\in \mathbb{N}$.

\begin{lemma}
\label{Lemma:divisible-iff-Ext}Let $A$ be a countable flat module, and $%
I\subseteq M$ be a finite ideal. Then $A$ is $I$-divisible if and only if $%
\mathrm{Ext}\left( R[I^{-1}],A\right) =0$.
\end{lemma}

\begin{proof}
We have that $R[I^{-1}]$ is isomorphic to the colimit of the inductive
sequence $\left( I^{-n}\right) _{n\in \omega }$. We have%
\begin{equation*}
\mathrm{Ext}\left( R[I^{-1}],A\right) \cong \mathrm{lim}_{n}^{1}\mathrm{Hom}%
\left( I^{-n},A\right) \text{.}
\end{equation*}%
The tower $\left( \mathrm{Hom}\left( I^{-n},A\right) \right) $ is isomorphic
to the tower $\left( I^{n}A\right) _{n\in \omega }$ with inclusions as
bonding maps. Thus, we have that $\mathrm{Ext}\left( R[I^{-1}],A\right) =0$
if and only if such a tower satisfies the Mittag-Leffler condition, if and
only if the sequence $\left( I^{n}A\right) $ of submodules of $A$ is
eventually constant. Since $R$ is a Pr\"{u}fer domain, $I$ is invertible.
Thus, there exists $n\in \omega $ such that $I^{n+1}A=I^{n}A$ if and only if 
$A$ is $I$-divisible, concluding the proof.
\end{proof}

\begin{lemma}
\label{Lemma:compact-iff-divisible}Let $A$ be a countable flat module. Then $%
A$ is algebraically compact if and only if it is divisible.
\end{lemma}

\begin{proof}
If $A$ is algebraically compact and countable, then by Corollary \ref%
{Corollary:DCC} the family of pp-definable submodules satisfies the DCC. In
particular, for every nonzero $r\in A$ the sequence $\left( r^{n}A\right) $
is eventually constant, whence $A$ is $r$-divisible.

Conversely, if $A$ is divisible every pp-definable submodule is either $%
\left\{ 0\right\} $ or $A$ and the conclusion follows.
\end{proof}

\begin{proposition}
\label{Proposition:characterize-injectives}Suppose that $A$ is a countable
flat module. The following assertions are equivalent:

\begin{enumerate}
\item $A$ is injective in $\mathbf{Flat}\left( R\right) $;

\item $A$ is pure-injective in $\mathbf{Thin}\left( R\right) $;

\item $A$ is pure-injective in $\mathbf{Mod}\left( R\right) $;

\item $A$ is algebraically compact;

\item $A$ is divisible;

\item $\mathrm{Ext}\left( R[1/r],A\right) =0$ for every $r\in R\setminus
\left\{ 0\right\} $;

\item $\mathrm{Ext}\left( R[I^{-1}],A\right) =0$ for every finite ideal $%
I\subseteq R$.
\end{enumerate}
\end{proposition}

\begin{proof}
Considering that every short-exact sequence of flat modules is pure, the
implications (3)$\Rightarrow $(2)$\Rightarrow $(1)$\Rightarrow $(7)$%
\Rightarrow $(6) are obvious. The equivalence of (3) and (4) is part of \cite%
[Theorem 4.1]{prest_pure-injective_2009}. The equivalence of (4) and (5) is
Lemma \ref{Lemma:compact-iff-divisible}, and the equivalence of (5) and (6)
is Lemma \ref{Lemma:divisible-iff-Ext}.
\end{proof}

Suppose that $A$ is a countable flat module. Then $A$ has a largest
divisible submodule, denoted by $\Delta A$, which is pure in $A$. Therefore, 
$\Delta A$ is a direct summand of $A$. We say that $A$ is \emph{reduced }if $%
\Delta A=0$.

\begin{proposition}
\label{Proposition:characterize-projectives}Suppose that $C$ is a countable
flat module. The following assertions are equivalent:

\begin{enumerate}
\item $C$ is projective in $\mathbf{Flat}\left( R\right) $;

\item $C$ is projective in $\mathbf{Mod}\left( R\right) $;

\item $C$ is a direct sum of finite ideals of $R$;

\item $\mathrm{Ext}\left( C,R\right) =0$.
\end{enumerate}
\end{proposition}

\begin{proof}
(4)$\Rightarrow $(2) Suppose that $\mathrm{Ext}\left( C,R\right) =0$. Write $%
C=\mathrm{co\mathrm{lim}}_{n}C_{n}$ for some increasing sequence $\left(
C_{n}\right) $ of finite submodules of $C$. Then we have that 
\begin{equation*}
0=\mathrm{Ext}\left( C,R\right) =\mathrm{lim}_{n}^{1}\mathrm{Hom}\left(
C_{n},R\right) =\mathrm{lim}_{n}^{1}C_{n}^{\ast }
\end{equation*}
where $C_{n}^{\ast }=\mathrm{Hom}\left( C_{n},R\right) $. Notice that $%
X\mapsto X^{\ast }$ is an equivalence between the category of finite
projective modules and its opposite. As $\mathrm{lim}_{n}^{1}C_{n}^{\ast }=0$%
, the tower $\left( C_{n}^{\ast }\right) $ is isomorphic to an epimorphic
tower. Thus, after replacing $\left( C_{n}\right) $ with an isomorphic
inductive sequence, we can assume that $\left( C_{n}^{\ast }\right) $ is an
epimorphic tower. Since $C_{n}^{\ast }$ is projective, the map $%
C_{n+1}^{\ast }\rightarrow C_{n}^{\ast }$ is a split epimorphism, whence the
map $C_{n}\rightarrow C_{n+1}$ is a split monomorphism. This shows that $C$
is isomorphic to a direct sum of finite projective modules, whence
projective.

The implication (2)$\Rightarrow $(3) is part of Proposition \ref%
{Proposition:characterize-prufer}, and the other implications are trivial.
\end{proof}

The following result is a particular case of \cite[Theorem 13]%
{macias-diaz_generalization_2010}.

\begin{lemma}
\label{Lemma:Pontryagin--Hill}Let $R$ be a countable Pr\"{u}fer domain. A
module $M$ is projective if and only if there exists an increasing sequence
of projective pure submodules of $M$ with union equal to $M$.
\end{lemma}

\subsection{Duality\label{Subsection:duality}}

Let $M$ be a module. If $M$ is either countable or pro-finite, we set $%
M^{\ast }:=\mathrm{Hom}\left( M,R\right) $. Notice that if $M$ is countable
and flat, then $M^{\ast }$ is pro-finiteflat, while if $M$ is
pro-finiteflat, then $M^{\ast }$ is countable and flat. Furthermore, there
is a canonical morphism $M\rightarrow M^{\ast \ast }$.

\begin{lemma}
\label{Lemma:duality-Hom}Let $R$ be a countable Pr\"{u}fer domain.

\begin{enumerate}
\item If $M$ is either countable projective or pro-finiteflat, then the
canonical morphism $M\rightarrow M^{\ast \ast }$ is an isomorphism.

\item The functor $M\mapsto M^{\ast }$ establishes a duality between
countable projective modules and pro-finiteflat modules.
\end{enumerate}
\end{lemma}

\begin{proof}
(1) A countable projective module is a direct sum of finite flat modules,
while a pro-finiteflat module is a direct product of finite flat modules.
The conclusion easily follows from this and Lemma \ref%
{Lemma:finiteflat-duality}.

(2) This is a consequence of (1).
\end{proof}

Let $M$ be a module. If $M$ is either countable coreduced and flat, or
phantom pro-finiteflat, define $M^{\vee }:=\mathrm{Ext}\left( M,R\right) $.\
Notice that if $M$ is countable coreduced and flat, then $M^{\vee }$ is
phantom pro-finiteflat, while if $M$ is phantom pro-finiteflat, then $%
M^{\vee }$ is countable and flat. Furthermore, there is a canonical morphism 
$M\mapsto M^{\vee \vee }$, obtained by considering a projective resolution $%
S\rightarrow P\rightarrow M$ of $M$ and, adopting the notation above,
considering the identifications%
\begin{equation*}
M^{\vee }\cong \frac{S^{\ast }}{P^{\ast }}
\end{equation*}%
and%
\begin{equation*}
M^{\vee \vee }\cong \frac{S^{\ast \ast }}{P^{\ast \ast }}\text{.}
\end{equation*}

\begin{proposition}
\label{Proposition:duality-Ext}Let $R$ be a countable Pr\"{u}fer domain.

\begin{enumerate}
\item If $M$ is either countable coreduced and flat or phantom
pro-finiteflat, then the canonical morphism $M\rightarrow M^{\vee \vee }$ is
an isomorphism;

\item The functor $M\mapsto M^{\vee }$ establishes a duality between
countable coreduced flat modules and phantom pro-finiteflat modules.
\end{enumerate}
\end{proposition}

\begin{proof}
The first assertions easily follows from Lemma \ref{Lemma:duality-Hom}. The
second assertion is a consequence of the first one.
\end{proof}

The following corollary is in fact a particular instance of Proposition \ref%
{Proposition:phantom-category}(2).

\begin{corollary}
\label{Corollary:fully-faithful-lim1}Let $R$ be a countable Pr\"{u}fer
domain. The functor $\mathrm{lim}^{1}$ establishes an equivalence from the
category of reduced towers of finite flat modules to the category of phantom
pro-finiteflat modules.
\end{corollary}

\begin{proof}
Choosing for each countable flat module $M$ a cofinal increasing sequence $%
\left( M_{n}\right) $ of finite submodules defines a functor $M\mapsto
\left( \mathrm{Hom}\left( M_{n},R\right) \right) $ that establishes an
equivalence between countable coreduced flat modules and reduced towers of
finite flat modules. By Theorem \ref{Theorem:phantom-Ext}, via this
equivalence, the functor $\mathrm{lim}^{1}$ corresponds to $\mathrm{Ext}%
\left( -,R\right) $. The conclusion thus follow from Proposition \ref%
{Proposition:duality-Ext}.
\end{proof}

As a particular instance of Proposition \ref{Proposition:phantom-category}%
(1) we also have the following:

\begin{proposition}
\label{Proposition:fully-faithful}Let $R$ be a countable Pr\"{u}fer domain.
The functor $\mathrm{lim}^{1}$ from the category $\mathrm{Ph}\left( \mathbf{%
Mod}_{\aleph _{0}}\left( R\right) \right) $ of reduced towers of
countably-presented modules to the category of phantom
pro-countably-presented modules is faithful and essentially surjective.
Furthermore:

\begin{itemize}
\item the essential image of the category of reduced towers of flat modules
is the category of phantom pro-countableflat modules;

\item the essential image of the category of reduced towers of finiteflat
modules is the category of phantom pro-finiteflat modules.
\end{itemize}

For a phantom pro-countably-presented module $G/N$, let $\left( U_{n}\right) 
$ be a basis open submodules of $G$, and $\left( W_{n}\right) $ be a basis
of open submodules of $N$ with $W_{n}\subseteq U_{n}$ for every $n\in \omega 
$. Let $A^{\left( n\right) }$ be the countably-presented module%
\begin{equation*}
\frac{U_{n}\cap N}{W_{n}}\text{.}
\end{equation*}%
This defines a reduced tower $\boldsymbol{A}$ of countably-presented modules
with $\mathrm{lim}^{1}\boldsymbol{A}\cong G/N$.
\end{proposition}

\begin{proof}
We just prove the last assertion. For every $k\in \omega $, the inclusion $%
N\rightarrow G$ induces a continuous homomorphism $N/W_{k}\rightarrow
G/U_{k} $, which is surjective since $\left\{ 0\right\} $ is dense in $G/N$%
.\ For every $k\in \omega $ we have a short exact sequence%
\begin{equation*}
Q^{\left( k\right) }\rightarrow N/W_{k}\rightarrow G/U_{k}
\end{equation*}%
where 
\begin{equation*}
Q^{\left( k\right) }=\frac{\left( N\cap U_{k}\right) }{W_{k}}
\end{equation*}%
Taking the limit we obtain an exact sequence%
\begin{equation*}
N=\mathrm{\mathrm{\mathrm{\mathrm{lim}}}}_{k}\left( N/W_{k}\right)
\rightarrow \mathrm{lim}_{k}\left( G/U_{k}\right) \rightarrow \mathrm{lim}%
^{1}Q^{\left( k\right) }\rightarrow \mathrm{lim}^{1}N/W_{k}=0\text{.}
\end{equation*}%
This yields an isomorphism%
\begin{equation*}
\mathrm{lim}^{1}Q^{\left( k\right) }\cong G/N\text{.}
\end{equation*}
\end{proof}

\subsection{Duality over a PID}

In this subsection we assume that $R$ is a countable PID with field of
fractions $K$, to extend the duality from Section \ref{Subsection:duality}.

\begin{lemma}
Let $R$ be a countable PID. The functor $\mathrm{Hom}\left( -,K/R\right) $
from countable modules to pro-finite modules is faithful
\end{lemma}

\begin{proof}
Suppose that $\varphi :C\rightarrow C^{\prime }$ is a homomorphism between
countable modules that induces the trivial homomorphism $\mathrm{Hom}\left(
C^{\prime },K/R\right) \rightarrow \mathrm{Hom}\left( C,K/R\right) $. If $%
x\in C^{\prime }$ is nonzero and $\langle x\rangle $ is the submodule of $%
C^{\prime }$ generated by $x$, then $\langle x\rangle \cong R$ or $\langle
x\rangle \cong R/rR$ for some $r\in R$. Either way, there exists a
homomorphism $\langle x\rangle \rightarrow K/R$ that maps $x$ to a nonzero
element of $K/R$. As $K/R$ is injective, this extends to a nonzero
homomorphism $\alpha :C^{\prime }\rightarrow K/R$ such that $\alpha \left(
x\right) \neq 0$. By hypothesis, we have that $\alpha \circ \varphi =0$.
This implies that $x$ does not belong to the image of $\varphi $. As this
holds for every nonzero element of $C^{\prime }$, $\varphi =0$.
\end{proof}

Let us say that a Polish module $C$ is pro-finite-torsion if it is the limit
of a tower of finite torsion modules. When $C$ is either a countable torsion
module or a pro-finite-torsion module, define as before $C^{\vee }:=\mathrm{%
Ext}\left( C,R\right) $. Considering the exact sequence%
\begin{equation*}
\mathrm{Hom}\left( C,K\right) \rightarrow \mathrm{Hom}\left( C,K/R\right)
\rightarrow \mathrm{Ext}\left( C,R\right) \rightarrow 0
\end{equation*}%
shows that in this case $C^{\vee }\cong \mathrm{Hom}\left( C,K/R\right) $.

\begin{proposition}
Let $R$ be a countable PID. The assignment $C\mapsto C^{\vee }$ establishes
a duality:

\begin{enumerate}
\item between countable torsion modules and pro-finite-torsion modules;

\item between countable modules and modules with a Polish cover $G$ such
that $G/s_{0}(G)$ is pro-finite-torsion;
\end{enumerate}
\end{proposition}

\begin{proof}
(1) is easily established using the Ulm classification of countable torsion
modules by induction on the Ulm length. (2) follows from (1) and Proposition %
\ref{Proposition:duality-Ext}, considering the short exact sequence%
\begin{equation*}
C_{\mathrm{t}}\rightarrow C\rightarrow C/C_{\mathrm{t}}
\end{equation*}%
associated with a countable torsion module $C$, and the short exact sequence%
\begin{equation*}
s_{0}\left( G\right) \rightarrow G\rightarrow G/s_{0}\left( G\right)
\end{equation*}%
associated with a module with a Polish cover $G$.
\end{proof}

\subsection{Module trees for phantom modules}

Pro-countable modules can be described in terms of \textquotedblleft module
trees\textquotedblright\ on a sequence of countable modules. Let $%
\boldsymbol{M}:=\left( M_{i}\right) _{i\in \omega }$ be a sequence of
countable modules. A \emph{module tree }on $\boldsymbol{M}$ is a rooted tree 
$T$ on%
\begin{equation*}
M^{\left( \omega \right) }:=\bigoplus_{i\in \omega }M_{i}
\end{equation*}%
such that, for every $i\in \omega $, the set $T^{\left( n\right) }$ of nodes
of height $n$ in $T$ is a submodule of%
\begin{equation*}
M^{\left( n\right) }:=\bigoplus_{i\in n}M_{i}\text{.}
\end{equation*}%
Then the body $M_{T}$ of $T$ is a closed submodule of%
\begin{equation*}
\prod_{n\in \omega }M_{n}\text{.}
\end{equation*}%
Every pro-countable module can be represented in this fashion; see \cite%
{solecki_equivalence_1995,ding_non-archimedean_2017} for more on
\textquotedblleft module trees\textquotedblright\ in the case of $\mathbb{Z}$%
-modules.

Likewise, one can represent a \emph{phantom }pro-countable module by a
module tree $T$ on a sequence $\boldsymbol{M}$ of countable modules. Indeed,
such a tree defines a tower of countable modules $\left( T^{\left( n\right)
}\right) $ with bonding maps given by truncation, whence a phantom module $%
\mathfrak{M}_{T}:=\mathrm{lim}_{n}^{1}T^{\left( n\right) }$. Proposition \ref%
{Proposition:fully-faithful} shows that every phantom pro-countable module
arises in this fashion for some well-founded module tree. Furthermore, the
Solecki submodules of $\mathfrak{M}_{T}$ correspond to the derivatives of $T$%
, as $s_{\alpha }\mathfrak{M}_{T}$ is the phantom module associated with the 
$\omega \alpha $-th derivative $T^{\left( \omega \alpha \right) }$ of $T$;
see Theorem \ref{Theorem:Solecki-lim1}.

\subsection{Projective length\label{Subsection:projective-length}}

We define the notion of \emph{projective length }of a countable flat module,
as a relaxation of the notion of projective module.

\begin{definition}
Suppose that $C$ is a coreduced countable flat module.

\begin{itemize}
\item If $A$ is a countable flat module, then the (plain) $A$-projective
length of $C$ is the (plain) Solecki length of $\mathrm{Ext}\left(
C,A\right) $;

\item $C$ has (plain) projective length at most $\alpha $ if and only if it
has (plain) $A$-projective length at most $\alpha $ for every countable flat
module $A$;

\item $C$ is a \emph{phantom projective} if it has projective length at most 
$\alpha $ for some $\alpha <\omega _{1}$.
\end{itemize}
\end{definition}

Consider the quasi-abelian category $\mathbf{Flat}\left( R\right) $ of
countable \emph{flat} modules. Suppose that $\mathcal{C}_{1}$ is the class
of finite-rank flat modules. We let $\mathcal{E}_{1}$ be the exact structure 
\emph{projectively generated }by $\mathcal{C}_{1}$ on the quasi-abelian
category $\mathbf{Flat}\left( R\right) $. A similar proof as the one of \cite%
[Theorem 87.1]{fuchs_infinite_1973} yields the following.

\begin{lemma}
\label{Lemma:finite-rank-decomposable}A countable flat module $A$ is a
direct sum of finite-rank modules if and only if every finite subset of $A$
is contained in a finite-rank direct summand of $A$.
\end{lemma}

\begin{proof}
Suppose that $A$ is a countable flat module such that every finite subset $F$
of $A$ is contained in a finite-rank direct summand of $A$. Let $\left(
a_{n}\right) $ be an enumeration of $A$. By assumption, there exists a
finite-rank direct summand $A_{1}$ of $A$ containing $a_{1}$. Suppose that
we have defined an ascending chain $A_{1}\subseteq \cdots \subseteq A_{n-1}$
of finite-rank direct summands of $A$ such that $\left\{ a_{1},\ldots
,a_{i}\right\} \subseteq A_{i}$ for $i<n$. Then by assumption there exists a
finite-rank direct summand $A_{n}$ of $A$ that contains $a_{n}$ and a finite
subset $G_{n-1}$ of $A_{n-1}$ such that $A_{n-1}=\langle G_{n-1}\rangle
_{\ast }$. Since $A_{n}$ is a direct summand of $A$, this implies that $%
A_{n-1}\subseteq A_{n}$. This procedure defines an increasing sequence $%
\left( A_{n}\right) $ of finite-rank direct summands of $A$ with union equal
to $A$. For every $n\in \omega $, one has that $A_{n+1}=A_{n}\oplus B_{n+1}$%
, where $B_{n+1}$ has finite rank. Setting $B_{0}=A_{0}$, one has that $%
A=\bigoplus_{n\in \omega }B_{n}$. As the converse implication is obvious,
this concludes the proof.
\end{proof}

\begin{proposition}
\label{Proposition:resolution}Suppose that $C$ is a countable flat module.
Then there exist a short exact sequence%
\begin{equation*}
0\rightarrow A\rightarrow B\rightarrow C\rightarrow 0
\end{equation*}%
in $\mathcal{E}_{1}$ where $A$ and $B$ are countable direct sums of
finite-rank flat modules.
\end{proposition}

\begin{proof}
Let $\left( C_{n}\right) $ be an increasing sequence of finite-rank pure
submodules of $C$ with union equal to $C$. Define%
\begin{equation*}
A=B=\bigoplus_{n\in \omega }C_{n}\text{.}
\end{equation*}%
Set $\phi :B\rightarrow C$ be the homomorphism induced by the inclusion map $%
C_{n}\rightarrow C$ for $n\in \omega $, and $\psi :A\rightarrow B$ be the
homomorphism induced by the morphisms $C_{i}\rightarrow C_{j}$, for $i,j\in
\omega $, defined by 
\begin{equation*}
x\mapsto \left\{ 
\begin{array}{ll}
x & \text{if }i=j\text{;} \\ 
-x & \text{if }i+1=j\text{;} \\ 
0 & \text{otherwise.}%
\end{array}%
\right.
\end{equation*}%
This defines a short exact sequence $A\rightarrow B\rightarrow C$ in $%
\mathcal{E}_{1}$.
\end{proof}

Recall that a \emph{retraction }in a category is a morphism that has a right
inverse, and a \emph{section }is a morphism that has a left inverse.

\begin{lemma}
Suppose that $r:\boldsymbol{B}\rightarrow \boldsymbol{A}$ is a retraction in
the category of towers of countable modules. Then the induced map $\mathrm{%
lim}^{1}\boldsymbol{B}\rightarrow \mathrm{lim}^{1}\boldsymbol{A}$ maps $%
s_{\alpha }\left( \mathrm{lim}^{1}\boldsymbol{B}\right) $ onto $s_{\alpha
}\left( \mathrm{lim}^{1}\boldsymbol{A}\right) $ for every $\alpha <\omega
_{1}$.
\end{lemma}

\begin{proof}
As $r$ is a retraction, the induced map $\mathrm{lim}^{1}\boldsymbol{B}%
\rightarrow \mathrm{lim}^{1}\boldsymbol{A}$ is also a retraction. Since $%
G\mapsto s_{\alpha }G$ is a subfunctor of the identity, the conclusion
follows.
\end{proof}

It is clear that a countable direct sum of finite-rank flat modules is
projective in $\left( \mathcal{A},\mathcal{E}_{1}\right) $. As a consequence
of this observation and Proposition \ref{Proposition:resolution} we obtain
the following.

\begin{proposition}
\label{Proposition:projectives-1-pure}The idempotent-complete exact category 
$\left( \mathbf{Flat}\left( R\right) ,\mathcal{E}_{1}\right) $ is hereditary
and has enough projectives. Its projective objects are precisely the direct
summands of the countable direct sums of finite-rank flat modules.
\end{proposition}

It follows that the functor $\mathrm{Hom}:\left( \mathbf{Flat}\left(
R\right) ,\mathcal{E}_{1}\right) \times \left( \mathbf{Flat}\left( R\right) ,%
\mathcal{E}_{1}\right) \rightarrow \boldsymbol{\Pi }(\mathbf{Flat}\left(
R\right) )$ has a total right derived functor \textrm{RH}$\mathrm{om}_{%
\mathcal{E}_{1}}^{\bullet }$. We can then define $\mathrm{Ext}_{\mathcal{E}%
_{1}}^{n}:=\mathrm{H}^{n}\circ \mathrm{RHom}_{\mathcal{E}_{1}}^{\bullet }$
for $n\in \mathbb{Z}$. For $n\geq 2$ we have that $\mathrm{Ext}_{\mathcal{E}%
_{1}}^{n}=0$.

\begin{corollary}
\label{Corollary:Ext-and-lim1-2}Suppose that $C$ and $A$ are countable flat
module. If $\left( F_{n}\right) $ is an increasing sequence of finite-rank
pure submodules of $C$ such that $\bigcup_{n\in \omega }F_{n}=C$, then%
\begin{equation*}
\mathrm{Ext}_{\mathcal{E}_{1}}\left( C,A\right) \cong \mathrm{\mathrm{lim}}%
_{n}^{1}\mathrm{Hom}\left( F_{n},A\right)
\end{equation*}
\end{corollary}

\begin{proof}
This follows from the definition of $\mathrm{Ext}_{\mathcal{E}_{1}}$ by
applying\ Lemma \ref{Lemma:Ext-and-lim1}.
\end{proof}

Suppose that $C$ and $A$ are countable flat modules. Notice that, by
definition, $\mathrm{Ext}_{\mathcal{E}^{1}}\left( C,A\right) $ is the
submodule of $\mathrm{Ext}\left( C,A\right) $ obtained as the intersection
of the kernels of the homomorphisms $\mathrm{Ext}\left( C,A\right)
\rightarrow \mathrm{Ext}\left( B,A\right) $ induced by the inclusion $%
B\rightarrow C$, where $B$ varies among the finite-rank pure submodules of $%
C $.

\begin{lemma}
\label{Lemma:finite-rank}Suppose that $A,C$ are countable flat modules such
that $C$ is finite-rank. Then $\mathrm{Ext}\left( C,A\right) $ has plain
Solecki length at most $1$.
\end{lemma}

\begin{proof}
Let $D:=K\otimes A$ be the divisible hull of $A$. Thus, $D$ is a divisible
flat module. Consider then the exact sequence%
\begin{equation*}
0\rightarrow \mathrm{Hom}\left( C,A\right) \rightarrow \mathrm{Hom}\left(
C,D\right) \rightarrow \mathrm{Hom}\left( C,D/A\right) \rightarrow \mathrm{%
Ext}\left( C,A\right) \rightarrow 0\text{.}
\end{equation*}%
Thus, we have that\textrm{\ Ext}$\left( C,A\right) $ is isomorphic to%
\begin{equation*}
\frac{\mathrm{Hom}\left( C,D/A\right) }{\mathrm{Ran}(\mathrm{Hom}\left(
C,D\right) \rightarrow \mathrm{Hom}\left( C,D/A\right) )}\text{.}
\end{equation*}%
Notice that $\mathrm{Hom}\left( C,D\right) $ is countable since $C$ has
finite-rank.
\end{proof}

\begin{lemma}
\label{Lemma:dense-PExtJ}Suppose that $C,A$ are countable flat modules. Then 
$\left\{ 0\right\} $ is dense in $\mathrm{Ext}_{\mathcal{E}^{1}}\left(
C,A\right) $.
\end{lemma}

\begin{proof}
It suffices to prove that $\left\{ 0\right\} $ is dense in $\mathrm{\mathrm{%
Ker}}\left( \mathrm{Ext}\left( C,A\right) \rightarrow \mathrm{Ext}\left(
B,A\right) \right) $ for every finite-rank pure submodule $B$ of $C$. Let $B$
be a finite-rank submodule of $C$. Then we have an exact sequence%
\begin{equation*}
\mathrm{Hom}\left( B,A\right) \rightarrow \mathrm{Ext}\left( C/B,A\right)
\rightarrow \mathrm{Ext}\left( C,A\right) \rightarrow \mathrm{Ext}\left(
B,A\right) \rightarrow 0
\end{equation*}%
Thus, we have an isomorphism 
\begin{equation*}
\mathrm{\mathrm{Ker}}(\mathrm{Ext}\left( C,A\right) \rightarrow \mathrm{Ext}%
\left( B,A\right) )\cong \frac{\mathrm{\mathrm{Ext}}\left( C/B,A\right) }{%
\mathrm{Ran}(\mathrm{Hom}\left( B,A\right) \rightarrow \mathrm{\mathrm{Ext}}%
\left( C/B,A\right) )}
\end{equation*}%
Thus, it suffices to prove that $\mathrm{Ran}(\mathrm{Hom}\left( B,A\right)
\rightarrow \mathrm{\mathrm{Ext}}\left( C/B,A\right) )$ is dense in $\mathrm{%
Ext}\left( C/B,A\right) $. Since $B$ is pure in $C$, $C/B$ is flat and hence 
$\mathrm{Ext}\left( C/B,A\right) =\mathrm{PExt}\left( C/B,A\right) $. Thus, $%
\left\{ 0\right\} $ is dense in $\mathrm{Ext}\left( C/B,A\right) $, and 
\emph{a fortiori} $\mathrm{Ran}(\mathrm{Hom}\left( B,A\right) \rightarrow 
\mathrm{\mathrm{Ext}}\left( C/B,A\right) )$ is dense in $\mathrm{Ext}\left(
C/B,A\right) $.
\end{proof}

\begin{theorem}
\label{Theorem:PExt-Solecki}Suppose that $A,C$ are countable flat modules.
Then $\mathrm{Ext}_{\mathcal{E}_{1}}\left( C,A\right) $ is equal to the
first Solecki submodule of $\mathrm{Ext}\left( C,A\right) $.
\end{theorem}

\begin{proof}
By Lemma \ref{Lemma:finite-rank}, $\left\{ 0\right\} $ is a $\boldsymbol{%
\Sigma }_{2}^{0}$ submodule of $\mathrm{Ext}\left( B,A\right) $ for every
finite-rank submodule $B$ of $C$. Since $\mathrm{Ext}_{\mathcal{E}%
_{1}}\left( C,A\right) $ is the intersection of the kernels of the Borel
homomorphisms $\mathrm{Ext}\left( C,A\right) \rightarrow \mathrm{Ext}\left(
B,A\right) $ where $B$ varies among the finite-rank pure submodules of $C$,
it follows that $\mathrm{Ext}_{\mathcal{E}_{1}}\left( C,A\right) $ is a $%
\boldsymbol{\Pi }_{3}^{0}$ submodule of $\mathrm{Ext}\left( C,A\right) $ as
there are countably many such submodules $B$.

Consider a projective resolution $S\rightarrow F\rightarrow C$ of $C$, where 
$F$ is a free module. Thus, we have a natural isomorphism%
\begin{equation*}
\mathrm{Ext}\left( C,A\right) \cong \frac{\mathrm{Hom}\left( S,A\right) }{%
\mathrm{Hom}\left( F|S,A\right) }\text{,}
\end{equation*}%
where%
\begin{equation*}
\mathrm{Hom}\left( F|S,A\right) =\mathrm{Ran}\left( \mathrm{Hom}\left(
F,A\right) \rightarrow \mathrm{Hom}\left( S,A\right) \right) \text{.}
\end{equation*}%
For a submodule $B$ of $C$, define $F_{B}=\left\{ x\in F:x+S\in B\right\} $.
Under such an isomorphism $\mathrm{Ext}_{\mathcal{E}_{1}}\left( C,A\right) $
corresponds to%
\begin{equation*}
\frac{\mathrm{Hom}^{1}\left( S,A\right) }{\mathrm{Hom}\left( F|S,A\right) }
\end{equation*}%
where $\mathrm{Hom}^{1}\left( S,A\right) $ is the submodule of homomorphisms 
$S\rightarrow A$ that have an extension $F_{B}\rightarrow A$ for every
finite-rank (pure) submodule $B\ $of $C$.

By Lemma \ref{Lemma:1st-Solecki} and Lemma \ref{Lemma:dense-PExtJ}, it
suffices to prove that if $V\subseteq \mathrm{Hom}\left( F|S,A\right) $ is
an open neighborhood of $0$ in $\mathrm{Hom}\left( F|S,A\right) $, then the
closure $\overline{V}^{\mathrm{Hom}\left( S,A\right) }$ of $V\ $inside $%
\mathrm{Hom}\left( S,A\right) $ contains an open neighborhood of $0$ in $%
\mathrm{Hom}^{1}\left( S,A\right) $. Suppose that $V\subseteq \mathrm{Hom}%
\left( F|S,A\right) $ is an open neighborhood of $0$. Thus, there exist $%
x_{0},\ldots ,x_{n}\in F$ such that 
\begin{equation*}
\left\{ \varphi |_{S}:\varphi \in \mathrm{Hom}\left( F,A\right) ,\forall
i\leq n,\varphi \left( x_{i}\right) =0\right\} \subseteq V\text{.}
\end{equation*}%
Let $B$ be a finite-rank pure submodule of $C$ such that $\left\{
x_{0},\ldots ,x_{n}\right\} \subseteq F_{B}$. Thus, we have that%
\begin{equation*}
U=\left\{ \varphi \in \mathrm{Hom}^{1}\left( S,A\right) :\exists \hat{\varphi%
}\in \mathrm{Hom}(F_{B},A),\forall i\leq n,\hat{\varphi}\left( x_{i}\right)
=0,\hat{\varphi}|_{S}=\varphi \right\}
\end{equation*}%
is an open neighborhood of $0$ in $\mathrm{Hom}^{1}\left( S,A\right) $. We
claim that $U\subseteq \overline{V}^{\mathrm{Hom}\left( S,A\right) }$.
Indeed, suppose that $\varphi _{0}\in U$ and let $W$ be an open neighborhood
of $\varphi _{0}$ in $\mathrm{Hom}\left( S,A\right) $. Thus, there exist $%
y_{0},\ldots ,y_{\ell }\in S$ such that%
\begin{equation*}
\left\{ \varphi \in \mathrm{Hom}\left( S,A\right) :\forall i\leq \ell
,\varphi \left( y_{i}\right) =\varphi _{0}\left( y_{i}\right) \right\}
\subseteq W\text{.}
\end{equation*}%
We need to prove that $V\cap W\neq \varnothing $. To this purpose, it
suffices to show that there exists $\varphi \in \mathrm{Hom}\left(
F,A\right) $ such that $\varphi \left( y_{i}\right) =\varphi _{0}\left(
y_{i}\right) $ for $i\leq \ell $ and $\varphi \left( x_{i}\right) =0$ for $%
i\leq n$. Since $\varphi _{0}\in U$, there exists $\hat{\varphi}_{0}\in 
\mathrm{Hom}(F_{B},A)$ such that $\hat{\varphi}_{0}\left( x_{i}\right) =0$
for $i\leq n$ and $\hat{\varphi}_{0}|_{S}=\varphi _{0}$.

Let $\left( z_{k}\right) _{k\in I}$ be a free $R$-basis of $F$. The elements 
$x_{0},\ldots ,x_{n},y_{0},\ldots ,y_{\ell }$ of $F$ are contained in the
submodule $F_{0}$ of $F$ generated by $\left\{ z_{i}:i\in I_{0}\right\} $
for some \emph{finite }$I_{0}\subseteq I$. Letting $\langle
F_{0},F_{B}\rangle $ be the submodule of $F$ generated by $F_{0}$ and $F_{B}$%
, we have that, since $B$ is pure in $C$, 
\begin{equation*}
\langle F_{0},F_{B}\rangle /F_{B}\subseteq F/F_{B}\cong C/B
\end{equation*}%
is finitely-generated and flat, and hence projective. Therefore, $\mathrm{Ext%
}\left( \langle F_{0},F_{B}\rangle /F_{B},A\right) =0$ and there exists $%
\psi \in \mathrm{Hom}(\langle F_{0},F_{B}\rangle ,A)$ that extends $\hat{%
\varphi}_{0}$. One can then define $\varphi \in \mathrm{Hom}\left(
F,A\right) $ by setting%
\begin{equation*}
\varphi \left( z_{i}\right) =\left\{ 
\begin{array}{ll}
\psi (z_{i}) & \text{if }i\in I_{0}\text{,} \\ 
0 & \text{if }i\in I\setminus I_{0}\text{.}%
\end{array}%
\right.
\end{equation*}%
This concludes the proof.
\end{proof}

\begin{corollary}
An extension of countable flat modules is in $\mathcal{E}_{1}$ if and only
if it is phantom of order $1$.
\end{corollary}

We now completely characterize the countable flat modules that have (plain)
projective rank at most one.

\begin{theorem}
\label{Theorem:structure-rank-1}Given a countable coreduced flat module $C$,
the following assertions are equivalent:

\begin{enumerate}
\item $C$ has plain $R$-projective length at most one;

\item $C$ has plain projective length at most one;

\item $C$ has finite rank.
\end{enumerate}
\end{theorem}

\begin{proof}
Let $\left( C_{n}\right) $ be an increasing of finite-rank pure submodules
of $C$ with union equal to $C$.

(1)$\Rightarrow $(2) By Corollary \ref{Corollary:Ext-and-lim1-2} and Theorem %
\ref{Theorem:PExt-Solecki}, we have that%
\begin{equation*}
\mathrm{Ext}\left( C,R\right) \cong \mathrm{lim}_{n}\mathrm{Ext}\left(
C_{n},R\right) \text{.}
\end{equation*}%
By Lemma \ref{Lemma:monomorphic-tower}, after passing to a subsequence of $%
\left( C_{n}\right) $ we can assume that $\mathrm{Ext}\left( C_{n},R\right)
\rightarrow \mathrm{Ext}\left( C_{n-1},R\right) $ is an isomorphism for
every $n\geq 1$. For $n\geq 0$ we have%
\begin{equation*}
0=\mathrm{Hom}\left( C_{n},R\right) \rightarrow \mathrm{Ext}\left(
C_{n+1}/C_{n},R\right) \rightarrow \mathrm{Ext}\left( C_{n+1},R\right)
\rightarrow \mathrm{Ext}\left( C_{n},R\right) \rightarrow 0\text{.}
\end{equation*}%
This shows that $\mathrm{Ext}\left( C_{n+1}/C_{n},R\right) =0$ and hence $%
C_{n+1}/C_{n}$ is a projective module by Proposition \ref%
{Proposition:characterize-projectives}. As this holds for every $n\geq 1$,
and since $C$ is coreduced, $C=C_{0}$.

(3)$\Rightarrow $(2) This follows from Lemma \ref{Lemma:finite-rank}.
\end{proof}

\begin{theorem}
\label{Theorem:structure-rank-1-bis}Given a countable coreduced flat module $%
C$, the following assertions are equivalent:

\begin{enumerate}
\item $C$ has $R$-projective length at most one;

\item $C\cong \mathrm{co\mathrm{lim}}_{k}C_{k}$ for some inductive sequence $%
\left( C_{k}\right) $ of coreduced flat modules of finite rank.
\end{enumerate}
\end{theorem}

\begin{proof}
Let $\left( C_{n}\right) $ be an increasing of finite-rank pure submodules
of $C$ with union equal to $C$. By Theorem \ref{Theorem:PExt-Solecki}, we
have%
\begin{equation*}
s_{1}\mathrm{Ext}\left( C,A\right) \cong \mathrm{lim}^{1}\mathrm{Hom}\left(
C_{n},R\right) \text{.}
\end{equation*}

(2)$\Rightarrow $(1) If $C_{n}$ is coreduced for every $n\in \omega $, then $%
\mathrm{Hom}\left( C_{n},R\right) =0$, whence $s_{1}\mathrm{Ext}\left(
C,A\right) =0$.

(1)$\Rightarrow $(2) If $s_{1}\mathrm{Ext}\left( C,A\right) =0$, then since $%
C$ is coreduced we have that the tower $\left( \mathrm{Hom}\left(
C_{n},R\right) \right) $ of finite flat modules is isomorphic to a trivial
tower. Thus, after passing to a subsequence, we can assume that $\mathrm{Hom}%
\left( C_{n},R\right) =0$ for every $n\in \omega $, i.e., that $C_{n}$ is
coreduced for every $n\in \omega $.
\end{proof}

\begin{theorem}
\label{Theorem:structure-rank-1-tris}Given a countable coreduced flat module 
$C$, the following assertions are equivalent:

\begin{enumerate}
\item $C\mathrm{\ }$has projective length at most one;

\item $C$ has $A$-projective length at most one for every finite-rank
countable flat module $A$;

\item $C$ is a direct summand of a countable direct sum of finite-rank flat
modules.
\end{enumerate}
\end{theorem}

\begin{proof}
(2)$\Rightarrow $(3) The hypothesis implies that $\mathrm{Ext}\left(
C,A\right) $ has Solecki length at most $1$ for every finite-rank
decomposable countable flat module. By Proposition \ref%
{Proposition:resolution} we have a $1$-pure short exact sequence%
\begin{equation*}
0\rightarrow A\rightarrow B\rightarrow C\rightarrow 0
\end{equation*}%
where $A$ and $B$ are finite-rank decomposable. By\ Theorem \ref%
{Theorem:PExt-Solecki} we have $\mathrm{Ext}_{\mathcal{E}_{1}}\left(
C,A\right) =0$ and hence such a sequence splits. This shows that $C$ is a
direct summand of $B$.

(3)$\Rightarrow $(1) This follows from Lemma \ref{Lemma:finite-rank}.
\end{proof}

Note that a direct summand of a countable direct sum of finite-rank flat
modules is not necessarily a countable direct sum of finite-rank flat
modules; see \cite{corner_note_1969} for an example in the case of $\mathbb{Z%
}$-modules.

\section{Phantom projective resolutions\label{Section:phantom-resolutions}}

For each limit ordinal $\alpha $, we let $\left( \alpha _{n}\right) $ be a
sequence of successor ordinals converging to $\alpha $. When $\alpha $ is a
successor we set $\alpha _{n}=\alpha $ for every $n\in \omega $. In this
section we provide a description of the functor $\mathrm{Ph}^{\alpha }%
\mathrm{Ext}$ on the category $\mathbf{Flat}\left( R\right) $ of countable
flat modules as a derived functor with respect to a canonical exact
structure $\mathcal{E}_{\alpha }$ on $\mathbf{Flat}\left( R\right) $.

\subsection{Wedge sums}

We begin with introducing a natural operation between modules which we term 
\emph{wedge sum}. Recall the sets of indices $I_{\alpha }$ and $I_{\alpha }^{%
\mathrm{plain}}$ from Section \ref{Subsection:indices}. In particular, $%
I_{2}^{\mathrm{plain}}$ is the complete countable rooted tree of height $1$.
A presheaf of flat modules over $I_{2}^{\mathrm{plain}}$ is given by a flat
module $M_{1}$ together with flat modules $M_{\left( i;0\right) }$ for $i\in
\omega $ and module homomorphisms $\psi _{i}:M_{1}\rightarrow M_{\left(
i;0\right) }$ for $i\in \omega $. We define the \emph{wedge sum}%
\begin{equation*}
M:=\bigoplus_{\left( M_{1},\psi _{i}\right) }M_{\left( i;0\right) }
\end{equation*}%
to be its colimit in $\mathbf{Flat}\left( R\right) $. We can identify $M$
with the quotient of the direct sum $M_{\infty }$ of $M_{1}$ and $M_{\left(
i;0\right) }$ for $i\in \omega $ by the pure submodule generated by the
submodule $N$ consisting of elements of the form 
\begin{equation*}
(\sum_{i\in \omega }y_{i},-\psi _{0}(y_{0}),-\psi _{1}\left( y_{1}\right)
,\ldots )
\end{equation*}%
for some $\left( y_{i}\right) \in M_{1}^{\left( \omega \right) }$. We can
also identify $M_{1}$ with a submodule of $M$, via the map that assigns $%
a\in M_{1}$ to the element of $M$ represented by the sequence $\left(
a,0,0,\ldots \right) $. Notice that the sequence%
\begin{equation*}
M_{1}\rightarrow M\rightarrow \bigoplus_{i}\mathrm{Coker}\left( \psi
_{i}:M_{1}\rightarrow M_{\left( i;0\right) }\right)
\end{equation*}%
is exact at $M$.

It is easily verified that, adopting the notation above, if $\psi _{i}$ is
injective with pure image for every $i\in \omega $, then $N$ is pure in $%
M_{\infty }$ and $M_{1}$ is pure in $M$.

\subsection{Tree length}

Recall that a countable flat module $C$ is \emph{finite }if it is
finitely-generated. We say that a countable flat module has plain tree
length $0$ if and only if it is finite. We now define by recursion on $%
\alpha $ the notion of (plain) tree length at most $\alpha $.

\begin{definition}
Let $C$ be a countable flat module, and $\alpha <\omega _{1}$:

\begin{itemize}
\item if $\alpha $ is a successor, then $C$ has plain tree length at most $%
\alpha $ if and only if there exist countable flat modules $M_{\left(
i;0\right) }$ of plain tree length at most $\left( \alpha -1\right) _{i}$, a
finiteflat module $M_{1}$, and injective module homomorphisms $\psi
_{i}:M_{1}\rightarrow M_{\left( i;0\right) }$ such that $C$ is a direct
summand of the wedge sum%
\begin{equation*}
\bigoplus_{\left( M_{1},\psi _{n}\right) }M_{\left( i;0\right) }\text{;}
\end{equation*}

\item $C$ has tree length at most $\alpha $ if and only if $C$ is a direct
summand of a direct sum of modules $C_{n}$ of plain tree length at most $%
\alpha _{n}$ for $n<\omega $.
\end{itemize}
\end{definition}

\begin{definition}
For $\alpha <\omega _{1}$, we let:

\begin{itemize}
\item $\mathcal{S}_{\alpha }$ be the class of countable flat modules of
plain tree length at most $\alpha $;

\item $\mathcal{E}_{\alpha }$ be the exact structure on $\mathbf{Flat}\left(
R\right) $ projectively generated by the union of $\mathcal{S}_{\alpha _{n}}$
for $n\in \omega $;

\item $\mathcal{P}_{\alpha }$ be the class of $\mathcal{E}_{\alpha }$%
-projective countable flat modules.
\end{itemize}
\end{definition}

Considering that colimits commute with colimits, a straightforward induction
on $\alpha $ proves the following characterization of modules of (plain)
tree length at most $\alpha $.

\begin{proposition}
\label{Proposition:tree-presheaf}Suppose that $C$ is a countable flat
module. Then the following assertions are equivalent:

\begin{enumerate}
\item $C$ has plain tree length at most $\alpha $;

\item $C$ is isomorphic to a direct summand of a colimit of a presheaf of
finiteflat modules over $I_{1+\alpha }^{\mathrm{plain}}$;

\item $C$ is isomorphic to a direct summand of a colimit of a presheaf of
finiteflat modules over a countable well-founded tree of rank at most $%
1+\alpha $.
\end{enumerate}

Furthermore, the following assertions are equivalent:

\begin{enumerate}
\item $C$ has tree length at most $\alpha $;

\item $C$ is isomorphic to a direct summand of a colimit of a presheaf of
finiteflat modules over $I_{1+\alpha }$;

\item $C$ is isomorphic to a direct summand of a colimit of a presheaf of
finiteflat modules over a countable well-founded forest of rank at most $%
1+\alpha $.
\end{enumerate}
\end{proposition}

We now relate the (plain) tree length with the (plain) projective length of
a module.

\begin{lemma}
\label{Lemma:monic-wedge}Suppose that a countable flat module $C$ is a wedge
sum%
\begin{equation*}
\bigoplus_{\left( M_{1},\psi _{i}\right) }M_{\left( i;0\right) }
\end{equation*}%
for some finiteflat module $M_{1}$, flat modules $M_{\left( i;0\right) }$,
and injective module homomorphisms $\psi _{i}:M_{1}\rightarrow M_{\left(
i;0\right) }$. For $n\in \omega $, let $C_{n}$ be the submodule of $C$
corresponding to $M_{1}\oplus M_{\left( i;0\right) }\oplus \cdots \oplus
M_{\left( i;n-1\right) }$. Let $A$ be a countable flat module, and $\alpha
,\beta <\omega _{1}$. For $n<\omega $, set%
\begin{equation*}
\mathrm{Hom}\left( C|C_{n},A\right) :=\mathrm{Ran}\left( \mathrm{Hom}\left(
C,A\right) \rightarrow \mathrm{Hom}\left( C_{n},A\right) \right) \text{.}
\end{equation*}

\begin{enumerate}
\item The tower%
\begin{equation*}
\left( \frac{\mathrm{Hom}\left( C_{n},A\right) }{\mathrm{Hom}\left(
C|C_{n},A\right) }\right) _{n\in \omega }
\end{equation*}%
is monomorphic;

\item If $M_{\left( i;0\right) },\ldots ,M_{\left( i;n\right) }$ have plain $%
A$-projective length at most $\beta $, then $C_{n}$ has plain $A$-projective
length at most $\beta $;

\item If $M_{\left( i;0\right) }$ has plain $A$-projective length at most $%
\left( \alpha -1\right) _{i}$ for every $i\in \omega $, then $C$ has plain $%
A $-projective length at most $\alpha $.
\end{enumerate}
\end{lemma}

\begin{proof}
(1) Suppose that $\varphi \in \mathrm{Hom}\left( C_{n+1},A\right) $ is such
that $\varphi |_{C_{n}}\in \mathrm{Hom}\left( C|C_{n},A\right) $. Thus,
there exists a homomorphism $\Psi \in \mathrm{Hom}\left( C,A\right) $ such
that $\Psi |_{C_{n}}=\varphi |_{C_{n}}$. One can then define a homomorphism $%
\Phi \in \mathrm{Hom}\left( C,A\right) $ by setting%
\begin{equation*}
\Phi \left( x\right) :=\left\{ 
\begin{array}{cc}
\varphi \left( x\right) & x\in C_{n+1}\text{;} \\ 
\Psi \left( x\right) & x\notin C_{n+1}\text{.}%
\end{array}%
\right.
\end{equation*}%
Then we have that $\Phi |_{C_{n+1}}=\varphi $.

(2) Fix $n\in \omega $. Notice that we have a short exact sequence%
\begin{equation*}
M_{1}^{n}\rightarrow M_{1}\oplus M_{\left( 0;0\right) }\oplus \cdots \oplus
M_{\left( 0;n-1\right) }\rightarrow C_{n}\text{.}
\end{equation*}%
Since $M_{\left( 0;0\right) },\ldots ,M_{\left( 0;n-1\right) }$ have plain $%
A $-projective length at most $\beta $, the same holds for $M_{1}\oplus
M_{\left( 0;0\right) }\oplus \cdots \oplus M_{\left( 0;n-1\right) }$. Set%
\begin{equation*}
G_{n}:=\mathrm{Ext}\left( C_{n},A\right)
\end{equation*}%
and%
\begin{equation*}
H_{n}:=\mathrm{Ran}\left( \mathrm{Hom}\left( M_{1}^{n},A\right) \rightarrow 
\mathrm{Ext}\left( C_{n},A\right) \right) =\mathrm{\mathrm{Ker}}\left( 
\mathrm{Ext}\left( C_{n},A\right) \rightarrow \mathrm{Ext}\left( M_{1}\oplus
M_{\left( 0;0\right) }\oplus \cdots \oplus M_{\left( 0;n-1\right) },A\right)
\right)
\end{equation*}%
Then $G_{n}/H_{n}$ is isomorphic to a submodule of $\mathrm{Ext}\left(
M_{1}\oplus M_{\left( 0;0\right) }\oplus \cdots \oplus M_{\left(
0;n-1\right) },A\right) $, whence it has plain Solecki length at most $\beta 
$. Since $M_{1}^{n}$ is finite, $H_{n}$ is countable, and hence $G_{n}$ has
plain Solecki length at most $\beta $ by Lemma \ref%
{Lemma:complexity-exact-sequence}(3).

(3) By (2), we have%
\begin{equation*}
\mathrm{Ph}^{\alpha -1}\mathrm{Ext}\left( C,A\right) \subseteq \bigcap_{n\in
\omega }\mathrm{\mathrm{\mathrm{Ker}}}\left( \mathrm{Ext}\left( C,A\right)
\rightarrow \mathrm{Ext}\left( C_{n},A\right) \right) \cong \mathrm{lim}%
_{n}^{1}\mathrm{Hom}\left( C_{n},A\right) \cong \mathrm{lim}_{n}^{1}\frac{%
\mathrm{Hom}\left( C_{n},A\right) }{\mathrm{Hom}\left( C|C_{n},A\right) }
\end{equation*}%
By (1), the tower%
\begin{equation*}
\left( \frac{\mathrm{Hom}\left( C_{n},A\right) }{\mathrm{Hom}\left(
C|C_{n},A\right) }\right) _{n\in \omega }
\end{equation*}%
is monomorphic. Furthermore,%
\begin{equation*}
\mathrm{lim}_{n}^{1}(\frac{\mathrm{Hom}\left( C_{n},A\right) }{\mathrm{Hom}%
\left( C|C_{n},A\right) })\cong \mathrm{\mathrm{li}m}_{n}^{1}\mathrm{Hom}%
\left( C_{n},A\right) \text{.}
\end{equation*}%
It follows that $\mathrm{lim}_{n}^{1}\mathrm{Hom}\left( C_{n},A\right) $ has
plain Solecki length at most $1$. As $\mathrm{Ph}^{\alpha -1}\mathrm{Ext}%
\left( C,A\right) $ is isomorphic to a submodule of $\mathrm{lim}_{n}^{1}%
\mathrm{Hom}\left( C_{n},A\right) $, the same holds for $\mathrm{Ph}^{\alpha
-1}\mathrm{Ext}\left( C,A\right) $. Thus, we have that $\mathrm{Ext}\left(
C,A\right) $ has plain Solecki length at most $\alpha $.
\end{proof}

\begin{lemma}
\label{Lemma:wedge-projective}Suppose that $C$ is a countable flat module,
and $\alpha <\omega _{1}$:

\begin{enumerate}
\item if $C$ has plain tree length at most $\alpha $, then $C$ has plain
projective length at most $\alpha $;

\item if $C$ has tree length at most $\alpha $, then $C$ has projective
length at most $\alpha $.
\end{enumerate}
\end{lemma}

\begin{proof}
We prove this by induction on $\alpha $. It suffices to notice that (2) for $%
\alpha $ is a consequence of (1) for $\alpha _{n}$ for $n<\omega $, while
(1) for $\alpha $ follows immediately from (1) for $\alpha _{n}$ for $%
n<\omega $ by Lemma \ref{Lemma:monic-wedge}(3).
\end{proof}

\subsection{Resolutions}

Let $C$ be a countable flat module. Write $C=\mathrm{co\mathrm{lim}}%
_{n}C_{n} $ for some finite submodules $C_{n}$ of $C$. We define by
recursion, for every countable successor ordinal $\alpha $, countable flat
modules $P_{n}^{\alpha }C$ of plain length at most $\alpha $ containing $%
C_{n}$ as a pure submodule, together with module homomorphisms $%
P_{n}^{\left( \alpha -1\right) _{n}}C\rightarrow P_{n+1}^{\left( \alpha
-1\right) _{n+1}}C$ and $P_{n}^{\alpha }C\rightarrow C$ that restrict to the
inclusion $C_{n}\rightarrow C$ and yield an isomorphism%
\begin{equation*}
C\cong \mathrm{co\mathrm{lim}}_{n}P_{n}^{\alpha _{n}}C\text{.}
\end{equation*}%
Define $P_{n}^{0}C=C_{n}$ for $n\in \omega $. Suppose that $P_{n}^{\beta
}\left( C\right) $ has been defined for every successor ordinal $\beta
<\alpha $ and $n\in \omega $. Define then $P_{n}^{\alpha }\left( C\right) $
to be the wedge sum of $P_{k}^{\left( \alpha -1\right) _{k}}C$ for $k\geq n$
over $C_{n}$. The inclusion $C_{n}\rightarrow P_{n}^{\alpha }C$ is the
canonical one given by the definition of wedge sum. The homomorphism $%
P_{n}^{\alpha }C\rightarrow C$ is induced by the homomorphisms $%
P_{k}^{\left( \alpha -1\right) _{k}}C\rightarrow C$ for $k\geq n$. By
definition, $P_{n}^{\alpha }C$ has plain tree length at most $\alpha $ for
every $n\in \omega $. Let $\mathfrak{R}^{\alpha }C$ be the short exact
sequence%
\begin{equation*}
0\rightarrow P^{\alpha }C\rightarrow P^{\alpha }C\rightarrow C\rightarrow 0
\end{equation*}%
realizing $C$ as $\mathrm{co\mathrm{lim}}_{n}P_{n}^{\alpha _{n}}C$.

\begin{theorem}
\label{Theorem:phantom-projective}Let $C$ and $A$ be countable flat modules.
Let $\left( C_{n}\right) $ be an increasing sequence of finite submodules of 
$C$ with union equal to $C$. Define%
\begin{equation*}
\mathrm{Hom}\left( C_{n},A\right) :=\mathrm{Hom}_{0}\left( C_{n},A\right)
\end{equation*}%
and recursively%
\begin{equation*}
\mathrm{Hom}_{\alpha }\left( C_{n},A\right) :=\bigcap_{k\geq n}\mathrm{Ran}(%
\mathrm{Hom}_{\alpha _{k}}\left( C_{k},A\right) \rightarrow \mathrm{Hom}%
\left( C,A\right) )\text{.}
\end{equation*}%
Then we have that the tower $\left( \mathrm{Hom}\left( C_{n},A\right)
\right) _{n\in \omega }$ has $(\mathrm{Hom}_{\alpha _{n}}\left(
C_{n},A\right) )_{n\in \omega }$ as its $\alpha $-th derived tower.

\begin{enumerate}
\item for $\alpha <\omega _{1}$ and $n<\omega $%
\begin{equation*}
\mathrm{Hom}_{\alpha }\left( C_{n},A\right) =\mathrm{Hom}\left(
P_{n}^{\alpha }C|C_{n},A\right) :=\mathrm{Ran}\left( \mathrm{Hom}\left(
P_{n}^{\alpha }C,A\right) \rightarrow \mathrm{Hom}\left( C_{n},A\right)
\right) \text{;}
\end{equation*}

\item $P_{n}^{\alpha }\left( C\right) $ has plain projective length at most $%
\alpha $;

\item 
\begin{equation*}
\mathrm{Ph}^{\alpha }\mathrm{Ext}\left( C,A\right) =\bigcap_{n\in \omega }%
\mathrm{\mathrm{Ker}}\left( \mathrm{Ext}\left( C,A\right) \rightarrow 
\mathrm{Ext}(P_{n}^{\alpha _{n}}C,A)\right) \cong \mathrm{lim}_{n}^{1}%
\mathrm{Hom}_{\alpha _{n}}\left( C,A\right)
\end{equation*}

\item a countable flat module has tree length at most $\alpha $ if and only
if it has projective length at most $\alpha $;

\item for every countable flat module $B$ of plain projective length $\alpha 
$, finite submodule $A\subseteq B$, $d_{0}\in \omega $, and homomorphism $%
\varphi :B\rightarrow C$ that maps $A$ to $C_{d_{0}}$, there exist $d\geq
d_{0}$ and a homomorphism $\overline{\varphi }:B\rightarrow P_{d}^{\alpha }C$
that maps $A$ to $C_{d}\subseteq P_{d}^{\alpha }C$ and such $\varphi $ is
equal to the composition of $\overline{\varphi }$ and the canonical
homomorphism $P_{n}^{\alpha }C\rightarrow C$;

\item $\mathfrak{R}^{\alpha }C$ represents an element of $\mathrm{Ph}%
^{\alpha }\mathrm{Ext}\left( C,P^{\alpha }C\right) $ and it is a projective
resolution of $C$ in the exact category $\left( \mathbf{Flat}\left( R\right)
,\mathcal{E}_{\alpha }\right) $;

\item the exact category $\left( \mathbf{Flat}\left( R\right) ,\mathcal{E}%
_{\alpha }\right) $ is hereditary with enough projectives;

\item if $\mathbf{Ext}_{\mathcal{E}_{\alpha }}$ is the total right derived
functor of $\mathrm{Hom}$ on the exact category $\left( \mathbf{Flat}\left(
R\right) ,\mathcal{E}_{\alpha }\right) $, then 
\begin{equation*}
\mathrm{Ph}^{\alpha }\mathrm{Ext}=\mathrm{Ext}_{\mathcal{E}_{\alpha }}^{1}=%
\mathrm{H}^{1}\circ \mathbf{Ext}_{\mathcal{E}_{\alpha }}\text{.}
\end{equation*}
\end{enumerate}
\end{theorem}

\begin{proof}
By induction on $\alpha $. As the case $\alpha =0$ is obvious, we can assume 
$\alpha >0$.

(1) By definition, we have%
\begin{equation*}
\mathrm{Hom}_{\alpha }\left( C_{n},A\right) =\bigcap_{k\geq n}\mathrm{Hom}%
_{\left( \alpha -1\right) _{k}}\left( C_{k}|C_{n},A\right) \text{.}
\end{equation*}%
By the inductive hypothesis, we have that%
\begin{equation*}
\mathrm{Hom}_{\left( \alpha -1\right) _{k}}\left( C_{k},A\right) =\mathrm{Hom%
}(P_{k}^{\left( \alpha -1\right) _{k}}C|C_{k},A)
\end{equation*}%
and hence%
\begin{equation*}
\mathrm{Hom}_{\alpha }\left( C_{n},A\right) =\bigcap_{k\geq n}\mathrm{Hom}%
(P_{k}^{(\alpha -1)_{k}}C|C_{n},A)\text{.}
\end{equation*}%
Considering that by definition $P_{n}^{\alpha }C$ is the wedge sum of $%
(P_{k}^{(\alpha -1)_{k}}C)$ over $C_{n}$, we obtain%
\begin{equation*}
\mathrm{Hom}_{\alpha }\left( C_{n},A\right) =\mathrm{Hom}\left(
P_{n}^{\alpha }C|C_{n},A\right) \text{.}
\end{equation*}

(2) Since $P_{n}^{\alpha }C$ has plain tree length at most $\alpha $, it
also has plain projective length at most $\alpha $ by Lemma \ref%
{Lemma:wedge-projective}.

(3) When $\alpha $ is limit, we have%
\begin{equation*}
\mathrm{Ph}^{\alpha }\mathrm{Ext}\left( C,A\right) =\bigcap_{n\in \omega }%
\mathrm{Ph}^{\alpha _{n}}\mathrm{Ext}\left( C,A\right) \text{.}
\end{equation*}%
By the inductive hypothesis, we have%
\begin{equation*}
\mathrm{Ph}^{\alpha _{n}}\mathrm{Ext}\left( C,A\right) =\bigcap_{k\geq n}%
\mathrm{\mathrm{Ker}}\left( \mathrm{Ext}\left( C,A\right) \rightarrow 
\mathrm{Ext}(P_{k}^{\alpha _{n}}C,A)\right) \text{.}
\end{equation*}%
Therefore,%
\begin{eqnarray*}
\mathrm{Ph}^{\alpha }\mathrm{Ext}\left( C,A\right) &=&\bigcap_{n\in \omega
}\bigcap_{k\geq n}\mathrm{\mathrm{Ker}}\left( \mathrm{Ext}\left( C,A\right)
\rightarrow \mathrm{Ext}(P_{k}^{\alpha _{n}}C,A)\right) \\
&=&\bigcap_{n\in \omega }\mathrm{\mathrm{Ker}}\left( \mathrm{Ext}\left(
C,A\right) \rightarrow \mathrm{Ext}(P_{n}^{\alpha _{n}}C,A)\right) \\
&\cong &\mathrm{lim}_{n}^{1}\mathrm{Hom}_{\alpha _{n}}\left( C,A\right) 
\text{.}
\end{eqnarray*}%
Suppose now that $\alpha \mathrm{\ }$is a successor. By the inductive
hypothesis, we have%
\begin{eqnarray*}
\mathrm{Ph}^{\alpha -1}\mathrm{Ext}\left( C,A\right) &=&\bigcap_{n\in \omega
}\mathrm{\mathrm{Ker}}\left( \mathrm{Ext}\left( C,A\right) \rightarrow 
\mathrm{Ext}\left( P_{n}^{\alpha -1}C,A\right) \right) \\
&\cong &\mathrm{lim}_{n}^{1}\mathrm{Hom}_{\alpha -1}\left( C_{n},A\right) 
\text{.}
\end{eqnarray*}%
Under such an isomorphism we have that 
\begin{equation*}
\mathrm{Ph}^{\alpha }\mathrm{Ext}\left( C,A\right) =s_{1}\mathrm{Ph}^{\alpha
-1}\mathrm{Ext}\left( C,A\right)
\end{equation*}%
corresponds to%
\begin{equation*}
\mathrm{lim}_{n}^{1}\mathrm{Hom}_{\alpha }\left( C_{n},A\right) =\mathrm{lim}%
_{n}^{1}\mathrm{Ran}(\mathrm{Hom}\left( P_{n}^{\alpha }C,A\right)
\rightarrow \mathrm{Hom}\left( C,A\right) )
\end{equation*}%
which in turn corresponds to%
\begin{equation*}
\bigcap_{n\in \omega }\mathrm{\mathrm{Ker}}\left( \mathrm{Ext}\left(
C,A\right) \rightarrow \mathrm{Ext}\left( P_{n}^{\alpha }C,A\right) \right) 
\text{.}
\end{equation*}

(4) If $C$ has tree length at most $\alpha $, then it has projective length
at most $\alpha $ by Lemma \ref{Lemma:wedge-projective}. Conversely, suppose
that $C$ has projective length at most $\alpha $. Then we have that $\mathrm{%
Ph}^{\alpha }\mathrm{Ext}\left( C,P^{\alpha }C\right) =0$ and hence the
short exact sequence $\mathfrak{R}^{\alpha }C$ splits, and $C$ is a direct
summands of $P^{\alpha }C$. This shows that $C$ has tree length at most $%
\alpha $.

(5) We have that $B$ is the wedge sum%
\begin{equation*}
\bigoplus_{\left( B_{1},\psi _{i}\right) }B_{\left( i;0\right) }
\end{equation*}%
of a sequence $\left( B_{\left( i;0\right) }\right) _{i\in \omega }$ of
modules of plain tree length $\left( \alpha -1\right) _{i}$ over a finite
module $B_{1}$ with respect to homomorphisms $\psi _{i}:B_{1}\rightarrow
B_{\left( i;0\right) }$. After replacing $A$ with the submodule generated by 
$A$ and $B_{1}$, and $d_{0}$ with a larger integer, we can assume that $%
B_{1}\subseteq A$. For every $n\in \omega $ let $\varphi _{n}:B_{\left(
n;0\right) }\rightarrow C$ be the homomorphism induced by $\varphi
:B\rightarrow C$. Then we have that there exists a homomorphism $\overline{%
\varphi }_{n}:B_{\left( n;0\right) }\rightarrow P_{d_{n}}^{\left( \alpha
-1\right) _{n}}C$ for some $d_{n}\geq d_{0}$ that maps $A_{n}:=A\cap
B_{\left( n;0\right) }$ to $C_{d_{n}}$ and such that the $\varphi _{n}$ is
equal to the composition of $\overline{\varphi }_{n}$ with the canonical
homomorphism $P_{d_{n}}^{\left( \alpha -1\right) _{n}}C$. If $n_{0}\in
\omega $ is such that $A_{n}=B_{n}$ for $n\geq n_{0}$, then $\left( 
\overline{\varphi }_{n}\right) $ induce a homomorphism $\overline{\varphi }%
:B\rightarrow P_{d_{n_{0}}}^{\alpha }C$ that maps $A$ to $C_{d_{n_{0}}}$.

(6) The fact that $\mathfrak{R}_{\alpha }C$ represents an element of $%
\mathrm{Ph}^{\alpha }\mathrm{Ext}\left( C,P^{\alpha }C\right) $ follows from
(3). The fact that $\mathfrak{R}_{\alpha }C$ is $\mathcal{E}_{\alpha }$%
-exact follows from (5) and the fact that $P_{n}^{\alpha }C$ has plain tree
length at most $\alpha $ for every $n\in \omega $. By definition, we have
that $P^{\alpha }C$ is $\mathcal{E}_{\alpha }$-projective.

(7) This is an immediate consequence of (6).

(8) Let $\mathfrak{S}$ be the exact sequence $A\rightarrow X\rightarrow C$.\
If $\mathfrak{S}$ represents an element of \textrm{Ph}$^{\alpha }\mathrm{Ext}%
\left( C,A\right) $, then by (3) we have that $\mathfrak{S}$ represents an
element of%
\begin{equation*}
\mathrm{\mathrm{Ker}}\left( \mathrm{Ext}\left( C,A\right) \rightarrow 
\mathrm{Ext}(P_{n}^{\alpha _{n}}C,A)\right) \text{.}
\end{equation*}%
If $B$ has plain tree length at most $\alpha $ and $\varphi :B\rightarrow C$
is a homomorphism, then by (5) $\varphi $ factors through $P_{n}^{\alpha
_{n}}C\rightarrow C$ for some $n\in \omega $. Thus, $\mathfrak{S}$
represents an element of%
\begin{equation*}
\mathrm{\mathrm{Ker}}\left( \mathrm{Ext}\left( C,A\right) \rightarrow 
\mathrm{Ext}\left( B,A\right) \right) \text{.}
\end{equation*}%
As this holds for every countable flat module $B$ of plain tree length at
most $\alpha $, $\mathfrak{S}$ is $\mathcal{E}_{\alpha }$-exact. Conversely,
if $\mathfrak{S}$ is $\mathcal{E}_{\alpha }$-exact, then since $%
P_{n}^{\alpha _{n}}C$ has plain tree length at most $\alpha _{n}$, $%
\mathfrak{S}$ represents an element of%
\begin{equation*}
\mathrm{\mathrm{Ker}}\left( \mathrm{Ext}\left( C,A\right) \rightarrow 
\mathrm{Ext}(P_{n}^{\alpha _{n}}C,A)\right) \text{.}
\end{equation*}%
As this holds for every $n\in \omega $, we have that $\mathfrak{S}$
represents an element of $\mathrm{Ph}^{\alpha }\mathrm{Ext}\left( C,A\right) 
$ by (3).
\end{proof}

\begin{corollary}
\label{Corollary:derived-functor}Fix $\alpha <\omega _{1}$. The functor 
\begin{equation*}
\mathrm{Hom}:\left( \mathbf{Flat}\left( R\right) ,\mathcal{E}_{\alpha
}\right) \times \left( \mathbf{Flat}\left( R\right) ,\mathcal{E}_{\alpha
}\right) \rightarrow \boldsymbol{\Pi }(\mathbf{Flat}\left( R\right) )
\end{equation*}%
has a total right derived functor $\mathrm{RHom}_{\mathcal{E}_{\alpha
}}^{\bullet }$. We can then define $\mathrm{Ext}_{\mathcal{E}_{\alpha
}}^{n}:=\mathrm{H}^{n}\circ \mathrm{RHom}_{\mathcal{E}_{\alpha }}^{\bullet }$
for $n\in \mathbb{Z}$. Then we have that $\mathrm{Ext}_{\mathcal{E}_{\alpha
}}^{0}\cong \mathrm{Hom}$, $\mathrm{Ext}_{\mathcal{E}_{\alpha }}^{1}\cong 
\mathrm{Ph}^{\alpha }\mathrm{Ext}$, and, for $n\in \mathbb{Z}\setminus
\left\{ 0,1\right\} $, $\mathrm{Ext}_{\mathcal{E}_{\alpha }}^{n}=0$.
\end{corollary}

\begin{proof}
This follows from Theorem \ref{Theorem:phantom-projective} by Proposition %
\ref{Proposition:derived-functor}.
\end{proof}

\begin{corollary}
\label{Corollary:alpha-projective}Let $C$ be a countable flat module. The
following assertions are equivalent:

\begin{enumerate}
\item $C$ projective in $\left( \mathbf{Flat}\left( R\right) ,\mathcal{E}%
_{\alpha }\right) $;

\item $C$ has projective length at most $\alpha $;

\item $C$ has tree length at most $\alpha $;

\item $C$ has $A$-projective length at most $\alpha $ for every $n\in \omega 
$ and countable flat module $A$ of plain tree length at most $\alpha _{n}$.
\end{enumerate}
\end{corollary}

\begin{proof}
This follows from Theorem \ref{Theorem:phantom-projective}.
\end{proof}

\subsection{Tree* length}

We now define a modified version of the notion of tree length, which we call
tree* length, obtained by replacing finite modules with finite-rank modules.
Thus, we say that a countable flat module has plain tree* length $1$ if and
only if it has finite rank. We now define by recursion on $\alpha $ the
notion of (plain) tree* length at most $\alpha $.

\begin{definition}
Let $C$ be a countable flat module, and $2\leq \alpha <\omega _{1}$:

\begin{itemize}
\item if $\alpha $ is a successor, then $C$ has plain tree* length at most $%
\alpha $ if and only if there exist countable flat modules $M_{\left(
i;0\right) }$ of plain tree* length at most $\left( \alpha -1\right) _{i}$,
a finite-rank module $M_{1}$, and injective module homomorphisms $\psi
_{i}:M_{1}\rightarrow M_{\left( i;0\right) }$ with pure image such that $C$
is isomorphic to the wedge sum%
\begin{equation*}
\bigoplus_{\left( M_{1},\psi _{i}\right) }M_{\left( i;0\right) }\text{;}
\end{equation*}

\item $C$ has tree* length at most $\alpha $ if and only if $C$ is a direct
summand of a direct sum of modules $C_{n}$ of plain tree* length at most $%
\alpha _{n}$ for $n<\omega $.
\end{itemize}
\end{definition}

\begin{definition}
For $\alpha <\omega _{1}$, we let:

\begin{itemize}
\item $\mathcal{S}_{\alpha }^{\ast }$ be the class of countable flat modules
of plain tree* length at most $\alpha $;

\item $\mathcal{E}_{\alpha }^{\ast }$ be the exact structure on $\mathbf{Flat%
}\left( R\right) $ projectively generated $\mathcal{S}_{\alpha _{n}}$ for $%
n\in \omega $.
\end{itemize}
\end{definition}

The same proofs as Proposition \ref{Proposition:tree-presheaf} and Lemma \ref%
{Lemma:wedge-projective} yield the following:

\begin{proposition}
\label{Proposition:tree-presheaf*}Suppose that $C$ is a countable flat
module. Then the following assertions are equivalent:

\begin{enumerate}
\item $C$ has plain tree* length at most $\alpha $;

\item $C$ is isomorphic to a direct summand of a colimit of a presheaf of
finite-rank flat modules over $I_{\alpha }^{\mathrm{plain}}$;

\item $C$ is isomorphic to a direct summand of a colimit of a presheaf of
finite-rank flat modules over a countable well-founded tree of rank at most $%
\alpha $.
\end{enumerate}

Furthermore, the following assertions are equivalent:

\begin{enumerate}
\item $C$ has tree* length at most $\alpha $;

\item $C$ is isomorphic to a direct summand of a colimit of a presheaf of
finite-rank flat modules over $I_{\alpha }$;

\item $C$ is isomorphic to a direct summand of a colimit of a presheaf of
finite-rank flat modules over a countable well-founded forest of rank at
most $\alpha $.
\end{enumerate}
\end{proposition}

\begin{lemma}
\label{Lemma:wedge*-projective}Suppose that $C$ is a countable flat module,
and $\alpha <\omega _{1}$:

\begin{enumerate}
\item if $C$ has plain tree* length at most $\alpha $, then $C$ has plain
projective length at most $\alpha $;

\item if $C$ has tree* length at most $\alpha $, then $C$ has projective
length at most $\alpha $.
\end{enumerate}
\end{lemma}

Let $C$ be a countable flat module. Write $C=\mathrm{co\mathrm{lim}}%
_{n}C_{n} $ for some \emph{finite-rank} pure submodules $C_{n}$ of $C$. We
define as before by recursion for every countable successor ordinal $\alpha $
countable flat modules $Q_{n}^{\alpha }C$ of plain tree* length at most $%
\alpha $ containing $C_{n}$ as a pure submodule. We then let $\mathfrak{R}%
_{\ast }^{\alpha }C$ be the short exact sequence%
\begin{equation*}
0\rightarrow Q^{\alpha }C\rightarrow Q^{\alpha }C\rightarrow C\rightarrow 0
\end{equation*}%
realizing $C$ as $\mathrm{co\mathrm{lim}}_{n}Q_{n}^{\alpha _{n}}C$. The same
proof as Theorem \ref{Theorem:phantom-projective} yields.

\begin{theorem}
\label{Theorem:phantom-projective*}Let $C$ and $A$ be countable flat
modules. Let $\left( C_{n}\right) $ be an increasing sequence of finite-rank
pure submodules of $C$ with union equal to $C$. Define%
\begin{equation*}
\mathrm{Hom}\left( C_{n},A\right) :=\mathrm{Hom}_{1}\left( C_{n},A\right)
\end{equation*}%
and recursively%
\begin{equation*}
\mathrm{Hom}_{\alpha }\left( C_{n},A\right) :=\bigcap_{k\geq n}\mathrm{Ran}(%
\mathrm{Hom}_{\alpha _{k}}\left( C_{k},A\right) \rightarrow \mathrm{Hom}%
\left( C,A\right) )\text{.}
\end{equation*}%
Then we have that $\left( \mathrm{Hom}\left( C_{n},A\right) \right) $ has $(%
\mathrm{Hom}_{\alpha _{n}}\left( C_{n},A\right) )$ as its $\alpha $-th
derived tower.

\begin{enumerate}
\item for $\alpha <\omega _{1}$ and $n<\omega $%
\begin{equation*}
\mathrm{Hom}_{\alpha }\left( C_{n},A\right) =\mathrm{Hom}\left(
Q_{n}^{\alpha }C|C,A\right) :=\mathrm{Ran}\left( \mathrm{Hom}\left(
Q_{n}^{\alpha }C,A\right) \rightarrow \mathrm{Hom}\left( C,A\right) \right) 
\text{;}
\end{equation*}

\item $Q_{n}^{\alpha }\left( C\right) $ has plain projective length at most $%
\alpha $;

\item 
\begin{equation*}
\mathrm{Ph}^{\alpha }\mathrm{Ext}\left( C,A\right) =\bigcap_{n\in \omega }%
\mathrm{\mathrm{Ker}}\left( \mathrm{Ext}\left( C,A\right) \rightarrow 
\mathrm{Ext}(Q_{n}^{\alpha _{n}}C,A)\right) \cong \mathrm{lim}_{n}^{1}%
\mathrm{Hom}_{\alpha _{n}}\left( C,A\right)
\end{equation*}

\item a countable flat module has tree* length at most $\alpha $ if and only
if it has projective length at most $\alpha $;

\item for every countable flat module $B$ of plain projective length $\alpha 
$ and finite-rank submodule $A\subseteq B$, for every homomorphism $\varphi
:B\rightarrow C$ that maps $A$ to $C_{d_{0}}$ there exists $d\geq d_{0}$ and
a homomorphism $\overline{\varphi }:B\rightarrow Q_{d}^{\alpha }C$ that maps 
$A$ to $C_{d}\subseteq Q_{d}^{\alpha }C$ and such $\varphi $ is equal to the
composition of $\overline{\varphi }$ and the canonical homomorphism $%
Q_{n}^{\alpha }C\rightarrow C$;

\item $\mathfrak{R}_{\ast }^{\alpha }C$ represents an element of $\mathrm{Ph}%
^{\alpha }\mathrm{Ext}\left( C,Q^{\alpha }C\right) $ and it is a projective
resolution of $C$ in the exact category $\left( \mathbf{Flat}\left( R\right)
,\mathcal{E}_{\alpha }^{\ast }\right) $;

\item the exact structures $\mathcal{E}_{\alpha }$ and $\mathcal{E}_{\alpha
}^{\ast }$ coincide.
\end{enumerate}
\end{theorem}

\begin{proof}
By induction on $\alpha $. The case $\alpha =1$ is established in Section %
\ref{Subsection:projective-length}. The inductive step is the same as in the
proof of Theorem \ref{Theorem:phantom-projective}.
\end{proof}

\begin{corollary}
For every $\alpha <\omega _{1}$, the tree length and the tree* length of a
countable flat module coincide, as they are both equal to the projective
length.
\end{corollary}

\section{Extractable modules\label{Section:higher}}

In this section we continue to assume that $R$ is a countable Pr\"{u}fer
domain, and all modules are $R$-modules. For each countable limit ordinal $%
\alpha $ we let $\left( \alpha _{k}\right) $ be an increasing sequence of 
\emph{successor} ordinals cofinal in $\alpha $. If $\alpha $ is a successor
ordinal, we let $\alpha _{k}$ be equal to $\alpha $.

\subsection{Derivation}

Suppose that $C$ is a coreduced flat module. We define, by recursion on $%
\alpha <\omega _{1}$, functorial \emph{pure coreduced }submodules $\sigma
_{\alpha }C\subseteq C$ and flat quotients $\partial _{\alpha }C=C/\sigma
_{\alpha }C$, as well as a countable upward-directed family $\mathcal{B}%
_{\alpha }\left( C\right) $ of pure submodules of $C$, which we call the $%
\alpha $-constituents of $C$.

Let $\mathcal{B}_{1}\left( C\right) $ be the family of finite-rank pure
submodules of $C$. For $B\in \mathcal{B}_{1}\left( C\right) $, we can write 
\begin{equation*}
B=\Sigma B\oplus \Phi B
\end{equation*}%
by Lemma \ref{Lemma:coreduced-radical}. We let $\left( B_{1}^{k}C\right) $
be a cofinal increasing sequence in $\mathcal{B}_{1}\left( C\right) $.
Define 
\begin{equation*}
\sigma _{1}C:=\mathrm{co\mathrm{lim}}_{B\in \mathcal{B}_{1}\left( C\right)
}\Sigma B=\mathrm{co\mathrm{lim}}_{k}\Sigma B_{1}^{k}C
\end{equation*}%
and%
\begin{equation*}
\partial _{1}C=C/\sigma _{1}C=\mathrm{co\mathrm{lim}}_{B\in \mathcal{B}%
_{1}\left( C\right) }\Phi B=\mathrm{co\mathrm{lim}}_{k}\Phi B_{1}^{k}C\text{.%
}
\end{equation*}%
Define $\mathcal{B}_{\alpha +1}\left( C\right) $ to be the collection of
preimages of finite-rank pure submodules of $\partial _{\alpha }C$ under the
quotient map $C\rightarrow \partial _{\alpha }C$. Define then 
\begin{equation*}
\sigma _{\alpha +1}C:=\mathrm{co\mathrm{lim}}_{B\in \mathcal{B}_{\alpha
+1}\left( C\right) }\Sigma B=\mathrm{co\mathrm{lim}}_{k}\Sigma B_{\alpha
+1}^{k}C
\end{equation*}%
and%
\begin{equation*}
\partial _{\alpha +1}C:=C/\sigma _{\alpha +1}C\cong \mathrm{co\mathrm{lim}}%
_{B\in \mathcal{B}_{\alpha +1}\left( C\right) }\Phi B\cong \mathrm{co\mathrm{%
lim}}_{k}\Phi B_{\alpha +1}^{k}C\text{.}
\end{equation*}%
Notice that $\sigma _{\alpha +1}C$ is the preimage of $\sigma _{1}\left(
\partial _{\alpha }C\right) $ under the quotient map $C\rightarrow \partial
_{\alpha }C$. If $\alpha $ is limit, define $\mathcal{B}_{\alpha }\left(
C\right) $ to be the union of $\mathcal{B}_{\beta }\left( C\right) $ for $%
\beta <\alpha $, and $\sigma _{\alpha }C$ to be the union of $\sigma _{\beta
}C$ for $\beta <\alpha $. We have that $\left( B_{\alpha _{k}}^{k}C\right) $
is a cofinal increasing sequence in $\mathcal{B}_{\alpha }\left( C\right) $.
Notice that 
\begin{equation*}
\sigma _{\alpha }C=\mathrm{co\mathrm{lim}}_{B\in \mathcal{B}_{\alpha }\left(
C\right) }\Sigma B=\mathrm{co\mathrm{lim}}_{k}\Sigma B_{\alpha _{k}}^{k}C%
\text{.}
\end{equation*}%
We also set%
\begin{equation*}
\partial _{\alpha }C:=C/\sigma _{\alpha }C\cong \mathrm{co\mathrm{lim}}%
_{B\in \mathcal{B}_{\alpha }\left( C\right) }\Phi B\cong \mathrm{co\mathrm{%
lim}}_{k}\Phi B_{\alpha _{k}}^{k}C\text{.}
\end{equation*}%
For an arbitrary $\alpha $, define now $F_{\alpha }C$ to be the direct sum
of $B_{\alpha _{k}}^{k}C$ for $k\in \omega $. Consider the short exact
sequence%
\begin{equation*}
0\rightarrow L_{\alpha }C\overset{\eta _{\alpha }^{C}}{\rightarrow }%
F_{\alpha }C\overset{\pi _{\alpha }^{C}}{\rightarrow }C\rightarrow 0
\end{equation*}%
that realizes $C$ as the colimit $\mathrm{co\mathrm{lim}}_{k\in \omega
}B_{\alpha _{k}}^{k}C$, where the homomorphism $\pi _{\alpha }^{C}:F_{\alpha
}C\rightarrow C$ is induced by the inclusion maps $B_{\alpha
_{k}}^{k}C\rightarrow C$ for each $k\in \omega $. Notice that $L_{\alpha }C$
is also isomorphic to a direct sum of copies $B_{\alpha _{k}}^{k}C$ for $%
k<\omega $.

\subsection{Extractable length}

We define by induction on $\alpha <\omega _{1}$ what it means for a
coreduced countable flat module to have (plain)\emph{\ extractable length }%
at most $\alpha $.

\begin{definition}
A coreduced countably-presented module $C$:

\begin{itemize}
\item has extractable length $0$ if and only if it is trivial;

\item has plain extractable length at most $\alpha +1$ if and only if there
exists a short exact sequence%
\begin{equation*}
0\rightarrow B\rightarrow C\rightarrow A
\end{equation*}%
where $B$ is a countably-presented module of extractable length at most $%
\alpha $ and $A$ is a countably-presented module with $A/A_{\mathrm{t}}$
finite-rank;

\item has extractable length at most $\alpha $ if and only if it is a direct
summand of a countable direct sum of modules of plain extractable length at
most $\alpha $.
\end{itemize}

We say that a countably-presented flat module is \emph{extractable }if it
has extractable length $\alpha $ for some $\alpha <\omega _{1}$.
\end{definition}

\begin{proposition}
\label{Proposition:constructible-s}Suppose that $C$ is a extractable flat
module. Then for every $\alpha <\omega _{1}$, $\sigma _{\alpha }C$ and $%
\partial _{\alpha }C$ are extractable.
\end{proposition}

\begin{proof}
We prove that the conclusion holds by induction on the extractable length of 
$C$. If 
\begin{equation*}
C=\bigoplus_{n}C_{n}
\end{equation*}%
and the conclusion holds for every $n\in \omega $, then it clearly holds for 
$C$. The same applies if $C$ is a direct summand of a module $D$ for which
the conclusion applies. Suppose now that there is an extension%
\begin{equation*}
0\rightarrow D\rightarrow C\rightarrow E\rightarrow 0
\end{equation*}%
where $E/E_{\mathrm{t}}$ has finite-rank and for every $\alpha <\omega _{1}$
we have that $\sigma _{\alpha }D$ and $\partial _{\alpha }D$ are
extractable. Then for every $\alpha <\omega _{1}$ we have short exact
sequences%
\begin{equation*}
0\rightarrow \sigma _{\alpha }D\rightarrow \sigma _{\alpha }C\rightarrow
E_{\alpha }\rightarrow 0
\end{equation*}%
and%
\begin{equation*}
0\rightarrow \partial _{\alpha }D\rightarrow \partial _{\alpha }D\rightarrow
E_{\alpha }^{\prime }\rightarrow 0
\end{equation*}%
where $E_{\alpha }/E_{\alpha ,\mathrm{t}}$ and $E_{\alpha }^{\prime
}/E_{\alpha ,\mathrm{t}}^{\prime }$ have finite rank. Thus, $\sigma _{\alpha
}C$ and $\partial _{\alpha }D$ are extractable.
\end{proof}

\begin{proposition}
Suppose that $C$ is a countably-presented module of (plain) extractable
length at most $\alpha $, and $E\subseteq C$ has finite rank. Then $C/E$ has
(plain) extractable length at most $\alpha $.
\end{proposition}

\begin{proof}
By induction on $\alpha $. This is obvious for length $0$. Suppose that this
holds for length $\alpha $. If $C$ has plain extractable length at most $%
\alpha +1$, then we have an extension%
\begin{equation*}
0\rightarrow B\rightarrow C\rightarrow A\rightarrow 0
\end{equation*}%
where $B$ has extractable length at most $\alpha $ and $A/A_{\mathrm{t}}$
has finite rank. If $E\subseteq B$ then we have a short exact sequence%
\begin{equation*}
0\rightarrow B/E\rightarrow C/E\rightarrow A\rightarrow 0
\end{equation*}%
and the conclusion follows from the inductive hypothesis. After replacing $E$
with $E/\left( B\cap E\right) $ and $C$ with $C/\left( B\cap E\right) $ we
can assume that $E\cap B=0$. In this case, we have a short exact sequence%
\begin{equation*}
0\rightarrow B\rightarrow C/E\rightarrow A^{\prime }\rightarrow 0
\end{equation*}%
for some module $A^{\prime }$ such that $A^{\prime }/A_{\mathrm{t}}^{\prime
} $ has finite rank. The conclusion again follows.

Suppose now that the conclusion holds for plain extractable length less than 
$\alpha $. If $C$ has extractable length $\alpha $, then $C$ is isomorphic
to the direct sum of a sequence $\left( C_{n}\right) $ where for every $n\in
\omega $, $C_{n}$ has plain extractable length less than $\alpha $. Without
loss of generality we can assume that $E\subseteq C_{0}$, whence the
conclusion follows from the inductive hypothesis.
\end{proof}

\begin{proposition}
\label{Proposition:projective-less}Suppose that $C$ is a extractable
countable flat module. Then the (plain) projective length of $C$ is less
than or equal to the (plain) extractable length of $C$.
\end{proposition}

\begin{proof}
Fix a countable flat module $A$. We prove by induction on $\alpha <\omega
_{1}$ that if $C$ has (plain) extractable length at most $\alpha $, then it
has (plain) $A$-projective length at most $\alpha $. Suppose that $C$ has
plain extractable length at most $\alpha $. Then we have an exact sequence%
\begin{equation*}
0\rightarrow D\rightarrow C\rightarrow E\rightarrow 0
\end{equation*}%
where $E/E_{\mathrm{t}}$ has finite rank and $D$ has extractable length at
most $\alpha -1$. Then for every countable flat module $A$, we have an exact
sequence%
\begin{equation*}
\mathrm{Ext}\left( E,A\right) \rightarrow \mathrm{Ext}\left( C,A\right)
\rightarrow \mathrm{Ext}\left( D,A\right) \rightarrow 0\text{.}
\end{equation*}%
By the inductive hypothesis, we have that $s_{\alpha -1}$\textrm{Ext}$\left(
D,A\right) =0$, whence 
\begin{equation*}
s_{\alpha -1}\mathrm{Ext}\left( C,A\right) \subseteq \mathrm{Ran}(\mathrm{Ext%
}\left( E,A\right) \rightarrow \mathrm{Ext}\left( C,A\right) ).
\end{equation*}%
Since $E/E_{\mathrm{t}}$ has finite rank, $\left\{ 0\right\} \in \boldsymbol{%
\Sigma }_{2}^{0}\left( \mathrm{Ext}\left( E,A\right) \right) $ by Theorem %
\ref{Theorem:structure-rank-1}, and hence 
\begin{equation*}
\left\{ 0\right\} \in \boldsymbol{\Sigma }_{2}^{0}\left( \mathrm{Ran}\left( 
\mathrm{Ext}\left( E,A\right) \rightarrow \mathrm{Ext}\left( C,A\right)
\right) \right)
\end{equation*}%
by Lemma \ref{Lemma:compact-phantom2}. Thus, 
\begin{equation*}
\left\{ 0\right\} \in \mathbf{\Sigma }_{2}^{0}\left( s_{\alpha -1}\mathrm{Ext%
}\left( C,A\right) \right) \text{.}
\end{equation*}%
This shows that $C$ is has plain projective length at most $\alpha $.

Suppose now that $C$ has extractable length at most $\alpha $. Without loss
of generality, we can assume that $C$ is a direct sum of countable coreduced
modules $C_{i}$ for $i\in \omega $ of plain extractable length at most $%
\alpha $.\ Then for every $i\in \omega $ we have that $C_{i}$ has projective
length at most $\alpha $. Therefore the same conclusion holds for $C$.
\end{proof}

\subsection{Extensions by projective modules}

We now characterize the Solecki submodules of $\mathrm{Ext}\left( C,A\right) 
$ in terms of the derivatives of $C$ when $A$ is a projective module.

\begin{theorem}
\label{Theorem:derivation}Let $C$ be a countable flat coreduced module, and $%
A$ be a countable \emph{projective} module. Fix $\alpha <\omega _{1}$.

\begin{enumerate}
\item the homomorphism%
\begin{equation*}
\mathrm{Ext}\left( \partial _{\alpha }C,A\right) \rightarrow \mathrm{Ext}%
\left( C,A\right)
\end{equation*}%
is injective;

\item if $\alpha $ is limit,%
\begin{equation*}
\mathrm{Ext}\left( \partial _{\alpha }C,A\right) =\bigcap_{\beta <\alpha }%
\mathrm{Ext}\left( \partial _{\beta }C,A\right) \text{;}
\end{equation*}

\item for $\alpha $ arbitrary,%
\begin{equation*}
s_{\alpha }\mathrm{Ext}\left( C,A\right) =\mathrm{Ext}\left( \partial
_{\alpha }C,A\right) \text{;}
\end{equation*}

\item $\mathrm{Ext}\left( C,A\right) $ has (plain) Solecki length $\alpha $
if and only if $\partial _{\alpha }C=0$ (and $\partial _{\alpha -1}C$ has
finite rank).
\end{enumerate}
\end{theorem}

\begin{proof}
(1) Since $C$ is coreduced, we have that $\mathrm{Hom}\left( \sigma _{\alpha
}C,A\right) =0$.\ Thus, the conclusion follows from the long exact sequence
relating $\mathrm{Hom}$ and $\mathrm{Ext}$;

(2) If $\alpha $ is limit, we have a short exact sequence%
\begin{equation*}
0\rightarrow \mathrm{lim}_{k}^{1}\mathrm{Hom}\left( \sigma _{\alpha
_{k}}C,A\right) \rightarrow \mathrm{Ext}\left( \sigma _{\alpha }C,A\right)
\rightarrow \mathrm{lim}_{k}\mathrm{Ext}\left( \sigma _{\alpha
_{k}}C,A\right) \rightarrow 0
\end{equation*}%
As $\mathrm{lim}_{k}^{1}\mathrm{Hom}\left( \sigma _{\alpha _{k}}C,A\right)
=0 $, this yields an isomorphism%
\begin{equation*}
\mathrm{Ext}\left( \sigma _{\alpha }C,A\right) \cong \mathrm{lim}_{k}\mathrm{%
Ext}\left( \sigma _{\alpha _{k}}C,A\right) \text{.}
\end{equation*}%
Therefore, we have that%
\begin{eqnarray*}
\bigcap_{k\in \omega }\mathrm{Ext}\left( \partial _{\alpha _{k}}C,A\right)
&=&\mathrm{\mathrm{Ker}}\left( \mathrm{Ext}\left( C,A\right) \rightarrow 
\mathrm{lim}_{k}\mathrm{Ext}\left( \sigma _{\alpha _{k}}C,A\right) \right) \\
&=&\mathrm{\mathrm{Ker}}\left( \mathrm{Ext}\left( C,A\right) \rightarrow 
\mathrm{Ext}\left( \sigma _{\alpha }C,A\right) \right) \\
&=&\mathrm{Ran}\left( \mathrm{Ext}\left( \partial _{\alpha }C,A\right)
\rightarrow \mathrm{Ext}\left( C,A\right) \right) \text{.}
\end{eqnarray*}%
(3) We prove that the conclusion holds by induction on $\alpha $. The case
of limits follows from (2). Suppose that $\alpha $ is a successor ordinal.
Then by the inductive hypothesis, we have that%
\begin{equation*}
s_{\alpha -1}\mathrm{Ext}\left( C,A\right) =\mathrm{Ext}\left( \partial
_{\alpha -1}C,A\right) \text{.}
\end{equation*}%
Thus, we have that%
\begin{eqnarray*}
s_{1}\mathrm{Ext}\left( \partial _{\alpha -1}C,A\right) &=&\bigcap_{k\in
\omega }\mathrm{\mathrm{Ker}}\left( \mathrm{Ext}\left( \partial _{\alpha
-1}C,A\right) \rightarrow \mathrm{Ext}\left( B_{1}^{k}\left( \partial
_{\alpha -1}C\right) ,A\right) \right) \\
&\cong &\mathrm{lim}_{k}^{1}\mathrm{Hom}\left( B_{1}^{k}\left( \partial
_{\alpha -1}C\right) ,A\right) \\
&\cong &\mathrm{lim}_{k}^{1}\mathrm{Hom}\left( \Phi \left( B_{\alpha
-1}^{k}C\right) ,A\right) \\
&\cong &\mathrm{Ext}\left( \mathrm{co\mathrm{lim}}_{k}\Phi \left( B_{\alpha
-1}^{k}C\right) ,A\right) \\
&\cong &\mathrm{Ext}\left( \partial _{\alpha }C,A\right) \text{.}
\end{eqnarray*}

(4) This is a consequence of (3) and Theorem \ref%
{Theorem:structure-rank-1-tris}.
\end{proof}

\begin{corollary}
Suppose that $C$ is a countable flat module. If $\alpha <\omega _{1}$, then $%
\sigma _{\alpha }C$ is the largest coreduced pure submodule of $C$ of $R$%
-projective length at most $\alpha $.
\end{corollary}

The same proof as Theorem \ref{Theorem:derivation} gives the following more
general result.

\begin{theorem}
\label{Theorem:derivation-A}Let $C,A$ are countable flat modules, with $C$
coreduced. Suppose that 
\begin{equation*}
\mathrm{Hom}\left( \sigma _{\beta }C,A\right) =0\text{ and }\mathrm{Hom}%
\left( \Sigma \left( B_{1}^{k}\left( \partial _{\beta }C\right) \right)
,A\right) =0
\end{equation*}%
for every $\beta <\omega _{1}$ and for all but finitely many $k\in \omega $.
Then for every $\alpha <\omega _{1}$, we have:

\begin{enumerate}
\item the homomorphism%
\begin{equation*}
\mathrm{Ext}\left( \partial _{\alpha }C,A\right) \rightarrow \mathrm{Ext}%
\left( C,A\right)
\end{equation*}%
is injective;

\item if $\alpha $ is limit%
\begin{equation*}
\mathrm{Ext}\left( \partial _{\alpha }C,A\right) =\bigcap_{\beta <\alpha }%
\mathrm{Ext}\left( \partial _{\beta }C,A\right) \text{;}
\end{equation*}

\item for arbitrary $\alpha $,%
\begin{equation*}
s_{\alpha }\mathrm{Ext}\left( C,A\right) =\mathrm{Ext}\left( \partial
_{\alpha }C,A\right) \text{;}
\end{equation*}

\item $C$ has (plain) $A$-projective length at most $\alpha $ if and only if 
$\partial _{\alpha }C=0$ (and $\mathrm{Ext}\left( \partial _{\alpha
-1}C,A\right) \rightarrow \mathrm{Ext}\left( B_{1}^{k}\left( \partial
_{\alpha -1}C\right) ,A\right) $ is an isomorphism for some $k\in \omega $).
\end{enumerate}
\end{theorem}

\begin{proof}
(1) This follows from the hypothesis that $\mathrm{Hom}\left( \sigma
_{\alpha }C,A\right) =0$ and the long exact sequence relating $\mathrm{Hom}$
and $\mathrm{Ext}$;

(2) If $\alpha $ is limit, we have a short exact sequence%
\begin{equation*}
0\rightarrow \mathrm{lim}_{k}^{1}\mathrm{Hom}\left( \sigma _{\alpha
_{k}}C,A\right) \rightarrow \mathrm{Ext}\left( \sigma _{\alpha }C,A\right)
\rightarrow \mathrm{lim}_{k}\mathrm{Ext}\left( \sigma _{\alpha
_{k}}C,A\right) \rightarrow 0\text{.}
\end{equation*}%
As $\mathrm{lim}_{k}^{1}\mathrm{Hom}\left( \sigma _{\alpha _{k}}C,A\right)
=0 $, this yields an isomorphism%
\begin{equation*}
\mathrm{Ext}\left( \sigma _{\alpha }C,A\right) \cong \mathrm{lim}_{k}\mathrm{%
Ext}\left( \sigma _{\alpha _{k}}C,A\right) \text{.}
\end{equation*}%
Therefore, we have that%
\begin{eqnarray*}
\bigcap_{k\in \omega }\mathrm{Ext}\left( \partial _{\alpha _{k}}C,A\right)
&=&\mathrm{\mathrm{Ker}}\left( \mathrm{Ext}\left( C,A\right) \rightarrow 
\mathrm{lim}_{k}\mathrm{Ext}\left( \sigma _{\alpha _{k}}C,A\right) \right) \\
&=&\mathrm{\mathrm{Ker}}\left( \mathrm{Ext}\left( C,A\right) \rightarrow 
\mathrm{Ext}\left( \sigma _{\alpha }C,A\right) \right) \\
&=&\mathrm{Ran}\left( \mathrm{Ext}\left( \partial _{\alpha }C,A\right)
\rightarrow \mathrm{Ext}\left( C,A\right) \right) \text{.}
\end{eqnarray*}

(3) We prove that the conclusion holds by induction on $\alpha $. The case
of limits follows from (2). Suppose that $\alpha $ is a successor ordinal.
Then by the inductive hypothesis, we have that%
\begin{equation*}
s_{\alpha -1}\mathrm{Ext}\left( C,A\right) =\mathrm{Ext}\left( \partial
_{\alpha -1}C,A\right) \text{.}
\end{equation*}%
Thus, using that%
\begin{equation*}
\mathrm{Hom}\left( \Sigma \left( B_{1}^{k}\left( \partial _{\alpha
-1}C\right) \right) ,A\right) =0
\end{equation*}%
for all but finitely many $k\in \omega $, we have%
\begin{eqnarray*}
s_{1}\mathrm{Ext}\left( \partial _{\alpha -1}C,A\right) &=&\bigcap_{k\in
\omega }\mathrm{\mathrm{Ker}}\left( \mathrm{Ext}\left( \partial _{\alpha
-1}C,A\right) \rightarrow \mathrm{Ext}\left( B_{1}^{k}\left( \partial
_{\alpha -1}C\right) ,A\right) \right) \\
&\cong &\mathrm{lim}_{k}^{1}\mathrm{Hom}\left( B_{1}^{k}\left( \partial
_{\alpha -1}C\right) ,A\right) \\
&\cong &\mathrm{lim}_{k}^{1}\mathrm{Hom}\left( \Phi \left( B_{\alpha
-1}^{k}C\right) ,A\right) \\
&\cong &\mathrm{Ext}\left( \mathrm{co\mathrm{lim}}_{k}\Phi \left( B_{\alpha
-1}^{k}C\right) ,A\right) \\
&\cong &\mathrm{Ext}\left( \partial _{\alpha }C,A\right) \text{.}
\end{eqnarray*}

(4) This is a consequence of (3) and Lemma \ref{Lemma:monomorphic-tower}.
\end{proof}

\subsection{Structure}

We collect in the following theorem a number of results we will establish by
induction on countable ordinals. As in the inductive step we will need the
inductive hypotheses to include all these statements, we need to run a
single inductive proof of all these results simultaneously.

\begin{theorem}
\label{Theorem:structure-1}For every countable ordinal $\alpha $, and
coreduced countable extractable flat module $C$:

\begin{enumerate}
\item if $C$ has (plain) extractable length at most $\alpha $, then $C$ has
(plain) projective length at most $\alpha $;

\item if $\alpha $ is a successor, then $C$ has plain extractable length at
most $\alpha $ if and only if $C$ has plain $R$-projective length at most $%
\alpha $;

\item if $\alpha $ is a successor, then each $\alpha $-constituent of $C$
has plain extractable length at most $\alpha $;

\item if $E$ is a pure coreduced submodule of $C$ of plain extractable
length at most $\alpha $, then there exists an $\alpha $-constituent $B$ of $%
C$ such that $E\subseteq \Sigma B$;

\item if $\left( E_{k}\right) $ is an increasing sequence of coreduced
submodules of $C$ of plain extractable length at most $\alpha $ containing $%
\sigma _{\alpha }C$ with union equal to $C$, then for every $\alpha $%
-constituent $B$ of $C$ there exists $\ell \in \omega $ such that $\Sigma
B\subseteq E_{\ell }$;

\item for every countable flat module $A$,%
\begin{equation*}
s_{\alpha }\mathrm{Ext}\left( C,A\right) =\bigcap_{B}\mathrm{\mathrm{Ker}}%
\left( \mathrm{Ext}\left( C,A\right) \rightarrow \mathrm{Ext}\left(
B,A\right) \right) =\bigcap_{E}\mathrm{\mathrm{Ker}}\left( \mathrm{Ext}%
\left( C,A\right) \rightarrow \mathrm{Ext}\left( E,A\right) \right)
\end{equation*}%
where $B$ ranges among the $\alpha $-constituents of $C$ and $E$ ranges
among the pure coreduced submodules of $C$ of plain extractable length at
most $\alpha $;

\item we have that%
\begin{equation*}
0\rightarrow L_{\alpha }C\rightarrow F_{\alpha }C\rightarrow C\rightarrow 0
\end{equation*}%
is a phantom exact sequence of order $\alpha $, where $L_{\alpha }C$ and $%
F_{\alpha }C$ have extractable length at most $\alpha $;

\item $C$ has (plain) extractable length at most $\alpha $ if and only if it
has (plain) projective length at most $\alpha $;

\item $C$ has (plain) extractable length at most $\alpha $ if and only if it
has (plain) $R$-projective length at most $\alpha $;

\item for every countable flat module $A$, the short exact sequence 
\begin{equation*}
0\rightarrow L_{\alpha }C\rightarrow F_{\alpha }C\rightarrow C\rightarrow 0
\end{equation*}
induces an isomorphism%
\begin{equation*}
s_{\alpha }\mathrm{Ext}\left( C,A\right) \cong \frac{\mathrm{Hom}\left(
L_{\alpha }C,A\right) }{\mathrm{Hom}\left( F_{\alpha }C|L_{\alpha
}C,A\right) }
\end{equation*}%
where $\mathrm{Hom}\left( F_{\alpha }C|L_{\alpha }C,A\right) $ is the
Polishable submodule of $\mathrm{Hom}\left( L_{\alpha }C,A\right) $
comprising the homomorphisms $L_{\alpha }C\rightarrow A$ that extend to a
homomorphism $F_{\alpha }C\rightarrow A$.
\end{enumerate}
\end{theorem}

\begin{proof}
We can assume without loss of generality that $C$ is coreduced. We prove
that the conclusion holds by induction on $\alpha <\omega _{1}$. Suppose
that the conclusion holds for $\alpha $. We set $S_{\alpha }^{k}C=B_{\alpha
}^{k}C/\sigma _{\alpha -1}C=B_{1}^{k}\partial _{\alpha -1}C$ for a successor
ordinal $\alpha <\omega _{1}$ and $k<\omega $.

(1) This is the content of Proposition \ref{Proposition:projective-less}.

(2) Suppose that $\alpha $ is a successor and $\left\{ 0\right\} \in 
\boldsymbol{\Sigma }_{2}^{0}\left( s_{\alpha -1}\mathrm{Ext}\left(
C,R\right) \right) $. By Theorem \ref{Theorem:derivation}, we have that%
\begin{equation*}
s_{\alpha -1}\mathrm{Ext}\left( C,R\right) =\mathrm{Ext}\left( \partial
_{\alpha -1}C,R\right) \text{.}
\end{equation*}%
Therefore, by Theorem \ref{Theorem:structure-rank-1} we have that $\partial
_{\alpha -1}C$ is finite-rank. Furthermore, we have a short exact sequence%
\begin{equation*}
0\rightarrow \sigma _{\alpha -1}C\rightarrow C\rightarrow \partial _{\alpha
-1}C\rightarrow 0
\end{equation*}%
where $\sigma _{\alpha -1}C$ has extractable length at most $\alpha -1$.
This shows that $C$ has plain extractable length at most $\alpha $. The
converse implication follows from (1).

(3) If $B$ is an $\alpha $-constituent of $C$, then we have an extension%
\begin{equation*}
0\rightarrow \sigma _{\alpha -1}C\rightarrow B\rightarrow B/\sigma _{\alpha
-1}C\rightarrow 0
\end{equation*}%
where $B/\sigma _{\alpha -1}C$ is a finite-rank submodule of $\partial
_{\alpha -1}C$ by definition, and $\sigma _{\alpha -1}C$ has extractable
length at most $\alpha -1$ by the inductive hypothesis;

(4) Suppose that $E\subseteq C$ is a coreduced submodule of plain
extractable length at most $\alpha $. Then we have that $\sigma _{\alpha
-1}E\subseteq \sigma _{\alpha -1}C$ and $\partial _{\alpha -1}E$ is
finite-rank. Thus the inclusion $E\rightarrow C$ induces a map $\partial
_{\alpha }E\rightarrow \partial _{\alpha }C$ whose image is contained in $%
S_{\alpha }^{k}C$ for some $k\in \omega $. This shows that $E$ is contained
in $B_{\alpha }^{k}C$ and hence in $\Sigma \left( B_{\alpha }^{k}C\right) $
(being coreduced).

(5) Suppose that $\left( E_{k}\right) $ is an increasing sequence of
submodules of $C$ of plain projective length at most $\alpha $ with union
equal to $C$ such that $\sigma _{\alpha -1}C\subseteq E_{\ell }$ for every $%
\ell \in \omega $. We have that $\partial _{\alpha -1}C$ is the union of the
images $\partial _{\alpha -1}E_{\ell }$ of the maps $E_{\ell }\rightarrow
\partial _{\alpha -1}C$ for $\ell \in \omega $. As $E_{\ell }$ is coreduced
of plain projective length at most $\alpha $, we have that $\partial
_{\alpha -1}E_{\ell }$ has finite rank for every $\ell \in \omega $.
Therefore, for every $k\in \omega $ there exists $\ell \in \omega $ such
that $S_{\alpha }^{k}C\subseteq \partial _{\alpha -1}E_{\ell }$ and hence $%
B_{\alpha }^{k}C\subseteq E_{\ell }$.

(6) By (3) and (4), the second and third term are equal. Furthermore, by (1)
and (3), we have that $\mathrm{Ext}\left( B_{\alpha _{k}}^{k}C,A\right) $
has Solecki length at most $\alpha _{k}$ for every $k\in \omega $. Whence we
have%
\begin{equation*}
G:=\bigcap_{k\in \omega }\mathrm{\mathrm{Ker}}\left( \mathrm{Ext}\left(
C,A\right) \rightarrow \mathrm{Ext}\left( B_{\alpha _{k}}^{k}C,A\right)
\right)
\end{equation*}%
has Borel rank at most $\alpha $ in $\mathrm{Ext}\left( C,A\right) $, and
contains $s_{\alpha }\mathrm{Ext}\left( C,A\right) =s_{1}s_{\alpha -1}%
\mathrm{Ext}\left( C,A\right) $. We now prove that $G$ is contained in $%
s_{1}s_{\alpha -1}\mathrm{Ext}\left( C,A\right) $. If $\alpha $ is limit,
this follows from the inductive hypothesis, considering that in this case%
\begin{equation*}
s_{\alpha }\mathrm{Ext}\left( C,A\right) =\bigcap_{k}s_{\alpha _{k}}\mathrm{%
Ext}\left( C,A\right) .
\end{equation*}%
Suppose now that $\alpha $ is a successor. By the inductive hypothesis (10)
applied to $\alpha -1$, the short exact sequence%
\begin{equation*}
L_{\alpha -1}C\rightarrow F_{\alpha -1}C\rightarrow C
\end{equation*}%
induces an isomorphism%
\begin{equation*}
s_{\alpha -1}\mathrm{Ext}\left( C,A\right) \cong \frac{\mathrm{Hom}\left(
L_{\alpha -1}C,A\right) }{\mathrm{Hom}\left( F_{\alpha -1}C|L_{\alpha
-1}C,A\right) }
\end{equation*}%
For $k\in \omega $, define 
\begin{equation*}
\hat{B}_{\alpha -1}^{k}C:=\left\{ x\in F_{\alpha -1}C:x+L_{\alpha -1}C\in
B_{\alpha -1}^{k}C\right\} \text{.}
\end{equation*}%
Under such an isomorphism $G$ corresponds to%
\begin{equation*}
\frac{\mathrm{Hom}^{1}\left( L_{\alpha -1}C,A\right) }{\mathrm{Hom}\left(
F_{\alpha -1}C|L_{\alpha -1}C,A\right) }
\end{equation*}%
where $\mathrm{Hom}^{1}\left( L_{\alpha -1}C,A\right) $ is the submodule
comprising the homomorphisms $L_{\alpha -1}C\rightarrow A$ that have an
extension $\hat{B}_{\alpha -1}^{k}C\rightarrow A$ for every $k\in \omega $.\
Therefore, it remains to prove that%
\begin{equation*}
\frac{\mathrm{Hom}^{1}\left( L_{\alpha -1}C,A\right) }{\mathrm{Hom}\left(
F_{\alpha -1}C|L_{\alpha -1}C,A\right) }\subseteq s_{1}\left( \frac{\mathrm{%
Hom}\left( L_{\alpha -1}C,A\right) }{\mathrm{Hom}\left( F_{\alpha
-1}C|L_{\alpha -1}C,A\right) }\right) \text{.}
\end{equation*}%
By Lemma \ref{Lemma:1st-Solecki} and Lemma \ref{Lemma:dense-PExtJ}, it
suffices to prove that if $V\subseteq \mathrm{Hom}\left( F_{\alpha
-1}C|L_{\alpha -1}C,A\right) $ is an open neighborhood of $0$ in $\mathrm{Hom%
}\left( F_{\alpha -1}C|L_{\alpha -1}C,A\right) $, then the closure $%
\overline{V}^{\mathrm{Hom}\left( L_{\alpha -1}C,A\right) }$ of $V\ $inside $%
\mathrm{Hom}\left( L_{\alpha -1}C,A\right) $ contains an open neighborhood
of $0$ in $\mathrm{Hom}^{1}\left( L_{\alpha -1}C,A\right) $. Suppose that $%
V\subseteq \mathrm{Hom}\left( F_{\alpha -1}C|L_{\alpha -1}C,A\right) $ is an
open neighborhood of $0$. Thus, there exist $x_{0},\ldots ,x_{n}\in
F_{\alpha -1}C$ such that 
\begin{equation*}
\left\{ \varphi |_{L_{\alpha -1}C}:\varphi \in \mathrm{Hom}\left( F_{\alpha
-1}C,A\right) \text{ and }\forall i\leq n\text{, }\varphi \left(
x_{i}\right) =0\right\} \subseteq V\text{.}
\end{equation*}%
Let $k\in \omega $ be such that $\left\{ x_{0},\ldots ,x_{n}\right\}
\subseteq \hat{B}_{\alpha -1}^{k}C$. Thus, we have that%
\begin{equation*}
U=\left\{ \varphi \in \mathrm{Hom}^{1}\left( L_{\alpha -1},A\right) :\exists 
\hat{\varphi}\in \mathrm{Hom}(\hat{B}_{\alpha -1}^{k}C,A),\forall i\leq n,%
\hat{\varphi}\left( x_{i}\right) =0,\hat{\varphi}|_{L_{\alpha -1}C}=\varphi
\right\}
\end{equation*}%
is an open neighborhood of $0$ in $\mathrm{Hom}^{1}\left( L_{\alpha
-1}C,A\right) $. We claim that $U\subseteq \overline{V}^{\mathrm{Hom}\left(
L_{\alpha -1}C,A\right) }$. Indeed, suppose that $\varphi _{0}\in U$ and let 
$W$ be an open neighborhood of $\varphi _{0}$ in $\mathrm{Hom}\left(
L_{\alpha -1}C,A\right) $. There exist $y_{0},\ldots ,y_{\ell }\in L_{\alpha
-1}C$ such that%
\begin{equation*}
\left\{ \varphi \in \mathrm{Hom}\left( L_{\alpha -1}C,A\right) :\forall
i\leq \ell \text{, }\varphi \left( y_{i}\right) =\varphi _{0}\left(
y_{i}\right) \right\} \subseteq W\text{.}
\end{equation*}%
Define%
\begin{equation*}
F_{\alpha -1}C_{\Sigma }:=\bigoplus_{k}\Sigma B_{\left( \alpha -1\right)
_{k}}^{k}C
\end{equation*}%
and%
\begin{equation*}
F_{\alpha -1}C_{\Phi }:=\bigoplus_{k}\Phi B_{\left( \alpha -1\right)
_{k}}^{k}C
\end{equation*}%
Then we have that, by definition, 
\begin{equation*}
F_{\alpha -1}C=F_{\alpha -1}C_{\Sigma }\oplus F_{\alpha -1}C_{\Phi }\text{.}
\end{equation*}%
We need to prove that $V\cap W\neq \varnothing $. To this purpose, it
suffices to show that there exists $\varphi \in \mathrm{Hom}(F_{\alpha
-1}C,A)$ such that $\varphi \left( y_{i}\right) =\varphi _{0}\left(
y_{i}\right) $ for $i\leq \ell $ and $\varphi \left( x_{i}\right) =0$ for $%
i\leq n$. Since $\varphi _{0}\in U$, there exists $\hat{\varphi}_{0}\in 
\mathrm{Hom}(\hat{B}_{\alpha -1}^{k}C,A)$ such that $\hat{\varphi}_{0}\left(
x_{i}\right) =0$ for $i\leq n$ and $\hat{\varphi}_{0}|_{L_{\alpha
-1}C}=\varphi _{0}$. For $i\leq \ell $ and $j\leq n$, we can write%
\begin{equation*}
y_{i}=y_{i}^{\Phi }+y_{i}^{\Sigma }
\end{equation*}%
\begin{equation*}
x_{j}=x_{j}^{\Phi }+x_{j}^{\Sigma }
\end{equation*}%
where $x_{j}^{\Phi },y_{i}^{\Phi }\in F_{\alpha -1}C_{\Phi }$ and $%
x_{j}^{\Sigma },y_{i}^{\Sigma }\in F_{\alpha -1}C_{\Sigma }$. Let $L$ be a
finite direct summand of $F_{\alpha -1}C_{\Phi }$ that contains $y_{0}^{\Phi
},\ldots ,y_{\ell }^{\Phi },x_{0}^{\Phi },\ldots ,x_{n}^{\Phi }$ (which
exists since $F_{\alpha -1}C_{\Phi }$ is a sum of finite projective
modules), and write $F_{\alpha -1}C_{\Phi }=L\oplus L^{\prime }$. Let $T$ be
the submodule of $F_{\alpha -1}C_{\Phi }$ generated by $L$ and by $\hat{B}%
_{\alpha -1}^{k}C\cap F_{\alpha -1}C_{\Phi }$. Then we have that%
\begin{equation*}
\frac{T}{\hat{B}_{\alpha -1}^{k}C\cap F_{\alpha -1}C_{\Phi }}\cong \frac{T+%
\hat{B}_{\alpha -1}^{k}}{\hat{B}_{\alpha -1}^{k}}\subseteq \frac{F_{\alpha
-1}C}{\hat{B}_{\alpha -1}^{k}C}\cong C/B_{\alpha -1}^{k}C
\end{equation*}%
is finite (since $L$ is finite) and flat (since $C/B_{\alpha -1}^{k}C$ is
flat), and hence projective. Thus there exists $\psi \in \mathrm{Hom}\left(
T,A\right) $ that extends $\hat{\varphi}_{0}|_{\hat{B}_{\alpha -1}^{k}C\cap
F_{\alpha -1}C_{\Phi }}$. One can define $\varphi \in \mathrm{Hom}\left(
F_{\alpha -1}C,A\right) $ by setting%
\begin{equation*}
\varphi \left( z+w+w^{\prime }\right) =\hat{\varphi}_{0}\left( z\right)
+\psi \left( w\right)
\end{equation*}%
for $z\in F_{\alpha -1}C_{\Sigma }$ and $w\in L$ and $w^{\prime }\in
L^{\prime }$. For $i\leq \ell $ and $j\leq n$ we have%
\begin{equation*}
\varphi \left( y_{i}\right) =\hat{\varphi}_{0}\left( y_{i}^{\Sigma }\right)
+\psi \left( y_{i}^{\Phi }\right) =\hat{\varphi}_{0}\left( y_{i}\right)
=\varphi \left( y_{i}\right)
\end{equation*}%
and%
\begin{equation*}
\varphi \left( x_{j}\right) =\hat{\varphi}_{0}(x_{j}^{\Sigma })+\psi
(x_{j}^{\Phi })=\hat{\varphi}_{0}\left( x_{j}\right) =0\text{.}
\end{equation*}%
This concludes the proof.

(7)\ It follows from (4) that%
\begin{equation*}
0\rightarrow L_{\alpha }C\rightarrow F_{\alpha }C\rightarrow C\rightarrow 0
\end{equation*}%
is a phantom short exact sequence of order $\alpha $, where $L_{\alpha }C$
and $F_{\alpha }C$ have extractable length at most $\alpha $ by definition
and the inductive hypothesis.

(8) If $C$ has (plain) extractable length at most $\alpha $, then $C$ has
(plain) projective length at most $\alpha $ by (1). Conversely, if $C$ has $%
L_{\alpha }$-projective length at most $\alpha $, then the short exact
sequence in (7) splits. Thus, $C$ is a direct summand of $F_{\alpha }C$,
whence it has extractable length at most $\alpha $.

If $C$ has plain projective length at most $\alpha $, then $\partial
_{\alpha -1}C$ has finite rank by Theorem \ref{Theorem:derivation}. By the
inductive hypothesis, $\sigma _{\alpha -1}C$ has extractable length at most $%
\alpha -1$. Considering the short exact sequence $\sigma _{\alpha
-1}C\rightarrow C\rightarrow \partial _{\alpha -1}C$ we conclude that $C$
has plain extractable length at most $\alpha $.

(9) If $C$ has plain $R$-projective length at most $\alpha $, then the proof
of (8) above shows that $C$ has plain extractable length at most $\alpha $.
Suppose now that $C$ has $R$-projective length at most $\alpha $. We claim
that $C$ has extractable length at most $\alpha $. Without loss of
generality, we can assume that $C$ is isomorphic to the direct sum of a
sequence $\left( C_{n}\right) $ of countable flat modules such that $C_{n}$
has plain extractable length $\beta _{n}$ for every $n\in \omega $. Then we
have that%
\begin{equation*}
\mathrm{Ext}\left( C,R\right) \cong \prod_{n}\mathrm{Ext}\left(
C_{n},R\right) \text{.}
\end{equation*}%
Therefore, the projective length of $C$ is equal by the inductive hypothesis
to $\mathrm{sup}_{n}\beta _{n}$. This shows that $\mathrm{sup}_{n}\beta
_{n}\leq \alpha $ and hence $C$ has extractable length at most $\alpha $.

(10) Let $A$ be a countable flat module. By (7) we have that the extension $%
\mathfrak{R}_{C}$%
\begin{equation*}
0\rightarrow L_{\alpha }C\rightarrow F_{\alpha }C\rightarrow C\rightarrow 0
\end{equation*}%
is phantom of order $\alpha $. If $\varphi :L_{\alpha }C\rightarrow A$ is a
homomorphism, then the image of $\varphi $ under the homomorphism%
\begin{equation*}
\mathrm{Hom}\left( L_{\alpha }C,A\right) \rightarrow \mathrm{Ext}\left(
C,A\right)
\end{equation*}%
is the extension obtained via the pushout diagram%
\begin{equation*}
\begin{array}{ccccc}
L_{\alpha }C & \rightarrow & F_{\alpha }C & \rightarrow & C \\ 
\varphi \downarrow &  & \downarrow &  & \downarrow \\ 
A & \rightarrow & X & \rightarrow & C%
\end{array}%
\end{equation*}%
Since $\mathfrak{R}_{C}$ is phantom of order $\alpha $, by (6) the exact
sequence induced by%
\begin{equation*}
0\rightarrow L_{\alpha }C\rightarrow F_{\alpha }C\rightarrow C\rightarrow 0
\end{equation*}%
and the inclusion $B_{\alpha _{k}}^{k}C\rightarrow C$ splits. Thus,t he same
holds for the exact sequence induced by%
\begin{equation*}
0\rightarrow A\rightarrow X\rightarrow C\rightarrow 0
\end{equation*}%
and the inclusion $B_{\alpha _{k}}^{k}C\rightarrow C$. As this holds for
every $k\in \omega $, by (6) again the short exact sequence%
\begin{equation*}
A\rightarrow X\rightarrow C
\end{equation*}%
is phantom of order $\alpha $. This shows that the image of the homomorphism%
\begin{equation*}
\mathrm{Hom}\left( L_{\alpha }C,A\right) \rightarrow \mathrm{Ext}\left(
C,A\right)
\end{equation*}%
is contained in $s_{\alpha }\mathrm{Ext}\left( C,A\right) $.\ Conversely,
suppose that%
\begin{equation*}
A\rightarrow X\rightarrow C
\end{equation*}%
is a phantom short exact sequence of order $\alpha $. Then by (6) we have
that for every $k\in \omega $ the induced short exact sequence%
\begin{equation*}
\begin{array}{ccccc}
A & \rightarrow & X & \rightarrow & C \\ 
\downarrow &  & \downarrow &  & \downarrow \\ 
A & \rightarrow & Y_{k} & \rightarrow & B_{\alpha _{k}}^{k}C%
\end{array}%
\end{equation*}%
splits. Thus, there exists a homomorphism $\psi _{k}:B_{\alpha
_{k}}^{k}C\rightarrow Y_{k}$ that is a right inverse for $Y_{k}\rightarrow
B_{\alpha _{k}}^{k}C$. For every $k\in \omega $, we have that 
\begin{equation*}
\varphi _{k}:=\psi _{k+1}|_{B_{\alpha _{k}}^{k}C}-\psi _{k}
\end{equation*}%
is a homomorphism $B_{\alpha _{k}}^{k}C\rightarrow A$.\ For $k\in \omega $,
these define a homomorphism $\varphi :L_{\alpha }C\rightarrow C$. The image
of $\varphi $ under the homomorphism%
\begin{equation*}
\mathrm{Hom}\left( L_{\alpha }C,A\right) \rightarrow \mathrm{Ext}\left(
C,A\right)
\end{equation*}%
is the extension of $C$ by $A$ represented by 
\begin{equation*}
A\rightarrow X\rightarrow C\text{.}
\end{equation*}%
This concludes the proof.
\end{proof}

\begin{corollary}
\label{Corollary:constructible}Let $C$ be an extractable countable flat
module. Then the extractable length, the $R$-projective length, and the
projective length of $C$ coincide, as well as the extractable rank, the $R$%
-projective rank, and the projective rank.
\end{corollary}

In consideration of Corollary \ref{Corollary:constructible}, we can
formulate the following:

\begin{definition}
We define the (plain) length of a countable \emph{extractable }flat module
to be the common value of its (plain) extractable, projective, and $R$%
-projective lengths.
\end{definition}

\subsection{Independence}

Suppose that $A$ is a countable flat module. If $C$ is an extractable
coreduced flat module of (plain) length $\alpha $, then in general the
(plain) $A$-projective length of $A$ is \emph{at most }$\alpha $. We now
consider a natural condition on $C$, which we term $A$-independence, which
guarantees that the (plain) $A$-projective length of $A$ is indeed equal to $%
\alpha $.

\begin{definition}
Suppose that $C$ is a extractable coreduced countable flat module, and $A$
is a countable flat module. We say that $C$ is $A$-\emph{independent} if%
\begin{equation*}
\mathrm{Hom}\left( \sigma _{\beta }C,A\right) =\mathrm{Hom}\left( \partial
_{\beta }C,A\right) =0
\end{equation*}%
for every $\beta <\omega _{1}$.
\end{definition}

Notice that every extractable coreduced countable flat module is $R$%
-independent. The following lemma is easily established by induction on the
length.

\begin{lemma}
\label{Lemma:A-independent}Suppose that $C\ $is a extractable coreduced
countable flat module. The following assertions are equivalent:

\begin{enumerate}
\item $C$ is $A$-independent;

\item $\mathrm{Hom}\left( \Sigma (B_{1}^{k}(\partial _{\beta }C)),A\right)
=0 $ for every $\beta <\omega _{1}$ and $k\in \omega $.
\end{enumerate}
\end{lemma}

\begin{theorem}
\label{Theorem:A-independent}Suppose that $C$ is a extractable coreduced
countable flat module, and $A$ is a countable flat module. Assume that $C$
is $A$-independent. Then:

\begin{enumerate}
\item the (plain) length of $C$ is equal to the (plain) $A$-projective
length of $C$;

\item for every $\beta <\omega _{1}$, the quotient map $M\rightarrow
\partial _{\beta }M$ induces an isomorphism%
\begin{equation*}
\mathrm{Ext}\left( \partial _{\beta }M,A\right) \cong \mathrm{Ph}^{\beta }%
\mathrm{Ext}\left( M,A\right) \text{.}
\end{equation*}
\end{enumerate}
\end{theorem}

\begin{proof}
This follows from Theorem \ref{Theorem:derivation-A} considering that its
hypotheses are verified by Lemma \ref{Lemma:A-independent}.
\end{proof}

\subsection{Rigidity}

Suppose that $M$ is a countable flat module. Let $\mathrm{\mathrm{Aut}}%
\left( M\right) $ be the group of automorphisms of $M$ endowed with the
topology of pointwise convergence. Then we have that $\mathrm{Aut}\left(
M\right) $ is a non-Archimedean Polish group. This means that $\mathrm{Aut}%
\left( M\right) $ is a topological group whose topology is Polish, and
admits a basis of neighborhoods of the identity consisting of open
subgroups. The automorphisms $\mu _{r}:x\mapsto rx$ for $r\in R^{\times }$,
which we call \emph{inner}, form a central closed subgroup isomorphic to $%
R^{\times }$.

Likewise, we have that \emph{endomorphisms }of $M$ form a Polish ring $%
\mathrm{\mathrm{End}}\left( M\right) $. An endomorphism of $M$ is called 
\emph{inner} if it is of the form $x\mapsto rx$ for some $x\in X$. Inner
endomorphisms form a central closed subring isomorphic to $R$. We say that $%
M $ is \emph{rigid} if every endomorphism of $M$ is inner.

We say that two countable flat modules $N$ and $M$ are \emph{orthogonal} if $%
\mathrm{Hom}\left( M,N\right) =0$ and $\mathrm{Hom}\left( N,M\right) =0$. We
say that $\left( N_{i}\right) _{i\in I}$ is an \emph{orthonormal family} if
for every $i\in I$, $N_{i}$ is rigid, and for every distinct $i,j\in I$, $%
N_{i}$ and $N_{j}$ are orthogonal.

\subsection{Minimal elements}

We say that an element of a module $M$ is \emph{minimal }if it generates a
pure submodule. A \emph{pointed module }is a module with a distinguished
element, and a \emph{well-pointed module} is a module with a distinguished
minimal element. Recall that a sequence $\left( \xi _{n}\right) $ of
elements of a module $M$ is $\emph{independent}$ if the submodule $N$
generated by $\left( \xi _{n}\right) $ is the direct sum of the submodules $%
N_{n}$ generated by $\xi _{n}$ for $n\in \omega $.

\begin{definition}
\label{Definition:minimal-family}A family $\left( \xi _{i}\right) _{i\in I}$
of elements of a module $M$ is \emph{minimal }if it is independent and the
submodule $N$ generated by $\left( \xi _{i}\right) $ is \emph{pure}.
\end{definition}

Notice that a family $\left( \xi _{i}\right) _{i\in I}$ in $M$ is minimal if
and only if, letting $N$ be the submodule generated by $\left( \xi
_{i}\right) $, an element of $rM\cap N$ can be written uniquely as%
\begin{equation*}
\sum_{i\in I}\mu _{i}\xi _{i}
\end{equation*}%
for $\mu _{i}\in R$, and in this case one has $\mu _{i}\in rR$ for every $%
i\in I$.

\begin{lemma}
\label{Lemma:minimal}Suppose that $B$ is a countable flat module, $%
A\subseteq B$ is a pure submodule, and $C:=B/A$. Let $\left( \xi _{i}\right)
_{i\in I}$ be family of elements of $B$. Suppose $I_{A}\subseteq I$ is such
that $\left( \xi _{i}\right) _{\in I_{A}}$ is a minimal family of elements
of $A$.\ Set $I_{C}:=I\setminus I_{A}$, and assume that $\left( \xi
_{i}+A\right) _{i\in I_{C}}$ is a minimal family of elements of $C$. Then $%
\left( \xi _{i}\right) _{i\in I}$ is a minimal family of elements of $B$.
\end{lemma}

\begin{proof}
Suppose that $\left( \lambda _{i}\right) \in R^{\left( I\right) }$ is such
that 
\begin{equation*}
\sum_{i\in I}\lambda _{i}\xi _{i}=0\text{.}
\end{equation*}%
Then we have that%
\begin{equation*}
\sum_{i\in I_{C}}\lambda _{i}\xi _{i}+A=A
\end{equation*}%
and hence $\lambda _{i}=0$ for $i\in I_{C}$. Hence%
\begin{equation*}
\sum_{i\in I_{A}}\lambda _{i}\xi _{i}=0
\end{equation*}%
and $\lambda _{i}=0$ for $i\in I_{A}$. This shows that $\left( \xi
_{i}\right) $ is independent.

Suppose that $b\in B$ and $r\in R\setminus \left\{ 0\right\} $ are such that 
$rb$ belongs to the submodule $N$ generated by $\left( \xi _{i}\right) $. We
claim that $b\in N$. We can write%
\begin{equation*}
rb=\sum_{i\in I}\lambda _{i}\xi _{i}
\end{equation*}%
for $\left( \xi _{i}\right) \in R^{\left( I\right) }$. Then we have that $%
rb+A$ belongs to the submodule of $C$ generated by $\left( \xi _{i}+A\right)
_{i\in I_{C}}$. As such a submodule is pure, the same holds for $b+A$. Thus,
there exist $\left( \mu _{i}\right) \in R^{\left( I_{C}\right) }$ and $a\in
A $ such that 
\begin{equation*}
b-a=\sum_{i\in I_{C}}\mu _{i}\xi _{i}
\end{equation*}%
and hence%
\begin{equation*}
\sum_{i\in I}\lambda _{i}\xi _{i}-ra=rb-ra=\sum_{i\in I_{C}}r\mu _{i}\xi _{i}
\end{equation*}%
Thus, we have $r\mu _{i}=\lambda _{i}$ for $i\in I_{C}$. Hence,%
\begin{equation*}
\sum_{i\in I_{C}}\left( \lambda _{i}-r\mu _{i}\right) \xi _{i}+\sum_{i\in
I_{A}}\lambda _{i}\xi _{i}=ra\in A\text{.}
\end{equation*}%
Thus, we have%
\begin{equation*}
\sum_{i\in I_{C}}\left( \lambda _{i}-r\mu _{i}\right) \xi _{i}\equiv 0\ 
\mathrm{\mathrm{mod}}A\text{.}
\end{equation*}%
This implies that $\lambda _{i}=r\mu _{i}$ for $i\in I_{C}$. Thus, we have
that%
\begin{equation*}
r(b-\sum_{i\in I_{C}}\mu _{i}\xi _{i})=\sum_{i\in I}\lambda _{i}\xi
_{i}-\sum_{i\in I_{C}}r\mu _{i}\xi _{i}=\sum_{i\in I_{A}}\lambda _{i}\xi _{i}
\end{equation*}%
belongs to the submodule $N_{A}$ of $A$ generated by $\left( \xi _{i}\right)
_{i\in A}$. As such a submodule is pure, we also have%
\begin{equation*}
b-\sum_{i\in I_{C}}\mu _{i}\xi _{i}\in N_{A}\subseteq N\text{.}
\end{equation*}%
Thus, we have%
\begin{equation*}
b=(b-\sum_{i\in I_{C}}\mu _{i}\xi _{i})+\sum_{i\in I_{C}}\mu _{i}\xi _{i}\in
N\text{.}
\end{equation*}%
This concludes the proof.
\end{proof}

\subsection{Types}

In this section we consider the natural generalization of the notion of type
of an element from torsion-free abelian groups to flat modules \cite[Chapter
XIII]{fuchs_infinite_1973}.

\begin{definition}
Define $\mathfrak{I}\left( R\right) $ to be the collection of \emph{finite }%
ideals of $R$ ordered by inclusion. Let $P$ and $P^{\prime }$ be countable
upward-directed ordered sets. Suppose that $\mathfrak{t}:P\rightarrow 
\mathfrak{I}\left( R\right) $ and $\mathfrak{t}^{\prime }:P^{\prime
}\rightarrow \mathfrak{I}\left( R\right) $ are order-reversing functions.
Say that $\mathfrak{t}$ and $\mathfrak{t}^{\prime }$ are \emph{cofinally
equivalent} if for all $p\in P$ there exists $p^{\prime }\in P^{\prime }$
such that $\mathfrak{t}^{\prime }\left( p^{\prime }\right) \subseteq 
\mathfrak{t}\left( p\right) $ and vice versa. A \emph{type }(of finite
ideals) is a cofinal equivalence class of order-reversing functions $%
\mathfrak{t}:P\rightarrow \mathfrak{I}\left( R\right) $, where $P$ is an
upward directed ordered set.
\end{definition}

Clearly, any type admits a representative $\mathfrak{t}$ that is \emph{tower
of ideals}, i.e., a decreasing sequence of ideals. In what follows by abuse
of notation we identify a type with any of its representatives. We define an
order relation among types by setting 
\begin{equation*}
\left( I_{n}\right) \leq \left( J_{n}\right) \Leftrightarrow \forall
n\exists k\text{, }I_{n}\subseteq J_{k}\text{.}
\end{equation*}%
The smallest type $\left( I_{n}\right) $ is such that $I_{n}=0$ for every $%
n\in \omega $, and the largest type $\left( I_{n}\right) $ is such that $%
I_{n}=R$ for every $n\in \omega $. Types form a lattice with%
\begin{equation*}
\left( I_{n}\right) \vee \left( J_{n}\right) :=\left( I_{n}+J_{n}\right)
\end{equation*}%
and meet%
\begin{equation*}
\left( I_{n}\right) \wedge \left( J_{n}\right) :=\left( I_{n}\cap
J_{n}\right) \text{.}
\end{equation*}

\begin{definition}
Two types $\mathfrak{t}$ and $\mathfrak{t}^{\prime }$ are \emph{orthogonal},
in which case we write $\mathfrak{t}\bot \mathfrak{t}^{\prime }$, if their
join is the largest type.
\end{definition}

We can define the type of an element, generalizing the usual notion in the
case of $\mathbb{Z}$-modules; see \cite[Chapter XIII]{fuchs_infinite_1973}.

\begin{definition}
Suppose that $M$ is a coreduced countable flat module, and $a\in M$ is a
nonzero element.\ Fix an increasing sequence $\left( M_{n}\right) $ of
finite submodules of $M$ with $a\in M_{0}$ and union of $\left\{ M_{n}:n\in
\omega \right\} $ equal to $M$. Let 
\begin{equation*}
\boldsymbol{A}=\left( \mathrm{Hom}\left( M_{n},R\right) \right) _{n\in
\omega }
\end{equation*}%
be the corresponding tower of finite flat modules. Given a subtower $%
\boldsymbol{B}$ of $\boldsymbol{A}$, one defines the $\boldsymbol{B}$-type
of $a$ to be%
\begin{equation*}
\boldsymbol{B}[a]=\left( \left\{ \varphi \left( a\right) :\varphi \in
B_{n}\right\} \right) _{n\in \omega }\text{.}
\end{equation*}%
In particular, we define the $\alpha $-type $\mathfrak{t}_{\alpha }\left(
a\right) $ of $a$ as the $\boldsymbol{A}_{\alpha }$-type, where $\boldsymbol{%
A}_{\alpha }$ is the $\alpha $-th derived tower of $\boldsymbol{A}$. The $0$%
-type is also called the type and denoted by $\mathfrak{t}\left( a\right) $.
\end{definition}

Notice that $\left( \mathfrak{t}_{\alpha }\left( a\right) \right) _{\alpha
<\omega _{1}}$ forms a decreasing sequence of types.

\subsection{Straight length}

We now introduce a more stringent notion than the one of extractable module,
which we term straight extractable module.

\begin{definition}
\label{Definition:straight-constructible}Suppose that $M$ is a countable
coreduced flat module. We say that $M$ has plain straight length $1$ if and
only if $M$ is nontrivial and has finite rank. For a countable successor
ordinal $\alpha >1$, $M$ has plain straight length $\alpha $ if and only if
there exist:

\begin{enumerate}
\item a submodule $L\subseteq M$;

\item an inductive sequence $\left( J_{n},\varphi _{n}\right) $ of finite
free modules;

\item bases $\left( e_{n,i}\right) _{i<d_{n}}$ for $J_{n}$ for $n<\omega $
with $\sup_{n}d_{n}<\infty $;

\item elements $f_{n,i}\in M$ for $n<\omega $ and $i<d_{n}$;

\item pointed submodules $\left( L_{n,i},\xi _{n,i}\right) \subseteq L$ for $%
n<\omega $ and $i<d_{n}$;
\end{enumerate}

such that

\begin{itemize}
\item $L_{n,i}$ has plain straight length $\beta \lbrack n]$ and $%
\sup_{n}\beta \lbrack n]=\alpha -1$;

\item $L$ is the direct sum of $L_{n,i}$ for $n<\omega $ and $i<d_{n}$;

\item there exists an isomorphism between the quotient $J$ of $M$ by $L$ and 
$\mathrm{co\mathrm{lim}}_{n}\left( J_{n},\varphi _{n}\right) $, which maps $%
f_{n,i}+L$ to the element of $\mathrm{co\mathrm{lim}}_{n}\left(
J_{n},\varphi _{n}\right) $ corresponding to $e_{n,i}\in J_{n}$;

\item for $n<\omega $ and $i<d_{n}$, if%
\begin{equation*}
\varphi _{n}\left( e_{n,i}\right) =\sum_{j<d_{n+1}}a_{j}e_{n+1,j}
\end{equation*}%
for $a_{j}\in R$, then%
\begin{equation*}
f_{n,i}=\sum_{j<d_{n+1}}a_{j}f_{n+1,j}+\xi _{n,i}
\end{equation*}

\item for $n<\omega $ and $i<d_{n}$,%
\begin{equation*}
\mathfrak{t}_{\alpha _{n}-1}^{L_{n,i}}\left( \xi _{n,i}\right) \bot 
\mathfrak{t}^{J}\left( e_{n,i}\right) \text{.}
\end{equation*}%
We say that $M$ has straight length $\alpha $ if and only if $M$ is
isomorphic to the direct sum of a sequence of modules $M_{n}$ of plain
straight length $\beta \lbrack n]$ with $\sup_{n}\beta \lbrack n]=\alpha $.
\end{itemize}

The module $M$ is \emph{straight extractable} if and only if it has (plain)
straight length $\alpha $ for some $\alpha <\omega _{1}$.
\end{definition}

Any straight extractable module is, in particular, extractable. In
particular, its (plain) $R$-projective and extractable lengths coincide.

\begin{proposition}
\label{Proposition:straight-constructible}Suppose that $\alpha >1$ is a
successor ordinal, $A$ is a countable flat module, and $M$ is an countable
coreduced flat module of plain straight length $\alpha $. Let $%
L_{n,i}\subseteq L\subseteq M$ for $n<\omega $ and $i<d_{n}$ and $J$ be as
in the definition of plain straight length $\alpha $.

\begin{enumerate}
\item $L=\sigma _{\alpha -1}M$ and $J\cong \partial _{\alpha -1}M$;

\item $M$ is $A$-independent if and only if $L_{n,i}$ is $A$-independent for 
$n<\omega $ and $i<d_{n}$ and $\mathrm{Hom}\left( J,A\right) =0$;

\item the plain length of $M$ is $\alpha $.
\end{enumerate}
\end{proposition}

\begin{proof}
We prove that the conclusion holds by induction on $\alpha $. Observe that
(2) and (3) easily follow from (1), so it remains to prove (1).

For $i<d_{n}$ let $\left( L_{n,i,k}\right) _{k\in \omega }$ be an increasing
sequence of finite submodules with union equal to $L_{n,i}$ and $\xi
_{n,i}\in L_{n,i,0}$. Set 
\begin{equation*}
L[n]:=L_{n,0}\oplus \cdots \oplus L_{n,d_{n}-1}
\end{equation*}%
and%
\begin{equation*}
L_{k}[n]:=L_{n,0,k}\oplus \cdots \oplus L_{n,d_{n}-1,k\text{.}}
\end{equation*}%
Define 
\begin{equation*}
M_{k}:=L_{k}[0]\oplus \cdots \oplus L_{k}[k]\oplus L_{k}[-1]
\end{equation*}%
where $L_{k}[-1]$ is the submodule of $M$ generated by $f_{k,i}$ for $%
i<d_{k} $. Then we have that $\left( M_{k}\right) $ is an increasing
sequence of finite submodules of $M$ with union equal to $M$.\ We have that%
\begin{equation*}
\boldsymbol{C}:=\left( \mathrm{Hom}\left( M_{n},A\right) \right) _{n\in
\omega }
\end{equation*}%
is the tower associated with the fishbone of towers $\left( \boldsymbol{A},%
\boldsymbol{B}[n]\right) _{n\in \omega }$, which is straight of length $%
\alpha $; see Definition \ref{Definition:fishbone-tower}. Here $\boldsymbol{A%
}$ is the tower%
\begin{equation*}
\left( \mathrm{Hom}\left( J_{n},A\right) \right) _{n\in \omega }
\end{equation*}%
while $\boldsymbol{B}[n]$ is the tower%
\begin{equation*}
\left( \mathrm{Hom}\left( L_{k}[n],A\right) \right) _{k\in \omega }\text{.}
\end{equation*}%
The isomorphism%
\begin{equation*}
\mathrm{Ext}\left( M,A\right) \cong \mathrm{lim}^{1}\boldsymbol{C}\text{,}
\end{equation*}%
induces an isomorphism%
\begin{equation*}
\mathrm{Ph}^{\beta }\mathrm{Ext}\left( M,A\right) \cong \mathrm{lim}^{1}%
\boldsymbol{C}_{\beta }
\end{equation*}%
for $\beta <\omega _{1}$, considering the description of the Solecki
subgroups of $\mathrm{Ext}$ and $\mathrm{lim}^{1}$; see Theorem \ref%
{Theorem:phantom-Ext} and Theorem \ref{Theorem:Solecki-lim1}.

The description of the derived towers of a fishbone tower given by
Proposition \ref{Proposition:fishbone-tower} yields%
\begin{equation*}
\mathrm{Ph}^{\alpha -1}\mathrm{Ext}\left( M,A\right) \cong \mathrm{lim}^{1}%
\boldsymbol{A}\text{.}
\end{equation*}%
This shows that the canonical projection $M\rightarrow J$ corresponds to $%
M\rightarrow \partial _{\alpha -1}M$ and hence $\sigma _{\alpha -1}M=L$.
\end{proof}

\begin{corollary}
If $C$ is a countable coreduced straight extractable flat module, then its
(plain) straight length equals its (plain) length.
\end{corollary}

\subsection{Fishbones of modules}

We now define the analogue for modules of the notion of \emph{fishbone} of
towers.

\begin{definition}
A \emph{fishbone }of modules $(\left( L_{n,i},\xi
_{n,i},J_{n},e_{n,i},\varphi _{n}\right) _{i<d_{n}})_{n\in \omega }$ is
given by:

\begin{enumerate}
\item an inductive sequence $\left( J_{n},\varphi _{n}\right) $ of nonzero
finite free modules with $\mathrm{rank}(J_{n})=d_{n}$ and $d_{0}=1$;

\item free bases $\left( e_{n,i}\right) _{i<d_{n}}$ of $J_{n}$ for $n<\omega 
$;

\item pointed countable coreduced flat modules $\left( L_{n,i},\xi
_{n,i}\right) $ for $n\in \omega $ and $i<d_{n}$.
\end{enumerate}

The module%
\begin{equation*}
L_{n}:=L_{n,0}\oplus \cdots \oplus L_{n,d_{n}-1}
\end{equation*}%
is called the $n$-th rib of the fishbone.

We can define the homomorphism $\psi _{n}:J_{n}\rightarrow L_{n}$ by way of
the universal property of the free module by $e_{n,i}\mapsto \xi _{n,i}$.
The \emph{fishbone colimit }$M$ is the colimit of the sequence%
\begin{equation*}
(L_{0}\oplus \cdots \oplus L_{n}\oplus J_{n})
\end{equation*}%
with respect to the morphisms 
\begin{eqnarray*}
L_{0}\oplus \cdots \oplus L_{n}\oplus J_{n} &\rightarrow &L_{0}\oplus \cdots
\oplus L_{n+1}\oplus J_{n+1} \\
\left( x_{0},\ldots ,x_{n},a_{n}\right) &\mapsto &\left( x_{0},\ldots
,x_{n},\psi _{n}(a_{n}),\varphi _{n}\left( a_{n}\right) \right) \text{.}
\end{eqnarray*}
\end{definition}

Notice that by definition $M$ contains the direct sum $L$ of $\left(
L_{n}\right) _{n\in \omega }$. This defines a short exact sequence%
\begin{equation*}
0\rightarrow L\rightarrow M\rightarrow J\rightarrow 0\text{.}
\end{equation*}%
Furthermore, $M$ is isomorphic to the quotient of $L\oplus J$ by the\ (pure)
submodule generated by 
\begin{equation*}
\xi _{n,i}-e_{n,i}
\end{equation*}%
for $n<\omega $ and $i<d_{n}$.

We define the notion of \emph{fishbone length }of a module, capturing the
complexity of describing it in terms of fishbones.

\begin{definition}
\label{Definition:fishbone-length}Let $M$ be a countable coreduced flat
module, and $A$ be a countable flat module.

\begin{itemize}
\item We say that $M$ has plain fishbone length $1$ if and only if it is
nontrivial and finite-rank.

\item The module $M$ has plain fishbone length a successor ordinal $\alpha
>1 $ if and only if there exists a fishbone of modules $(\left( L_{n,i},\xi
_{n,i},J_{n},e_{n,i},\varphi _{n}\right) _{i<d_{n}})_{n\in \omega }$ with
fishbone colimit $M$ such that for every $n\in \omega $ and $i<d_{n}$:

\begin{enumerate}
\item $L_{n,i}$ has plain fishbone length $\beta \lbrack n]$ where $%
\sup_{n}\beta \lbrack n]=\alpha -1$;

\item $\mathfrak{t}_{\beta \lbrack n]-1}^{L_{n,i}}\left( \xi _{n,i}\right)
\bot \mathfrak{t}^{J}\left( e_{n,i}\right) $.
\end{enumerate}

\item The module $M$ has fishbone length $\alpha $ if and only if it is the
direct sum of a sequence of modules $M_{n}$ of plain fishbone length $\beta
\lbrack n]$ with $\sup_{n}\beta \lbrack n]=\alpha $.
\end{itemize}

A module is\emph{\ fishbone extractable }if and only if it has (plain)
fishbone length $\alpha $ for some countable ordinal $\alpha $.
\end{definition}

It is easily shown by induction that any fishbone extractable module is
straight extractable.

\begin{proposition}
\label{Proposition:fishbone-module}Suppose that $M$ is a countable coreduced
flat module of plain fishbone length $\alpha >1$, and let $(\left(
L_{n,i},\xi _{n,i},J_{n},e_{n,i},\varphi _{n}\right) _{i<d_{n}})_{n\in
\omega }$ be a fishbone of modules as in Definition \ref%
{Definition:fishbone-length} with fishbone colimit $L$ witnessing that $M$
has plain fishbone length $\alpha $. Let $L$ be the module 
\begin{equation*}
\mathrm{co\mathrm{lim}}_{n}(L_{n,0}\oplus \cdots \oplus L_{n,d_{n-1}})
\end{equation*}%
identified with a submodule of $M$. Suppose that $A$ is a countable flat
module. Then

\begin{enumerate}
\item $L=\sigma _{\alpha -1}M$ and $\partial _{\alpha -1}M\cong J$;

\item $M$ is $A$-independent if and only if $L_{n,i}$ is $A$-independent for 
$n<\omega $ and $i<\omega $ and $\mathrm{Hom}\left( J,A\right) =0$;

\item $M$ has plain length $\alpha $.
\end{enumerate}
\end{proposition}

\begin{proof}
This is a consequence of Proposition \ref{Proposition:straight-constructible}%
.
\end{proof}

\begin{corollary}
If $C$ is a countable coreduced fishbone extractable flat module, then its
(plain) fishbone length equals its (plain) length.
\end{corollary}

\section{Constructions for modules over a Dedekind domain\label%
{Section:constructions-Dedekind}}

In this section, we assume that $R$ is not only a Pr\"{u}fer domain but a 
\emph{Dedekind domain}. We let $\mathbb{P}$ be the collection of nonzero 
\emph{prime ideals }of $R$. For $\mathfrak{p}\in \mathbb{P}$, we let $R_{%
\mathfrak{p}}$ be the localization of $R$ at $\mathfrak{p}$, which is a
discrete valuation domain (DVR), namely a PID that is also a local ring. We
furthermore assume that $R$ is countable and not a field, which implies that 
$R$ is countably infinite (as a finite division ring is necessarily a field).

\subsection{Factorization of ideals}

Recall that in a Dedekind domain, every ideal can be written uniquely (up to
the order of the factors) as a product of prime ideals \cite[Corollary 9.4]%
{atiyah_introduction_1969}. Furthermore, it has \emph{Krull dimension }$1$,
whence the nonzero prime ideals are precisely the maximal ideals. Every
ideal is generated by two elements \cite[Proposition C-5.99]%
{rotman_advanced_2015}. Notice that if $\mathfrak{p},\mathfrak{q}$ are
distinct elements of $\mathbb{P}$, then $\mathfrak{p}$ and $\mathfrak{q}$
are \emph{coprime}, namely $\mathfrak{p}+\mathfrak{q}=R$. More generally, if 
$\mathfrak{p}_{0},\ldots ,\mathfrak{p}_{n},\mathfrak{q}_{0},\ldots ,%
\mathfrak{q}_{n}\in \mathbb{P}$ satisfy%
\begin{equation*}
\left\{ \mathfrak{p}_{0},\ldots ,\mathfrak{p}_{n}\right\} \cap \left\{ 
\mathfrak{q}_{0},\ldots ,\mathfrak{q}_{n}\right\} =\varnothing
\end{equation*}%
then $\mathfrak{p}_{0}\cdots \mathfrak{p}_{n}$ and $\mathfrak{q}_{0}\cdots 
\mathfrak{q}_{n}$ are coprime. If $I$ is an ideal of $R$, then by the Krull
Intersection Theorem \cite{anderson_krull_1975}, the sequence $\left(
I^{n}\right) _{n\in \omega }$ has trivial intersection.

\begin{lemma}
\label{Lemma:minimal-combination}Suppose that $B$ is a countable flat
module. Let $\left( \xi _{i}\right) _{i\in \omega }$ be a family of minimal
elements of $B$. Let $N$ be the submodule generated by $\left( \xi
_{i}\right) _{i\in \omega }$. Then $rB\cap N$ is the submodule generated by $%
\left( r\xi _{i}\right) _{i\in \omega }$. Suppose that $a_{n}\in R^{\left(
\omega \right) }$ for $n\in \omega $ are such that, for every $\mathfrak{p}%
\in \mathbb{P}$, the elements $a_{n}+\mathfrak{p}\in \left( R/\mathfrak{p}%
\right) ^{\left( \omega \right) }$ for $n\in \omega $ are linearly
independent. Define $\eta _{n}:=\sum_{i\in \omega }a_{n}\left( i\right) \xi
_{i}$ for $n\in \omega $. Then $\left( \eta _{n}\right) _{n\in \omega }$ is
a family of minimal elements of $B$.
\end{lemma}

\begin{proof}
Since every finite ideal $J$ of $R$ is a product of prime ideals, we also
have that $\left( a_{n}+J\right) _{n\in \omega }\in \left( R/J\right)
^{\left( \omega \right) }$ are linearly independent.

The assertion about $rB\cap N$ follows from minimality of $\left( \xi
_{i}\right) $. Furthermore, if $\left( \lambda _{n}\right) $ are such that%
\begin{equation*}
\sum_{n}\lambda _{n}\eta _{n}=0
\end{equation*}%
then we have%
\begin{equation*}
\sum_{i}\left( \sum_{n}\lambda _{n}a_{n}\left( i\right) \right) \xi _{i}=0
\end{equation*}%
Hence, for every $i$%
\begin{equation*}
\sum_{n}\lambda _{n}a_{n}\left( i\right) =0
\end{equation*}%
This shows%
\begin{equation*}
\sum_{n}\lambda _{n}a_{n}=0\in R^{\left( \omega \right) }\text{.}
\end{equation*}%
By hypothesis, $\left( a_{n}\right) $ are an independent family, whence $%
\lambda _{n}=0$ for every $n\in \omega $. This shows that $\left( \eta
_{n}\right) $ is an independent family.

Let $M$ be the submodule generated by $\left( \eta _{n}\right) $. Consider $%
x\in rB\cap M\subseteq rB\cap N$. Then we have%
\begin{equation*}
x=\sum_{i}x_{i}\xi _{i}
\end{equation*}%
for some $x_{i}=ry_{i}\in rR$. We also have%
\begin{equation*}
x=\sum_{n}\mu _{n}\eta _{n}
\end{equation*}%
for some $\mu _{n}\in R$. Thus, we have 
\begin{equation*}
\sum_{i}x_{i}\xi _{i}=\sum_{j}\left( \sum_{n}\mu _{n}a_{n}\left( j\right)
\right) \xi _{j}\text{.}
\end{equation*}%
This implies%
\begin{equation*}
ry_{i}=x_{i}=\sum_{n}\mu _{n}a_{n}\left( i\right)
\end{equation*}%
for $i\in \omega $. Hence, we have that $\mu _{n}\in rR$ for every $n\in
\omega $, and $x\in rM$. This concludes the proof.
\end{proof}

\subsection{Completions}

We define the $\mathfrak{p}$-\emph{completion }of $R$ to be $\hat{R}_{%
\mathfrak{p}}:=\mathrm{lim}_{n}R/\mathfrak{p}^{n}R$, with the canonical
embedding $\iota _{\mathfrak{p}}:R\rightarrow \hat{R}_{\mathfrak{p}}$. As
this can be seen as the completion of the DVR $R_{\mathfrak{p}}$, we have
that $\hat{R}_{\mathfrak{p}}$ is a DVR. The $\mathfrak{p}$-completion of a
countable flat module $M$ is the pro-countable Polish flat $\hat{R}_{%
\mathfrak{p}}$-module 
\begin{equation*}
\hat{M}_{\mathfrak{p}}:=\mathrm{lim}_{n}M/\mathfrak{p}^{n}M\text{.}
\end{equation*}%
We let the completion of $R$ to be $\hat{R}:=\mathrm{lim}_{I}R/I$ where $I$
ranges among the ideals of $R$. We identify $R$ with a subring of $\hat{R}$.
By the Chinese Remainder Theorem, we have an isomorphism%
\begin{equation*}
\hat{R}\cong \prod_{\mathfrak{p}\in \mathbb{P}}\hat{R}_{\mathfrak{p}}\text{,}
\end{equation*}%
and under this isomorphism the inclusion $R\rightarrow \hat{R}$ corresponds
to the diagonal embedding%
\begin{equation*}
a\mapsto \left( \iota _{\mathfrak{p}}\left( a\right) \right) _{\mathfrak{p}%
\in \mathbb{P}}\text{.}
\end{equation*}%
We define $\hat{K}_{\mathfrak{p}}$ to be the field of fractions of the DVR $%
\hat{R}_{\mathfrak{p}}$, and set $\hat{K}$ to be the divisible hull $%
K\otimes \hat{R}$ of the $R$-module $\hat{R}$. This can be identified with
the\emph{\ restricted product}%
\begin{equation*}
\prod_{\mathfrak{p}\in \mathbb{P}}(\hat{K}_{\mathfrak{p}}:\hat{R}_{\mathfrak{%
p}})=\left\{ \left( a_{\mathfrak{p}}\right) \in \prod_{\mathfrak{p}\in 
\mathbb{P}}\hat{K}_{\mathfrak{p}}:\left\{ \mathfrak{p}\in \mathbb{P}:a_{%
\mathfrak{p}}\notin \hat{R}_{\mathfrak{p}}\right\} \text{ is finite}\right\} 
\text{.}
\end{equation*}%
If $M$ is a flat module, we let 
\begin{equation*}
\hat{M}:=\mathrm{lim}_{I}M/IM\cong \prod_{\mathfrak{p}\in \mathbb{P}}\hat{M}%
_{\mathfrak{p}}
\end{equation*}%
be its completion, which we can regard as an $\hat{R}$-module. Its divisible
hull $K\otimes \hat{M}$ as an $\hat{R}$-module can be seen as a $\hat{K}$%
-module.

The kernel of the canonical homomorphism $M\rightarrow \hat{M}$ is the
largest divisible submodule $\Delta M$ of $M$, which is a direct summand of $%
M$. We say that $M$ is \emph{reduced }if $\Delta M$ is trivial. In this
case, we can and will identify $M$ with a submodule of its completion.
Notice that if $M$ is a reduced nonzero flat module, then it must be
infinite, whence $\hat{M}$ is uncountable.

For $\mathfrak{p}\in \mathbb{P}$, the kernel of the canonical homomorphism $%
M\rightarrow \hat{M}_{\mathfrak{p}}$ is the intersection $\mathfrak{p}%
^{\infty }M$ of $\mathfrak{p}^{n}M$ for $n\in \mathbb{N}$. Notice that 
\begin{equation*}
\mathfrak{p}^{\infty }M\cap \mathfrak{q}^{\infty }M=\left( \mathfrak{pq}%
\right) ^{\infty }M
\end{equation*}%
for $\mathfrak{p},\mathfrak{q}\in \mathbb{P}$. Furthermore, by the Ideal
Factorization Theorem, 
\begin{equation*}
\bigcap_{\mathfrak{p}\in \mathbb{P}}\mathfrak{p}^{\infty }M=\Delta M\text{.}
\end{equation*}%
We say that $M$ is \emph{completely reduced }if $\mathfrak{p}^{\infty }M=0$
for every $\mathfrak{p}\in \mathbb{P}$. In this case, we can identify $M$
with a submodule of $\hat{M}_{\mathfrak{p}}$ for any $\mathfrak{p}\in 
\mathbb{P}$.

\subsection{Divisibility}

For $\mathfrak{p}\in \mathbb{P}$ and $\lambda \in \hat{K}_{\mathfrak{p}}$,
we define $K_{\mathfrak{p}}\left( \lambda \right) $ to be the (countable)
subfield of $\hat{K}_{\mathfrak{p}}$ generated by $K$ and $\lambda $.

\begin{definition}
\label{Definition:essential-equivalence}Given elements $\lambda $ and $\mu $
of $\hat{R}_{\mathfrak{p}}$, we say that $\lambda $ and $\mu $ are \emph{%
essentially equivalent }if $K_{\mathfrak{p}}\left( \lambda \right) =K_{%
\mathfrak{p}}\left( \mu \right) $.
\end{definition}

Thus, we have that $\lambda $ and $\mu $ are essentially equivalent if there
exist nonzero $s,t\in R$ such that $r\lambda =s\mu $. Notice that the
relation of essential equivalence has countable equivalence classes, since $%
R $ is countable.

\begin{lemma}
\label{Lemnma:countably-dependent}Fix $\mathfrak{p}\in \mathbb{P}$. The set
of $\lambda \in \hat{R}_{\mathfrak{p}}$ such that the elements $\left(
\lambda ^{n}\right) _{n\in \omega }$ of $\hat{K}_{\mathfrak{p}}$ are
linearly dependent over $K$ is \emph{countable}.
\end{lemma}

\begin{proof}
Suppose by contradiction that there exists an uncountable set $X\subseteq 
\hat{K}_{\mathfrak{p}}$ such that for every $\lambda \in X$, $\left( \lambda
^{n}\right) _{n\in \omega }$ are linearly dependent over $K$. Then for every 
$\lambda \in X$ there exists $a\in K^{\left( \omega \right) }$ such that $%
\sum_{n}a_{n}\lambda ^{n}=0$. Since $X$ is uncountable while $K^{\left(
\omega \right) }$ is countable, after passing to an uncountable subset of $X$
we can assume that there exists $a\in K^{\left( \omega \right) }$ such that $%
\sum_{n}a_{n}\lambda ^{n}=0$ for all $\lambda \in X$. Since $\hat{K}_{p}$ is
infinite, a nonzero polynomial function has only finitely many roots. This
implies that $a=0$.
\end{proof}

We let:

\begin{itemize}
\item for $\mathfrak{p}\in \mathbb{P}$, $J_{\mathfrak{p}}\left( R\right) $
be the (uncountable) set of $\lambda \in \hat{R}_{\mathfrak{p}}\setminus K$
such that $\lambda \equiv 1\ \mathrm{mod}\mathfrak{p}$ and $\left( \lambda
^{n}\right) _{n\in \omega }$ are linearly independent over $K$ in $\hat{K}_{%
\mathfrak{p}}$;

\item for $\mathfrak{p}\in \mathbb{P}$, $\boldsymbol{J}_{\mathfrak{p}}(R)$
be the set of essential equivalence classes of elements of $J_{\mathfrak{p}%
}\left( R\right) $;

\item $J\left( R\right) $ to be the product of $J_{\mathfrak{p}}\left(
R\right) $ for $\mathfrak{p}\in \mathbb{P}$;

\item $\boldsymbol{J}\left( R\right) $ to be the product of $\boldsymbol{J}_{%
\mathfrak{p}}\left( R\right) $ for $\mathfrak{p}\in \mathbb{P}$.
\end{itemize}

If $\lambda \in J_{\mathfrak{p}}\left( R\right) $, we let $[\lambda ]\in 
\boldsymbol{J}_{\mathfrak{p}}\left( R\right) $ be its essential equivalence
class. Likewise, if $\lambda =\left( \lambda _{\mathfrak{p}}\right) \in
J\left( R\right) $, we let $[\lambda ]=([\lambda _{\mathfrak{p}}])_{p\in 
\mathbb{P}}\in \boldsymbol{J}\left( R\right) $.

\begin{definition}
\label{Definition:E}Let $M$ be a countable flat module. For $\mathfrak{p}\in 
\mathbb{P}$, identify $M/\mathfrak{p}^{\infty }M$ with a submodule of $\hat{M%
}_{\mathfrak{p}}$. If $\tau \in J\left( R\right) $, and $\boldsymbol{\tau }%
=[\tau ]\in \boldsymbol{J}\left( R\right) $, we let $E_{\tau }\left(
M\right) =E_{\boldsymbol{\tau }}\left( M\right) $ be the pure submodule of $%
M $ generated by the set of $x\in M$ such that for every $\mathfrak{p}\in 
\mathbb{P}$ there exists $\mu $ essentially equivalent to $\tau _{\mathfrak{p%
}}$ and $z\in M$ such that either 
\begin{equation*}
\mu (x+\mathfrak{p}^{\infty }M)=z+\mathfrak{p}^{\infty }M\in M/\mathfrak{p}%
^{\infty }M\subseteq \hat{M}_{\mathfrak{p}}
\end{equation*}%
or%
\begin{equation*}
\mu \left( z+\mathfrak{p}^{\infty }M\right) =x+\mathfrak{p}^{\infty }M\in M/%
\mathfrak{p}^{\infty }M\subseteq \hat{M}_{\mathfrak{p}}\text{.}
\end{equation*}
\end{definition}

It is clear that $M\mapsto E_{\tau }\left( M\right) $ is a subfunctor of the
identity.

\begin{definition}
\label{Definition:Omega}Suppose that $A$ is countable flat module. For $%
\mathfrak{p}\in \mathbb{P}$, define $\Omega _{\mathfrak{p}}\left( A\right) $
to be the set of $\tau \in J_{\mathfrak{p}}\left( R\right) $ such that for
every $\mu $ essentially equivalent to $\tau $ and for every nonzero $x\in A/%
\mathfrak{p}^{\infty }A$, $\mu x\notin A/\mathfrak{p}^{\infty }A$. Let also $%
\Omega \left( A\right) =\prod_{\mathfrak{p}\in \mathbb{P}}\Omega _{\mathfrak{%
p}}\left( A\right) $.
\end{definition}

\begin{lemma}
\label{Lemma:Omega}Suppose that $A$ is a countable flat module. For every $%
\mathfrak{p}\in \mathbb{P}$, $J_{\mathfrak{p}}(R)\setminus \Omega _{%
\mathfrak{p}}\left( A\right) $ is countable.
\end{lemma}

\begin{proof}
If $A$ is $\mathfrak{p}$-divisible there is nothing to prove. If $A$ is not $%
\mathfrak{p}$-divisible, then after replacing $A$ with $A/\mathfrak{p}%
^{\infty }A$ we can assume without loss of generality that $\mathfrak{p}%
^{\infty }A=0$ and hence $A$ can be identified with an $R$-submodule of its $%
\mathfrak{p}$-completion $\hat{A}_{\mathfrak{p}}$. As $\hat{A}_{\mathfrak{p}%
} $ is a flat $\hat{R}_{\mathfrak{p}}$-module, and $A$ is countable, the set
of $\tau \in \hat{R}_{\mathfrak{p}}$ such that there exists a nonzero $x\in
A $ such that $\tau x\in A$ is also countable. This concludes the proof.
\end{proof}

\begin{lemma}
\label{Lemma:Hom(-,A)}Suppose that $A$ is a countable reduced flat module,
and $\tau \in \Omega \left( A\right) $.\ Then $E_{\tau }\left( A\right) =0$.
If $M$ is a countable coreduced flat module, then $\mathrm{Hom}\left(
E_{\tau }\left( M\right) ,A\right) =0$.
\end{lemma}

\begin{proof}
The second assertion follows from the first one since $E_{\tau }$ is a
subfunctor of the identity.

If $x\in E_{\tau }\left( A\right) $, then we have that for every $\mathfrak{p%
}\in \mathbb{P}$ there exists $\mu $ essentially equivalent to $\tau _{%
\mathfrak{p}}$ and $z\in A$ such that either%
\begin{equation*}
\mu (x+\mathfrak{p}^{\infty }A)=z+\mathfrak{p}^{\infty }A\in A/\mathfrak{p}%
^{\infty }A\subseteq \hat{A}_{\mathfrak{p}}
\end{equation*}%
or%
\begin{equation*}
\mu \left( z+\mathfrak{p}^{\infty }A\right) =x+\mathfrak{p}^{\infty }A\in A/%
\mathfrak{p}^{\infty }A\subseteq \hat{A}_{\mathfrak{p}}\text{.}
\end{equation*}%
Since $\tau _{\mathfrak{p}}\in \Omega _{\mathfrak{p}}\left( A\right) $, this
implies that $x\in \mathfrak{p}^{\infty }A$. As this holds for every $%
\mathfrak{p}\in \mathbb{P}$, $x\in \Delta A$. Hence, $x=0$ since $A$ is
reduced.
\end{proof}

\subsection{The modules $\Xi \left( \protect\tau \right) $\label%
{Subsection:XI}}

Fix an enumeration $\left( \mathfrak{p}_{n}\right) $ of $\mathbb{P}$.
Consider $\tau =\left( \tau _{\mathfrak{p}}\right) _{\mathfrak{p}\in \mathbb{%
P}}\in J\left( R\right) $. For $\mathfrak{p}\in \mathbb{P}$ and $n\in \omega 
$ pick $\tau _{\mathfrak{p},n}\in R$ such that 
\begin{equation*}
\tau _{\mathfrak{p}}\equiv \tau _{\mathfrak{p},n}\ \mathrm{mod}\mathfrak{p}%
^{n+1}\text{.}
\end{equation*}%
For $n\geq 1$, let $\Xi _{n}\left( \tau \right) $ be submodule of $K\oplus K$
generated by $e_{\mathfrak{p},0}:=\left( 1,0\right) =e$, $f_{\mathfrak{p}%
,n}:=\left( 0,1\right) =f$, and 
\begin{equation*}
I_{\mathfrak{p},n}:=\mathfrak{p}^{-n}(e-\tau _{\mathfrak{p},n}f)
\end{equation*}%
for $\mathfrak{p}\in \left\{ \mathfrak{p}_{0},\ldots ,\mathfrak{p}%
_{n}\right\} $. Clearly,%
\begin{equation*}
\mathfrak{p}^{n}I_{\mathfrak{p},n}+R\tau _{\mathfrak{p},n}f=Re
\end{equation*}%
and in particular%
\begin{equation*}
I_{\mathfrak{p},0}+R\tau _{\mathfrak{p},0}=Re\text{.}
\end{equation*}%
Notice that $\Xi _{n}\left( \tau \right) \subseteq \Xi _{n+1}\left( \tau
\right) $ for $n\in \omega $. Indeed, we have for every $\mathfrak{p}\in 
\mathbb{P}$,%
\begin{equation*}
I_{\mathfrak{p},n}=\mathfrak{p}^{-n}\left( e-\tau _{\mathfrak{p},n}f\right)
\subseteq \mathfrak{p}^{-n}\left( e-\tau _{\mathfrak{p},n+1}f+\mathfrak{p}%
^{n}f\right) \subseteq I_{\mathfrak{p},n+1}+Rf\text{.}
\end{equation*}%
Define $\Xi \left( \tau \right) =\mathrm{co\mathrm{lim}}_{n}\Xi _{n}\left(
\tau \right) \subseteq K\oplus K$.

\begin{lemma}
\label{Lemma:rank-one-E}Suppose that $M$ is a countable flat module. If $M$
has rank one, then $E_{\tau }\left( M\right) =\left\{ 0\right\} $ for every $%
\tau \in J\left( R\right) $.
\end{lemma}

\begin{proof}
If $M$ is not reduced, then $M=K$. In this case, if $\mathfrak{p}\in \mathbb{%
P}$, $x,z\in M$, and $r\in R$ are nonzero such that $\tau _{\mathfrak{p}%
}z=rx $ then $\tau _{\mathfrak{p}}\in K$, which is not possible since $\tau
_{\mathfrak{p}}\notin K$.

Assume that $M$ is reduced. Suppose that $\mathfrak{p}\in \mathbb{P}$, $%
x,z\in M$, and $r\in R\setminus \left\{ 0\right\} $ are such that $\tau _{%
\mathfrak{p}}z=rx$. If $z$ is nonzero, then there exists $s\in K$ such that $%
x=sz$ and hence $\tau _{\mathfrak{p}}z=rsz$ and $\tau _{\mathfrak{p}}=rs\in
K $, which is again a contradiction.
\end{proof}

\begin{lemma}
\label{Lemma:pure-completion}Suppose that $R$ is a\ DVR with maximal ideal
generated by $p$ and $M$ is a reduced $R$-module. Then $M$ is a pure
submodule of its completion $\hat{M}$.
\end{lemma}

\begin{proof}
Suppose that $x\in \hat{M}$ is such that $p^{k}x=a\in M$ for some $k\in 
\mathbb{N}$. We can write $a=p^{d}b$ where $b\in M$, $d\geq 0$, and $p$ does
not divide $b$. Then for every $n\geq d$ there exists $x_{n}\in M$ such that 
$p^{k}x_{n}\equiv p^{d}b\ \mathrm{mod}p^{n}M$. This implies that $k=d$, and $%
x_{n}\equiv b\ \mathrm{mod}p^{n}M$ for every $n\geq d$. Thus, $x=b\in M$.
\end{proof}

Recall that, for a countable flat $R$-module $M$, we let $M_{\mathfrak{p}}$
be $R_{\mathfrak{p}}$-module obtained as the localization of $M$ at $%
\mathfrak{p}$. Explicitly, $M_{\mathfrak{p}}$ is the $R_{\mathfrak{p}}$%
-submodule of $K\otimes _{R}M$ generated by $M$, thus its elements are of
the form $ax$ for $x\in M$ and $a\in R_{\mathfrak{p}}$.

\begin{lemma}
\label{Lemma:reduce-local}Suppose that $M$ is a countable flat $R$-module,
and $E$ is a submodule of $M$. Then:

\begin{enumerate}
\item $M$ is completely reduced if and only if $M_{\mathfrak{p}}$ is reduced
for every $\mathfrak{p}\in \mathbb{P}$;

\item $E$ is pure in $M$ if and only if $E_{\mathfrak{p}}$ is pure in $M_{%
\mathfrak{p}}$ for every $\mathfrak{p}\in \mathbb{P}$.
\end{enumerate}
\end{lemma}

\begin{proof}
(1) Suppose that $\mathfrak{p}\in \mathbb{P}$. If $M_{\mathfrak{p}}$ is
reduced then $\mathfrak{p}^{\infty }M\subseteq \mathfrak{p}^{\infty }M_{%
\mathfrak{p}}=\left\{ 0\right\} $. If $M$ is completely reduced, and $y\in 
\mathfrak{p}^{\infty }M_{\mathfrak{p}}$, then one can write $y=ax$ with $%
a\in R_{\mathfrak{p}}$ and $x\in M$. Since $y\in \mathfrak{p}^{\infty }M_{%
\mathfrak{p}}$, for every $n\in \omega $, $\mathfrak{p}^{-n}y\subseteq M_{%
\mathfrak{p}}$. Thus, there is an ideal $J_{n}$ coprime with $\mathfrak{p}$
such that $\mathfrak{p}^{-n}y\subseteq J_{n}^{-1}M$. Thus, $J_{n}y\subseteq 
\mathfrak{p}^{n}M$. Since $J_{n}$ is coprime with $\mathfrak{p}$ we have $%
J_{n}+\mathfrak{p}^{n}=R$. Thus, we have 
\begin{equation*}
Ry=\left( J_{n}+\mathfrak{p}^{n}\right) y\subseteq J_{n}y+\mathfrak{p}%
^{n}y\subseteq \mathfrak{p}^{n}M\text{.}
\end{equation*}%
As this holds for every $n\in \omega $, we conclude $y\in \mathfrak{p}%
^{\infty }M=0$. This concludes the proof.

(2) Suppose that $E$ is pure in $M$. Then if $y\in M_{\mathfrak{p}}$ and $%
r\in R_{\mathfrak{p}}$ are such that $ry\in E_{\mathfrak{p}}$, there exists
an ideal $J$ coprime with $\mathfrak{p}$ such that $ry\in J^{-1}E$ and $r\in
J^{-1}R$. Thus, we have that $Jry\subseteq E$ and $Jr\subseteq R$. If $s\in
Jr$ then we have $sy\in E$ and hence $y\in E$ since $E$ is pure in $M$.

Conversely, suppose that $E_{\mathfrak{p}}$ is pure in $M_{\mathfrak{p}}$
for every $\mathfrak{p}\in \mathbb{P}$. Suppose that $y\in M$ and $r\in R$
are such that $ry\in E$. Then there exist prime ideals $\mathfrak{p}%
_{1},\ldots ,\mathfrak{p}_{n}$ such that $Rr=\mathfrak{p}_{1}\cdots 
\mathfrak{p}_{n}$. We prove that $y\in E$ by induction on $n$. If $n=1$ then
we have $Rr=\mathfrak{p}$. Thus, we have that $\mathfrak{p}y\subseteq
E\subseteq E_{\mathfrak{p}}$. Since $E_{\mathfrak{p}}$ is pure in $M_{%
\mathfrak{p}}$, this implies that $y\in E_{\mathfrak{p}}$. Thus, there
exists an ideal $J$ coprime with $\mathfrak{p}$ such that $y\in J^{-1}E$.
Thus, we have $Jy\subseteq E$. Hence $Ry=\left( J+\mathfrak{p}\right)
y\subseteq E$ and $y\in E$. This shows that $y\in E$ when $n=1$. If the
conclusion holds for $n-1$, then we have $\mathfrak{p}_{1}\cdots \mathfrak{p}%
_{n-1}y\subseteq E$ by the case $n=1$, and then $y\in E$ by the case $n-1$.
\end{proof}

We now establish some properties of the modules $\Xi \left( \tau \right) $.
Recall the definition of $\Omega \left( A\right) $ for some countable flat
module $A$ from \ref{Definition:Omega}.

\begin{lemma}
\label{Lemma:XI}For $\tau =\left( \tau _{\mathfrak{p}}\right) _{\mathfrak{p}%
\in \mathbb{P}}\in J\left( R\right) $, consider the pointed module $\left(
\Xi \left( \tau \right) ,e\right) $ defined above. Denote by $E$ the
submodule of $\Xi \left( \tau \right) $ generated by $e$. Then:

\begin{enumerate}
\item $\Xi \left( \tau \right) $ is completely reduced;

\item $\left( \Xi \left( \tau \right) ,e\right) $ is a well-pointed module;

\item $E$ is the submodule of $\Xi \left( \tau \right) $ comprising the $%
\tau $-divisible elements;

\item $E$ is a characteristic submodule of $\Xi \left( \tau \right) $;

\item $E_{\tau }\left( \Xi \left( \tau \right) \right) =\Xi \left( \tau
\right) $;

\item $\Xi \left( \tau \right) $ is rigid;

\item for $\lambda \in J\left( R\right) $ not essentially equivalent to $%
\tau $, $E_{\lambda }\left( \Xi \left( \tau \right) \right) =0$;

\item if $A$ is a countable reduced flat module and $\tau \in \Omega \left(
A\right) $, then $\mathrm{Hom}\left( \Xi \left( \tau \right) ,A\right) =0$;
\end{enumerate}
\end{lemma}

\begin{proof}
We have that (8) is a consequence of (5) in view of Lemma \ref%
{Lemma:Hom(-,A)}.

By Lemma \ref{Lemma:reduce-local}, after localizing at a prime, we can
assume that $R$ is a DVR. In this case, $\mathbb{P}$ contains a single
maximal ideal $\mathfrak{p}$, which is furthermore principal \cite[Example
4.82]{rotman_introduction_2009}. We can thus fix a generator $p$ of $%
\mathfrak{p}$. In the notation below, we will use $p$ for the subscript
instead of $\mathfrak{p}$, to simplify the notation. In this case, the ideal 
$I_{p,n}$ is principal, and generated by $g_{p,n}:=p^{-n}\left( e-\tau
_{p,n}f\right) $. Therefore, we have that $e\equiv \tau _{p}f\ \mathrm{mod}%
p^{n}$ for every $n\in \omega $, and $e\in E_{\tau }\left( \Xi \left( \tau
\right) \right) $.

(1) We need to prove that $\Xi \left( \tau \right) $ is reduced. If $\Xi
\left( \tau \right) $ is not reduced, as it has rank $2$, it must be
isomorphic to $K\oplus L$ where $L$ has rank one. This implies that $E_{\tau
}\left( \Xi \left( \tau \right) \right) =\left\{ 0\right\} $ by Lemma \ref%
{Lemma:rank-one-E}, contradicting the fact that $e\in E_{\tau }\left( \Xi
\left( \tau \right) \right) $.

(2) Suppose that $x\in \Xi \left( \tau \right) $ and $d\geq 0$ are such that 
$p^{d}x=e$. We can write $x=ag_{p,n}+bf$ for some $a,b\in R$ and $n\geq d$.
Thus, we have that%
\begin{eqnarray*}
e &=&p^{d}x=p^{d}ag_{p,n}+p^{d}bf \\
&=&p^{d}a\left( p^{-n}(e-\tau _{p,n}f)\right) +p^{d}bf \\
&=&p^{d-n}ae+\left( p^{d}b-p^{d-n}a\tau _{p,n}\right) f
\end{eqnarray*}%
This implies that%
\begin{equation*}
p^{d-n}a=1
\end{equation*}%
and%
\begin{equation*}
p^{d}b=p^{d-n}a\tau _{p,n}=\tau _{p,n}\text{.}
\end{equation*}%
This shows that $p^{d}$ divides $\tau _{p}$. Since $\tau _{p}\equiv 1\ 
\mathrm{mod}p$ this implies that $d=0$;

(3) Suppose that $x\in \Xi \left( \tau \right) $ is $\tau _{p}$-divisible,
whence $x=\tau _{p}y$ for some $y\in \Xi \left( \tau _{p}\right) $.\ We work
in $K_{p}\left( \tau _{p}\right) \otimes \Xi \left( \tau \right) $, which is
a $1$-dimensional $K_{p}\left( \tau _{p}\right) $-vector space. We can write 
$x=se+tf$ and $y=s^{\prime }e+t^{\prime }f$ for some $s,s^{\prime
},t,t^{\prime }\in K$. Thus, we have%
\begin{equation*}
s\tau _{p}f+tf=x=\tau _{p}y=s^{\prime }\tau _{p}^{2}f+t^{\prime }\tau _{p}f
\end{equation*}%
This implies that%
\begin{equation*}
s^{\prime }\tau _{p}^{2}+\left( t^{\prime }-s\right) \tau _{p}-t=0\text{.}
\end{equation*}%
Since $1$, $\tau _{p}$, $\tau _{p}^{2}$ are linearly independent over $K$ by
the choice of $\tau _{p}$, this implies $t=s^{\prime }=0$ and $t^{\prime }=s$%
. Thus, we have $x\in E$.

(4) This is a consequence of (3).

(5) This follows from the fact that $e,f\in E_{\tau }\left( \Xi \left( \tau
\right) \right) $.

(6) If $\varphi $ is an endomorphism of $\Xi \left( \tau \right) $, we have $%
\varphi \left( e\right) =re$ for some $r\in R$ by (4). Thus, since $e=\tau f$
we also have%
\begin{equation*}
\tau rf=re=\varphi \left( \tau f\right) =\tau \varphi \left( f\right)
\end{equation*}%
and hence $\varphi \left( f\right) =rf$.\ Since the pure submodule of $\Xi
\left( \tau \right) $ generated by $e$ and $f$ is $\Xi \left( \tau \right) $%
, this shows that $\varphi $ is the endomorphism $x\mapsto rx$.

(7) Suppose that $\lambda \in J(R)$ and $x,y\in \Xi \left( \tau \right) $
are such that $\lambda _{p}y=x$. Consider the subfield $K\left( \tau
_{p},\lambda _{p}\right) $ of $\hat{K}_{p}$ generated by $\tau _{p},\lambda
_{p}$. We work in $K\left( \tau _{p},\lambda _{p}\right) \otimes \Xi \left(
\tau \right) $. We can write $x=se+tf$ and $y=s^{\prime }e+t^{\prime }f$ for
some $s,s^{\prime },t,t^{\prime }\in K$. Then we have that%
\begin{equation*}
\left( s\tau _{p}+t\right) f=se+tf=\lambda _{p}y=\lambda _{p}\left(
s^{\prime }\tau _{p}+t^{\prime }\right) f
\end{equation*}%
and hence%
\begin{equation*}
\left( s\tau _{p}+t\right) =\lambda _{p}\left( s^{\prime }\tau
_{p}+t^{\prime }\right) \text{.}
\end{equation*}%
This shows that $\lambda _{p}\in K\left( \tau _{p}\right) $ and hence $%
\lambda _{p}$, $\tau _{p}$ are essentially equivalent.
\end{proof}

\subsection{$\Xi $-extractable modules}

We define what it means for a countable flat module to have (plain) $\Xi $-%
\emph{extractable} length $\alpha $. We also define the $\Xi $-invariant $%
\tau _{M}$ and the minimal family $\xi _{M}$ of $M$.\ Recall the definition
of the sets of indices $I_{\alpha }$ and $I_{\alpha }^{\mathrm{plain}}$ for $%
\alpha <\omega _{1}$.

\begin{definition}
Let $M$ be a coreduced countable flat module.

\begin{itemize}
\item $M$ has plain $\Xi $-length $1$, $\Xi $-invariant $\boldsymbol{\tau }%
\in M^{I_{1}^{\mathrm{plain}}}$ and minimal family $\xi \in M^{I_{1}^{%
\mathrm{plain}}}$ if, for some $\tau \in J\left( R\right) $, 
\begin{equation*}
M=\Xi \left( \tau \right) \text{,}
\end{equation*}
\begin{equation*}
\xi \left( 0\right) =e_{\Xi \left( \tau \right) }\text{,}
\end{equation*}
\begin{equation*}
\boldsymbol{\tau }\left( 0\right) =[\tau ]\text{;}
\end{equation*}

\item $M$ has $\Xi $-length $\alpha $, $\Xi $-invariant $\boldsymbol{\tau }%
\in M^{I_{\alpha }}$ and minimal family $\xi \in M^{I_{\alpha }}$ if%
\begin{equation*}
M=\bigoplus_{n}M_{n}
\end{equation*}%
\begin{equation*}
\boldsymbol{\tau }\left( -;n\right) =\boldsymbol{\tau }_{n}
\end{equation*}%
\begin{equation*}
\xi \left( -;n\right) =\xi _{n}
\end{equation*}%
where $M_{n}$ has plain $\Xi $-length $\alpha _{n}$, $\Xi $-invariant $%
\boldsymbol{\tau }_{n}\in M_{n}^{I_{\alpha _{n}}^{\mathrm{plain}}}$, and
minimal family $\xi \in M^{I_{\alpha _{n}}^{\mathrm{plain}}}$, and the
function $\boldsymbol{\tau }$ is injective;

\item $M$ has plain $\Xi $-length $\alpha $, $\Xi $-invariant $\boldsymbol{%
\tau }\in M^{I_{\alpha }^{\mathrm{plain}}}$, and minimal family $\xi \in
M^{I_{\alpha }^{\mathrm{plain}}}$ if there exists a fishbone of modules 
\begin{equation*}
(\left( L_{n,i},\xi _{n,i}\left( \left( \alpha -1\right) _{n}-1\right)
,J_{n},e_{n,i},\varphi _{n}\right) _{i<d_{n}})_{n\in \omega }
\end{equation*}%
with fishbone colimit $M$ and $\lambda \in J\left( R\right) $ such that:

\begin{enumerate}
\item for every $n\in \omega $ and $i<d_{n}$, $L_{n,i}$ has plain $\Xi $%
-length $\left( \alpha -1\right) _{n}$, $\Xi $-invariant $\boldsymbol{\tau }%
_{n,i}$, and minimal family $\xi _{n,i}$;

\item if $L$ is the submodule of the fishbone colimit $M$ corresponding to 
\begin{equation*}
\mathrm{co\mathrm{lim}}_{n}\left( L_{n,0}\oplus \cdots \oplus
L_{n,d_{n}-1}\right) \text{,}
\end{equation*}%
then there is an isomorphism between the quotient of $M$ by $L$ and $\Xi
\left( \lambda \right) $ that maps the element corresponding to $\xi \left(
\alpha -1\right) $ to $e_{\Xi \left( \lambda \right) }$;

\item $\boldsymbol{\tau }\left( \alpha -1\right) =[\lambda ]$;

\item $\boldsymbol{\tau }\left( -;d_{0}+\cdots +d_{n-1}+i\right) =%
\boldsymbol{\tau }_{n,i}$ for $n<\omega $ and $i<d_{n}$;

\item $\xi \left( -;d_{0}+\cdots +d_{n-1}+i\right) =\xi _{n,i}$ for $%
n<\omega $ and $i<d_{n}$;

\item $\boldsymbol{\tau }$ is injective.
\end{enumerate}
\end{itemize}

The module $M$ is $\Xi $-extractable if it has (plain) $\Xi $-length $\alpha 
$ for some $\alpha <\omega _{1}$.
\end{definition}

It follows easily by induction that the (plain) fishbone length of a module
is less than or equal to the (plain) $\Xi $-length. Thus, any $\Xi $%
-extractable module is fishbone extractable. Furthermore, the minimal family
of a $\Xi $-extractable module is indeed a minimal family in the sense of
Definition \ref{Definition:minimal-family}.

\begin{proposition}
\label{Proposition:existence-Xi}For every $\alpha <\omega _{1}$ and infinite 
$S\subseteq \boldsymbol{J}\left( R\right) $ there exists a well-pointed
countable flat module $M$ that has plain $\Xi $-length $\alpha $ and $\Xi $%
-invariant $\boldsymbol{\lambda }$ with image contained in $S$.
\end{proposition}

\begin{proof}
By induction on $\alpha $. For $\alpha =1$ define $M$ to be $\Xi \left(
\lambda \right) $ for some $\lambda \in S$.

Suppose that this holds for $\beta <\alpha $. Pick $\lambda \in S$ and
define $J:=\Xi \left( \lambda \right) $ with $e=e_{\Xi \left( \lambda
\right) }$ as a distinguished element. Write $J=\mathrm{co\mathrm{lim}}%
_{n}J_{n}$ for finite submodules $J_{n}\subseteq J$ with $\dim \left(
J_{n}\right) =d_{n}$ and $d_{0}=1$. Let $\left( e_{n,i}\right) _{i<d_{n}}$
be a free basis of $J_{n}$. By the inductive hypothesis, for $n<\omega $ and 
$i<d_{n}$ there exists a countable flat module $M_{n,i}$ of plain $\Xi $%
-length $\left( \alpha -1\right) _{n}$, $\Xi $-invariant $\boldsymbol{%
\lambda }_{n}$, and minimal family $\xi _{n}$ such that $\left( \boldsymbol{%
\lambda }_{n}\right) _{n\in \omega }$ have images that are pairwise disjoint
and contained in $S\setminus \left\{ \lambda \right\} $. Consider the
fishbone of modules 
\begin{equation*}
\left( M_{n,i},\xi _{n,i}\left( \left( \alpha -1\right) _{n}-1\right)
,J_{n},\varphi _{n},e_{n,i}\right)
\end{equation*}%
Then its corresponding fishbone colimit has plain $\Xi $-length $\alpha $,
and $\Xi $-invariant $\boldsymbol{\lambda }$ with image contained in $S$.
\end{proof}

The following proposition is easily established by induction on $\alpha $.
Recall the definition of $\Omega \left( A\right) $ for some countable flat
module $A$ from \ref{Definition:Omega}.

\begin{proposition}
\label{Proposition:XI-independent}Let $M$ be a coreduced countable flat
module of (plain) $\Xi $-length $\alpha $ and $\Xi $-invariant $\boldsymbol{%
\lambda }$. Let also $A$ be a countable flat module with $\Omega \left(
A\right) $ disjoint from the image of $\boldsymbol{\lambda }$. Then $M$ is
extractable of (plain) length $\alpha $ and $A$-independent.
\end{proposition}

\subsection{Zipping}

Let $\alpha $ be a countable ordinal. We define a \emph{successor} function $%
s_{\alpha }:I_{\alpha }\rightarrow I_{\alpha }$ by induction on $\alpha $.
For $\alpha =1$ define%
\begin{equation*}
s_{\alpha }\left( 0;n\right) :=\left( 0;n+1\right) \text{.}
\end{equation*}%
If it has been defined for ordinals less than $\alpha $, define%
\begin{equation*}
s_{\alpha }\left( \left( \alpha -1\right) _{n}\right) :=\left( \alpha
-1\right) _{n+1}
\end{equation*}%
and%
\begin{equation*}
s_{\alpha }\left( i;n\right) =\left( s_{\alpha _{n}}\left( i\right) ;n\right)
\end{equation*}%
for $i\in I_{\left( \alpha -1\right) _{n}}$ and $n\in \omega $. This is
indeed the successor function when $I_{\alpha }$ is regarded as a linear
order isomorphic to $\omega ^{\alpha }$ as described in\ Section \ref%
{Subsection:indices}.

Given a module and a function $\xi \in M^{I_{\alpha }}$, we define $\delta
\xi \in M^{I_{\alpha }}$ to be the function $\xi -\xi \circ s_{\alpha }$.
Suppose now that $M$ is a coreduced and completely reduced $\Xi $%
-extractable module of length $\alpha \geq 1$ with $\Xi $-invariant $%
\boldsymbol{\tau }$ and minimal family $\xi $. Let also $\boldsymbol{\lambda 
}\in \boldsymbol{J}\left( R\right) ^{I_{\alpha }}$ be an injective function
whose image is disjoint from the image of $\tau $. For each $i\in I_{\alpha
} $ we let $\lambda \left( i\right) \in J\left( R\right) $ be such that $%
[\lambda \left( i\right) ]=\boldsymbol{\lambda }\left( i\right) $. Let $S$
be the direct sum of $\Xi \left( \lambda \left( i\right) \right) $ for $i\in
I_{\alpha }$. Define then the \emph{zipping} $Z\left( M,\boldsymbol{\lambda }%
\right) $ of $M$ along $\boldsymbol{\lambda }$ to be the quotient of $%
M\oplus S$ by the submodule $N$ generated by $\left( \delta \xi \right)
\left( i\right) -e_{\lambda \left( i\right) }$ for $i\in I_{\alpha }$. We
can identify $M$ and $S$ as submodules of $Z\left( M,\boldsymbol{\lambda }%
\right) $.

\begin{lemma}
\label{Lemma:zipping}Adopt the notation above. Then we have that:

\begin{enumerate}
\item $\left( \left( \delta \xi \right) \left( i\right) -e_{\lambda \left(
i\right) }\right) _{i\in I_{\alpha }}$ is a minimal family in $M\oplus S$;

\item the submodule $N$ of $M\oplus S$ generated by $\left( \left( \delta
\xi \right) \left( i\right) -e_{\lambda \left( i\right) }\right) _{i\in
I_{\alpha }}$ is pure;

\item $Z\left( M,\boldsymbol{\lambda }\right) $ is a countable flat module;

\item $\sigma _{\alpha }Z\left( M,\boldsymbol{\lambda }\right) =Z\left(
\sigma _{\alpha }M,\boldsymbol{\lambda }\right) $ for $0\leq \alpha <\omega
_{1}$;

\item $\partial _{\alpha }Z\left( M,\boldsymbol{\lambda }\right) \cong
\partial _{\alpha }M$ for $0\leq \alpha <\omega _{1}$;

\item if $A$ is a countable flat module, $E_{\lambda \left( i\right) }\left(
A\right) =0$ for every $i\in I_{\alpha }$, and $M$ is $A$-independent, then $%
Z\left( M,\boldsymbol{\lambda }\right) $ has $A$-projective length $\alpha $%
, and 
\begin{equation*}
s_{\beta }\mathrm{Ext}\left( Z\left( M,\boldsymbol{\lambda }\right)
,A\right) \cong \mathrm{Ext}\left( \partial _{\beta }M,A\right)
\end{equation*}%
for $\beta \geq 1$;

\item $M$ is a characteristic submodule of $Z\left( M,\boldsymbol{\lambda }%
\right) $;

\item $Z\left( M,\boldsymbol{\lambda }\right) $ is rigid;

\item $Z\left( M,\boldsymbol{\lambda }\right) $ is completely coreduced.
\end{enumerate}
\end{lemma}

\begin{proof}
We prove that the conclusion holds by induction on $\alpha $. Define $\eta
_{i}:=\left( \delta \xi \right) \left( i\right) -e_{\lambda \left( i\right)
} $ for $i\in I_{\alpha }$.

(1) For $\alpha =1$, this follows from the fact that $\left( \delta \xi
\left( i\right) \right) _{i\in I_{1}}$ is the image of $\left( \xi \left(
i\right) \right) _{i\in I_{1}}$ under the isomorphism $M\rightarrow M$ that
maps $\left( x_{i}\right) $ to $\left( x_{i}-x_{i+1}\right) $. Suppose that
the conclusion holds for $\beta <\alpha $. Then we can write%
\begin{equation*}
M\cong \bigoplus_{n}M_{n}
\end{equation*}%
where for $n\in \omega $, $M_{n}$ is $\Xi $-extractable of plain length $%
\alpha _{n}$. Define 
\begin{equation*}
L:=\bigoplus_{n}\sigma _{\alpha _{n}-1}M_{n}\subseteq M
\end{equation*}%
and%
\begin{equation*}
J:=\bigoplus_{n}\partial _{\alpha _{n}-1}M_{n}\text{.}
\end{equation*}%
We have an exact sequence%
\begin{equation*}
L\rightarrow M\rightarrow J\text{.}
\end{equation*}%
For $n\in \omega $, $\sigma _{\alpha _{n}-1}M_{n}$ is extractable of length $%
\alpha _{n}-1$. Thus, by the inductive hypothesis, $\left( \eta _{i}\right)
_{i\in I_{\alpha _{n}-1}}$ is minimal. By the case $\alpha =1$, the family 
\begin{equation*}
\left( \eta _{\left( \alpha _{n}-1,n\right) }+L\right) _{n\in \omega }
\end{equation*}%
is minimal in $J$. It follows that $\left( \eta _{i}\right) _{i\in I}$ is
minimal in $M$ by Lemma \ref{Lemma:minimal}.

(2) and (3) are consequences of (1).

(4) and (5) Write $M=\mathrm{co\mathrm{lim}}_{n}M_{n}$ for some finite-rank
pure submodules $M_{n}\subseteq M$. Define $Z\left( M_{n},\boldsymbol{%
\lambda }\right) $ to be the submodule of $Z\left( M,\boldsymbol{\lambda }%
\right) $ generated by $M_{n}$ and $\Xi \left( \lambda \left( i\right)
\right) $ for $i\leq n$. Then $Z\left( M,\boldsymbol{\lambda }\right) =%
\mathrm{co\mathrm{lim}}_{n}{}Z\left( M_{n},\boldsymbol{\lambda }\right) $.
Since $\Xi \left( \lambda \left( i\right) \right) $ is coreduced for every $%
i\in \omega $, the inclusion $M_{n}\rightarrow Z\left( M_{n},\boldsymbol{%
\lambda }\right) $ induces an isomorphism%
\begin{equation*}
\mathrm{Hom}\left( M_{n},R\right) \cong \mathrm{Hom}\left( Z\left( M_{n},%
\boldsymbol{\lambda }\right) ,R\right) \text{.}
\end{equation*}%
Thus, the inclusion $M\rightarrow Z\left( M,\boldsymbol{\lambda }\right) $
induces an isomorphism%
\begin{eqnarray*}
\mathrm{Ext}\left( \partial _{1}M,R\right) &\cong &\mathrm{Ph}^{1}\mathrm{Ext%
}\left( M,R\right) \\
&\cong &\mathrm{\mathrm{lim}}^{1}\mathrm{Hom}\left( M_{n},R\right) \\
&\cong &\mathrm{lim}^{1}\mathrm{Hom}\left( Z\left( M_{n},\boldsymbol{\lambda 
}\right) ,R\right) \\
&\cong &\mathrm{Ph}^{1}\mathrm{Ext}\left( Z\left( M,\boldsymbol{\lambda }%
\right) ,R\right) \\
&\cong &\mathrm{Ext}\left( \partial _{1}Z\left( M,\boldsymbol{\lambda }%
\right) ,R\right) \text{.}
\end{eqnarray*}%
The conclusion thus follows by induction.

(6) Since $M$ is $A$-independent, it has projective $A$-rank $\alpha $ by
Theorem \ref{Theorem:A-independent}. Since $E_{\lambda \left( i\right)
}\left( A\right) =0$ for every $i\in I_{\alpha }$, we have $\mathrm{Hom}%
\left( \Xi \left( \lambda \left( i\right) \right) ,A\right) =0$ for every $%
i\in I_{\alpha }$ by Lemma \ref{Lemma:Hom(-,A)}. The conclusion thus can be
obtained as in (4) and (5) above.

(7) We have that $\sigma _{1}{}Z\left( M,\boldsymbol{\lambda }\right)
=Z\left( \sigma _{1}M,\boldsymbol{\lambda }\right) $ is a characteristic
submodule of $Z\left( M,\boldsymbol{\lambda }\right) $. Thus, it suffices to
consider the case when $\alpha =1$. In this case, we have that $M$ is the
submodule $E_{\boldsymbol{\tau }}\left( Z\left( M,\boldsymbol{\lambda }%
\right) \right) $ and hence it is characteristic.

(8) By induction on $\alpha $. For an element $a$ of $Z\left( M,\boldsymbol{%
\lambda }\right) $, we denote by $\langle a\rangle $ the submodule of $%
Z\left( M,\boldsymbol{\lambda }\right) $ generated by $a$. For $\alpha =1$,
we have%
\begin{equation*}
E_{\lambda \left( k\right) }\left( Z\left( M,\boldsymbol{\lambda }\right)
\right) =\langle \xi _{\left( 0;k\right) }-\xi _{\left( 0;k+1\right) }\rangle
\end{equation*}%
and%
\begin{equation*}
E_{\tau \left( 0;k\right) }\left( Z\left( M,\boldsymbol{\lambda }\right)
\right) =\langle \xi _{\left( 0;k\right) }\rangle
\end{equation*}%
for every $k\in \omega $. This implies that for every $k\in \omega $ there
exists $r_{k},s_{k}\in R$ such that%
\begin{equation*}
\varphi \left( \xi _{\left( 0;k\right) }\right) =r_{k}\xi _{\left(
0;k\right) }
\end{equation*}%
and%
\begin{equation*}
\varphi \left( \xi _{\left( 0;k\right) }-\xi _{\left( 0;k+1\right) }\right)
=s_{k}\xi _{\left( 0;k\right) }-s_{k}\xi _{\left( 0;k+1\right) }\text{.}
\end{equation*}%
In turn, this implies that $r_{k}=s_{k}=r_{0}$ for every $k\in \omega $.
Thus, $\varphi $ is the endomorphism $x\mapsto r_{0}x$.

Assume the conclusion holds for $\beta <\alpha $. As before, we can write%
\begin{equation*}
M\cong \bigoplus_{n}M_{n}
\end{equation*}%
where for $n\in \omega $, $M_{n}$ is $\Xi $-extractable of plain rank $%
\alpha _{n}$. Again, define 
\begin{equation*}
L:=\bigoplus_{n}\sigma _{\alpha _{n}}M_{n}\text{,}
\end{equation*}%
\begin{equation*}
J:=\bigoplus_{n}\partial _{\alpha _{n}}M_{n}\text{,}
\end{equation*}%
and consider the short exact sequence%
\begin{equation*}
L\rightarrow M\rightarrow J\text{.}
\end{equation*}%
By the inductive hypothesis, for every $n\in \omega $ there exists $r_{n}\in
R$ such that $\varphi |_{M_{n}}:x\mapsto r_{n}x$. Furthermore, we have that 
\begin{equation*}
E_{\tau \left( \alpha _{n}-1;n\right) }\left( Z\left( M,\boldsymbol{\lambda }%
\right) \right) =\langle \xi _{\left( \alpha _{n}-1;n\right) }\rangle
\end{equation*}%
and%
\begin{equation*}
E_{\lambda \left( \alpha _{n}-1;n\right) }\left( Z\left( M,\boldsymbol{%
\lambda }\right) \right) =\langle \xi _{\left( \alpha _{n}-1;n\right) }-\xi
_{\left( \alpha _{n+1}-1;n+1\right) }\rangle \text{.}
\end{equation*}%
As above, this implies that $r_{n}=r_{0}$ for every $n\in \omega $ and hence 
$\varphi $ is the endomorphism $x\mapsto r_{0}x$.

(9) After localization one can assume that $R$ is a DVR. The conclusion is
then obtained by induction on $\alpha $. For $\alpha =1$ it is an easy a
calculation.\ The case of $\alpha $ follows from the cases of $1$ and $%
\alpha _{n}-1$ for $n\in \omega $ and considering the extension%
\begin{equation*}
L\rightarrow M\rightarrow J
\end{equation*}%
as above.
\end{proof}

\subsection{Dichotomy}

Summarizing the results established above, we obtained the following
Dichotomy Theorem for modules over a countable Dedekind domain.

\begin{theorem}
\label{Theorem:dichotomy}Suppose that $R$ is a countable Dedekind domain,
and $A$ is a countable flat module. If $A$ is divisible, then $\mathrm{Ext}%
\left( -,A\right) =0$. If $A$ is not divisible, then for every $\alpha
<\omega _{1}$:

\begin{enumerate}
\item there exists an extractable $A$-independent countable coreduced and
completely reduced module of (plain) length $\alpha $ and (plain) $A$%
-projective length $\alpha $;

\item there exists a rigid countable coreduced and completely reduced module
of $R$-projective and $A$-projective length $\alpha $.
\end{enumerate}
\end{theorem}

\begin{proof}
If $A$ is divisible, then $\mathrm{Ext}\left( -,A\right) =0$ by Proposition %
\ref{Proposition:characterize-injectives}. If $A$ is not divisible, then by
Proposition \ref{Proposition:XI-independent} for every $\alpha <\omega _{1}$
there exists a extractable $A$-independent countable coreduced flat module $%
M $ of (plain) length $\alpha $. By Theorem \ref{Theorem:A-independent}, $M$
has (plain) $A$-projective length $\alpha $.

Considering such an $M$ of length $\alpha $, we can obtain by Lemma \ref%
{Lemma:zipping} a new \emph{rigid} module which is still countable,
coreduced, completely reduced, and of $R$-projective and $A$-projective
length $\alpha $.
\end{proof}

\begin{corollary}
\label{Corollary:chain-exact}Suppose that $R$ is a countable Dedekind
domain. Let $\mathcal{E}_{\alpha }$ be the exact structure projectively
generated by the countable flat module of plain tree length at most $\alpha
_{n}$ for $n\in \omega $. Then $\left( \mathcal{E}_{\alpha }\right) _{\alpha
<\omega _{1}}$ is a strictly decreasing chain of nontrivial exact structures
on the category $\mathbf{Flat}\left( R\right) $ of countable flat modules
over $R$.
\end{corollary}

\begin{proof}
This follows from Theorem \ref{Theorem:dichotomy} and Theorem \ref%
{Theorem:phantom-projective}.
\end{proof}

\subsection{Modules that are not phantom projective}

While the modules we have considered so far are extractable, there exist
modules that fail to be so. Indeed, we now provide examples of modules of
that are not phantom projective, which necessarily are not extractable.

\begin{theorem}
\label{Theorem:not-constructible}Suppose that $R$ is a countable Dedekind
domain. Then there exists a countable module $C$ of $R$-projective length $1$
such that:

\begin{enumerate}
\item for every phantom pro-finiteflat module $\boldsymbol{M}$ there exists
a countable flat module $A$ such that $\mathrm{Ph}^{1}\mathrm{Ext}\left(
C,A\right) \cong \boldsymbol{M}$;

\item $C$ is not a phantom projective.
\end{enumerate}
\end{theorem}

\begin{proof}
After replacing $R$ with the localization of $R$ at a maximal ideal, we can
assume that $R$ is a DVR with maximal ideal generated by $p\in R$. Let $%
\omega ^{<\omega }$ be the set of finite tuples of elements of $\omega $.
The idea is to construct a module with generators and relations that code
any possible tower of finite flat modules. Let $\mu \left( \varphi ,i\right) 
$ and $\lambda \left( \varphi ,i,\psi \right) $ be pairwise not essentially
equivalent elements of $J\left( R\right) $ for $n\in \omega $, $\varphi \in 
\mathrm{Hom}\left( R^{n},R^{n-1}\right) $, and $\psi \in \mathrm{Hom}\left(
R^{n-1},R^{n-2}\right) $, where by convention $\mathrm{Hom}\left(
R^{n},R^{n-1}\right) =\varnothing $ for $n\in \left\{ 0,1\right\} $. Set%
\begin{equation*}
B\left( \varphi ,i\right) :=\Xi \left( \mu \left( \varphi ,i\right) \right)
\end{equation*}%
and%
\begin{equation*}
e\left( \varphi ,i\right) :=e_{\Xi \left( \mu \left( \varphi ,i\right)
\right) }\in \Xi \left( \mu \left( \varphi ,i\right) \right) =B\left(
\varphi ,i\right) \text{.}
\end{equation*}%
Likewise we define $B\left( \varphi ,i,\psi \right) $ and $e\left( \varphi
,i,\psi \right) $. Let $B$ be the direct sum of $B\left( \varphi ,i\right) $
and $B\left( \varphi ,i,\psi \right) $ as above. Thus, $B$ is the pure
submodule of itself generated by $e\left( \varphi ,i\right) $ and $e\left(
\varphi ,i,\psi \right) $. Fix $\varphi \in \mathrm{Hom}\left(
R^{n},R^{n-1}\right) $, and let $a_{\varphi }\left( j,i\right) \in R$ such
that%
\begin{equation*}
\varphi \left( \delta _{i}\right) =\sum_{j<n-1}a_{\varphi }\left( j,i\right)
\delta _{j}
\end{equation*}%
for $i<n$, where $\left( \delta _{i}\right) _{i<n}$ is the canonical basis
of $R^{n}$. We want to add the relation%
\begin{equation*}
e\left( \varphi ,i\right) =\sum_{j<n-1}a_{\varphi }\left( j,i\right) e\left(
\psi ,j\right) +e\left( \varphi ,i,\psi \right)
\end{equation*}%
for any $\varphi \in \mathrm{Hom}\left( R^{n},R^{n-1}\right) $, $\psi \in 
\mathrm{Hom}\left( R^{n-1},R^{n-2}\right) $, and $i<n$. We therefore
consider the submodule $N$ of $B$ generated by the elements%
\begin{equation*}
\rho \left( \varphi ,i,\psi \right) :=e\left( \varphi ,i\right) -\left(
\sum_{j<n-1}a_{\varphi }\left( j,i\right) e\left( \psi ,j\right) +e\left(
\varphi ,i,\psi \right) \right) \text{.}
\end{equation*}%
It follows easily from Lemma \ref{Lemma:minimal-combination} that $\left(
\rho \left( \varphi ,i,\psi \right) \right) $ is a minimal family in $B$.\
Thus, $N$ is a pure submodule of $B$. We let $c\left( \varphi ,i\right) $ be
the element $e\left( \varphi ,i\right) +N\in C$, and likewise $c\left(
\varphi ,i,\psi \right) $ be the element $\rho \left( \varphi ,i,\psi
\right) +N\in C$.

(1) Given a phantom pro-finiteflat module $\boldsymbol{M}$, consider a tower 
$\left( M_{n}\right) $ of finiteflat modules such that $\mathrm{lim}%
^{1}M_{n}\cong \boldsymbol{M}$. We can assume that $M_{n}\cong R^{n}$ and
the bonding map $M_{n+1}\rightarrow M_{n}$ corresponds to a homomorphism $%
\varphi _{n+1}:R^{n+1}\rightarrow R^{n}$. Now consider the countable flat
module $A$ obtained as follows. Set $\varphi _{0}=\varnothing $. Consider%
\begin{equation*}
A=\bigoplus_{n\in \omega }\bigoplus_{i<n}B\left( \varphi _{n},i\right) \text{%
.}
\end{equation*}%
Then it is easily seen that one can write an increasing sequence $\left(
C_{k}\right) $ of finite-rank pure submodules of $C$ with $C=\mathrm{co%
\mathrm{lim}}_{k}C_{k}$ such that $\left( \mathrm{Hom}\left( C_{k},A\right)
\right) \cong \left( M_{k}\right) $. Indeed, let for every $n\in \omega $, $%
\left( \psi _{n,i}\right) $ be an enumeration of $\mathrm{Hom}\left(
R^{n},R^{n-1}\right) $ with $\psi _{n,0}=\varphi _{n}$. Define $C_{k}$ to be
the pure submodule of $C$ generated by $c\left( \psi _{n,i},j\right) $ for $%
n,i\leq k$ and $j\leq n$. Then we have%
\begin{equation*}
\mathrm{Hom}\left( C_{k},A\right) \cong R^{k}
\end{equation*}%
and under such an isomorphism the inclusion $C_{k-1}\rightarrow C_{k}$
induces the map $\varphi _{k}:R^{k}\rightarrow R^{k-1}$. Thus, $s_{1}\mathrm{%
Ext}\left( C,A\right) \cong \mathrm{lim}_{k}^{1}\mathrm{Hom}\left(
C_{k},A\right) \cong \boldsymbol{M}$.

(2) is a consequence of (1).
\end{proof}

\subsection{Phantom projective modules that are not extractable}

We now prove that when $R$ is a Dedekind domain, there exist phantom
projective modules that are not extractable. Again, after replacing $R$ with
its localization at a prime, we can assume that $R$ is a DVR with maximal
ideal generated by $p$. We recall the definition of the module $\Xi \left(
\tau \right) $ for $\tau \in J\left( R\right) $ in this case. For $n\in
\omega $ pick $\tau _{n}\in R$ such that $\tau \equiv \tau _{n}\ \mathrm{mod}%
p^{n}$. For $n\geq 1$, let $\Xi _{n}\left( \tau \right) $ be submodule of $%
K\oplus K$ generated by $e_{0}:=\left( 1,0\right) =e$, $f_{n}:=\left(
0,1\right) =f$, and 
\begin{equation*}
g_{n}:=p^{-n}(e-\tau _{n}f)\text{.}
\end{equation*}%
Then%
\begin{equation*}
p^{n}g_{n}+\tau _{n}f=e
\end{equation*}%
for every $n\in \omega $. Notice that $\Xi _{n}\left( \tau \right) \subseteq
\Xi _{n+1}\left( \tau \right) $ for $n\in \omega $. Define $\Xi \left( \tau
\right) =\mathrm{co\mathrm{lim}}_{n}\Xi _{n}\left( \tau \right) \subseteq
K\oplus K$. Then we have that $\Xi \left( \tau \right) $ is a rigid module.

Notice that if $M$ is a countable coreduced flat module of plain tree length 
$1$, then it has plain projective length $1$. This implies that $M$ has
finite rank, and $M$ has also plain $R$-projective length $1$. We now show
that this fails for $2$.

\begin{proposition}
Let $R$ be a countable Dedekind domain. Then there exists a countable flat
module that has plain tree length $2$ and $R$-projective length $1$. In
particular, $M$ is a phantom projective module that is not extractable.
\end{proposition}

\begin{proof}
As noticed above, we can assume without loss of generality that $R$ is a
DVR. Fix an element $\tau $ of $J\left( R\right) $. We define $M$ as a wedge
sum 
\begin{equation*}
\bigoplus_{\left( C_{1},\varphi _{i}\right) }C_{\left( i;0\right) }
\end{equation*}%
where $C_{1}=R$, $C_{\left( i;0\right) }=\Xi \left( \tau \right) $, and $%
\psi _{i}:R\rightarrow C_{\left( i;0\right) }$, $1\mapsto p^{i}e$. Let $%
M_{n} $ be the submodule of $M$ corresponding to $C_{1}\oplus C_{\left(
0;0\right) }\cdots \oplus C_{\left( 0;n-1\right) }$. Then we have that%
\begin{equation*}
\mathrm{Ph}^{1}\mathrm{Ext}\left( M,R\right) \cong \mathrm{lim}^{1}\mathrm{%
Hom}\left( M_{n}|C_{1},R\right) =0\text{.}
\end{equation*}%
On the other hand,%
\begin{equation*}
\mathrm{Ph}^{1}\mathrm{Ext}\left( M,\Xi \left( \tau \right) \right) \cong 
\mathrm{lim}^{1}\mathrm{Hom}\left( M_{n}|C_{1},\Xi \left( \tau \right)
\right) \cong \mathrm{lim}^{1}\left( p^{n}R\right) \cong \hat{R}/R\text{.}
\end{equation*}%
This shows that $M$ has plain tree length $2$ and plain projective length $2$%
, but $R$-projective length $1$. In particular, $M$ is not extractable by
Corollary \ref{Corollary:constructible}.
\end{proof}

\section{Complexity of extensions\label{Section:complexity}}

In this last section we recall some fundamental notions from descriptive set
theory, and isolate some complexity-theoretic consequence of the main
results of the previous sections.

\subsection{Complexity of classification problems}

\emph{Invariant complexity theory }provides a framework to compare the
complexity of classification problems in mathematics. Every (reasonable)
classification problem in mathematics can be regarded (perhaps after a
suitable parametrization) as an analytic equivalence relation $E$ on a
Polish space $\boldsymbol{X}$. Recall that a subset $A$ of a Polish space $%
\boldsymbol{X}$ is \emph{analytic} or $\boldsymbol{\Sigma }_{1}^{1}$ if it
is the image of a Borel function $\boldsymbol{Z}\rightarrow \boldsymbol{X}$
for some standard Borel space $\boldsymbol{Z}$. Every Borel set is analytic,
but the converse does not hold. An equivalence relation $E$ on $\boldsymbol{X%
}$ is analytic if it is analytic as a subset of $\boldsymbol{X}\times 
\boldsymbol{X}$.

The notion of \emph{Borel reducibility }allows one to compare the complexity
of different classification problems.

\begin{definition}
Let $E,F$ be analytic equivalence relations on Polish spaces $\boldsymbol{X},%
\boldsymbol{Y}$, respectively. The equivalence relation $E$ is \emph{Borel
reducible }to $F$ if there exists a Borel function $f:\boldsymbol{X}%
\rightarrow \boldsymbol{Y}$ such that, for $x,x^{\prime }\in \boldsymbol{X}$,%
\begin{equation*}
xEx^{\prime }\Leftrightarrow f\left( x\right) Ff\left( x^{\prime }\right) 
\text{.}
\end{equation*}%
The relations $E,F$ are Borel bireducible if both $E$ is Borel reducible to $%
F$ and $F$ is Borel reducible to $E$.
\end{definition}

Special equivalence relations are used as\emph{\ benchmarks }to assess the
complexity of classification problems.

If $G$ is a Polish group and $X$ is a Polish $G$-space, the one can consider
the corresponding \emph{orbit equivalence relation }$E_{G}^{X}$ on $X$. An
equivalence relation is\emph{\ classifiable by orbits} if it is Borel
reducible to $E_{G}^{X}$ for some Polish group $G$ and Polish $G$-space $X$.
By restricting $G$ to smaller classes of Polish groups one obtains stronger
notions of classifiability. Recall that a Polish group $G$ is \emph{%
non-Archimedean} if it has a basis of neighborhoods of the identity
consisting of open subgroups. This is equivalent to the assertion that $G$
is isomorphic to a closed subgroup of the Polish group $S_{\infty }$ of
permutations of $\mathbb{N}$.

\begin{definition}
An equivalence relation $E$ is \emph{classifiable by countable structures}
if it is Borel reducible to the orbit equivalence relation of a Polish $G$%
-space for some non-Archimedean Polish group $G$.
\end{definition}

In fact, one can always take $G$ to be the Polish group $S_{\infty }$ by 
\cite[Theorem 3.5.2]{gao_invariant_2009}. If $\mathcal{L}$ is a countable
first-order language, then the space of countable $\mathcal{L}$-structures
can be regarded as a Polish space $\mathrm{Mod}\left( \mathcal{L}\right) $.
An equivalence relation $E$ is classifiable by countable structures if and
only if it is Borel reducible to the relation of \emph{isomorphism }of
countable $\mathcal{L}$-structures for some countable language $\mathcal{L}$%
, thus justifying the name.

One obtains other benchmarks of complexity by considering distinguished
equivalence relations. We let $\wp \left( \mathbb{N}\right) $ be the
collection of subsets of $\mathbb{N}$. This is a Polish space when
identified with the product $\left\{ 0,1\right\} ^{\mathbb{N}}$. As any
other uncountable Polish space, it admits a Borel bijection with $\mathbb{R}$%
, whence its elements are also called reals. The relations of equality in $%
\wp \left( \mathbb{N}\right) $, $\mathbb{R}$, or any other uncountable
Polish space are all Borel bireducible.

\begin{definition}
An equivalence relation $E$ is smooth or concretely classifiable if it is
Borel reducible to the relation of equality of real numbers.
\end{definition}

A more generous notion of classifiability is obtained by considering on the
space $\left\{ 0,1\right\} ^{\mathbb{N}}$ of binary sequences a coarser
relation than equality.

\begin{definition}
An equivalence relation $E$ is \emph{essentially hyperfinite} if it is Borel
reducible to the relation of \emph{tail equivalence }of binary sequences.
\end{definition}

The relation of tail equivalence of binary sequences is the orbit
equivalence relation associated with a Borel action of a \emph{countable }%
group on $\left\{ 0,1\right\} ^{\mathbb{N}}$. More generally, a Borel
equivalence relation on a Polish space is called \emph{countable} if it is
the orbit equivalence relation associated with a Borel action of a countable
group or, equivalently, if its equivalence classes are countable.

\begin{definition}
An equivalence relation $E$ is \emph{essentially countable }if it is Borel
reducible to a countable Borel equivalence relations.
\end{definition}

Every essentially hyperfinite equivalence relation is essentially countable.
The converse holds for orbit equivalence relations associated with actions
of \emph{abelian }non-Archimedean Polish groups on Polish spaces; see \cite[%
Corollary 6.3]{ding_non-archimedean_2017}. Among the (essentially) countable
Borel equivalence relations there exists one of maximum complexity, usually
denoted by $E_{\infty }$.

We will consider an equivalent presentation of the study of Borel complexity
of classification problems, in terms of Borel-definable and $\boldsymbol{%
\Sigma }_{1}^{1}$-definable sets.

\subsection{Definable and semidefinable set}

By a $\boldsymbol{\Sigma }_{1}^{1}$-definable set we mean a pair $(%
\boldsymbol{X},E)$ where $\boldsymbol{X}$ is a Polish space and $E$ is an
analytic equivalence relation on $\boldsymbol{X}$. We think of the pair $%
\left( \boldsymbol{X},E\right) $ as an explicit presentation of the space $X=%
\boldsymbol{X}/E$ of $E$-equivalence classes. A subset $Z$ of $X$ is Borel
(respectively, analytic) if its \emph{lift }$\boldsymbol{Z}:=\left\{ z\in 
\boldsymbol{X}:[z]_{E}\in Z\right\} $ is a Borel (respectively, analytic)
subset of $\boldsymbol{X}$. A function $f:X\rightarrow Y$ between $%
\boldsymbol{\Sigma }_{1}^{1}$-definable sets is Borel if it is induced by a
Borel function $\boldsymbol{X}\rightarrow \boldsymbol{Y}$. We consider $%
\boldsymbol{\Sigma }_{1}^{1}$-definable sets as objects of a category with
Borel functions as morphisms. The more restrictive notion of Borel-definable
set $X:=\boldsymbol{X}/E$ is obtained by requiring that $E$ be an
equivalence relation on $\boldsymbol{X}$ that satisfies stronger
definability requirements.

We recall the notion of idealistic equivalence relation, which we present in
a slightly stronger formulation than in \cite[Definition 3.1]%
{bergfalk_definable_2024}. Given an equivalence relation $E$ on $\boldsymbol{%
X}$ we define the equivalence relation $E^{\mathbb{N}}$ on $\boldsymbol{X}^{%
\mathbb{N}}$ by setting%
\begin{equation*}
\left( x_{n}\right) E^{\mathbb{N}}\left( y_{n}\right) \Leftrightarrow
\forall n\in \mathbb{N}\text{, }x_{n}Ey_{n}\text{.}
\end{equation*}

\begin{definition}
An equivalence relation $E$ on a Polish space $X$ is \emph{idealistic }if
there is a Borel function $s:X\rightarrow X$ such that $xEs\left( x\right) $
for every $x\in X$, and an assignment $[x]_{E}\mapsto \mathcal{F}_{[x]_{E}}$
of a nontrivial $\sigma $-filter on $[x]_{E}$ to each $E$-class $[x]_{E}$,
that is Borel in the sense that for every Polish space $Z$ and every Borel
set $A\subseteq X\times Z\times X$ the set%
\begin{equation*}
\left\{ \left( x,z\right) \in X\times Z:\mathcal{F}_{[x]_{E}}y\text{, }%
\left( s\left( x\right) ,z,y\right) \in A\right\}
\end{equation*}%
is Borel, where%
\begin{equation*}
\mathcal{F}_{[x]_{E}}y\text{, }\left( s\left( x\right) ,z,y\right) \in
A\Leftrightarrow \left\{ y\in \lbrack x]_{E}:\left( s\left( x\right)
,z,y\right) \in E\right\} \in \mathcal{F}_{[x]_{E}}\text{.}
\end{equation*}%
We say that $E$ is \emph{stably idealistic }if the equivalence relation $E^{%
\mathbb{N}}$ on $X^{\mathbb{N}}$ defined by is idealistic.
\end{definition}

\begin{definition}
A \emph{Borel-definable set }is a $\boldsymbol{\Sigma }_{1}^{1}$-definable
set $X:=\boldsymbol{X}/E$ where $E$ is Borel and stably idealistic.
\end{definition}

It is proved in \cite[Section 3]{bergfalk_definable_2024} that
Borel-definable sets form an isomorphism-invariant full subcategory of $%
\boldsymbol{\Sigma }_{1}^{1}$-definable sets, which satisfies natural
properties that can be seen as analogues of corresponding elementary facts
concerning the category of sets. Furthermore, a function $f:X\rightarrow Y$
between Borel-definable sets is Borel if and only if its graph is a Borel
subset of $X\times Y$.

Every group with a Polish cover yields a Borel-definable set, as the
orbit-equivalence relation associated with a continuous action of a Polish
group on a Polish space is stably idealistic \cite[Proposition 5.4.19]%
{gao_invariant_2009}. The notion of\emph{\ complexity class} of the trivial
subgroup of a group with a Polish cover can be extended naturally to
Borel-definable sets.

\begin{definition}
Let $\Gamma $ be a complexity class of Borel subsets of Polish spaces. The
Borel-definable set $X$ is $\Gamma $-definable or potentially $\Gamma $ if
it is isomorphic in the category of Borel-definable sets to $Y=\boldsymbol{Y}%
/E$ where $E$ is an equivalence relation on $\boldsymbol{Y}$ such that $E\in
\Gamma \left( \boldsymbol{Y}\times \boldsymbol{Y}\right) $.
\end{definition}

If $\Gamma $ is a complexity class of Borel subsets of a Polish space, then
its \emph{dual class }$\check{\Gamma}$ is defined by letting, for every
Polish space $X$, $\check{\Gamma}\left( X\right) $ be the collection of 
\emph{complements }of elements of $\Gamma \left( X\right) $. The class $%
\Gamma $ is self-dual if $\Gamma =\check{\Gamma}$. If $\Gamma $ is \emph{not 
}self-dual, then we say that $\Gamma $ is the complexity class of a
Borel-definable set $X$ if $X$ is $\Gamma $-definable and not $\check{\Gamma}
$-definable.

\subsection{Orbit spaces}

Suppose that $G$ is a Polish group. A Polish $G$-space is a Polish space $%
\boldsymbol{X}$ endowed with a continuous action $G\curvearrowright X$, $%
\left( g,x\right) \mapsto g\cdot x$. The corresponding orbit equivalence
relation $E_{G}^{X}$ is always stably idealistic, although it might fail to
be Borel. Therefore, the space of $G$-orbits $\boldsymbol{X}/G$ is a $%
\boldsymbol{\Sigma }_{1}^{1}$-definable set, which is Borel-definable if and
only if $E_{G}^{X}$ is Borel. This holds if and only if there exists $\alpha
<\omega _{1}$ such that every $G$-orbit $Gx$ is $\boldsymbol{\Pi }_{\alpha
}^{0}$; see \cite[Therem 8.2.2]{gao_invariant_2009}.

More generally, suppose that $X=\boldsymbol{X}/E$ is a Borel-definable set,
and $G$ is a Polish group. A Borel action of $G$ on $X$ is an action that is
defined by a Borel function $G\times X\rightarrow X$. The corresponding
orbit equivalence relation $E_{G}^{X}$ is the equivalence relation on $%
\boldsymbol{X}$ defined by $xE_{G}^{X}y\Leftrightarrow g\cdot \lbrack
x]_{E}=[y]_{E}$. In this case, we say that $X$ is a Borel-definable $G$%
-space. The orbit space is the $\boldsymbol{\Sigma }_{1}^{1}$-definable set $%
X/G:=\boldsymbol{X}/E_{G}^{X}$. A similar proof as \cite[Therem 8.2.2]%
{gao_invariant_2009} gives the following.

\begin{proposition}
\label{Proposition:orbit-space}Suppose that $G$ is a Polish group, and $X$
is a Borel-definable $G$-space. Then the orbit equivalence relation $%
E_{G}^{X}$ is stably idealistic. Thus, if $E_{G}^{X}$ is also Borel, then $%
X/G$ is a Borel-definable set.
\end{proposition}

For example, given a Borel-definable set $X=\boldsymbol{X}/E$, one can
consider the Borel-definable set $\wp _{\aleph _{0}}\left( X\right) $ of
countable subsets of $X$. This can be seen as the orbit space $Y/S_{\infty }$%
, where $Y\subseteq X^{\mathbb{N}}$ is the $S_{\infty }$-invariant Borel set
of sequences that are injective modulo $E$ and the action $S_{\infty
}\curvearrowright X^{\mathbb{N}}$ is the shift. The corresponding
equivalence relation $F$ is the \emph{Friedman--Stanley jump }of the
equivalence relation $E$ \cite[Definition 8.3.1]{gao_invariant_2009}. In
particular, when $X=\mathbb{R}$ (or any other uncountable Polish space),
then $\wp _{\aleph _{0}}\left( \mathbb{R}\right) $ can be seen as the
definable set of \emph{countable sets of reals}.

\begin{definition}
A definable set $X$ can be \emph{parametrized by countable sets of reals} if
it admits a Borel bijection onto a Borel subset of $\wp _{\aleph _{0}}\left( 
\mathbb{R}\right) $.
\end{definition}

\subsection{Homogenous spaces with a Polish cover}

In Section \ref{Section:pro-countable} we have introduced the notion of a
module with a Polish cover. This is a module $M$ explicitly presented as a
quotient $M=\hat{M}/N$ where $\hat{M}$ is a Polish module and $N$ is a
Polishable submodule of $\hat{M}$ (not necessarily closed in $\hat{M}$).
More generally, one can consider a Borel-definable set $X$ presented as the
quotient $G/N$ of a (not necessarily commutative) Polish group $G$ by a (not
necessarily closed nor normal) Polishable subgroup $N$ of $G$. The fact that 
$N$ is a Polishable subgroup of $G$ means that $N$ is itself a Polish group
with respect to a topology that makes the inclusion $N\rightarrow G$
continuous. Such a Polish group topology on $N$ is unique, and it is
characterized by the fact that is Borel sets are precisely the Borel subsets
of $G$ contained in $N$. (In this sense, the Polish topology on $N$ is
\textquotedblleft induced\textquotedblright\ by the one of $G$, albeit not
in the sense of the subspace topology.) In particular, $N$ is a Borel subset
of $G$, whence the orbit equivalence relation associated with the left
translation action $N\curvearrowright G$ is Borel and stably idealistic.

Classically, a homogeneous space is a space $X$ endowed with a transitive a
action of a group $G$. Letting $N$ be the stabilizer of a fixed point of $X$%
, one obtains an identification of $X$ with the space $G/N$ of left $N$%
-cosets of $G$. This motivates the following definition:

\begin{definition}
A \emph{homogeneous space with a Polish cover} is a Borel-definable set $X$
of the form $\boldsymbol{X}/E$ where $\boldsymbol{X}$ is a Polish group $G$
and $E$ is the left coset relation of a Polishable subgroup $N$ of $G$. In
this case, we denote $X$ by $G/N$.

A\emph{\ phantom homogeneous Polish space} is a homogeneous space with a
Polish cover $G/N$ where $N$ is dense in $G$.
\end{definition}

If $G/N$ is a homogeneous space with a Polish cover, then a homogeneous
subspace with a Polish cover of $G/N$ is a subspace of the form $H/N$ where $%
H$ is a Polishable subgroup of $G$ containing $N$. Clearly, $H/N$ is itself
a homogeneous space with a Polish cover, such that the inclusion $%
H/N\rightarrow G/N$ is Borel. Likewise, the quotient $G/H=(G/N)/(H/N)$ is a
homogeneous space with a Polish cover. When $H$ is \emph{normal }in $G$, $%
G/H $ is actually a group, and in this case we call it a group with a Polish
cover. Plainly, every module with a Polish cover is, in particular, a group
with a Polish cover.

The results about complexity of modules with a Polish cover from Section \ref%
{Section:modules-with-polish-cover} holds more generally in the context of
homogeneous spaces with a Polish cover; see \cite{lupini_complexity_2024}.
The same holds for the notion and properties of Solecki submodules, which in
the context of homogeneous spaces with a Polish cover will be called Solecki
subspaces; see \cite{lupini_complexity_2024}. The notion of (plain) Solecki
length of a module with a Polish cover can be defined for homogeneous spaces
with a Polish cover in the same manner. Notice in particular that a
homogeneous space with a Polish cover $G/N$ has Solecki length $0$ if and
only if $N$ is closed in $G$, which is equivalent to the assertion that $G/N$
is a Polish space (with the quotient topology). Thus, the\emph{\ }%
homogeneous \emph{Polish }spaces can be thought of as the homogeneous spaces
with a Polish cover of Solecki length zero. If $G/N$ is a homogeneous space
with a Polish cover, and $H$ is the closure of $N$ in $G$, then $H/N$ is a
phantom homogeneous Polish space, and $G/H$ is a homogeneous Polish space.
We let $\rho \left( G/N\right) $ be the Solecki length of $G/N$, which is
equal to the Solecki length of $H/N$. If $\alpha :=\rho \left( G/N\right) $,
then we say that $G/N$ is:

\begin{itemize}
\item \emph{semiplain} if $\alpha $ is a successor and $s_{\alpha -1}\left(
G/N\right) $ is $D(\boldsymbol{\Sigma }_{2}^{0})$-definable, which is
equivalent to the assertion that it is $\boldsymbol{\Sigma }_{3}^{0}$%
-definable as shown in \cite{lupini_complexity_2024};

\item \emph{plain }if $\alpha $ is a successor and $s_{\alpha -1}\left(
G/N\right) $ is $\boldsymbol{\Sigma }_{2}^{0}$-definable.
\end{itemize}

Clearly, if $G/N$ is semiplain, then it is plain. The converse holds when $N$
is non-Archimedean a shown in \cite{lupini_complexity_2024}. Let us recall
the definition of the complexity classes $\Gamma _{\alpha }$, $\Gamma
_{\alpha }^{\mathrm{semiplain}}$, and $\Gamma _{\alpha }^{\mathrm{plain}}$
for $\alpha <\omega _{1}$ from Section \ref{Section:solecki-submodules}:%
\begin{equation*}
\Gamma _{\lambda +n}:=\left\{ 
\begin{array}{ll}
\boldsymbol{\Pi }_{1}^{0} & \text{for }\lambda =n=0\text{;} \\ 
\boldsymbol{\Pi }_{\lambda } & \text{for }n=0\text{ and }\lambda >0\text{;}
\\ 
\boldsymbol{\Pi }_{1+\lambda +n+1}^{0} & \text{for }\lambda >0\text{ and }n>0%
\text{;}%
\end{array}%
\right.
\end{equation*}%
and%
\begin{equation*}
\Gamma _{\lambda +n}^{\mathrm{semiplain}}:=\left\{ 
\begin{array}{ll}
\Gamma _{\lambda +n} & \text{for }n=0\text{;} \\ 
D(\boldsymbol{\Pi }_{1+\lambda +n}^{0}) & \text{for }n\geq 1\text{;}%
\end{array}%
\right.
\end{equation*}%
and%
\begin{equation*}
\Gamma _{\lambda +n}^{\mathrm{plain}}:=\left\{ 
\begin{array}{ll}
\Gamma _{\lambda +n}^{\mathrm{semiplain}} & \text{for }n\neq 1\text{;} \\ 
\boldsymbol{\Sigma }_{1+\lambda +1}^{0} & \text{for }n=1\text{.}%
\end{array}%
\right.
\end{equation*}
The following is a reformulation of \cite[Theorem 6.1]%
{lupini_complexity_2024}.

\begin{theorem}
\label{Theorem:complexity2}Suppose that $X$ is phantom homogeneous Polish
space.\ Let $\alpha $ be the Solecki length of $X$.

\begin{enumerate}
\item If $X$ is plain, then $\Gamma _{\alpha }^{\mathrm{plain}}$ is the
complexity class of $X$;

\item If $X$ is semiplain and not plain, then $\Gamma _{\alpha }^{\mathrm{%
semiplain}}$ is the complexity class of $X$;

\item If $X$ is not semiplain, then $\Gamma _{\alpha }$ is the complexity
class of $X$.
\end{enumerate}
\end{theorem}

\begin{lemma}
\label{Lemma:plain-extension}Suppose that Suppose that $G=\hat{G}/N$ is a
homogeneous space with a Polish cover and $H=\hat{H}/N$ a homogeneous
subspace with a Polish cover of $G$.

\begin{enumerate}
\item If $H$ has length at most $1$ and $G/H$ has plain length at most $1$,
then $G$ has length at most length $1$;

\item If $H$ has plain length at most $1$ and $G/H$ has plain length at most 
$1$, then $G$ has plain length at most $1$;

\item If $H$ has plain length at most $1$ and $G/H$ has semiplain length at
most $1$, then $G$ has semiplain length at most $1$.
\end{enumerate}
\end{lemma}

\begin{proof}
(1) By hypothesis, we have that $\hat{H}$ contains a closed neighborhood $V$
of the identity that is closed in $\hat{G}$. Since $H$ has length at most $1$%
, we have that $\hat{H}$ has a basis of neighborhoods of the identity $U$
contained in $V$ such that $\overline{U}^{\hat{H}}\cap N\subseteq U$. For
such a set $U$, since $U\subseteq V$ and $V$ is closed in $\hat{H}$, we have
that $\overline{U}^{\hat{H}}=\overline{U}^{\hat{G}}$, and hence $\overline{U}%
^{\hat{G}}\cap N\subseteq U$. This shows that $\hat{G}$ has length at most $%
1 $.

(2) By hypothesis, we have that $\hat{H}$ contains a closed neighborhood $V$
of the identity that is closed in $\hat{G}$. Likewise, $N$ contains a closed
neighborhood $W$ of the identity contained in $V$ that is closed in $\hat{H}$%
. This implies that $W$ is closed in $\hat{G}$, whence $G$ has plain length $%
1$.

(3) As in (2), $N$ contains a neighborhood $W$ of the identity that is
closed in $\hat{H}$. By \cite[Lemma 5.1]{lupini_complexity_2024}, since $%
\hat{H}$ is $D(\boldsymbol{\Pi }_{2}^{0})$ in $G$, we conclude that $W$ is $%
D(\boldsymbol{\Pi }_{2}^{0})$ in $\hat{G}$. Thus, $N$ is $\boldsymbol{\Sigma 
}_{3}^{0}$ in $\hat{G}$, and hence also $D(\boldsymbol{\Pi }_{2}^{0})$ by
Theorem \ref{Theorem:complexity2}.
\end{proof}

\begin{lemma}
\label{Lemma:complexity-exact-sequence}Suppose that $G=\hat{G}/N$ is a
homogeneous space with a Polish cover and $H=\hat{H}/N$ a homogeneous
subspace with a Polish cover of $G$.

\begin{enumerate}
\item We have that 
\begin{equation*}
\rho \left( G\right) \leq \rho \left( G/H\right) +\rho \left( H\right) \text{%
.}
\end{equation*}

\item If $\rho \left( G/H\right) $ is successor, and $G/H$ is plain, then%
\begin{equation*}
\rho \left( G\right) \leq \left( \rho \left( G/H\right) -1\right) +\rho
\left( H\right) \text{.}
\end{equation*}

\item If $H$ has plain length $1$ and $G/H$ has (semi)plain length at most $%
\gamma $, then $G$ has (semi)plain length at most $\gamma $;

\item If $H$ has (semi)plain length $1$ and $G/H$ has length at most $\gamma 
$, then $G$ has (semi)plain length at most $\gamma +1$;

\item If $H$ is countable, then $\left\{ 0\right\} $ is $\boldsymbol{\Sigma }%
_{1+\gamma }^{0}$ in $G$ if and only if $H$ is $\boldsymbol{\Sigma }%
_{1+\gamma }^{0}$ in $G$.
\end{enumerate}
\end{lemma}

\begin{proof}
Set $\alpha :=\rho \left( G/H\right) $ and $\beta :=\rho \left( H\right) $.

(1) We have that $H\in \Gamma _{\alpha }\left( G\right) $. Thus, $s_{\alpha
}\left( G\right) \subseteq H$. We also have $\left\{ 0\right\} \in \Gamma
_{\beta }\left( H\right) $. Thus, $\left\{ 0\right\} \in \Gamma _{\beta
}\left( s_{\alpha }\left( G\right) \right) $. Hence, $\left\{ 0\right\}
=s_{\beta }\left( s_{\alpha }\left( G\right) \right) =s_{\alpha +\beta
}\left( G\right) $. Thus shows that $\rho \left( G\right) \leq \alpha +\beta 
$.

(2) For $\alpha =1$ and $\beta =1$ the conclusion follows from Lemma \ref%
{Lemma:plain-extension}(1) By induction on $\beta $ one can then show that
the conclusion holds for $\alpha =1$ and every $\beta <\omega _{1}$. Let $L$
be the preimage of $s_{\alpha -1}\left( G/H\right) $ under the quotient map $%
G\rightarrow G/H$. The conclusion for arbitrary $\alpha $ follows from the
case $\alpha =1$ and Part (1) applied to the exact sequences%
\begin{equation*}
H\rightarrow L\rightarrow s_{\alpha -1}\left( G/H\right)
\end{equation*}%
and%
\begin{equation*}
L\rightarrow G\rightarrow G/L=\frac{G/H}{s_{\alpha -1}\left( G/H\right) }%
\text{.}
\end{equation*}

(3) Since $G/H$ has (semi)plain length at most $\gamma $, there is $%
H\subseteq K\subseteq G$ such that $G/K$ has length at most $\gamma -1$ and $%
K/H$ has (semi)plain length at most $1$. Thus, by Lemma \ref%
{Lemma:plain-extension}(2) and Lemma \ref{Lemma:plain-extension}(3), since $%
H $ has plain length $1$, $K$ has (semi)plain length at most $1$. Thus shows
that $G$ has (semi)plain length at most $\gamma $.

(4) Suppose that $H$ has plain length $1$, whence $\left\{ 0\right\} $ is $%
\boldsymbol{\Sigma }_{2}^{0}$ in $H$. Since $G/H$ has length at most $\gamma 
$, we have $s_{\gamma }\left( G\right) \subseteq H$. Thus, $\left\{
0\right\} $ is $\boldsymbol{\Sigma }_{2}^{0}$ in $s_{\gamma }\left( G\right) 
$, and $G$ has plain length at most $1$. The same proof shows that if $H$
has semiplain length $1$, then $G$ has semiplain length at most $1$.

(5) If $\left\{ 0\right\} $ is $\boldsymbol{\Sigma }_{1+\gamma }^{0}$ in $G$%
, then $H$ is $\boldsymbol{\Sigma }_{1+\gamma }^{0}$ in $G$ as a countable
union of translates of $\left\{ 0\right\} $. Conversely, suppose that $H$ is 
$\boldsymbol{\Sigma }_{1+\gamma }^{0}$ in $G$. We can write $\gamma =\lambda
+n$ for a limit ordinal $\lambda $ and $n<\omega $. If $n=0$ then we must
have that $G/H$ has length less than $\lambda $, whence the same holds for $%
G $. This implies that $\left\{ 0\right\} $ is $\boldsymbol{\Pi }_{1+\beta
}^{0}$ for some $\beta <\lambda $ and hence $\boldsymbol{\Sigma }_{1+\lambda
}^{0}$. If $n=1$, then we have that $G/H$ has plain length at most $\lambda
+1$, whence the same holds for $G$ by (3). Thus, we have that $\left\{
0\right\} $ is $\boldsymbol{\Sigma }_{1+\lambda +1}^{0}$. If $n\geq 2$ then
we have that $H$ is $\boldsymbol{\Pi }_{1+\lambda +n-1}^{0}$ in $G$, and $%
G/H $ has length at most $1+\lambda +n-2$. Therefore, we have that $G$ has
plain length at most $1+\lambda +n-1$.\ This shows that $\left\{ 0\right\} $
is $D(\boldsymbol{\Pi }_{1+\lambda +n-1}^{0})$ and hence also $\boldsymbol{%
\Sigma }_{1+\lambda +n}^{0}$ in $G$. This concludes the proof.
\end{proof}

\subsection{Strict actions}

Given a homogeneous space with a Polish cover $G/H$ and a Polish group $%
\Lambda $, we define a notion of left \emph{strict action }of $\Lambda $ on $%
G/H$.

\begin{definition}
A left strict action $\alpha $ of a Polish group $\Lambda $ on a homogeneous
space with a Polish cover $X=G/H$ is a given by a continuous action $\alpha $
of $\Lambda $ on $G$ that leaves $H$ setwise invariant. We define the space $%
X/\Lambda =\mathrm{Out}\left( \Lambda \curvearrowright G/H\right) $ of
orbits of the action to be the $\boldsymbol{\Sigma }_{1}^{1}$-definable set $%
X=\boldsymbol{X}/E$ where $\boldsymbol{X}=G$ and 
\begin{equation*}
E=\left\{ \left( g,g^{\prime }\right) \in G\times G:\exists \lambda \in
\Lambda \text{, }\lambda Hg=Hg^{\prime }\right\} \text{.}
\end{equation*}
\end{definition}

Suppose that $M$ is a countable $R$-module. Let $\mathrm{\mathrm{Aut}}\left(
M\right) $ be the group of automorphisms of $M$ endowed with the topology of
pointwise convergence. Then we have that $\mathrm{Aut}\left( M\right) $ is a
non-Archimedean Polish group. The automorphisms $\mu _{r}:x\mapsto rx$ for $%
r\in R^{\times }$, which we call \emph{inner}, form a central closed subgroup%
\emph{\ }denoted by $\mathrm{Inn}\left( M\right) $.

Suppose now that $C$ and $A$ are countable thin modules. Let $\mathrm{Aut}%
\left( C\right) $ be the automorphism group of $C$, which is a Polish group
with respect to the topology of pointwise convergence. Then one can define
the canonical right action $\mathrm{Ext}\left( C,A\right) \curvearrowleft 
\mathrm{\mathrm{\mathrm{Au}}t}\left( C\right) $ as follows. Up to
isomorphism of modules with a Polish cover, we can write $\mathrm{Ext}\left(
C,A\right) $ as $\mathrm{Z}\left( C,A\right) /\mathrm{B}\left( C,A\right) $.
Here, as in Section \ref{Section:pro-countable}, $\mathrm{Z}\left(
C,A\right) $ is the Polish group of cocycles on $C$ with values in $A$,
while $\mathrm{B}\left( C,A\right) $ is the Polishable subgroup of cocycles
that are trivial. We have a canonical action $\mathrm{Z}\left( C,A\right)
\curvearrowleft \mathrm{\mathrm{\mathrm{Au}}t}\left( C\right) $ by
automorphisms that leave $\mathrm{B}\left( C,A\right) $ (setwise) invariant.
This defines a strict right action $\mathrm{Ext}\left( C,A\right)
\curvearrowleft \mathrm{\mathrm{Aut}}\left( C\right) $. In the same fashion,
one can define a strict left action of $\mathrm{\mathrm{Aut}}\left( A\right) 
$ on $\mathrm{Ext}\left( C,A\right) $.

Notice that if $M$ is a phantom pro-finiteflat module, then by Proposition %
\ref{Proposition:duality-Ext} we have that $M\cong \mathrm{Ext}\left(
M^{\vee },R\right) $ where $M^{\vee }:=\mathrm{Ext}\left( M,R\right) $ is a
countable flat module. Furthermore, $\mathrm{\mathrm{Aut}}\left( M\right) $
is a Polish group canonically isomorphic to $\mathrm{Aut}\left( M^{\vee
}\right) ^{\mathrm{op}}$, in such a way that the canonical left action $%
\mathrm{Aut}\left( M\right) \curvearrowright M$ corresponds to the canonical
right action $\mathrm{Ext}\left( M^{\vee },R\right) \curvearrowleft \mathrm{%
\mathrm{Aut}}\left( M^{\vee }\right) $. Thus, the action $\mathrm{Aut}\left(
M\right) \curvearrowright M$ can be seen as a strict action of the Polish
group $\mathrm{Aut}\left( M\right) $ on the homogeneous space with a Polish
cover $M$. Incidentally, this also shows that every phantom pro-finiteflat
module is \emph{rigid }in the sense of \cite[Definition 4.2]%
{bergfalk_definable_2024}.

\begin{lemma}
\label{Lemma:complexity-class-countable-action}Suppose that $M=G/N$ is a
module with a Polish cover. For $\alpha <\omega _{1}$ and $g\in G$, the
following assertions are equivalent:

\begin{enumerate}
\item $N$ is $\boldsymbol{\Sigma }_{1+\alpha }^{0}$ in $G$;

\item the relation 
\begin{equation*}
\left\{ \left( x,y\right) \in G:\exists r\in R^{\times },rx\equiv y\ \mathrm{%
mod}N\right\}
\end{equation*}%
is potentially $\boldsymbol{\Sigma }_{1+\alpha }^{0}$;

\item $N+R^{\times }g$ is $\boldsymbol{\Sigma }_{1+\alpha }^{0}$ in $G$ for
every $g\in G\setminus N$;

\item $N+Rg$ is $\boldsymbol{\Sigma }_{1+\alpha }^{0}$ in $G$ for some $g\in
G$;
\end{enumerate}
\end{lemma}

\begin{proof}
(1)$\Rightarrow $(2) and (3)$\Rightarrow $(4) are obvious. (2)$\Rightarrow $%
(3) follows from \cite[Proposition 3.1]{lupini_complexity_2024}. The
equivalence of (1) and (4) follows from Lemma \ref%
{Lemma:complexity-exact-sequence}(5).
\end{proof}

\subsection{Parametrization by hereditarily countable sets}

The following lemmas are essentially \cite[Lemma 3.7]{kechris_borel_2016}
and \cite[Proposition 2.3]{motto_ros_complexity_2012}.

\begin{lemma}[Kechris--Macdonald]
\label{Lemma:Kechris-Macdonald}Suppose that $E,F$ are equivalence relations
on Polish spaces $\boldsymbol{X},\boldsymbol{Y}$, respectively. If $E$ is
idealistic, $F$ is Borel, and $E$ is Borel reducible to $F$, then there
exists a Borel $F$-invariant subset $\boldsymbol{A}$ of $\boldsymbol{Y}$
such that $E$ is Borel bireducible to $F|_{\boldsymbol{A}}$.
\end{lemma}

In particular, it follows that a Borel injection $X\rightarrow Y$ between
Borel-definable sets has Borel image.

\begin{lemma}[Motto Ros]
Suppose that $X$ and $Y$ are Borel-definable sets.\ If there exist a
Borel-injection $X\rightarrow Y$ and a Borel-injection $Y\rightarrow X$,
then there exists a Borel bijection between $X$ and $Y$.
\end{lemma}

Combining these with \cite[Theorem 1.5]{friedman_borel_2000} we obtain the
following.

\begin{lemma}[Friedman]
\label{Lemma:Friedman}Suppose that $X=\boldsymbol{X}/E$ is a Borel-definable
set. Then $E$ is classifiable by countable structures if and only if $X$ is
Borel-isomorphic to an orbit space $\boldsymbol{Y}/G$ for some
non-Archimedean Polish group $G$ and Polish $G$-space $\boldsymbol{Y}$. In
fact, one can take $G=S_{\infty }$.
\end{lemma}

A canonical hierarchy of Borel $S_{\infty }$-orbit spaces is considered in 
\cite{hjorth_borel_1998}. One defines recursively $\wp ^{0}\left( \mathbb{N}%
\right) :=\mathbb{N}$, $\wp ^{<\alpha }\left( \mathbb{N}\right) $ to be the
union of $\wp ^{\beta }\left( \mathbb{N}\right) $ for $\beta <\alpha $, and $%
\wp ^{\alpha }\left( \mathbb{N}\right) $ to be the collection of countable
subsets of $\wp ^{<\alpha }\left( \mathbb{N}\right) $.\ Thus for example $%
\wp ^{1}\left( \mathbb{N}\right) $ comprises all the subsets of $\mathbb{N}$%
, whence $\wp ^{1}\left( \mathbb{N}\right) \cong \mathbb{R}$. In turn, $\wp
^{2}\left( \mathbb{N}\right) \cong \wp _{\aleph _{0}}\left( \mathbb{R}%
\right) $ is the collection of all countable sets of reals. More generally,
one can think of $\wp ^{\alpha }\left( \mathbb{N}\right) $ as the collection
of \emph{hereditarily countable sets of rank} $\alpha $. These can be
parametrized by the isomorphism classes of the countable models of some
infinitary sentence $\sigma _{\alpha }$ in a suitable countable language $%
\mathcal{L}_{\alpha }$. One lets $\cong _{\alpha }$ be the relation of
isomorphism of such countable $\mathcal{L}_{\alpha }$-structures, which is a
Borel equivalence relation \cite[Section 1]{hjorth_borel_1998}. The relation 
$\cong _{1}$ recovers the relation of equality of reals, while $\cong _{2}$
coincides with the Friedman--Stanly jump of $\cong _{1}$. Thus,

\begin{lemma}
Let $X$ be a $\boldsymbol{\Sigma }_{1}^{1}$-definable set. We say that $X$
is parametrized by\emph{\ hereditarily countable sets of rank }$\alpha $ if
it Borel-isomorphic to a Borel subset of $\wp ^{\alpha }\left( \mathbb{N}%
\right) $.
\end{lemma}

Note that this implies that $X$ is in fact a Borel-definable set.

In \cite{hjorth_borel_1998} it is also defined for a countable successor
ordinal $\alpha \geq 3$ a Borel subset $\wp _{\ast }^{\alpha }\left( \mathbb{%
N}\right) $ of $\wp ^{\alpha }\left( \mathbb{N}\right) $. Again, $\wp _{\ast
}^{a}\left( \mathbb{N}\right) $ can be parametrized as the isomorphism
classes of countable modules of some infinitary $\mathcal{L}_{\alpha }$%
-sentence $\sigma _{\alpha }^{\ast }$. One then lets $\cong _{\alpha }^{\ast
}$ be the corresponding isomorphism relation, which can be regarded as a
subequivalence relation of $\cong _{\alpha }$. We say that $\wp _{\ast
}^{\alpha }\left( \mathbb{N}\right) $ comprises the\emph{\ }hereditary
countable sets of \emph{plain }rank $\alpha $. For $\alpha =2$ we let $\wp
_{\ast }^{2}\left( \mathbb{N}\right) $ be the Borel-definable set associated
with the countable Borel equivalence relation of maximum complexity, while
for $\alpha =1$ we let $\wp _{\ast }^{1}\left( \mathbb{N}\right) =\wp
^{1}\left( \mathbb{N}\right) $. For a limit ordinal $\lambda $ we define $%
\wp _{\ast }^{\lambda }\left( \mathbb{N}\right) =\wp ^{\lambda }\left( 
\mathbb{N}\right) $. Combining Lemma \ref{Lemma:Friedman} with \cite[Theorem
2 and Theorem 3]{hjorth_borel_1998} we obtain the following.

\begin{theorem}[Hjorth--Kechris--Louveau]
\label{Theorem:HKL}Let $X:=\boldsymbol{X}/E$ be a $\boldsymbol{\Sigma }%
_{1}^{1}$-definable set, where $\boldsymbol{X}$ is a Polish space $E$ be an
idealistic equivalence relation classifiable by countable structures. Then $%
X $ is a Borel-definable set if and only if there exists $\alpha <\omega
_{1} $ such that $X\ $is parametrized by countable sets of rank $\alpha $
for some $\alpha <\omega _{1}$.

Furthermore, for every $\alpha <\omega _{1}$ the following assertions are
equivalent:

\begin{enumerate}
\item $X$ is $\Gamma _{\alpha }$-definable if and only if it is parametrized
by hereditarily countable sets of rank $1+\alpha $;

\item $X$ is $\Gamma _{\alpha }^{\mathrm{plain}}$-definable if and only if
it is parametrized by hereditarily countable sets of plain rank $1+\alpha $.
\end{enumerate}
\end{theorem}

\subsection{The classification problem for extensions}

We conclude by isolating the complexity-theoretic content of our main
results, in terms of the classification problem for extensions of countable
flat modules.

\begin{theorem}
\label{Theorem:parametrize-extensions}Let $R$ be a countable Dedekind domain
and $A$ be a countable flat module. Then for every $\alpha <\omega _{1}$
there exist countable flat modules $C_{\alpha }$ and $C_{\alpha }^{\mathrm{%
plain}}$ such that the potential Borel complexity class of the isomorphism
relation of extension of $C_{\alpha }$ (respectively, $C_{\alpha }^{\mathrm{%
plain}}$) by $A$ is $\Gamma _{\alpha }$ (respectively, $\Gamma _{\alpha }^{%
\mathrm{plain}}$). Furthermore, for every $\lambda <\omega _{1}$ either zero
or limit and for every $n<\omega $ we have that:

\begin{itemize}
\item for $n=0$, isomorphism classes of extensions of $C_{\lambda }$ by $A$
can be parametrized by hereditarily countable sets of rank $1+\lambda $ but
not by hereditarily countable sets of rank $1+\beta $ for any $\beta
<\lambda $;

\item for $n\geq 1$, isomorphism classes of extensions of $C_{\lambda +n}$
by $A$ for $n\geq 1$ can be parametrized by hereditarily countable sets of
rank $1+\lambda +n$ but not by hereditarily countable sets of \emph{plain}
rank $1+\lambda +n$;

\item for $n\geq 1$, isomorphism classes of extensions of $C_{\lambda +n}^{%
\mathrm{plain}}$ by $A$ for $n\geq 1$ can be parametrized by hereditarily
countable sets of \emph{plain} rank $1+\lambda +n$ but not by hereditarily
countable sets of rank $1+\lambda +n-1$.
\end{itemize}
\end{theorem}

\begin{proof}
This follows from Theorem \ref{Theorem:complexity2}, Theorem \ref%
{Theorem:dichotomy}, and Theorem \ref{Theorem:HKL}.
\end{proof}

\begin{corollary}
\label{Theorem:dichotomy-extensions}Let $R$ be a countable Dedekind domain.
Suppose that $A$ is a countable flat module. Then either:

\begin{enumerate}
\item for all countable flat module $C$, the relation of isomorphisms of
extensions of $C$ by $A$ is trivial (which happens precisely when $A$ is
divisible), or

\item there exist countable flat modules $C$ such that the relation of
isomorphisms of extensions of $C$ by $A$ has arbitrarily high potential
Borel complexity.
\end{enumerate}
\end{corollary}

\begin{proof}
This is a consequence of Theorem \ref{Theorem:complexity2} and Theorem \ref%
{Theorem:dichotomy}.
\end{proof}

\begin{theorem}
Suppose that $R$ is a countable Dedekind domain and $A$ is a countable flat
module that is not divisible. Then for every $\lambda +n<\omega _{1}$ with $%
\lambda $ limit and $n<\omega $ there exists a rigid countable flat module $%
C $ of plain extractable length and plain $A$-projective length $\lambda +n$%
. For such a module $C$, we have that%
\begin{equation*}
\mathrm{Out}\left( \mathrm{Ext}\left( C,A\right) \curvearrowleft \mathrm{Aut}%
\left( C\right) \right)
\end{equation*}%
is $\boldsymbol{\Sigma }_{\lambda +1}^{0}$-definable and not $\boldsymbol{%
\Sigma }_{\lambda }^{0}$-definable for $n=0$, and $\boldsymbol{\Sigma }%
_{1+\lambda +n+2}^{0}$-definable and not $\boldsymbol{\Sigma }_{1+\lambda
+n+1}^{0}$-definable for $n>0$.
\end{theorem}

\begin{proof}
This follows from Theorem \ref{Theorem:complexity2}, Theorem \ref%
{Theorem:dichotomy}, and Lemma \ref{Lemma:complexity-class-countable-action}.
\end{proof}

\begin{theorem}
Suppose that $R$ is a countable Dedekind domain. Then there exists a
countable flat module $C$ such that the isomorphism classes of extensions of 
$C$ by $R$ can be parametrized by countable sets of reals, and for every $%
\alpha <\omega _{1}$ there exist countable flat modules $A_{\alpha }$ and $%
A_{\alpha }^{\mathrm{plain}}$ such that the potential Borel complexity class
of the isomorphism relation of extension of $C$ by $A_{\alpha }$
(respectively, by $A_{\alpha }^{\mathrm{plain}}$) is $\Gamma _{\alpha }$
(respectively, $\Gamma _{\alpha }^{\mathrm{plain}}$).
\end{theorem}

\begin{proof}
This follows from Theorem \ref{Theorem:not-constructible} and Theorem \ref%
{Theorem:complexity2}.
\end{proof}

\subsection{Complexity of classes of modules}

We have defined $\mathbf{Flat}\left( R\right) $ to be the category of
countable flat modules. The category $\mathbf{Flat}\left( R\right) $ is
(equivalent to) a category whose objects can be seen as points of a Polish
space $\mathrm{Flat}\left( R\right) $. Indeed, a countable flat module is
isomorphic to one whose set of elements is the set $\mathbb{N}$ of natural
numbers. Its module operation is then a subset of $\mathbb{N}\times \mathbb{N%
}$, whence the space of such modules can be seen as a subspace of $2^{%
\mathbb{N}\times \mathbb{N}}$. This subspace is easily seen to be $G_{\delta
}$, whence Polish with the subspace topology \cite[Theorem 3.11]%
{kechris_classical_1995}.

It is therefore meaningful to talk about the Borel complexity of
isomorphism-invariant classes of countable flat modules, identified as
subspaces of $\mathrm{Flat}\left( R\right) $.

\begin{theorem}
\label{Therem:complexity-classes}Let $R$ be a countable Dededind domain.

\begin{enumerate}
\item The class of coreduced countable flat modules is co-analytic and not
Borel;

\item For every $\alpha <\omega _{1}$, the class $\mathcal{P}_{\alpha }$ of
countable flat modules of projective length at most $\alpha $ is Borel;

\item The class of $\mathcal{P}_{\infty }$ of countable flat phantom
projective modules, which is the union of $\mathcal{P}_{\alpha }$ for $%
\alpha <\omega _{1}$, is co-analytic not Borel.
\end{enumerate}
\end{theorem}

\begin{proof}
(1) Consider the co-analytic rank $\varphi $ on $\mathrm{Flat}\left(
R\right) $ defined by letting $\varphi \left( C\right) $ be the least $%
\alpha $ such that $\partial _{\alpha }C=0$ if it exists, and $\infty $
otherwise. Then we have that $\varphi $ is a co-analytic rank that is
unbounded on the class of coreduced countable flat modules by Theorem \ref%
{Theorem:dichotomy}. The conclusion follows from the Boundedness Theorem for
Coanalytic Ranks on Borel sets; see Proposition \ref%
{Proposition:bounded-rank}.

(2) A countable flat module $C$ has projective length at most $\alpha $ if
and only if $\mathrm{Ph}^{\alpha }\mathrm{Ext}\left( C,A\right) =0$ for
every countable flat module $A$. This shows that $\mathcal{P}_{\alpha }$ is
co-analytic. We also have that $C$ has projective length at most $\alpha $
if and only if it has tree length at most $\alpha $; see Corollary \ref%
{Corollary:alpha-projective}. The class of such modules is easily seen to be
analytic by induction on $\alpha $. Being both co-analytic and analytic, $%
\mathcal{P}_{\alpha }$ is Borel.

(3) A countable flat module $C$ has projective length at most $\alpha $ if
and only if $\mathrm{Ph}^{\alpha }\mathrm{Ext}\left( C,A\right) =0$ for
every countable flat module $A$. Thus defining $\varphi \left( C\right) $ to
be the least $\alpha $ such that $C$ is in $\mathcal{P}_{\alpha }$ and $%
\infty $ otherwise, defines a co-analytic rank on $\mathbf{Flat}\left(
R\right) $. As such a rank is unbounded by Theorem \ref{Theorem:dichotomy},
the class $\mathcal{P}_{\infty }$ is not Borel.

In order to see that $\mathcal{P}_{\infty }$ is co-analytic, it suffices to
observe that its complement $\mathcal{C}$ within the class of countable flat
modules is analytic. Indeed, a countable flat module $C$ is $\mathcal{C}$ if
and only if for every countable ordinal $\alpha $ there exist a countable
flat module $A$ and a nontrivial extension of $C$ by $A$ that is phantom of
order $\alpha $. As in the proof of Proposition \ref{Proposition:rank} one
can write this as an analytic condition on $C$.
\end{proof}

\bibliographystyle{amsalpha}
\bibliography{bibliography}

\end{document}